\documentclass[12pt]{amsart}
\usepackage{amsmath, amsthm, amssymb, eqnarray}
\usepackage{accents}
\usepackage{tikz-cd}

\input xy
\xyoption{all}

\usepackage{enumitem} 
\usepackage{amsfonts, calligra, amscd, bbm, eucal, latexsym, extarrows, mathrsfs}
\usepackage[utf8]{inputenc}
\usepackage[T1]{fontenc}
\usepackage[bookmarks=false]{hyperref}

\usepackage{fullpage}
\usepackage[all,cmtip]{xy}
\usepackage{microtype}
\usepackage{cancel}
\usepackage{fourier}
\usepackage{graphicx}
\usepackage{soul}

\newcommand{\nref}[2]{\hyperref[#1]{\ref*{#1}$_{#2}$}}

\usepackage{mathtools}

\theoremstyle{plain}
\newtheorem{theorem}{Theorem}[section]
\newtheorem{proposition-definition}{Proposition/Definition}[section]
\newtheorem{theorem-definition}{Theorem/Definition}[section]

\newtheorem*{proposition-definition-intro}{Proposition/Definition}
\newtheorem*{definition-intro}{Definition}
\newtheorem*{cor-intro}{Corollary}
\newtheorem*{prop-intro}{Proposition}
\newtheorem{proposition}[theorem]{Proposition}		
\newtheorem{corollary}[theorem]{Corollary}
\newtheorem{lemma}[theorem]{Lemma}
\newtheorem*{conjecture}{Conjecture}
\newtheorem*{conjecture-intro}{Conjecture}

\newtheorem{definition}[theorem]{Definition}
\newtheorem{construction-definition}[theorem]{Construction/Definition}

\newtheorem*{problem2prime}{Problem 2'}
\newtheorem*{theoremMeta}{Main theorem}
\newtheorem{thm}{Theorem}

\newtheorem{prop}[thm]{Proposition}
\newtheorem{defthm}[thm]{Definition / Theorem}

\newtheorem{problem-Deligne}{Problem}

\theoremstyle{remark}
\newtheorem{remark}[theorem]{Remark}
\newtheorem*{remark-intro}{Remark}

\newtheorem*{construction}{Construction}

\newcommand{\Abold}{{\bf A}}
\newcommand{\Bbold}{{\bf B}}
\newcommand{\Cbold}{{\bf C}}
\newcommand{\Dbold}{{\bf D}}
\newcommand{\Ebold}{{\bf E}}

\newcommand{\Kbold}{{\bf K}}

\newcommand{\Pbold}{{\bf P}}

\newcommand{\ov}{\overline}

\newcommand{\NBbb}{\mathbb N}

\newcommand{\PBbb}{\mathbb P}
\newcommand{\QBbb}{\mathbb Q}

\newcommand{\ZBbb}{\mathbb Z}

\newcommand{\Acal}{\mathcal A}
\newcommand{\Bcal}{\mathcal B}
\newcommand{\Ccal}{\mathcal C}
\newcommand{\Dcal}{\mathcal D}

\newcommand{\Gcal}{\mathcal G}

\newcommand{\Ical}{\mathcal I}

\newcommand{\Ocal}{\mathcal O}
\newcommand{\Pcal}{\mathcal P}

\newcommand{\Rcal}{\mathcal R}
\newcommand{\Scal}{\mathcal S}

\newcommand{\Ucal}{\mathcal U}
\newcommand{\Vcal}{\mathcal V}
\newcommand{\Wcal}{\mathcal W}

\newcommand{\Cfrak}{\mathfrak C}

\newcommand{\cfrak}{\mathfrak c}
\newcommand{\chfrak}{\mathfrak {ch}}

\newcommand{\sfrak}{\mathfrak s}

\newcommand{\tdfrak}{\mathfrak {td}}

\newcommand{\bis}{\prime\prime}
\newcommand{\tris}{\prime\prime\prime}

\newcommand{\fppf}{\emph{fppf}}

\providecommand{\td}{\operatorname{td}}
\providecommand{\ch}{\operatorname{ch}}
\providecommand{\CH}{\operatorname{CH}}
\providecommand{\ch}{\operatorname{ch}}

\providecommand{\rk}{\operatorname{rk}}
\providecommand{\Gr}{\operatorname{Gr}}

\providecommand{\Spec}{\operatorname{Spec}}

\DeclareMathOperator{\Hom}{Hom}
\DeclareMathOperator{\SheafHom}{\mathscr{H}\text{\kern -3pt {\calligra\large om}}\,}

\DeclareMathOperator{\Aut}{Aut}

\DeclareMathOperator{\Vect}{Vect}
\DeclareMathOperator{\id}{id}

\DeclareMathOperator{\Sym}{Sym}

\DeclareMathOperator{\Picfr}{\mathfrak{Pic}}
\newcommand{\Pic}{{\rm Pic}}

\DeclareMathOperator{\Imag}{Im}

\DeclareMathOperator{\iso}{iso}
\DeclareMathOperator{\qiso}{quasi-iso}

\DeclareMathOperator{\CHfrak}{\mathfrak{CH}}

\numberwithin{equation}{section}

\begin{document}
\setcounter{tocdepth}{1}
\title{ Deligne--Riemann--Roch and Intersection bundles}
\author{Dennis Eriksson}
\author{Gerard Freixas i Montplet}

\address{Dennis Eriksson \\ Department of Mathematics \\ Chalmers University of Technology and  University of Gothenburg}
\email{dener@chalmers.se}

\address{Gerard Freixas i Montplet \\ CNRS -- Centre de Math\'ematiques Laurent Schwartz - \'Ecole Polytechnique - Institut Polytechnique de Paris}
\email{gerard.freixas@polytechnique.edu}



\maketitle
\begin{abstract}
This article is part of a series of works by the authors with the goal of completing a far-reaching program propounded by Deligne, aiming to extend the codimension one part of the Grothendieck--Riemann--Roch theorem from isomorphism classes of line bundles to isomorphisms thereof. The paper develops a relative functorial intersection theory with values in line bundles, together with a formalism that generalizes previous constructions by Deligne and Elkik, related to the right-hand side of the theorem.

\end{abstract}

\tableofcontents

\newpage

\section{Introduction}
This is the first of a series of papers with the purpose of completing a far-reaching program propounded by Deligne in the foundational article \emph{Le d\'eterminant de la cohomologie} \cite{Deligne-determinant}, aiming to lift the Grothendieck--Riemann--Roch theorem to the level of line bundles. The current article generalizes previous constructions by Elkik \cite{Elkikfib} and others related to the right-hand side of the theorem, and develops a formalism that is suitable for adapting to this setting the standard proof of factoring a projective morphism in terms of a closed immersion and a projective bundle. It can be recast as a relative functorial intersection theory valued in line bundles, which addresses Problem \ref{problem1} below and is of independent interest. 

\subsection{The program of Deligne}
Recall that, for a vector bundle $E$ on a projective variety $X$, the determinant of the cohomology is the line
\begin{displaymath}
    \det H^{\bullet}(X,E) = \bigotimes \det H^i(X,E)^{(-1)^i}.
\end{displaymath}
More generally, the Knudsen--Mumford determinant extends the above construction to the families setting \cite{KnudsenMumford}. For a flat projective morphism of schemes\footnote{For the clarity of exposition, in the introduction we may oversimplify some geometric requirements.} $f\colon X\to S$ and a vector bundle $E$ on $X$, it furnishes a line bundle $\lambda_{f}(E)$ on $S$, with fibers $\det H^{\bullet}(X_{s},E|_{X_{s}})$. If the family $X\to S$ is moreover a local complete intersection, the isomorphism class of $\lambda_{f}(E)$, as a $\QBbb$-line bundle, is determined by the Grothendieck--Riemann--Roch theorem \cite{SGA6}, which expresses the first Chern class as 
\begin{equation}\label{eq:GRR}
    c_1(\lambda_{f}(E))  =  f_\ast \left(\ch(E) \cdot \td^{\ast}(\Omega_f) \right)^{(1)}.
\end{equation}
Here, the right-hand side is given by products of characteristic classes in the Chow ring, and the codimension one part of the direct image. 

With this in mind, the program in \cite{Deligne-determinant} takes the form of several problems. 

\begin{problem-Deligne}  [\textbf{Integrals of Chern classes}]\label{problem1}
 Develop a theory of integrals of Chern classes in terms of line bundles, with properties analogous to the corresponding ones from Chow theory. 
\end{problem-Deligne} 

In particular, this should allow one to represent the right-hand side of the Grothendieck--Riemann--Roch theorem in \eqref{eq:GRR} in terms of line bundles. Since the left-hand side is represented by $\lambda_f(E)$, given a positive solution to this first problem, one can pose the following:

\begin{problem-Deligne} [\textbf{Deligne--Riemann--Roch isomorphism}]\label{problem2}
Construct and characterize a canonical isomorphism between the $\QBbb$-line bundles representing the two sides of \eqref{eq:GRR}, and determine the denominator of the isomorphism. 
\end{problem-Deligne}

Finally, with a view toward a refined formulation of the Riemann--Roch theorem in Arakelov geometry, one can consider the complex geometric situation with hermitian metrics. In this setting, the bundle $\lambda_f(E)$ is equipped with the Quillen metric. Then, Deligne proposes: 
\begin{problem-Deligne}[\textbf{Analytification and metrics}]\label{problem3}
Extend the integrals of Chern classes to the hermitian setting, and compute the norm of the Deligne--Riemann--Roch isomorphism.
\end{problem-Deligne}

Berkovich analogs of the last problem have also proven to be interesting, for instance in \cite{Boucksom-Eriksson, SebWalMarDiff}, where these circles of ideas are applied to non-Archimedean pluripotential theory. One should hence seek extensions of this program to these spaces and situations. 

\subsection{Formulation of the main theorem} All the problems were solved, in an essentially unique way, for families of smooth curves, by Deligne himself in \cite{Deligne-determinant}, who for the analytical part relied on the work of Bismut and collaborators \cite{freed1, freed2, BGS1, BGS2, BGS3}. These results were also germinal for higher dimensional developments in Arakelov theory, and actually the arithmetic Riemann--Roch theorem of Gillet--Soul\'e \cite{GS-ARR} should follow from Problem \ref{problem3} by passing to isometry classes. In this sense, Deligne's program is a refinement of the work of  Gillet--Soul\'e.

The  program has attracted the attention of several authors, from which we highlight two of historical importance. Elkik, in a series of papers \cite{Elkikfib, Elkikmetr}, introduced a theory of integrals of Chern classes that she called intersection bundles, and also equipped them with metrics. On a related note, also see the references \cite{Boucksom-Eriksson, Ducrot, Munoz}, where these problems are considered in various generalities. Independently, in another series of papers \cite{FrankeChow, FrankeChern, Franke}, Franke studied the first two problems within a Chow homology-type formalism. The last paper gives, in this setting, a proof of the second problem, but remains unpublished. 

Notwithstanding these contributions, the Deligne program, in the whole generality, is still open. For example, the approach  by Franke demands certain regularity assumptions on the base schemes which are not satisfied for most moduli spaces, such as Hilbert schemes. The approach by Elkik makes Noetherian assumptions on the base, which are not fulfilled in many situations in non-Archimedean geometry. Most importantly, neither of these works nor their sequels, develop a formalism that faces the interaction between various properties of intersection bundles. 

In practical terms, this means in particular that one lacks functoriality properties which are fundamental to tackle a Deligne--Riemann--Roch isomorphism along the lines of the classical proof by deformation to the normal cone. To this end, our objective in this article is to expand upon Elkik's perspective and ultimately solve Problem \ref{problem1}. To accomplish this, we further draw inspiration from Franke's Chow formalism and introduce a category of 
Chern classes that encodes the fundamental functorial properties of Elkik's intersection bundles. Unlike Franke's categories, our approach is more manageable and specifically designed to address the limitations mentioned earlier.

To this effect, the main result of our article can be expressed as follows: 

\begin{theoremMeta} [Informal]
   For projective and flat morphisms over general base schemes, there is a relative functorial intersection theory, recovering Elkik's intersection bundles in terms of direct images. 
\end{theoremMeta}

In the rest of the introduction, we summarize the main lines of our work and the meaning of this statement.

\subsection{Virtual categories and Chern categories }

Let $f: X\to Y$ be a projective local complete intersection morphism. The classical formulation of the Grothendieck--Riemann--Roch theorem \cite{SGA6} states that, under appropriate assumptions, there is a commutative diagram

\begin{equation}\label{eq:GRR-diagram}
    \xymatrix{
        K_0(X) \ar[rr]^{\ch(-) \td^{\ast}(\Omega_f)} \ar[d]_{f_!} & & \CH^{\ast}(X)_\QBbb \ar[d]^{f_\ast} \\
        K_0(Y) \ar[rr]^{\ch(-)} & & \CH^{\ast}(Y)_\QBbb,
    }
\end{equation}
where $\CH^{\ast}$, here and elsewhere, stands for Grothendieck's Chow groups defined in terms of the $\gamma$-filtration of $K$-theory.

\subsubsection*{Virtual categories} Recall that, according to the homotopical tradition initiated by Quillen \cite{Quillen:K-theory-I}, the $K_0$-group is the $\pi_0$ of a $K$-theory spectrum. Hence, if we would be willing to categorify the Grothendieck--Riemann--Roch theorem in all degrees, we would undoubtedly be led to work with the entire $K$-theory spectrum and the recent developments on $\infty$-categories. 

Following Deligne \cite{Deligne-determinant}, for the categorification of the codimension one part of the \linebreak Grothendieck--Riemann--Roch theorem \eqref{eq:GRR}, it is enough to work with the $[0,1]$-truncation of the $K$-theory spectrum of the exact category of vector bundles, namely the virtual category $V(X)$. The group of isomorphism classes $\pi_0$ and the automorphism group $\pi_1$ of $V(X)$ are hence related to $K$-theory as follows:
\begin{displaymath}
    \pi_0(V(X))= K_0(X),\quad \pi_1(V(X)) = K_1(X).
\end{displaymath}

The virtual category of a general exact category $\Ccal$ admits a universal property, similar to that of $K_0(\Ccal)$ as a universal group for multiplicative maps from $\Ccal$ to the category of abelian groups. Briefly, abelian groups are replaced by categorical groups, known as commutative Picard categories and multiplicative maps are replaced by multiplicative functors. We refer to\linebreak Theorem \ref{thm:delignevirtual} for a precise formulation. If we denote by $\Picfr(S)$ the Picard category of line bundles on $S$, the main multiplicative functor we have in mind is $(\Vect_X, \iso) \to \Picfr(S)$ given by $E \mapsto \lambda_f(E)$ in the left-hand side of \eqref{eq:GRR}. Similarly, the right-hand side should lift to a multiplicative functor. 

It is natural to consider the various functoriality properties of the virtual categories of schemes. In particular, given \eqref{eq:GRR-diagram}, it is desirable to have a formalism of direct images. At the level of $K$-theory, a technical problem and central point in \cite{SGA6}, is that derived direct images of vector bundles are at best perfect complexes, and not bounded complexes of vector bundles.  This motivates an extension of the theory of virtual categories from exact categories to categories of complexes. The proper framework for these is complicial biWaldhausen categories, whose homotopy categories are usually thought of as derived categories. The sought extension is achieved by the work of Muro--Tonks--Witte \cite{Muro:determinant}. We prove that the virtual categories themselves admit descriptions similar to the usual $K$-theory. For example,  two complicial biWaldhausen categories with the same derived categories give rise to the same virtual categories (cf. Proposition \ref{prop:derivedvirtual}).

Following Thomason's work on $K$-theory and derived categories \cite{ThomasonTrobaugh}, we introduce the virtual category of perfect complexes (cf. Definition \ref{def:virtual-category-perfect}). This coincides with the usual virtual category for divisorial schemes (cf. Definition \ref{def:divisorial}). We then in particular address the problem of refining, from the $K$-theory to the virtual category, the direct image functor $f_{!}$ appearing in \eqref{eq:GRR-diagram}. We summarize the main conclusions, including functoriality properties involving pullback functors (cf. Proposition \ref{prop:properties-direct-image-virtual}):
\begin{prop}\label{prop:intro-direct}
Let $f\colon X\to Y$ be a morphism of divisorial schemes. 

\begin{enumerate} 
\item If $f$ is projective, of local complete intersection, then, the derived functor $Rf_{\ast}$ induces a functor of commutative Picard categories $f_{!}\colon V(X)\to V(Y)$. 

\item For arbitrary morphisms $f \colon X \to Y$ there are pullback morphisms $f^\ast \colon V(Y) \to V(X).$

\item These functors exhibit natural Tor-independent base change isomorphisms and projection formulas. 
\end{enumerate}
\end{prop}

The proofs of these and related properties follow from adaptions from the corresponding arguments on the level of $K$-theory spectra. This necessitates some background on Picard categories, which is developed in Section \ref{sec:monoidal-picard}.

\subsubsection*{Chern categories} Suppose, for simplicity, that $X$ is a connected scheme. While it would be natural to develop a categorification of Grothendieck's Chow groups $\CH^{\ast}(X)$ or $\CH^{\ast}(X)_\QBbb$, by introducing a $\gamma$-type filtration on the virtual category, it is also desirable to have a flexible categorification with a handy universal property. Along these lines, we introduce a category of formal sums and products of Chern classes, viewed as objects in their own right. The morphisms in this category reflect a minimal set of fundamental properties of Chern classes, the most important ones being isomorphisms induced from isomorphisms of vector bundles, and  Whitney-type isomorphisms associated with exact sequences. This category $\CHfrak(X)$ comes equipped with a grading induced by the degrees of the Chern classes  and a graded ring category structure. The morphisms in the category $\CHfrak(X)$ are tailored so that the total Chern class induces a natural multiplicative functor 
\begin{equation}\label{eq:Chern-functor-intro}
    \cfrak: V(X) \to \CHfrak(X).
\end{equation}
The piece of degree $k$ of this functor is denoted by $\cfrak_{k}$. We call similar functors $ V(X) \to \Rcal$, into graded ring categories, Chern functors. By a procedure reminiscent of the Bousfield localization of spectra, we also obtain a rational counterpart of this category, denoted by $\CHfrak(X)_{\QBbb}$. Here, rationality means that multiplication by $n$ induces an equivalence of categories on each graded piece. For the purposes of the introduction, we focus on $\CHfrak(X)_\QBbb$.

\begin{thm}\label{thm:B}
The category $\CHfrak(X)_\QBbb$, together with the functor induced by \eqref{eq:Chern-functor-intro}, is universal with respect to Chern functors into rational graded ring categories. 

\end{thm}
We refer to $\CHfrak(X)_\QBbb$ as the rational Chern category of $X$. By a direct limit procedure \emph{\`a la} Grothendieck, the construction can be extended to non-connected schemes. 

The proof of Theorem \ref{thm:B} and its consequences occupies \textsection \ref{subsec:universal-chern-categories}--\textsection \ref{subsec:Chern-category-quasi-compact}. The intermediate steps require some more background on commutative Picard categories, as well as ring categories, contained in Section \ref{sec:monoidal-picard} and Section \ref{section:ring-categories}. In particular, we treat the problem of turning a symmetric monoidal category into a Picard category, a procedure that we call Picardification. We apply this to extending Laplaza's rig categories \cite{Laplaza} into ring categories. A finer related result is due to Baas--Dundas--Richter--Rognes \cite{BDRR}, and a comparison with ours is briefly discussed in\linebreak Remark \ref{rem:picardification}.

An  advantage of our approach is that we can formally realize characteristic classes on the level of the Chern categories. As an excerpt of the formalism, developed in \textsection \ref{subsubsec:categorical-characteristic}, we reproduce the following proposition regarding the Chern character and the Todd class appearing in \eqref{eq:GRR}.

\begin{prop}\label{propB}
For any scheme $X$, there is a natural additive functor
\begin{displaymath}
    \chfrak : V(X) \to \CHfrak(X)_\QBbb,
\end{displaymath}
given on objects by the Chern formal power series, and a natural multiplicative functor

\begin{displaymath}
    \tdfrak^{\ast} : V(X) \to \CHfrak(X)_\QBbb,
\end{displaymath}
given on objects by the dual Todd formal power series.\footnote{The dual Todd power series is obtained by changing the Todd power series by the sign $(-1)^{k}$ in degree $k$.} Both functors commute with natural pullback functors. 

\end{prop}
The above categorical Chern and Todd functors are likewise defined on the bounded derived category of vector bundles on $X$, and additive (resp. multiplicative) on the level of true triangles. This is suitable to treat functorialities for $\tdfrak^{\ast}$ evaluated on the relative cotangent complex, which is most naturally considered on this level.

\subsubsection*{Line distributions}
The preceding theory lacks a formalism of direct images fulfilling the role of the right vertical morphism in \eqref{eq:GRR-diagram}. In this direction, we first interpret Problem \ref{problem1} as a theory taking formal power series in Chern classes and functorially associating line bundles. As such, it is expected to be  a line bundle-valued distribution, where the role of test forms is being played by Chern classes. While we have one particular line distribution in mind, discussed in the next subsection, this general interpretation permits us to highlight and systematize some key categorical points of the construction.  

Precisely, in \textsection \ref{subsec:line-distributions}, we define a line distribution for a morphism $X\to S$, as follows.

\begin{definition-intro}

A line distribution $T$ for $X\to S$ consists in associating, to every base change $h: S^{\prime}\to S$, a multiplicative functor from the rational Chern category of $X^{\prime}=X\times_{S}S^{\prime}$ to the category of $\QBbb$-line bundles on $S^{\prime}$: 

\begin{displaymath}
    T_{S^{\prime}} : (\CHfrak(X^{\prime})_\QBbb,+) \to \Picfr(S^{\prime})_{\QBbb}.
\end{displaymath}
We require:
\begin{itemize}
    \item There is a natural identification $h^{\ast}\circ T_{S} \simeq T_{S^{\prime}}\circ h^{\ast}.$
    \item It is trivial on Chern classes of degree $> N$, for some $N$ independent of $S^{\prime}$.
    \item There is an integer $m\neq 0$ such that $T^{\otimes m}$ is induced by functors without rational coefficients. 

\end{itemize}
If $P$ is an object of the Chern category $\CHfrak(X)_\QBbb$, the line distribution for $X\to S$ defined by 
\begin{displaymath}
    P^{\prime} \mapsto T_{S^{\prime}}(P^{\prime} \cdot h^\ast P)
\end{displaymath}
is denoted by $P\cdot T$ or $[P]_{X/S}$.
\end{definition-intro}

The structure of commutative Picard category of $\Picfr(S)_{\QBbb}$ induces such a structure on the line distributions for $X\to S$. We denote this category by $\Dcal(X/S)$. In analogy with usual distributions, one can define direct images of line distributions in terms of pullbacks of Chern classes. In particular, given a line distribution, if $f: X\to Y$ is a local complete intersection morphism of schemes over $S$, in this formalism the following expression has a meaning:
\begin{equation}\label{eq:introdirectimage}
    f_\ast [\chfrak(E) \cdot \tdfrak^{\ast} (\Omega_f)]_{X/S}.
\end{equation}

We think of these line distributions as a relative intersection theory valued in line bundles. The functorial condition in the definition implies that they are  of local nature, and that, for most constructions, one can assume that all the bases are in fact affine. 

As is per usual in intersection theory, one expects several helpful natural properties to hold, such as the multiplicativity of the Chern character in \eqref{eq:introdirectimage} with respect to the tensor product. This can not be realized within our Chern categories alone. However, based on our previous work with Wentworth \cite{Eriksson-Freixas-Wentworth}, we do prove in \textsection\ref{subsubsec:splitting-principle} that, on the level of line distributions, there are splitting principles that allow us to reduce Chern class identities for general vector bundles to the cases of direct sums of line bundles, and often ultimately to the case of line bundles.

\subsection{Intersection bundles and intersection distributions}
The theory of Chern categories and line distributions provide tools to finally show that intersection bundles can be organized into an actual relative intersection theory. We now elaborate on this and present the contents of\linebreak Section \ref{section:intersection-bundles} to Section \ref{subsec:int-bundles-line-distributions}.

\subsubsection*{Intersection bundles} Suppose that $f: X \to S$ is a faithfully flat, locally projective morphism of finite presentation, whose fibers have pure dimension $n$. We summarize all these properties by saying that $f$ satisfies the condition $(C_{n})$. Suppose also that we are given vector bundles $E_1, \ldots, E_m$ on $X$ and non-negative integers $k_1, \ldots, k_m$, whose sum is $n+1$. In this case, we associate in \textsection\ref{subsubsec:construction-intersection-bundles} a natural line bundle 
\begin{equation}\label{eq:def-int-bundle-intro}
    \langle \cfrak_{k_1}(E_1) \cdot \ldots \cdot \cfrak_{k_m}(E_m) \rangle_{X/S}
\end{equation}
on $S$, whose formation commutes with base change. We can extend the definition by multilinearity to general polynomials or even rational power series in the Chern classes, by taking the component of degree $n+1$. If $S$ is divisorial, for example, quasi-projective over a regular scheme, we prove in Proposition \ref{prop:chernclassofIntbundle} that there is an equality in the Picard group of $S$
\begin{equation}\label{eq:chern-classes-integral-intro}
    c_1\left(\langle \cfrak_{k_1}(E_1) \cdot \ldots \cdot \cfrak_{k_m}(E_m) \rangle_{X/S}\right) = f_\ast \left( c_{k_1}(E_1) \cdot \ldots \cdot c_{k_m}(E_m) \right),
\end{equation}
so that the bundles indeed provide a theory of integrals of Chern classes as in Problem \ref{problem1}. The line bundles of the form \eqref{eq:def-int-bundle-intro} are thus called intersection bundles.

Building on Elkik's approach \cite{Elkikfib}, the intersection bundles are developed first in the case of line bundles, in terms of generators and relations. These are the so-called Deligne pairings, that we recall and elaborate on in Section \ref{section:intersection-bundles}. The general case follows, by a method analogous to Segre's construction of Chern classes. In contrast with Elkik's work, where schemes are supposed to be Noetherian and morphisms to be Cohen--Macaulay, we take care of stating our results over general base schemes and replacing the Cohen--Macaulay hypothesis with having fibers of pure dimension $n$. The latter point was already studied by Mu\~noz Garc\'ia \cite{Munoz}, but functorial properties such as base change were not established in the generality we need. We also notice that the relationship \eqref{eq:chern-classes-integral-intro} was not considered in \cite{Elkikfib, Munoz}. 

The fundamental property \eqref{eq:chern-classes-integral-intro} suggests that typical features of characteristic classes should have intersection bundle counterparts. A key such property is the Whitney product formula for the Chern classes of an exact sequence of vector bundles. In this article, we establish several natural isomorphisms between intersection bundles, corresponding to this and other properties. This is accomplished by systematically applying the splitting-type principles proven in \cite{Eriksson-Freixas-Wentworth}. While for Chow rings one can perform these operations in any order, there is no reason why a composition of two natural isomorphisms should commute. One of our new contributions, which was not accounted for by Elkik, is the establishment of the compatibility between these natural isomorphisms. As a result, we can formally manipulate intersection bundles and such properties as if we were really in a Chow ring. 

\begin{thm}\label{thm:D}
All classical identities of Chern classes lift to isomorphisms of intersection bundles. Moreover, these isomorphisms commute with each other. 
\end{thm}
This statement refers to the results proven in \textsection\ref{subsec:properties-of-intersection-bundles} and \textsection\ref{subsec:compatibility}. A more concrete formulation in the language of line distributions is given in Theorem \ref{thm:F} below.

\subsubsection*{Intersection distributions}

The previous paragraph on intersection bundles and their properties was developed with the perspective of treating the Chern classes as individual objects, being composed with a functor that is producing line bundles. The functor formally acts as pairing with a fundamental class and then taking a direct image.  This leads us to the following definition of a line distribution, which is the main subject of \textsection \ref{subsec:Intbunaslindib}. 

\begin{defthm}
For $X \to S$ satisfying the condition $(C_{n})$, there is a line distribution, given on objects by 
    \begin{equation}\label{eq:thm-def-D}
        \cfrak_{k_1}(E_1) \cdot \ldots \cdot \cfrak_{k_m}(E_m) \mapsto \langle \cfrak_{k_1}(E_1) \cdot \ldots \cdot \cfrak_{k_m}(E_m) \rangle_{X/S}.
    \end{equation}
We call it the intersection distribution for $X\to S$. We equally call intersection distributions the following related constructions:

\begin{enumerate}
    \item If we fix an object $P$ of $\CHfrak(X)_\QBbb$, we denote by $[P]_{X/S}$ the distribution induced by
    \begin{displaymath}
        P^{\prime} \mapsto \langle P^{\prime} \cdot P \rangle_{X/S}.
    \end{displaymath}
    \item If $i\colon Y \to X$ is a closed subscheme of $X$ also satisfying the condition $(C_m)$ for some $m$, we denote by $\delta_{Y/S}$  the line distribution for $X\to S$ given by $ i_\ast [1]_{Y/S}$.
\end{enumerate}

\end{defthm}
We refer to Theorem \ref{thm:Elkik-distribution} for the statement of the theorem. The key to show that the intersection bundles indeed define multiplicative functors from $\CHfrak(X)_{\QBbb}$ to $\Picfr(S)_{\QBbb}$ is to combine\linebreak Theorem \ref{thm:D} with a study of the Whitney isomorphism for split exact sequences and its interaction with the action of permutations on the right-hand side of \eqref{eq:thm-def-D} (cf. Lemma \ref{lemma:signs}). The latter is the reason for the introduction of rational coefficients, which has the effect of identifying several natural isomorphisms which differ at most by a sign. While it is, in theory, possible to develop an integral formalism that accounts for all such signs, working rationally is of no harm for future applications to the Deligne--Riemann--Roch problem. 

When $X$ is equal to $S$ itself, the intersection distributions are, via a first Chern class isomorphism (cf. Theorem \ref{thm:F} below), equivalent to the category of $\QBbb$-line bundles. With this understood, we have an identification of $\QBbb$-line bundles
\begin{displaymath}
    \langle \cfrak_{k_1}(E_1) \cdot \ldots \cdot \cfrak_{k_m}(E_m) \rangle_{X/S} = f_\ast [ \cfrak_{k_1}(E_1) \cdot \ldots \cdot \cfrak_{k_m}(E_m) ]_{X/S},  
\end{displaymath}
hence lifting all the individual members of the equality \eqref{eq:chern-classes-integral-intro}. 

Using the language of intersection distributions, the bulk of Theorem \ref{thm:D} above can be recast as follows. We refer the reader to Corollary \ref{cor:properties-intersection-distribution} for a detailed formulation.

\begin{thm}\label{thm:F}
Assume, in properties \eqref{item:thmE-1}--\eqref{item:thmE-5} below, that all the morphisms satisfy the condition $(C_{m})$, for some $m$. We let $E$ denote a vector bundle on $X,$ of rank $r.$

\begin{enumerate}
    \item\label{item:thmE-1} (Projection formulas) Let $h\colon X^{\prime}\to X$ be a morphism of relative dimension $n^{\prime}$, and take $P$ in $\CHfrak(X)_{\QBbb}$ and $P^{\prime}$ in $\CHfrak(X^{\prime})_{\QBbb}$. 
    Then, there is a canonical isomorphism 
    \begin{displaymath}
        h_\ast (h^\ast P \cdot [P^{\prime}]_{X^{\prime}/S} )  \simeq P \cdot h_{\ast} [P^{\prime}]_{X^{\prime}/S}.
    \end{displaymath}
    Furthermore, there are canonical isomorphisms:
    \begin{displaymath}
            h_\ast [P']_{X^{\prime}/S} \simeq 
\begin{cases}
    \left[ \cfrak_1 (\langle P^{\prime} \rangle_{X^{\prime}/X})\right]_{X/S}, & \text{if }\; \deg P^{\prime}=n^{\prime}+1.\\ \\
    \left[ \int_{X^{\prime}/X}P^{\prime} \right]_{X/S}, &  \text{if }\; \deg P^{\prime} = n^{\prime}.  \\ \\
         0,  & \text{if }\; \deg P^{\prime} < n^{\prime}. \\
\end{cases}
\end{displaymath}
    \item (Whitney isomorphism) For a short exact sequence of vector bundles $0 \to E^{\prime}\to E \to E^{\bis} \to 0$, there is a canonical isomorphism
    \begin{displaymath}
        [\cfrak_{k}(E)]_{X/S}\simeq \sum_{i=0}^{k} [\cfrak_{i}(E^{\prime})\cdot\cfrak_{k-i}(E^{\bis})]_{X/S},
    \end{displaymath}
    in a way that is compatible with admissible filtrations.
    \item(First Chern class isomorphism) There is a canonical isomorphism 
 \begin{displaymath}
        [\cfrak_1(E)]_{X/S} \simeq [\cfrak_1(\det E)]_{X/S}
    \end{displaymath}
    in a way that is compatible with the Whitney isomorphism. 
    \item (Rank triviality) If $q >r$, there is a canonical isomorphism 
    \begin{displaymath}
        [ \cfrak_q(E)]_{X/S} \simeq 0.
        \end{displaymath}
    \item\label{item:thmE-5} (Restriction isomorphism) Suppose that $\sigma$ is a regular section of $E$,  whose zero locus $Y$ is flat over $S$. Then, there is a canonical isomorphism
    \begin{displaymath}
        [ \cfrak_{r}(E) ]_{X/S} \simeq \delta_{Y/S}.
    \end{displaymath}
    \item\label{item:prop-int-dist-birational1} (Birational invariance) Suppose $h \colon X' \to X$  is birational. Then, there is a canonical isomorphism
    \begin{displaymath}
        h_\ast \delta_{X'/S}\simeq \delta_{X/S}.
    \end{displaymath}
    In particular, $h_{\ast}[h^{\ast}P]_{X'/S}\simeq [P]_{X/S}$. 
\end{enumerate}
These operations can be composed with each other in a natural way, and  commute with each other. 

\end{thm}
In \textsection \ref{subsec:Intbunaslindib}, we also establish further expected intersection theoretical properties of the intersection distributions, for instance a canonical isomorphism $[\cfrak_{k}(E^{\vee})]_{X/S}\simeq (-1)^{k}[\cfrak_{k}(E)]_{X/S}$, proven in Proposition \ref{prop:Cherndual}.

\subsubsection*{The Riemann--Roch distribution}
We now explain how the above formalism provides a framework in which we can study the Deligne--Riemann--Roch problem. As an application, we construct an isomorphism in the setting of closed immersions realized as zeros of regular sections of a vector bundle.  

Let $X\to S$ and $Y\to S$ be morphisms of divisorial schemes, satisfying the conditions $(C_{n})$ and $(C_{m})$. Suppose that $f\colon X\to Y$ is a $S$-morphism of local complete intersection. Denote the cotangent complex in $D^{b}(\Vect_X)$ of $f$ by $L_{X/Y}$. In this case, by Proposition \ref{propB} and the surrounding discussion, for any vector bundle $E$ we can consider the following object in $\CHfrak(X)_{\QBbb}$
\begin{displaymath}
    \chfrak(E) \cdot \tdfrak^\ast(L_{X/Y}),
\end{displaymath}
and the associated intersection distribution
\begin{displaymath}
    [\chfrak(E) \cdot \tdfrak^{\ast}(L_{X/Y})]_{X/S}.
\end{displaymath}
In this language, given \eqref{eq:GRR-diagram} and \eqref{eq:chern-classes-integral-intro}, Problem \ref{problem2} can be generalized as the quest for a natural isomorphism of line distributions:

\begin{problem2prime}

With the notation as above, there is a canonical isomorphism:
\begin{equation}\label{eq:RR-isomorphism-intro}
 [\chfrak(f_!\ E)]_{Y/S}\to f_\ast [\chfrak(E) \cdot \tdfrak^{\ast}(L_{X/Y})]_{X/S}
\end{equation}
Moreover, in the variable $E$, it is an isomorphism of functors of commutative Picard categories $V(X)\to\Dcal(Y/S)$.
\end{problem2prime}

We notice that the left-hand side of \eqref{eq:RR-isomorphism-intro} is defined by Proposition \ref{prop:intro-direct}. When one specializes to the case $Y=S$, the left-hand side identifies with the determinant of the cohomology $\lambda_{f}(E)$, and we hence understand this to solve Problem \ref{problem2} as a particular case. We refer the reader to \textsection \ref{subsec:FormDRR} for a discussion of expected additional features of the isomorphism of distributions \eqref{eq:RR-isomorphism-intro}. 

As an application of our techniques, in Corollary \ref{cor:Borel-Serre-consequence} in \textsection \ref{subsec:Borel-Serre}, we prove the following instance of Problem 2', which, on the level of Chern classes, follows from the Borel--Serre identity \cite[Lemme 18]{Borel-Serre}. We expect it to play the role of a prototype for the Deligne--Riemann--Roch problem in the case of regular immersions. 

\begin{thm}
When $f=i$ is a closed immersion, determined by the zeros of a regular section of a vector bundle, there is a natural isomorphism 
\begin{displaymath}
[\chfrak(i_! \ \Ocal_Y)]_{X/S} \to i_\ast [\tdfrak^{\ast}(N_{X/Y}^\vee)^{-1}]_{Y/S}
\end{displaymath}
of line distributions. 
\end{thm}
When the immersion admits a retraction, this statement readily leads to a solution of Problem 2' for $i\colon Y\to X$. This situation arises in the approach to the Grothendieck--Riemann--Roch theorem by the deformation to the normal cone. We refer to Corollary \ref{cor:Borel-Serre-consequence2} for the proof.

We finish this introduction with a discussion of parallel contributions to the \linebreak Deligne--Riemann--Roch problem. 
The first author studied the Deligne--Riemann--Roch program in various contexts. We refer to \cite{EriComptes} for an announcement, although the integrality of this work remains unpublished. A running hypothesis is a regularity assumption on the involved schemes, which is unwieldy in many applications, and the current work is supposed to supersede this.  Concerning Problem \ref{problem2}, more recently, in \cite{Rossler-ARR} R\"ossler considered the \linebreak Adams--Riemann--Roch counterpart of the above program for the second Adams operation, for line bundles and projective smooth morphisms in the Noetherian setting, satisfying some additional assumptions. It would be interesting to understand the relationship between our work and his. 

\subsection{Conventions and notation}\label{subsec:conventions-notations}
We gather some conventions that will prevail throughout this article.

The subject of study of this article is intersection bundles, which depend on an $S$-scheme and a number of vector bundles. In this article, when we refer to these and similar constructions to be functorial, we mean that the formation is compatible with base change, and with isomorphisms of $S$-schemes and vector bundles. This type of functoriality being trivially true in all of our constructions, it will not be addressed in the proofs.

By a vector bundle (resp. vector bundle of constant rank $r$) on a scheme $X$ we will mean a locally free sheaf of $\Ocal_{X}$-modules of finite type  (resp. constant finite rank $r$). A vector bundle of constant rank $1$ will equivalently be called a line bundle. Given a vector bundle $E$ on $X$, our convention for the associated projective bundle $\pi\colon\PBbb(E)\to X$ is $\PBbb(E)=\mathrm{Proj}\ (\Sym E)$. In particular, there is a universal, or tautological, exact sequence on $\PBbb(E)$
\begin{equation}\label{eq:tautological-exact-sequence}
    0 \to Q \to \pi^\ast E \to \Ocal(1) \to 0.
\end{equation}
From Section \ref{subsubsec:general-intersection-bundles} onwards, most constructions will be trivial for the zero vector bundles. In such case, for ease of exposition, we will tacitly assume that our vector bundles are non-trivial.

In the theory of intersection bundles, we will deal with morphisms satisfying a number of good properties. In order to simplify the discussions, it is convenient to introduce some terminology, already anticipated in the introduction.

\begin{definition}
Let $f \colon X \to S$ be a morphism of schemes. Let $n\geq 0 $ be an integer. The following properties define the condition $(C_n)$ for $f$:

\begin{equation}\tag{$C_n$}\label{condition-Cn}
    \begin{split} 
        &f  \text{ is locally projective, faithfully flat of finite presentation (fppf),}\\
        & \text{and of pure relative dimension }n.
    \end{split}
\end{equation}
\end{definition}
We recall that locally projective means that, locally with respect to $S$, it factors through a closed embedding into some $\PBbb^N_S$, where $N$ is not fixed. The condition on the dimension means that all the fibers are equidimensional of dimension $n$. The condition $(C_{n})$ is stable under base change. We notice that, in the Noetherian case, a morphism satisfying the condition $(C_{n})$ is universally equidimensional, see \cite[Definition 2.1.2 \& Proposition 2.1.7 (2)]{Suslin-Voevodsky}.

We will encounter regular immersions of schemes. There are several variants of this notion, which are in general not equivalent. We follow \cite[\href{https://stacks.math.columbia.edu/tag/0638}{0638}]{stacks-project} and \cite[\href{https://stacks.math.columbia.edu/tag/067M}{067M}]{stacks-project}. Let $Y\hookrightarrow X$ be a closed immersion of schemes. It is said to be regular if it is locally given by a regular ideal sheaf $(f_{1},\ldots, f_{r})\subseteq\Ocal_{X}$. That is, for every $i$, multiplication by $f_{i}$ is injective on $\Ocal_{X}/(f_{1},\ldots, f_{i-1})$. Similarly, it is said to be Koszul-regular if the Koszul complex associated with the sequence $(f_{1},\ldots, f_{r})$ is a resolution of $\Ocal_{Y}$. A regular closed immersion is automatically Koszul-regular. Suppose now that $Y$ and $X$ are flat and locally of finite presentation over a base scheme $S$. In this case, Koszul-regularity entails regularity. Furthermore, these conditions can be checked on fibers, and hold after any base change. This follows from \cite[\href{https://stacks.math.columbia.edu/tag/063K}{063K}]{stacks-project} and \cite[Proposition 19.2.4]{EGAIV4}. In this article, we will always be in this relative situation, and hence regularity and Koszul-regularity will be equivalent notions.

If $E$ is a vector bundle on a scheme $X$, and $\sigma$ is a global section of $E$, let $Y$ be the zero-locus scheme of $\sigma$. We denote by $K(\sigma)$ the associated Koszul complex. We say that $\sigma$ is a regular section of $E$ if $K(\sigma)$ is a resolution of $\Ocal_{Y}$. If $X$ and $Y$ are flat over a base scheme $S$, and $X$ is locally of finite presentation over $S$, then by the previous paragraph $Y\hookrightarrow X$ is a regular closed immersion. Furthermore, in this situation, $\sigma$ remains regular after any base change $S^{\prime}\to S$.

A morphism of schemes $f\colon X\to Y$ is a local complete intersection if, locally on $X$, it factors as a Koszul-regular immersion $X\hookrightarrow P$ followed by a smooth morphism $P\to Y$. If $X$ and $Y$ are flat and locally of finite presentation over a base scheme $S$, then the closed immersion is necessarily regular. Furthermore, in this situation, $f$ remains a local complete intersection after any base change $S^{\prime}\to S$.

In dealing with general schemes, we will sometimes use the so-called Noetherian approximation. Our use of this technique is standard, and we rather refer \cite[\href{https://stacks.math.columbia.edu/tag/01YT}{01YT}]{stacks-project} for a compendium of results supporting our arguments.

In categorical considerations, it will be convenient, albeit not strictly necessary, to fix a Grothendieck universe. Scheme theoretic constructions will be assumed to take place within this universe. The categories we need or build will all be small with respect to a possibly larger universe. Up to equivalence of categories, the constructions of Chern categories below do not depend on the choice of universe, see \cite[Appendix F]{ThomasonTrobaugh}. The use of universes allows us to avoid some set-theoretic technicalities and to conform with references such as \cite{ThomasonTrobaugh}, on which we rely.

If $S$ is a scheme, we denote by $\Picfr(S)$ the category of line bundles on $S$ with isomorphisms. Together with the tensor product, it has a natural structure of a commutative Picard category. Similarly, we denote by $\Picfr(S)_{\mathrm{gr}}$ the Picard category of graded line bundles. Its objects are couples $(n,L)$, where $n$ is a locally constant function $S\to\ZBbb$, and $L$ is a line bundle on $S$. A morphism between two objects $(n,L)$ and $(m,M)$ consists in an identity $n=m$ and an isomorphism $L\to M$. Componentwise addition defines a monoidal structure, which is symmetric if taken with the Koszul rule of signs. With this understood, $\Picfr(S)_{\mathrm{gr}}$ is a commutative Picard category. 

Finally, $\Picfr(S)_{\QBbb}$ is the Picard category of $\QBbb$-line bundles. Morphisms of $\QBbb$-line bundles are isomorphisms of sufficiently high powers of them.

\subsection*{Acknowledgments}

The authors would like to thank Alexander Berglund, Christian Johansson, Dan Petersen, and David Rydh for discussions related to the present article.

\section{Monoidal and Picard categories}\label{sec:monoidal-picard}
The purpose of this section is twofold. The first purpose is to review basic facts and elaborate on some properties of symmetric monoidal categories and Picard categories. The second one, which is the main one, is two introduce and discuss the notion of Picardification. This consists in a functorial procedure to turn a symmetric monoidal category into a Picard category. The construction of the Picardification is not entirely new, but it has not been much studied \emph{per se} in the literature. Our treatment is in the spirit of Thomason's work on the $K$-theory of spectra of symmetric monoidal categories \cite{ThomasonSpectral}. We also provide an equivalent presentation that is close to Sinh's Picard envelope \cite{Sinh-these}.\footnote{To the best of our knowledge, Sinh was the first in addressing this question, attributed to Grothendieck.} We recall from \textsection \ref{subsec:conventions-notations} that all the categories we consider are implicitly assumed to be small. 

\subsection{Generalities}
We recall some basics on monoidal and Picard categories.  The reader is referred to \cite[Section 4]{Deligne-determinant}, \cite[Appendix A]{KnudsenMumford}, \cite{MacLane}, \cite{Richter}, \cite{Saavedra}, and \cite{Sinh-these} for further details and references.

\subsubsection{Monoidal categories}\label{subsub:monoidalcategories} 
A monoidal category is a category $\Acal$ equipped with a bifunctor \linebreak $\oplus: \Acal \times \Acal \to \Acal$, here called addition, with a functorial associativity isomorphism\linebreak  $A\oplus (B \oplus C) \to (A \oplus B) \oplus C$, which is supposed to satisfy the pentagonal axiom for sums of four objects and studied in detail in \cite[p. 33]{MacLane}. Any potentially commutative diagram, involving only the associativity isomorphism (or its inverse), does commute. This is the content of the coherence theorem proven in \cite[Theorem 3.1]{MacLane}, allowing one to consider finite ordered sums $\oplus_{i=1}^{n} A_i$, also denoted $\sum_{i=1}^{n} A_{i}$, without necessarily specifying the bracketing. In concrete constructions, though, we will fix the following convention for ordered sums: $\sum_{i=1}^{n}A_{i}$ is defined inductively by
\begin{equation}\label{eq:convention-sum}
    \sum_{i=1}^{n}A_{i}=A_{1}\oplus \left(\sum_{i=2}^{n}A_{i}\right).
\end{equation}

The monoidal category $\Acal$ is said to be symmetric if there is a functorial isomorphism\linebreak  $c_{A,B}: A\oplus B \to B \oplus A$, such that $c_{B,A}c_{A,B} = \id_{A\oplus B}, $ and satisfying the hexagonal axiom relating it to the associativity, explained in detail \cite[p. 38]{MacLane}. Any potentially commutative diagram, built out of the associativity and the commutativity transformations, does commute, by another coherence theorem also proven in \cite[Theorem 4.2]{MacLane}. In particular, it makes sense to consider not necessarily ordered finite sums $\oplus_{i \in I} A_i$, or $\sum_{i\in I}A_{i}$. A strictly symmetric monoidal category is one for which the commutativity isomorphisms satisfy $c_{A,A}=\id_{A}$. 

A monoidal category $\Acal$ is said to be unital if there is a zero object $0_{\Acal}$, or simply $0$, meaning there are natural contraction isomorphisms $A\oplus 0_{\Acal} \rightarrow A \leftarrow 0_{\Acal} \oplus A$. These are supposed to be compatible with the associativity isomorphisms, and the symmetry in the presence of it. In this article, all the monoidal categories will be supposed to be unital. Henceforth, by a monoidal category we always mean a unital one. One frequently encounters equivalent terminology for the zero object, such as neutral object or unit object, depending on the context. The coherence theorem for the associativity and commutativity rules extends to include the addition with the zero object \cite[Theorem 5.1]{MacLane}.

\begin{remark}\label{rmk:Maclane-strict}
 In \cite{MacLane}, Mac Lane states the coherence theorems in terms of isomorphisms of functors. An alternative formulation, albeit not equivalent, can be given in terms of regular objects, \emph{i.e.} which are sums of pairwise distinct objects or several zeroes. An instance of the coherence theorems in this form says that in a symmetric monoidal category, between two regular objects, there is at most one morphism involving associativity, commutativity, and contraction against zero. It follows that any diagram involving such morphisms between regular objects must commute. The regularity condition can be removed only in the strictly symmetric case. The formulation in terms of objects is then equivalent to Mac Lane's one. These subtleties will be relevant when we study rig categories in \textsection \ref{subsec:rig-and-ring-categories}, for which the coherence theorems in the literature are stated in terms of regular objects.  
\end{remark}

If $\Acal, \Bcal$ are monoidal categories, a functor $F\colon\Acal\to\Bcal$ is said to be monoidal\footnote{It is also called a lax monoidal functor.} if it is also equipped with a natural transformation 
\begin{equation}\label{eq:sym-mon-fun-1}
    F(A) \oplus_{\Bcal} F(B) \to F(A \oplus_\Acal B), 
\end{equation}
which is compatible with the respective associativity isomorphisms. We further require that there is the choice of a  morphism 
\begin{equation}\label{eq:sym-mon-fun-2}
    0_\Bcal \to F(0_\Acal)
\end{equation}
satisfying natural coherence conditions for the zero object. We say that $F$ is a strong monoidal functor if both \eqref{eq:sym-mon-fun-1}--\eqref{eq:sym-mon-fun-2} are isomorphisms. This condition is automatically fulfilled if all the morphisms in $\Bcal$ are isomorphisms, \emph{i.e.} $\Bcal$ is a groupoid. If $\Acal, \Bcal$ are symmetric, then $F$ is said to be symmetric if \eqref{eq:sym-mon-fun-1} is compatible with the respective commutativity constraints. A natural transformation of (symmetric) monoidal functors is a natural transformation compatible with the monoidal structure of the functors, that is the data \eqref{eq:sym-mon-fun-1} and \eqref{eq:sym-mon-fun-2}.

An equivalence of monoidal categories is a monoidal functor $F\colon \Acal\to \Bcal$ which is a weak equivalence of the underlying categories. For the purposes of this article, a quasi-inverse functor $G \colon \Bcal \to \Acal$ consists additionally in choosing, for any object $B$ of $\Bcal$, an object $A$ of $\Acal$ and an isomorphism $F(A)\to B$. It corresponds to a natural transformation$F \circ G \to \id_{\Bcal}$, and composing with $F$ one deduces a unique associated adjoint transformation $G \circ F \to \id_{\Acal}$. Furthermore, if $F$ is a strong monoidal functor, then there is a uniquely defined monoidal structure of $G$ such that $F \circ G \to \id_{\Bcal}$ and $G \circ F \to \id_{\Acal}$ are natural transformations of monoidal functors. Applied to two inverses $G$ and $G'$, we see that the deduced transformation $G \to G'$ is automatically a natural transformation of monoidal functors. In conclusion, associated with an equivalence\linebreak  $F \colon \Acal \to \Bcal$ which is strong monoidal, we can always consider an inverse monoidal functor $G \colon \Bcal \to \Acal$, unique up to canonical natural transformation of monoidal functors. This construction associated with a quasi-inverse $G$ is often assumed implicitly in the rest of the text where we often need to invert equivalences. The discussion extends to the symmetric monoidal case. 

More generally, we will also consider multisymmetric and multimonoidal functors. This consists in giving, for symmetric monoidal categories $\Acal_{1},\ldots,\Acal_{n},\Bcal$, a functor $F\colon\Acal_{1}\times\ldots\Acal_{n}\to\Bcal$ which is symmetric and monoidal in every entry separately. The notion of natural transformation of symmetric monoidal functors extends to the multi-setting, by requiring compatibility with the monoidal structure in each component. 

Given a symmetric monoidal category $\Acal$, we will have the need for the Grayson--Quillen completion $\Acal^{-1}\Acal$ \cite[pp. 218--220]{Grayson}, sometimes also denoted by $(-\Acal)\Acal$. This is a symmetric monoidal category, whose objects are couples $(A, B)$ of objects of $\Acal$, which informally represent differences $B-A$. A morphism $(A,B)\to (C,D)$ is an equivalence class of tuples $(X,f,g)$, where $X$ is an object of $\Acal$ and $f\colon A\oplus X\to C$, $g\colon B\oplus X\to D$ are morphisms in $\Acal$. Two tuples $(X,f,g)$ and $(X',f',g')$ are equivalent if there exists an isomorphism $\sigma\colon X\to X'$ which renders the following diagrams commutative:
\begin{displaymath}
    \xymatrix{
        A\oplus X\ar[dd]_{\id_{A}\oplus \sigma}\ar[rd]^{f}   &       &       &B \oplus X \ar[dd]_{\id_{B} \oplus \sigma}\ar[rd]^{g}      &\\
                                        &C     &                &                         & D\\
        A\oplus X^{\prime}\ar[ru]_{f^{\prime}}                           &       &       &B\oplus X^{\prime}.\ar[ru]_{g^{\prime}}   &
    }
\end{displaymath}
The symmetric monoidal structure is induced by componentwise addition. The construction of $\Acal^{-1}\Acal$ is functorial with respect to symmetric, monoidal functors, and natural transformations between those. Finally, we notice that even if $\Acal$ is a groupoid, $\Acal^{-1}\Acal$ does not need not be so.

\subsubsection{Picard categories}\label{subsub:PicardCat} A Picard category is a monoidal category $(\Pcal,\oplus)$, which is also a groupoid and such that for any object $X$ of $\Pcal$, the corresponding translation functors $X \oplus \bullet\colon \Pcal \to \Pcal $ and $\bullet \oplus X \colon \Pcal \to \Pcal$ are autoequivalences. In a Picard category, the unital axiom is automatically satisfied and is hence superfluous.  By analogy with Picard categories of line bundles, one frequently encounters the notation $\otimes$ for the monoidal structure and refers to this functor as product or tensor product. In this case, zero objects are rather called unit objects.

In a Picard category $\Pcal$, the fact that the translation functors by an object $X$ are essentially surjective, provides the existence of left and right opposites, or inverses, of $X$. Precisely, a left (resp. right) inverse for $X$ consists in the choice of an object $-X$ endowed with an isomorphism $(-X)\oplus X\simeq 0_{\Pcal}$ (resp. $X\oplus (-X)\simeq 0_{\Pcal})$. Such isomorphism is called a contraction. From the fact that the translation functors are fully faithful, we deduce that an inverse is unique up to unique isomorphism. The choice of the contraction isomorphism is usually not specified in the discussions. 

Given a Picard category $\Pcal$, we can consider its group of objects up to isomorphism, $\pi_0(\Pcal)$, and the group of automorphisms of the zero object, $\pi_1(\Pcal) = \Aut_{\Pcal}(0_{\Pcal})$. Since the functor $\bullet \oplus X$ is an autoequivalence, for any object $X$, we have a natural identification $\Aut_{\Pcal}(X) = \pi_1(\Pcal)$.

A Picard category which is also a symmetric (resp. strictly symmetric) monoidal category is rather called a commutative (resp. strictly commutative) Picard category. Any left inverse object $-X$ to $X$ is canonically equipped with the structure of a right inverse, by applying the commutativity isomorphism. In particular, there are two ways of constructing an isomorphism
\begin{displaymath}
    X \oplus (-X) \oplus X \to X,
\end{displaymath}
either by contracting on the left or the right. The two different choices differ by the automorphism $c_{X,X},$ which can be thought of as an automorphism of the zero object, necessarily of order 2. The construction of  $c_{X,X}$ defines a homomorphism
\begin{equation}\label{eq:isawthesign}
    \varepsilon \colon \pi_0(\Pcal) \to \pi_1(\Pcal),
\end{equation} 
whose images are referred to as signs. Hence, $\Pcal$ is strictly commutative if the subgroup of signs is trivial. In this case, we can consider finite unordered products and can unambiguously contract elements against their inverses. 

We notice that in a commutative Picard category $\Pcal$, the choice of an inverse for every object defines a functor $X\mapsto -X$. It is an equivalence of categories and its own quasi-inverse: it is involutive. In particular, it can naturally be given the structure of a symmetric monoidal functor. 

A functor between Picard categories $F: \Pcal \to \Pcal'$ is a monoidal functor. For commutative Picard categories, $F$ is said to be commutative if it is symmetric. A natural transformation of (commutative) functors between Picard categories is a natural transformation of (symmetric) monoidal functors. 

An equivalence of Picard categories $F: \Pcal \to \Pcal'$ is an equivalence of the underlying monoidal categories. We notice that since $\Pcal^{\prime}$ is a groupoid, $F$ is automatically a strong monoidal functor, that is the morphisms of the type \eqref{eq:sym-mon-fun-1}--\eqref{eq:sym-mon-fun-2} are isomorphisms. In particular, by the discussion in \textsection\ref{subsub:monoidalcategories} on equivalences of monoidal categories, there exists an inverse equivalence $\Pcal^{\prime}\to\Pcal$, which is unique up to canonical natural transformation of monoidal functors.

An equivalence of Picard categories induces an isomorphism on $\pi_0$ and $\pi_1$. For a functor of Picard categories, the converse is also true.\footnote{This in the spirit of Whitehead's theorem relating weak equivalences and homotopy equivalences for CW complexes.} Since this is used in various contexts in the text, we state it and recall the classical proof, together with a remark on natural transformations of functors between Picard categories. 
\begin{lemma}\label{lemma:isoPicard}
Let $F, G: \Pcal \to \Pcal'$ be functors between Picard categories. 
\begin{enumerate}
    \item If $F$ induces isomorphisms $\pi_0(\Pcal) \to \pi_0(\Pcal')$ and $\pi_1(\Pcal) \to \pi_1(\Pcal')$, then $F$ is a weak equivalence of categories.
    \item The set of natural transformations between $F$ and $G$ is either empty or a torsor under $\Hom_{\mathrm{Grp}}(\pi_0(\Pcal), \pi_1(\Pcal'))$.
\end{enumerate}
\end{lemma}
\begin{proof}
For the first point, since $F$ is an isomorphism on $\pi_0$, it is essentially surjective on objects. We need to show that $F$ is also fully faithful. Since $F$ is injective on $\pi_0$, there is an isomorphism $F(X) \simeq F(X')$ in $\Pcal'$ if and only if there is an isomorphism $X\to X'$ in $\Pcal.$ It follows that the Hom-sets $\Hom_{\Pcal}(X,X')$ and $\Hom_{\Pcal'}(F(X), F(X'))$ are either both empty or non-empty. If they are non-empty, choose an isomorphism $X\to X'$ to reduce the statement about full faithfulness to $X=X'$. Since $X\oplus \bullet$ is fully faithful, one reduces to the case when $X = 0_{\Pcal}$. Then, the full faithfulness of $F$ becomes equivalent to $F$ inducing an isomorphism on the level of automorphisms of the unit object, \emph{i.e.} on the level of $\pi_1.$

For the second point, we can reduce to the case $F=G$. A natural transformation of functors is the same thing as providing an isomorphism $F(X) \to F(X)$ for any object $X$, compatible with morphisms $X\to Y$. Since these morphisms $X\to Y$ are all isomorphisms, it passes to $\pi_0$. The automorphism of $F(X)$ is identified with an object in $\pi_1(\Pcal')$ and hence we obtain a map of sets $\pi_0(\Pcal) \to \pi_1(\Pcal').$ Since $F$ is a functor of Picard categories, the map $\pi_0(\Pcal) \to \pi_1(\Pcal')$ is seen to be a group homomorphism. The construction can be reversed, and it uniquely assigns an endomorphism of $F$ to a group homomorphism $\pi_0(\Pcal) \to \pi_1(\Pcal').$ The details are left to the reader.
\end{proof}

\subsubsection{Rationalization of Picard categories}\label{subsubsec:rationalization-picard-categories}
Let $(\Dcal,\oplus)$ be a commutative Picard category. Let $n\geq 1$ be an integer. Following the convention \eqref{eq:convention-sum}, we can define a multiplication by $n$ functor $[n]\colon\Dcal\to\Dcal$. By Mac Lane's coherence theorems, it has a natural structure of symmetric monoidal functor. For an object $D$ of $\Dcal$, we will rather write $n\cdot D$ or $nD$, instead of $[n]D$. We say that $\Dcal$ is divisible if $[n]: \Dcal \to \Dcal$ is an equivalence of categories for every $n\geq 1$. By Lemma \ref{lemma:isoPicard}, this is exactly the case when $\pi_0(\Dcal)$ and $\pi_1(\Dcal)$ are divisible groups.  

Let now $(\Pcal,\oplus)$ be a commutative Picard category. We will introduce a commutative Picard category $(\Pcal_{\QBbb}, \oplus)$, called the rationalization of $\Pcal$, which is divisible and endowed with a natural functor of commutative Picard categories $\Pcal \to \Pcal_{\QBbb}$. The objects of $\Pcal_{\QBbb}$ are simply objects of the form $(m, P)$, where $P$ is an object of $\Pcal$ and $m\geq 1$ is an integer. The set of morphisms $(m,P) \to (m', P')$ is 
\begin{displaymath} 
    \lim_{k \to \infty} \Hom(k \cdot m'\cdot P, k \cdot m\cdot P'),
\end{displaymath}
where the limit is taken with respect to the natural transition morphisms between the Hom sets. The resulting category $\Pcal_{\QBbb}$ naturally inherits the structure of a commutative Picard category from $\Pcal,$ by setting 
\begin{displaymath}
    (m,P) \oplus (m', P') = (m \cdot m', m' \cdot P \oplus m \cdot P').
\end{displaymath}
An application of Mac Lane's coherence theorems shows that this indeed defines a symmetric monoidal structure on $\Pcal$, with zero object $(1,0_{\Pcal})$. The functor determined on objects by $P \mapsto (1,P)$ is clearly a functor of commutative Picard categories. By construction, $\Pcal_{\QBbb}$ is divisible. Its fundamental groups are related to those of $\Pcal$ by $\pi_{k}(\Pcal_{\QBbb})=\pi_{k}(\Pcal)\otimes\QBbb$, as the following lemma explains.
\begin{lemma}
Let the notation be as above. Then, there are natural isomorphisms of groups $\pi_{k}(\Pcal_{\QBbb})=\pi_{k}(\Pcal)\otimes\QBbb$. In particular, $\Pcal \to \Pcal_\QBbb$ is an equivalence of categories if, and only if, $\Pcal$ is divisible. 
\end{lemma}
\begin{proof}
For $\pi_{0}$, this is clear. For $\pi_{1}$, the proof of this is based on the fact that, if one is given the composition of two automorphisms $X \stackrel{f}{\to} X \stackrel{g}{\to} X$, then under the isomorphism of groups \linebreak $\Hom_{\Pcal}(X, X) \simeq \Hom_{\Pcal} (2\cdot X, 2\cdot X)$, the composition corresponds to the automorphism $f\oplus g$. This, in turn, follows from a simple rewriting $f\oplus g = (f \oplus \id) \circ (\id \oplus g) = (f \oplus \id) \circ \varepsilon \circ (g \oplus \id) \circ \varepsilon$, where $\varepsilon$ is the sign morphism \eqref{eq:isawthesign}, and noting that for any $\varphi \in \pi_1(\Pcal)$, $\varphi \circ f = f \circ \varphi$. A similar argument applies to the composition of several isomorphisms. The last assertion is an application of Lemma \ref{lemma:isoPicard}.
\end{proof}

The proposition asserts that the construction above is characterized by a universal property. 

\begin{proposition}\label{prop:rationalization}
Let $\Pcal$ be a commutative Picard category and suppose we are given a divisible commutative Picard category $\Dcal$. Then the functor $\Pcal\to\Pcal_{\QBbb}$ induces an equivalence of categories
\begin{displaymath}
     \Hom(\Pcal_\QBbb, \Dcal) \to  \Hom(\Pcal, \Dcal).
\end{displaymath}
Here, the Hom-categories refer to functors of commutative Picard categories. 
\end{proposition}
\begin{proof}
We first prove the functor is essentially surjective. Since $\Dcal\to\Dcal_{\QBbb}$ is an equivalence of categories, we can fix a quasi-inverse $\Dcal_{\QBbb}\to\Dcal$ which is a functor of commutative Picard categories. Now, given a functor of Picard categories $F: \Pcal \to \Dcal, $ there is an induced functor of commutative Picard categories $F_\QBbb: \Pcal_\QBbb \to \Dcal_\QBbb$, given on objects by $F_\QBbb((m, P)) = (m, F(P)).$ Composing with $\Dcal_{\QBbb}\to\Dcal$, we obtain a functor of commutative Picard categories $\Pcal_{\QBbb}\to\Dcal$, which is isomorphic to $F: \Pcal \to \Dcal$.

Since the Hom-categories are commutative Picard categories, to prove the functor is fully faithful, it is enough to look at the induced functor on the groups of automorphisms. The automorphisms in the first (resp. the second) Hom-category are, by Lemma \ref{lemma:isoPicard}, \linebreak 
 $\Hom(\pi_0(\Pcal_\QBbb), \pi_1(\Dcal)) = \Hom(\pi_0(\Pcal)_\QBbb, \pi_1(\Dcal))$ (resp. $\Hom(\pi_0(\Pcal), \pi_1(\Dcal)).$ Because $\pi_{1}(\Dcal)$ is divisible, these are naturally identified by the above functor between Hom-categories.

\end{proof}

\begin{remark}\label{rem:Q-Bousfield}
\begin{enumerate}
    \item \label{item:Q-Bousfield-1} Since $\pi_1(\Pcal_\QBbb) = \pi_1(\Pcal)\otimes\QBbb$, the sign homomorphism in \eqref{eq:isawthesign} is necessarily the trivial map, and $\Pcal_\QBbb$ is hence strictly commutative. 
    \item \label{item:Q-Bousfield-2} The above is a version of the Bousfield localization at $\QBbb$ in the context of Picard categories.
\end{enumerate}
\end{remark}

\begin{corollary}\label{cor:rationalization}
Let $\Pcal_{1},\ldots,\Pcal_{n}$ be commutative Picard categories, and $\Dcal$ a divisible, commutative Picard category. Let $F\colon \Pcal_{1}\times\ldots\times\Pcal_{n}\to\Dcal$ be a multisymmetric and multimonoidal functor. Then there is an extension of $F$ into a multisymmetric and multimonoidal functor
\begin{displaymath}
    \widetilde{F}\colon \Pcal_{1\ \QBbb}\times\ldots \times \Pcal_{n\ \QBbb}\to \Dcal.
\end{displaymath}
The construction is functorial with respect to natural transformations. Hence, there is an equivalence of categories of multisymmetric and multimonoidal functors
\begin{displaymath}
    \Hom(\Pcal_{1\ \QBbb}\times\ldots\times\Pcal_{n\ \QBbb},\Dcal)\to\Hom(\Pcal_{1}\times\ldots\times\Pcal_{n},\Dcal).
\end{displaymath}
\end{corollary}
\begin{proof}
If $n=1$, this is the content of Proposition \ref{prop:rationalization}. In general, we interpret $F$ as a multimonoidal and multisymmetric functor $\Pcal_{1}\times\ldots\times\Pcal_{n-1}\to\Hom(\Pcal_{n},\Dcal)\simeq\Hom(\Pcal_{n\ \QBbb},\Dcal)$. The Hom-category has a natural structure of a commutative Picard category. Furthermore, because $\Dcal$ is divisible, the Hom-category is divisible too. We conclude by induction.
\end{proof}
\begin{remark}\label{rmk:order-rationalization}
In the proof of the corollary, in the induction step, there is a choice of order of the Picard categories. It is an exercise to check that any other order would yield an isomorphic extension of the functor. We will later encounter a similar, albeit more subtle, phenomenon in the Picardification theory. We refer to Remark \ref{rem:orderdevelopproduct} \eqref{item:orderdevelopproduct-1} for the details.
\end{remark}

\subsection{Picardification of symmetric monoidal categories}\label{subsubsec:picardification}
Let $(\Acal,\oplus)$ be a symmetric monoidal category, which is also a groupoid. For example, $\Acal$ could be a commutative Picard category. 

\begin{definition}\label{def:Picardification}
A Picardification of $(\Acal, \oplus)$ is a commutative Picard category $(V(\Acal), \oplus)$ together with a symmetric monoidal functor $i\colon\Acal\to V(\Acal)$, universal with respect to symmetric monoidal functors $\Acal \to \Pcal$, where $\Pcal$ is a commutative Picard category. That is, composition with the functor $i\colon \Acal\to V(\Acal)$ induces an equivalence of categories
\begin{equation}\label{eq:equiv-cat-VA-def}
    \Hom(V(\Acal),\Pcal)\to\Hom(\Acal,\Pcal).
\end{equation}
Here, the $\Hom$ refer to categories of functors of symmetric monoidal categories, together with the natural transformations.
\end{definition}
Let us assume, for the time being, that a Picardification $V(\Acal)$ exists. For the sake of clarity, and in order to fix some conventions, we elaborate on the practical use of the equivalence of categories \eqref{eq:equiv-cat-VA-def}. Let $F\colon\Acal\to\Pcal$ be a functor of symmetric monoidal categories as in the definition. That \eqref{eq:equiv-cat-VA-def} is essentially surjective means that there exists a couple $(\widetilde{F},\mu)$, where $\widetilde{F}\colon V(\Acal)\to\Pcal$ is a functor of commutative Picard categories, and $\mu\colon \widetilde{F}\circ i\to F$ is a natural transformation of functors of commutative Picard categories. From the fact that \eqref{eq:equiv-cat-VA-def} is fully faithful, we infer that the couple $(\widetilde{F},\mu)$ is  unique up to unique isomorphism. Explicitly, if $(\widetilde{F}^{\prime},\mu^{\prime})$ is another such couple, then we have an isomorphism $(\mu^{\prime})^{-1}\circ\mu\colon \widetilde{F}\circ i\to\widetilde{F}^{\prime}\circ i $. By full faithfulness, there exists a unique isomorphism of $\eta\colon\widetilde{F}\to\widetilde{F}^{\prime}$ with $(\mu^{\prime})^{-1}\circ\mu=\eta\circ i$. It is customary to refer to $F$ without specifying the natural transformation $\mu$, which is therefore implicit.

The discussion of the previous paragraph shows that if $V(\Acal)^{\prime}$ is another Picardification, then there exists an equivalence of commutative Picard categories $V(\Acal)\to V(\Acal)^{\prime}$, which is determined up to unique isomorphism. Therefore, from now on, we will refer to \emph{the}, rather than  \emph{a}, Picardification.  

\begin{theorem}\label{thm:VA}
Let $(\Acal, \oplus)$ be a symmetric monoidal category, which is also a groupoid. Then:
\begin{enumerate}
    \item The Picardification of $\Acal$ exists.  
    \item The assignment $\Acal\mapsto V(\Acal)$ is functorial with respect to symmetric monoidal functors, and natural transformations between those.
\end{enumerate}
\end{theorem}

\begin{proof}
We begin with the existence property. Without loss of generality, we may assume that the right, or equivalently left, translation functors of $\Acal$ are faithful. For this, we observe that there is a universal quotient category $\widetilde{\Acal}$ of $\Acal$, where this condition is satisfied. To see this, introduce a relation on morphisms of $\Acal$ as follows: two morphisms $f,f^{\prime}\colon A\to B$ are related if there exists an object $X$ such that $f\oplus\id_{X}=f^{\prime}\oplus\id_{X}$. This is an equivalence relation. For the proof, one uses that for any such morphism $f$ and objects $X,Y$, we have by bifunctoriality of $\oplus$ that
\begin{displaymath}
    (\id_{B}\oplus c_{X,Y})\circ (f\oplus\id_{X\oplus Y})=(f\oplus\id_{Y\oplus X})\circ (\id_{A}\oplus c_{Y,X}), 
\end{displaymath}
and that all the morphisms in $\Acal$ are isomorphisms. In the same vein, it is seen that if $f, f^{\prime}\colon A\to B$, resp. $g,g^{\prime}\colon B\to C$ are related, then so are $g\circ f$ and $g^{\prime}\circ f^{\prime}$. We can thus form the quotient category $\widetilde{\Acal}$, which inherits from $\Acal$ the structure of a symmetric monoidal category, whose morphisms are still isomorphisms. The quotient functor $\Acal\to\widetilde{\Acal}$ is naturally symmetric monoidal. If $\Pcal$ is a Picard category,  consider the induced functor $\Hom(\widetilde{\Acal},\Pcal)\to\Hom(\Acal,\Pcal)$, where the $\Hom$ stand for categories of symmetric monoidal functors. Using that the translation functors of $\Pcal$ are equivalences of categories, we see that $\Hom(\widetilde{\Acal},\Pcal)\to\Hom(\Acal,\Pcal)$ is an equivalence of categories. Therefore, it is enough to show the virtual category of $\widetilde{\Acal}$ exists. Hence, for the rest of the proof, we suppose that $\Acal$ is a unital symmetric monoidal category, whose morphisms are isomorphisms and whose translation functors are faithful.

Consider the Grayson--Quillen completion $\Acal^{-1} \Acal$ of $\Acal$, recalled in \textsection\ref{subsub:monoidalcategories} above. For a category $\Ccal$, let $B(\Ccal)$ be its classifying space, that is the geometric realization of its nerve. In these terms, we define $V(\Acal)$ to be the fundamental groupoid of $B(\Acal^{-1} \Acal)$, denoted by $\pi_{f}B(\Acal^{-1} \Acal)$. We first prove that $V(\Acal)$ is a Picard category and that there is a natural symmetric monoidal functor $i: \Acal \to V(\Acal)$. 

It follows by functoriality of the involved constructions that the fundamental groupoid of  $B(\Acal^{-1} \Acal)$ is a symmetric monoidal groupoid. For this, it is enough to notice that the construction $\Ccal\mapsto \pi_{f}B(\Ccal)$ induces a functor from the 2-category (small categories, functors, natural transformations) to the 2-category (small groupoids, functors, natural transformations). The statement on 2-morphisms is only true up to homotopy on the level of classifying spaces, and after applying the fundamental groupoid this homotopy is collapsed. 

We next prove that for any object $A$ of $V(\Acal)$, addition $A \oplus$ (or $\oplus A$) is essentially surjective and fully faithful. The path component set $\pi_0(V(\Acal))$ is the set of isomorphism classes of $V(\Acal)$, which in turn equals $\pi_{0}(\Acal^{-1}\Acal)$. By construction, it is the group completion of the abelian monoid $\pi_0(\Acal)$, see \cite[Lemma 13.3.4]{Richter}. It follows that addition is essentially surjective on $V(\Acal)$. Finally, we prove that $A \oplus$ is fully faithful. Since it induces a bijection on the level of $\pi_0$, there is hence an isomorphism $A \oplus B \to A \oplus B'$ if and only if there is an isomorphism $B \to B'$. Thus, we can reduce to the case of $B = B'$ and automorphisms of objects. The group of automorphisms of an object $x$ in $V(\Acal)$ is $\pi_1(V(\Acal), x),$ which is abelian since $B(\Acal^{-1} \Acal)$ is an $H$-space. This in turn is a component of the homology 
\begin{displaymath}
    H_{1}(V(\Acal))=H_{1}(\Acal^{-1}\Acal)\simeq\prod_{x\in\pi_{0}(V(\Acal))}\pi_{1}(V(\Acal),x).
\end{displaymath}
Now, the monoid $\pi_{0}(\Acal)$ acts invertibly on $H_{1}(V(\Acal))$ (cf. \cite[p. 221]{Grayson}). Therefore, this action extends to an action of its group completion $\pi_0(V(\Acal))$. As a result, we deduce that there are induced bijections $\pi_1(V(\Acal), B) \to \pi_1(V(\Acal), A+B)$. 

The functor  $i: \Acal \to V(\Acal)$ is the composition of the natural functors $\Acal \to \Acal^{-1}\Acal$ and $\Acal^{-1}\Acal\to V(\Acal)$, which are both symmetric monoidal by construction. 

Finally, to prove the universal property, composing a symmetric monoidal functor $V(\Acal) \to \Pcal $ with $i: \Acal \to V(\Acal)$ provides a functor in one direction. Conversely, given a symmetric monoidal functor $\Acal \to \Pcal$, there is an induced symmetric monoidal functor $\Acal^{-1} \Acal \to \Pcal^{-1}\Pcal$, and hence $V(\Acal) \to V(\Pcal)$. The below diagram commutes
\begin{displaymath}
    \xymatrix{
      \Acal \ar[r] \ar[d] & \Pcal \ar[d] \\
      \Acal^{-1} \Acal \ar[d] \ar[r] &  \Pcal^{-1} \Pcal \ar[d]\\
       V(\Acal) \ar[r] & V(\Pcal). \\
    }
\end{displaymath}
The composition $\Pcal \to V(\Pcal)$ is an equivalence of categories, as follows from Lemma \ref{lemma:isoPicard}, the computation of the homology of $V(\Pcal)$ in \cite[p. 221]{Grayson}, and the fact that $\pi_{0}(\Pcal)$ is a group. The fact that the diagram induces the equivalence promised in the theorem is a diagram chase, which we leave to the reader. 

The second point of the theorem follows directly from the construction since the \linebreak  Grayson--Quillen completion is itself functorial. It is also a standard consequence of the universal property, as follows. Let $F\colon \Acal\to\Bcal$ be a symmetric monoidal functor between symmetric monoidal groupoids. The natural functor $\Bcal\to V(\Bcal)$ induces a functor of categories of symmetric monoidal functors
\begin{displaymath}
    \Hom(\Acal,\Bcal)\to\Hom(\Acal,V(\Bcal)).
\end{displaymath}
By the universal property of $V(\Acal)$, the rightmost $\Hom$ category is equivalent to \linebreak  $\Hom(V(\Acal),V(\Bcal))$. Hence, to $F$ there is associated a functor $V(F)$, and an isomorphism of functors $F\to F^{\prime}$ induces an isomorphism of functors $V(F)\to V(F^{\prime})$, compatibly with compositions of natural transformations. In a similar vein, it is proven that if $F$ and $G$ are two composable functors of symmetric monoidal groupoids, then there is an isomorphism $V(F\circ G)\simeq V(F)\circ V(G)$, which is unique up to unique isomorphism. This concludes the proof.
\end{proof}

\begin{remark}\label{rmk:thomason}
 The astute reader will be familiar with a well-known issue of universal property of the construction $\Acal^{-1}\Acal$, exposed by Thomason in \cite{ThomasonPhony}. This is related to $(A,B) \mapsto (B,A)$ failing to functorially provide an inverse on the level of $\Acal^{-1}\Acal$. However, it does induce a homotopy inverse on the $H$-space $B(\Acal^{-1} \Acal)$ by \cite[Proposition 5.3]{ThomasonSpectral}, which is reminiscent of the fact that $V(\Acal)$ is a Picard category. This in turn is a reflection of the fact that $B(\Acal^{-1}\Acal)$ naturally has the structure of an infinite loop space, which is a well-known construction going back at least to Thomason, see \cite[Section 4]{ThomasonHomCol}. 
\end{remark}

In the proof of Theorem \ref{thm:VA}, the effect of taking the classifying space and then the fundamental groupoid amounts to inverting all the morphisms of $\Acal^{-1}\Acal$. We now render this precise, and discuss an equivalent description of the Picardification, which is reminiscent of Sinh's treatment in \cite[Chapitre III]{Sinh-these}. This will later simplify some verifications in \textsection\ref{subsubsec:strict-rig-categories} and \textsection\ref{subsubsec:coherence-property}. Let $\Sigma$ denote the set of all the morphisms of $\Acal^{-1}\Acal$, and introduce the localization $\Acal^{-1}\Acal[\Sigma^{-1}]$ where all the morphisms become invertible \cite[Chapter I]{Gabriel-Zisman}.
\begin{lemma}\label{lemma:VA-tilde}
Les $\Acal$ be a symmetric monoidal category, which is also a groupoid. Then, the natural functor $\Acal\to V(\Acal)$ extends to an equivalence of symmetric monoidal categories \linebreak $\Acal^{-1}\Acal[\Sigma^{-1}]\to V(\Acal)$.
\end{lemma}
\begin{proof}
Assume first that the translation functors of $\Acal$ are faithful. The natural symmetric monoidal functor $\Acal^{-1}\Acal\to V(\Acal)$ extends to a symmetric monoidal functor $\Acal^{-1}\Acal[\Sigma^{-1}]$ since $V(\Acal)$ is a groupoid. This functor is an equivalence of categories, by \cite[III, Corollary 1.2]{JG}. 

In general, let $\widetilde{\Acal}$ be the quotient category of $\Acal$, where the translation functors become faithful. This was constructed at the beginning of the proof of Theorem \ref{thm:VA}. By definition, $V(\Acal)=V(\widetilde{\Acal} )$. The quotient morphism $\Acal\to\widetilde{\Acal}$ is a symmetric, strong monoidal functor. Hence, it extends to a natural symmetric monoidal functor $j\colon \Acal^{-1}\Acal[\Sigma^{-1}]\to\widetilde{\Acal}^{-1}\widetilde{\Acal}\ [\widetilde{\Sigma}^{-1}]$, where $\widetilde{\Sigma}$ is the set of morphisms of $\widetilde{\Acal}^{-1}\widetilde{\Acal}$. Thus, it is enough to see that $j$ is an equivalence of categories. To this end, we will construct an inverse functor. 

We construct a functor $\widetilde{\Acal}^{-1}\widetilde{\Acal}\to\Acal^{-1}\Acal[\Sigma^{-1}]$. On objects, we take the identity correspondence. Let $\widetilde{h}\colon (A,B)\to (C,D)$ be a morphism in $\widetilde{\Acal}^{-1}\widetilde{\Acal}$. This is the equivalence class of a triple $(X,\widetilde{f},\widetilde{g})$, determined by an object $X$ in $\Acal$ and morphisms $\widetilde{f}\colon A\oplus X\to C$ and $\widetilde{g}\colon B\oplus X\to D$ in $\widetilde{\Acal}$. In turn, $\widetilde{f}$ and $\widetilde{g}$ are given by equivalence classes of morphisms $f\colon A\oplus X\to C$, $g\colon B\oplus X\to D$ in $\Acal$. We take the morphism in $\Acal^{-1}\Acal[\Sigma^{-1}]$ induced by $(f,g)$, and denoted by $h\colon (A,B)\to (C,D)$. We need to show that $h$ depends only on $\widetilde{h}$, and not the chosen representatives. 

Let $(X^{\prime},\widetilde{f}^{\prime},\widetilde{g}^{\prime})$ be another representative of $\widetilde{h}$. From this, we construct $(f^{\prime},g^{\prime})$ similarly to $(f,g)$. We have to justify that $(f^{\prime},g^{\prime})$ also induces the morphism $h$ in the category $\Acal^{-1}\Acal[\Sigma^{-1}]$. Unraveling the definitions, we see that there exists an isomorphism $\sigma\colon X\to X^{\prime}$ in $\Acal$, an object $Y$, and commutative diagrams in $\Acal$
\begin{equation}\label{eq:another-commutative-diagram}
    \xymatrix{
        A\oplus X\oplus Y\ar[dd]_{\id_{A}\oplus \sigma\oplus \id_{Y}}\ar[rd]^{f\oplus \id_{Y}}   &       &       &B\oplus X\oplus Y \ar[dd]_{\id_{B}\oplus \sigma \oplus \id_{Y}}\ar[rd]^{g\oplus \id_{Y}}      &\\
                                        &C\oplus Y     &                &                         & D\oplus Y\\
        A\oplus X^{\prime}\oplus Y\ar[ru]_{f^{\prime}\oplus \id_{Y}}                           &       &       &B\oplus X^{\prime}\oplus Y.\ar[ru]_{g^{\prime}\oplus \id_{Y}}   &
    }
\end{equation}
We notice that the following diagram commutes in $\Acal^{-1}\Acal$:
\begin{equation}\label{eq:I-lack-inspiration-for-diagrams}
    \xymatrix{
        (A,B)\ar[r]^{h}\ar[d]   &(C,D)\ar[d]\\
        (A\oplus X\oplus Y,B\oplus X\oplus Y)\ar[r]_{\substack{\vspace{0.5cm}}{(f\oplus  \id_{Y},g\oplus \id_{Y})}}   &(C\oplus Y,D\oplus Y),
    }
\end{equation}
where the vertical arrows are the natural ones induced by the identity maps of $A\oplus X\oplus Y$, $B\oplus X\oplus Y$, etc. From \eqref{eq:another-commutative-diagram}, we build yet another commutative diagram
\begin{displaymath}
    \xymatrix{
                                        &(A\oplus X\oplus Y,B\oplus X\oplus Y)\ar[dd]\ar[dr]      &   \\
        (A,B)\ar[ru]\ar[rd]             &                               &(C\oplus Y,D\oplus Y)\\                                              &(A\oplus X^{\prime}\oplus Y,B\oplus X^{\prime}\oplus Y).\ar[ru]     &
    }
\end{displaymath}
Combining the last two diagrams, we deduce that the analog of \eqref{eq:I-lack-inspiration-for-diagrams} for $(X^{\prime}, f^{\prime},g^{\prime})$ also commutes in $\Acal^{-1}\Acal$. This implies that $(f^{\prime},g^{\prime})$ induces the morphism $h$ in $\Acal^{-1}\Acal[\Sigma^{-1}]$, as required.

It is clear that $\widetilde{\Acal}^{-1}\widetilde{\Acal}\to\Acal^{-1}\Acal[\Sigma^{-1}]$ thus defined is a functor, which necessarily extends to $\widetilde{\Acal}^{-1}\widetilde{\Acal}\ [\widetilde{\Sigma}^{-1}]\to\Acal^{-1}\Acal[\Sigma^{-1}]$. It follows directly from the construction that this is indeed an inverse functor to $j$. It also has an obvious structure of symmetric monoidal functor. This concludes the proof.
\end{proof}

\begin{remark}\label{rmk:thomason-strict}
It is possible, but somewhat tedious, to write the proof of Theorem \ref{thm:VA} directly in terms of $\Acal^{-1}\Acal[\Sigma^{-1}]$. In particular, it can be checked that $(A,B)\mapsto (B,A)$ defines a functorial inverse on $\Acal^{-1}\Acal[\Sigma^{-1}]$. See also Remark \ref{rmk:thomason}. The contraction isomorphism $(A,B)+(B,A)\simeq 0$ can be chosen to be induced by the natural morphism $(0,0)\to (A,B)\oplus (B,A)$, defined by the couple $(\id_{A\oplus B}, c_{A,B})$. Notice that the latter already exists in $\Acal^{-1}\Acal$, but it is in general not invertible in this category. This is the reason why we need to localize along $\Sigma$.

\end{remark}

In Section \ref{section:ring-categories}, we will encounter the problem of turning a rig category into a ring category. A rig category is a symmetric monoidal category with the further structure of a product, and the problem at hand consists in extending the product to the Picardification. The following proposition provides a tool for the extension of this and more general structures. 

\begin{proposition}\label{prop:product}
Let $\Acal_{1},\ldots,\Acal_{n},\Bcal$ be symmetric monoidal categories, which are moreover groupoids, and $F\colon \Acal_{1}\times\ldots\times\Acal_{n}\to\Bcal$ be a multimonoidal and multisymmetric functor. Then there is an extension of $F$ into a multimonoidal and multisymmetric functor
\begin{displaymath}
    \widetilde{F}\colon V(\Acal_{1})\times\ldots \times V(\Acal_{n})\to V(\Bcal).
\end{displaymath}
The construction is functorial with respect to natural transformations. 
\end{proposition}
\begin{proof}
First of all, we compose $F$ with the natural functor $\Bcal\to V(\Bcal)$. Then, the proof proceeds by induction and is analogous to the proof of Corollary \ref{cor:rationalization}, replacing Proposition \ref{prop:rationalization} with Theorem \ref{thm:VA}. We leave the details to the reader. 
\end{proof}
\begin{remark}\label{rem:orderdevelopproduct}
\begin{enumerate}
    \item Proposition \ref{prop:product} can also be stated as an equivalence of categories of multimonoidal and multisymmetric functors.
    \item\label{item:orderdevelopproduct-1} As in Remark \ref{rmk:order-rationalization}, we notice that in the proof of Proposition \ref{prop:product}, there is a choice of order of the Picard categories. Contrary to the rationalization, the Picardification construction depends on the order. Let us first discuss the relevant case of a  product functor \linebreak $\otimes\colon\Acal\times\Bcal\to\Ccal$. Then the reasoning of the proof implicitly chooses to first develop a product $(\sum_{i} A_i) \otimes (\sum_{j} B_j)$ on the left into an object $\sum_i \left(A_i \otimes(\sum_{j} B_j)\right)$, and later develop the right-hand side. Doing it in the opposite order determines another extension of the product. They are equal up to sign. We refer the reader to \cite[\textsection 4.11]{Deligne-determinant} for a related discussion. The general case is analogous and reduces to the case of a product functor by induction. Hence, the different choices of order yield extensions which differ at most by signs.
\end{enumerate}
\end{remark}

\section{Ring categories}\label{section:ring-categories}
In this section, we review and discuss the notions and properties of rig and ring categories, which are categorifications of corresponding structures in commutative algebra.\footnote{A rig is defined analogously to a ri\emph{n}g but without the formation of the \emph{n}egative of an element.} The Picardification theory of the previous section allows us to turn a rig category into a ring category. We also consider such structures endowed with the additional datum of a grading.

\subsection{Generalities}\label{subsec:rig-and-ring-categories}
The results of this subsection recall the basics of rig and ring categories. Our definitions are more restrictive than in other references, such as Laplaza \cite{Laplaza} or \linebreak  Baas--Dundas--Richter--Rognes \cite{BDRR}, but they are adapted to our needs in the ulterior theory of Chern categories. 

\subsubsection{Rig and ring categories}

A rig category is a category $\Acal$ endowed with two symmetric monoidal structures $\oplus$ and $\otimes$, called addition and product (or multiplication), respectively. One requires that there are left and right functorial distributivity isomorphisms\footnote{In some references, such as \cite{Laplaza}, the distributivity property is more generally given by monomorphisms.}: \linebreak $A\otimes (B\oplus C)\to (A\otimes B)\oplus (A\otimes C)$ and $(A\oplus B)\otimes C\to (A\otimes C)\oplus (B\otimes C)$. There are also left and right functorial absorption isomorphisms $ 0_{\Acal}\otimes A\to 0_{\Acal}$ and $A\otimes 0_{\Acal}\to 0_{\Acal}$, where $0_{\Acal}$ is the neutral object for addition. Finally, several coherence conditions must hold. These are discussed by Laplaza in \cite[\textsection 1]{Laplaza}, to which we refer to for the precise list of 24 axioms. The coherence conditions in particular guarantee that $\otimes\colon\Acal\times\Acal\to\Acal$ is bisymmetric and bimonoidal, with respect to $\oplus$. In this paper, rig categories whose addition and product are strictly symmetric will be called strict rig categories. 

Rig categories satisfy Laplaza's coherence theorem \cite[\textsection 7]{Laplaza}. It asserts that, under a regularity condition on objects, any diagram involving combinations of addition, product, symmetry, associativity (or its inverse), distributivity, contractions with neutral elements (or their inverses), and absorption (or its inverse), must commute. The regularity condition is studied in \cite[\textsection 3]{Laplaza}. Roughly, after forgetting parentheses, a regular object is a sum of different objects, which are themselves products of different objects. In practice, in this article, we will only encounter rig categories for which the coherence theorem is satisfied without the regularity condition. We will refer to these as being strongly coherent. A strongly coherent rig category is necessarily strict. A careful examination of the proof of \cite[Proposition 10]{Laplaza}, shows that a strict rig category is strongly coherent. In other words, being strict or strongly coherent, are equivalent notions. 

A ring category $\Rcal$ is a rig category $(\Rcal,\oplus,\otimes)$, such that $(\Rcal,\oplus)$ is moreover a Picard category. In particular, $\Rcal$ is a groupoid. We say that $\Rcal$ is strict (resp. strongly coherent) if it is so as a rig category.

A functor between rig (resp. ring) categories $F\colon\Acal\to\Bcal$ is a symmetric functor of symmetric monoidal categories for both monoidal structures $\oplus$ and $\otimes$, which is moreover compatible with the distributivity and absorption properties. For clarity, let us indicate what this means for the left absorption property. Consider the diagram
\begin{displaymath}
    \xymatrix{
        0_{\Bcal}\otimes F(A)\ar[d] \ar[rr]     &                            &0_{\Bcal}\ar[d]\\
        F(0_{\Acal})\otimes F(A)\ar[r]        &F(0_{\Acal}\otimes A)\ar[r]  &F(0_{\Acal}).
    }
\end{displaymath}
It is constructed using the monoidal structure of $F$ and the absorption isomorphisms of $\Acal$ and $\Bcal$. The required compatibility consists in the commutativity of the diagram.

\subsubsection{Graded rig and graded ring categories}
We will need graded counterparts of rig and ring categories. A graded rig (resp. ring) category is a rig (resp. ring) category $(\Acal, \oplus, \otimes)$, of the form $\Acal=\prod_{k\geq 0}\Acal_{k}$. Here, we suppose that $(\Acal_{k},\oplus)$ is a symmetric monoidal category (resp. commutative Picard category) with respect to addition and that the symmetric monoidal structure $\oplus$ on $\Acal$ is induced by componentwise addition. In particular, the zero object of $\Acal$ is of the form $0_{\Acal}=(0_{\Acal_{0}},0_{\Acal_{1}},\ldots)$. We also suppose given bimonoidal bisymmetric functors $\otimes\colon\Acal_{k}\times\Acal_{k'}\to\Acal_{k+k'}$, called graded products, which induce the product $\otimes$ on $\Acal$ as follows:
\begin{equation}\label{eq:graded-product}
    \begin{split}
        \Acal\times\Acal\simeq \prod_{\ell}\prod_{k=0}^{\ell}\Acal_{k}\times\Acal_{\ell-k}&\longrightarrow\Acal=\prod_{\ell}\Acal_{\ell}\\
        ((A_{j})_{j}, (B_{j})_{j})&\longmapsto \left(\sum_{k=0}^{\ell}A_{k}\otimes B_{\ell-k}\right)_{\ell}.
    \end{split}
\end{equation}
In this expression, we adopt the convention \eqref{eq:convention-sum} for the addition of several terms.\footnote{Thanks to Mac Lane's coherence theorem \cite{MacLane} for symmetric monoidal categories, any other convention would yield an equivalent notion of graded product.} The category $\Acal_{k}$ is called the graded piece of degree $k$. Objects of $\Acal_{k}$ are called objects of pure degree $k$. 

In a graded rig category, the defining axioms (associativity, symmetry, distributivity, neutral elements, absorption, coherence conditions) are required to be compatible with corresponding axioms limited to objects of pure degree. The proper formulation involves the associativity and symmetry isomorphisms for addition and Mac Lane's coherence theorems. For the sake of clarity, the next two paragraphs discuss some instances of these axioms. The other axioms admit an analogous treatment, whose details are left to the reader.

The graded products are supposed to be equipped with functorial symmetry isomorphisms $c_{j,k}\colon A_{j}\otimes B_{k}\to B_{k}\otimes A_{j}$, satisfying $c_{k,j}\circ c_{j,k}=1$. Combining with the associativity and symmetry isomorphisms for the addition on the graded pieces, we obtain natural functorial isomorphisms
\begin{equation}\label{eq:graded-sym-product}
    \sum_{k=0}^{\ell}A_{k}\otimes B_{\ell-k}\to\sum_{j=0}^{\ell}B_{j}\otimes A_{\ell-j}.
\end{equation}
We notice that the associativity and symmetry isomorphisms for the addition can be performed in any order, by Mac Lane's coherence theorems. Also, applying these first and then the symmetry for the graded producs, or inversely, yields the same result, by functoriality of the associativity and symmetry for the addition. It is then checked that the functorial isomorphisms \eqref{eq:graded-sym-product} induce a symmetry constraint for the product \eqref{eq:graded-product}.

The category $\Acal_{0}$ has the structure of a rig category, and in particular, has a unit object $1_{\Acal_{0}}$. This is a neutral object for the graded products: there are functorial isomorphisms $1_{\Acal_{0}}\otimes A_{k}\simeq A_{k}$ and $A_{k}\otimes 1_{\Acal_{0}}\simeq A_{k}$ in $\Acal_{k}$. We also have left absorption isomorphisms $0_{\Acal_{k}}\otimes A_{k'}\simeq 0_{\Acal_{k+k'}}$, and similarly for the right absorption isomorphisms. These are all supposed to be compatible with the symmetry constraint for the graded products. We see that $1_{\Acal}:=(1_{\Acal_{0}},0_{\Acal_{1}},0_{\Acal_{2}},\ldots )$ provides a unit object for the product on $\Acal$, and that the zero object $0_{\Acal}=(0_{\Acal_{0}},0_{\Acal_{1}},\ldots)$ is equipped with natural left and right absorption isomorphisms.

A graded rig (resp. ring) category whose product and addition are strictly symmetric will be called a graded strict rig (resp. ring) category. As for the defining axioms of a graded rig category, the property of being strict can be stated in terms of the graded pieces.  

A functor of graded rig categories $F\colon\Acal\to\Bcal$ is a functor of rig categories that is induced by functors $F_{k}\colon \Acal_{k}\to\Bcal_{k}$, which are symmetric monoidal with respect to addition and are compatible with the product structure. This means that there are natural transformations $F_{k}(A_{k})\otimes F_{k'}(A_{k'})\to F_{k+k'}(A_{k}\otimes A_{k'})$, which are compatible with associativity, symmetry, distributivity, absorption, etc. 

We conclude this subsection by commenting on some standard constructions related to rig categories. If $\Acal$ is a graded rig category, we have natural sections of the projection functors $\Acal\to\Acal_{k}$. These consist in completing an object by zero. Precisely, if $A_{k}$ is an object of $\Acal_{k}$ and we let $\widetilde{A}_{k}$ be $(0_{\Acal_{0}},\ldots, 0_{\Acal_{k-1}},A_{k},0_{\Acal_{k+1}},\ldots)$, then the assignment $A_{k}\mapsto\widetilde{A}_{k}$ can be naturally upgraded into a functor $i_{k}\colon\Acal_{k}\to\Acal$, which is a commutative functor of Picard categories with respect to addition. This construction is functorial, in the following sense. Let $F\colon\Acal\to\Bcal$ be a functor of graded rig categories, and assume that $\Bcal$ is a groupoid. Then, the diagram
\begin{displaymath}
    \xymatrix{
        \Acal_{k}\ar[r]^{i_{k}}\ar[d]_{F_{k}}     &\Acal\ar[d]^{F}\\
        \Bcal_{k}\ar[r]_{i_{k}}    &\Bcal
    }
\end{displaymath}
is naturally 2-commutative: there is a natural isomorphism of functors of commutative Picard categories $F\circ i_{k}\simeq i_{k}\circ F_{k}$, which is induced by the monoidal structure of $F_{k}$. In a similar vein, if $N\geq 0$ is an integer and $\Acal^{\leq N}=\prod_{k\leq N}\Acal_{k}$, we have a natural section of the projection $\Acal\to\Acal^{\leq N}$, which is functorial with respect to functors of graded rig categories.

If $\Rcal$ is a graded ring category, we denote by $1+\Rcal^{+}$ the full subcategory of $\Rcal$ whose objects have $0$-th component isomorphic to $1_{\Rcal_{0}}$. It has a natural structure of symmetric monoidal category with respect to the product induced by $\Rcal$. It can be easily checked that $1+\Rcal^{+}$ is actually a commutative Picard category. The construction is clearly functorial with respect to functors of graded ring categories.

\subsection{Picardification of rig categories}
We apply the Picardification machinery in \textsection \ref{subsubsec:picardification}, in order to turn a rig category into a ring category.
\subsubsection{Ungraded rig categories}
We first treat the case of rig categories without grading.
\begin{proposition}\label{prop:rig-to-ring-1}
Let $(\Acal,\oplus,\otimes)$ be a rig category, which is also a groupoid. Let $V(\Acal)$ be the Picardification of $\Acal$ with respect to addition. 
\begin{enumerate}
    \item $V(\Acal)$ has a structure of ring category, such that the natural functor $\Acal\to V(\Acal)$ is a functor of rig categories.
    \item If $\Rcal$ is a ring category, then the natural functor $\Acal\to V(\Acal)$ induces an equivalence of categories of functors, of rig categories, $\Hom(V(\Acal),\Rcal)\simeq\Hom (\Acal,\Rcal)$.
\end{enumerate}
\end{proposition}
\begin{proof}
First of all, by Proposition \ref{prop:product} the symmetric monoidal structure $\otimes$ extends to $V(\Acal)$. For $V(\Acal)$ to be a ring category, we need to show that it satisfies Laplaza's coherence axioms. This can also be inferred from Proposition \ref{prop:product}. For concreteness, let us elaborate on the coherence diagram \cite[\textsection 1 (VI)]{Laplaza}, and leave the rest as an exercise. The diagram under consideration is 
\begin{equation}\label{eq:laplaza-vi}
    \xymatrix{
        A\otimes [B\otimes(C\oplus D)]\ar[r]\ar[d]   &A\otimes(B\otimes C\oplus B\otimes D)\ar[r]   &A\otimes (B\otimes C)\oplus A\otimes (B\otimes D)\ar[d]\\
        (A\otimes B)\otimes (C\oplus D)\ar[rr]    &          &(A\otimes B)\otimes C\oplus (A\otimes B)\otimes D.
    }
\end{equation}
We introduce the functor $F\colon\Acal\times\Acal\times\Acal^{\times 2}\to\Acal$, sending an object $(A,B,(C,D))$ to $A\otimes [B\otimes (C\oplus D)]$. If we endow $\Acal^{\times 2}$ with the monoidal structure given by the componentwise addition, then $F$ is a multimonoidal and multisymmetric functor. Similarly, we have a functor $G$ sending $(A,B,(C,D))$ to $(A\otimes B)\otimes C\oplus (A\otimes B)\otimes D$. The up-right composition corresponds to a natural transformation $\lambda\colon F\to G$. The left-bottom composition corresponds to another natural transformation $\mu$. The coherence condition requiring the commutativity of \eqref{eq:laplaza-vi} is recast as $\lambda=\mu$. Since $F$ and $G$ are multimonoidal and multisymmetric, from Proposition \ref{prop:product} we deduce that $F$ and $G$ extend to analogous functors for $V(\Acal)$, and that the natural transformations extend to $\widetilde{\lambda},\widetilde{\mu}$. Since $\lambda=\mu$, then we necessarily have $\widetilde{\lambda}=\widetilde{\mu}$, which is the coherence condition \cite[\textsection 1 (VI)]{Laplaza} for $V(\Acal)$.

The second claim of the statement is similarly deduced from Proposition \ref{prop:product}. We leave the details to the reader. 
\end{proof}

\subsubsection{Graded rig categories}  
In the case of a graded rig category, the Picardification is not sensitive to the grading, which is therefore lost. The following proposition fixes this issue.

\begin{proposition}\label{prop:rig-to-ring-2}
Let $\Acal=\prod_{k}\Acal_{k}$ be a graded rig category, which is also a groupoid. Define $V_{\mathrm{gr}}(\Acal)=\prod_{k}V(\Acal_{k})$, where $V(\Acal_{k})$ is the Picardification of $\Acal_{k}$ with respect to addition. 
\begin{enumerate}
    \item $V_{\mathrm{gr}}(\Acal)$ has a structure of graded ring category, such that the natural functor $\Acal\to V_{\mathrm{gr}}(\Acal)$ is a functor of graded rig categories.
    \item If $\Rcal$ is a graded ring category, then the natural functor $\Acal\to V_{\mathrm{gr}}(\Acal)$ induces an equivalence of categories of functors, of graded rig categories, $\Hom(V_{\mathrm{gr}}(\Acal),\Rcal)\simeq\Hom(\Acal,\Rcal)$.
\end{enumerate}
\end{proposition}

\begin{proof}
The argument is a slight variant of Proposition \ref{prop:rig-to-ring-1}. By Proposition \ref{prop:product}, the products $\Acal_{k}\times\Acal_{k'}\to\Acal_{k+k'}$ extend to products $V(\Acal_{k})\times V(\Acal_{k'})\to V(\Acal_{k+k'})$. We endow $V_{\mathrm{gr}}(\Acal)$, with the componentwise addition, and with the product defined by \eqref{eq:graded-product}. We claim that with these structures, $V_{\mathrm{gr}}(\Acal)$ is a graded ring category. We will only comment on the coherence conditions.

Consider a coherence condition on $\Acal$ expressed in terms of functors $F,G\colon\Acal^{\times N}\to\Acal$ and natural transformations $\mu,\lambda\colon F\to G$. We can decompose $F$ and $G$ according to the grading on the target $\Acal$. For every $k$, we obtain functors $F_{k},G_{k}\colon \Acal_{\ell_{1}}\times\ldots\times\Acal_{\ell_{r}}\to\Acal_{k}$ and natural transformations $\mu_{k},\lambda_{k}\colon F_{k}\to G_{k}$. The index $r$ and the degrees $\ell_{1},\ldots,\ell_{r}$ depend on $k$, but we don't specify the dependence to ease the notation.  

By Proposition \ref{prop:product}, the functors $F_{k}, G_{k}$, extend to analogous functors involving the Picardifications, and the natural transformations extend to natural transformations $\widetilde{\mu}_{k}$ and $\widetilde{\lambda}_{k}$. The point here is that $F_{k}, G_{k}$ can be rewritten as multisymmetric, multimonoidal functors of the form $\Acal_{m_{1}}^{\times N_{1}}\times\ldots\times\Acal_{m_{s}}^{\times N_{s}}\to\Acal_{k}$, where each $\Acal_{m_{i}}^{\times N_{i}}$ is endowed with the componentwise addition, and that there is a natural identification $V(\Acal_{m_{i}}^{\times N_{i}})=V(\Acal_{m_{i}})^{\times N_{i}}$. 

Finally, by the very definition of a graded rig category, the coherence condition $\mu=\lambda$ is equivalent to $\mu_{k}=\lambda_{k}$ for every $k$. This entails $\widetilde{\mu}_{k}=\widetilde{\lambda}_{k}$ for every $k$. Hence, the coherence conditions for $V_{\mathrm{gr}}(\Acal)$ hold.

The second claim of the statement is established along the same lines, and the verification is left to the reader.
\end{proof}

In order to distinguish the construction of the previous proposition from the usual Picardification, we introduce the following terminology.

\begin{definition}\label{def:gr-picardification}
With the hypotheses and notation as in Proposition \ref{prop:rig-to-ring-2}, we say that $V_{\mathrm{gr}}(\Acal)$ is the graded Picardification of $\Acal$.
\end{definition}

\begin{remark}\label{rem:picardification}
\begin{enumerate}
    \item\label{item:rem-picardification-1} If $\Acal$ is a graded rig groupoid, it is in general not true that $V(\Acal)=V_{\mathrm{gr}}(\Acal)$. By the universal property of the Picardification, there is however a natural functor\linebreak  $V(\Acal)\to V_{\mathrm{gr}}(\Acal)$. Taking into account the description of $V(\Acal)$ as $\Acal^{-1}\Acal[\Sigma^{-1}]$ provided by Lemma \ref{lemma:VA-tilde}, it can be seen that this functor is the identity on objects, while it is strictly injective on the $\Hom$ sets. 
    \item If $\Acal$ is a strict (graded) rig groupoid, then its Picardification is strict as well. In particular, it is strongly coherent. 
    \item\label{rem:picardification-3} The subject of ring completing a rig category has been considered by several authors. We mention the work of Baas--Dundas--Richter--Rognes \cite{BDRR}. From a rig category $\Acal$ whose translation functors are faithful, the authors produce functors of rig categories $\Acal\to Z\Acal$, $\Acal\to\ov{\Acal}$, where $\ov{\Acal}$ is a ring category and the induced maps on spectra are stable equivalences. We notice that the notion of ring category in \cite{BDRR} does actually not require $\ov{\Acal}$ to be a Picard category. To understand the difference between both constructions, we claim that if $\Acal$ is a groupoid, then our functor $\Acal\to V(\Acal)$ generally does not induce a stable equivalence on spectra. For this, we first recall that the spectrum of a Picard category such as $V(\Acal)$ is $[0,1]$-connected, meaning that its homotopy groups are concentrated in degrees 0 and 1. Next, by \cite[pp. 1657--1658]{ThomasonSpectral}, the functor $\Acal\to\Acal^{-1}\Acal$ induces a stable equivalence, while by construction $\Acal^{-1}\Acal\to V(\Acal)$ induces isomorphisms only on $\pi_{0}$ and $\pi_{1}$. We refer the reader to the appendix in \cite{Patel} for a thorough discussion of the homotopy equivalence between $[0,1]$-connected spectra and Picard categories.
\end{enumerate}
\end{remark}

\subsubsection{Picardification of strict rig categories}\label{subsubsec:strict-rig-categories}
Under a strictness assumption, the procedure of turning a rig category into a ring category, via Picardification, becomes rather concrete and simplifies some arguments. We will make this precise for rig categories without a grading, and leave the analogous graded case to the reader. 

Let $\Acal$ be a strict rig groupoid.\footnote{It would be enough to require that the addition be strictly symmetric.} Recall from Lemma \ref{lemma:VA-tilde} that the Picardification $V(\Acal)$ can be realized in terms of the Grayson--Quillen completion as $V(\Acal)=\Acal^{-1}\Acal[\Sigma^{-1}]$, where $\Sigma$ is the set of morphisms of $\Acal^{-1}\Acal$. With this presentation, the product $\otimes$ is extended to $V(\Acal)$ as follows. Given objects $(A,A')$, $(B,B')$, we set
\begin{equation}\label{eq:explicit-product-A-A}
    (A,A')\otimes (B,B')=(A\otimes B'\ \oplus\ A'\otimes B,\ A\otimes B\ \oplus\ A'\otimes B').
\end{equation}
For the product of morphisms, it suffices to consider $(A,A')\to (C,C')$ and $(B,B')\to (D,D')$ in $\Acal^{-1}\Acal$. We suppose these are represented by couples of morphisms $(f,f')$, with \linebreak $f\colon A\oplus X\to C$, $f'\colon A'\oplus X\to C'$, and $(g,g')$ with $g\colon B\oplus Y\to D$, $g'\colon B'\oplus Y\to D'$. We form the morphisms $f\otimes g'\ \oplus\ f'\otimes g$ and $f\otimes g\oplus f'\otimes g'$. After reordering terms by means of symmetry isomorphisms for addition, the first one is transformed into 
\begin{equation}\label{eq:product-morphism-1}
    A\otimes B'\ \oplus\ A'\otimes B\ \oplus\ Z\to C\otimes D'\ \oplus\ C'\otimes D,
\end{equation}
where $Z$ is given by
\begin{displaymath}
    Z=A\otimes Y\ \oplus\ X\otimes B'\ \oplus\ X\otimes Y\ \oplus\ A'\otimes Y\ \oplus\ X\otimes B\ \oplus\ X\otimes Y.
\end{displaymath}
A similar manipulation with $f\otimes g\ \oplus\ f'\otimes g'$ yields
\begin{equation}\label{eq:product-morphism-2}
    A\otimes B\ \oplus\ A'\otimes B'\ \oplus\ Z\to C\otimes D\ \oplus\ C'\otimes D'.
\end{equation}
The couple constituted by \eqref{eq:product-morphism-1}--\eqref{eq:product-morphism-2} defines a morphism $(A,A')\otimes (C,C')\to (B,B')\otimes (D,D')$. It is easy to see that it only depends on the equivalence class of $(f,f')$ and $(g,g')$ in the construction of $\Acal^{-1}\Acal$, recalled at the end of \textsection\ref{subsub:monoidalcategories}. It is tedious, but formal, to check that the construction indeed defines a bifunctorial product. The only non-trivial ingredients needed in the verification are Mac Lane's coherence theorems for addition and, most importantly, the fact that addition is strictly symmetric by assumption. That the procedure just described yields a well-behaved extension of the product in the strict case, was already observed by Thomason \cite[p. 572]{ThomasonPhony}.

Finally, for given objects $Z, W$ in $\Acal$, one can easily find natural isomorphisms \linebreak $(-Z)\otimes W\simeq -(Z\otimes W)\simeq Z\otimes (-W)$. For this, we recall that the inversion can be realized as $(A,B)\mapsto (B,A)$, see Remark \ref{rmk:thomason} and Remark \ref{rmk:thomason-strict}. Using the explicit formula for the product, the sought isomorphisms can be explicitly written down, componentwise, in terms of symmetries for the addition. There are unique such isomorphisms since $\Acal$ is strongly coherent. 

\subsection{Rationalization of ring categories}\label{subsec:rationalization-ring-categories}
Let $(\Rcal,\oplus,\otimes)$ be a ring category. Since $(\Rcal,\oplus)$ is a Picard category by definition, we can form the rationalization $\Rcal_{\QBbb}$. By Corollary \ref{cor:rationalization}, the product structure $\otimes$ extends to $\Rcal_{\QBbb}$. The argument of the proof of Proposition \ref{prop:rig-to-ring-1} can be adapted to show that $\Rcal_{\QBbb}$ has the structure of a ring category, and the natural functor $\Rcal\to\Rcal_{\QBbb}$ is a functor of ring categories. To adapt the argument of Proposition \ref{prop:rig-to-ring-1}, one needs to replace Proposition \ref{prop:product} with Corollary \ref{cor:rationalization}, and the rest of the reasoning is formally the same. Clearly, the construction is functorial with respect to functors of ring categories.

If $\Rcal=\prod_{k}\Rcal_{k}$ is a graded ring category, we define the graded rationalization $\Rcal_{\QBbb}:=\prod_{k}\Rcal_{k\ \QBbb}$. As in the previous paragraph, we can adapt the reasoning of Proposition \ref{prop:rig-to-ring-2}, and show that $\Rcal_{\QBbb}$ has the structure of a graded ring category, such that the natural functor $\Rcal\to\Rcal_{\QBbb}$ is a functor of graded ring categories. We notice that this construction does not in general coincide with the rationalization of $\Rcal$ as a commutative Picard category. Therefore, in the graded case, the notation $\Rcal_{\QBbb}$ is actually abusive. However, in the applications below, this should not be a source of confusion. As in the non-graded case, the construction just described is functorial with respect to functors of graded ring categories.

We refer to the procedures just presented as the rationalization of ring or graded ring categories. The rationalization of a ring category is universal, in a similar vein as in Proposition \ref{prop:rationalization}, among divisible ring categories. Divisibility here refers to the underlying additive Picard category. In the graded case, the notion of divisibility refers to the graded pieces taken individually. Then, a similar universal property of the rationalization holds. These universal properties will not be needed in the sequel. 

In the applications to the theory of Chern categories, we will combine the Picardification and the rationalization constructions. Precisely, we will first construct certain graded rig category $\Acal$. We will go on to perform the graded Picardification $V_{\mathrm{gr}}(\Acal)$. Finally, we will form the graded version of the rationalization $V_{\mathrm{gr}}(\Acal)_{\QBbb}$. 

\section{Multiplicative functors and virtual categories}\label{sec:multiplicative-functors-virtual-categories}
In this section, we review the notion of multiplicative functors, sometimes referred to as determinant functors, and of virtual categories. Given an exact category $\Ccal$, a multiplicative functor on $\Ccal$ is roughly a functor to a Picard category which behaves as the determinant of vector bundles with respect to exact sequences. The virtual category is a universal Picard category receiving a multiplicative functor from $\Ccal$. The formalism of virtual categories extends to more general Waldhausen categories. This leads us to the study of the virtual categories of schemes, attached to various categories of vector bundles and sheaves. 

\subsection{Multiplicative functors}\label{subsec:multfunctor}
We recall the notion of a multiplicative functor from an exact category to a monoidal groupoid. Most of the time, the latter will actually be a Picard category. We also stress that we are mainly interested in the commutative setting. 
\begin{definition}\label{def:multiplicative-functor}
 Let $\Ccal$ be an exact category and $(\Pcal,\otimes)$ a monoidal groupoid. We say that a functor $F: (\Ccal, \iso) \to \Pcal$ is multiplicative, if:
 \begin{enumerate}
     \item\label{item:mult-funct-1} for any exact sequence 
     \begin{displaymath}
         \Sigma\colon 0 \to A \to B \to C \to 0, 
     \end{displaymath}
     $F$ is equipped with an isomorphism
     \begin{displaymath}
         F(\Sigma): F(B) \to F(A) \otimes F(C),
     \end{displaymath}
     which is functorial with respect to isomorphisms of exact sequences. 
     \item\label{item:mult-funct-2} $F$ is equipped with the choice of an isomorphism 
     \begin{displaymath}
         F(0) \to 1_{\Pcal}.
     \end{displaymath}
     \item\label{item:mult-funct-3} if $\varphi\colon A \to B$ is an isomorphism, it gives rise to an exact sequence $\Sigma_{\varphi}: 0 \to A \to B \to C \to 0$ with $C=0$, and a sequence of isomorphisms
     \begin{displaymath}
         F(B) \overset{F(\Sigma_{\varphi})}{\longrightarrow} F(A) \otimes F(0) \to F(A) \otimes 1_{\Pcal} \to F(A).
     \end{displaymath}
     We require it to be equal to $F(\varphi)^{-1}.$ If we instead consider the induced exact sequence\linebreak $0 \to 0 \to A \to B \to 0$, we require that the analogously constructed morphism $F(A) \to F(B)$ is $F(\varphi).$
     \item\label{item:mult-funct-4} given an admissible filtration $A'' \subseteq A' \subseteq A$, the diagram 
     \begin{displaymath}
         \xymatrix{F(A) \ar[r] \ar[d] & F(A') \otimes F(A/A') \ar[d] \\ 
         F(A'') \otimes F(A/A'') \ar[r] & F(A'') \otimes F(A'/A'') \otimes F(A/A')}
     \end{displaymath}
     commutes. 
     
     \item\label{item:mult-funct-5} if moreover $\Pcal$ is symmetric, we say that $F$ is commutative\footnote{This terminology is preferred to symmetric since most of the time $\Pcal$ will be a commutative Picard category.} if for any objects $A,B$ of $\Ccal$, and given the exact sequences $\Sigma_{A,B}\colon 0 \to A \to A \oplus B \to B \to 0$ and $\Sigma_{B,A} \colon 0 \to B \to A \oplus B \to A \to 0 $, then 
    \begin{displaymath}
    \xymatrix{&  F(A\oplus B) \ar[dl]_{F(\Sigma_{A,B})} \ar[dr]^{F(\Sigma_{B,A})} & \\
    F(A) \otimes F(B) \ar[rr]_{c_{F(A), F(B)}} & & F(B) \otimes F(A)}
    \end{displaymath}
    commutes.
 \end{enumerate}
\end{definition}
The (commutative) multiplicative functors actually constitute the objects of a category, whose morphisms are natural transformations preserving the (commutative) multiplicative structure.

\begin{remark}\label{rem:detfunctor}
\begin{enumerate}
   \item\label{rem:detfunctor-1} In Definition \ref{def:multiplicative-functor} \eqref{item:mult-funct-4}, the quotients of the form $A/A'$ are only defined up to unique isomorphism. Nevertheless, the condition is meaningful since we have imposed compatibility with isomorphisms of exact sequences in \eqref{item:mult-funct-1}. Also, if we take brackets into account, the bottom and right arrows of the diagram actually have different targets. These are naturally identified via the associativity isomorphism, hence dispensing us from specifying the bracketing.
    \item \label{rem:detfunctor-0} For a natural transformation of (commutative) multiplicative functors $F$ and $G$ to preserve the multiplicative structure, it is enough to check the compatibility with exact sequences, as well as the compatibility of the trivializations $F(0) \simeq 1_{\Pcal}$ and $G(0)\simeq 1_{\Pcal}$.
\end{enumerate}
\end{remark}

\subsection{Virtual categories of exact and Waldhausen categories}
We briefly recapitulate known results about virtual categories of exact categories and Waldhausen categories. It parallels the Picardification construction of a symmetric monoidal category introduced in \textsection\ref{subsubsec:picardification}, in the sense that they both associate a universal Picard category with a category that enjoys a notion of a sum or a product structure. Since a symmetric monoidal category isn't necessarily an exact category, the two constructions are not obviously related.

The following summarizes the discussion in \cite[\textsection 4.4]{Deligne-determinant}. 
\begin{theorem}\label{thm:delignevirtual}
\begin{enumerate}
    \item Given an exact category $\Ccal,$ there is a commutative Picard category $V(\Ccal)$, and a commutative multiplicative functor $[-]\colon(\Ccal, \iso) \to V(\Ccal),$ such that for any commutative Picard category $\Pcal$,  $[-]$ induces an equivalence of the category of commutative multiplicative functors $(\Ccal, \iso) \to \Pcal$ and the category of commutative functors $V(\Ccal) \to \Pcal$ of commutative Picard categories.
    \item $V(\Ccal)$ can be realized as the fundamental groupoid of the Quillen construction $\Omega BQ\Ccal$. In particular, if $K_i(\Ccal)$ denotes the Quillen $K$-theory groups, we have $\pi_0(V(\Ccal)) = K_0(\Ccal)$ and $\pi_1(V(\Ccal)) = K_1(\Ccal)$.
\end{enumerate}
\end{theorem}
\qed

It was noted in \cite[\textsection 4.10]{Deligne-determinant}, extending an observation made in \cite{KnudsenMumford}, that for a commutative Picard category $\Pcal$, and any commutative multiplicative functor $(\Ccal, \iso) \to \Pcal$, there is a natural factorization 
\begin{equation}\label{eq:factorizationderived}
(\Ccal, \iso) \to (D^b(\Ccal), \hbox{iso}) \to \Pcal.
\end{equation} 
In \cite{Knudsen, Knudsen-err} multiplicative functors were called determinant functors. There this notion was extended to the category of bounded complexes on an exact category, with formally the same axioms. The main difference is that one more generally considers functors from the category of bounded complexes on $\Ccal$ together with the quasi-isomorphisms,
\begin{displaymath}
    (C^b(\Ccal), \hbox{quasi-iso}) \to \Pcal,
\end{displaymath} 
and requires functoriality with respect to quasi-isomorphisms of complexes. In\linebreak \cite[Theorem 2.3]{Knudsen} it is proven that the category of multiplicative functors $(C^b(\Ccal), \hbox{quasi-iso}) \to \Pcal$ is equivalent to that of multiplicative functors $(\Ccal, \iso) \to \Pcal$ as in \eqref{eq:factorizationderived}.  

In \cite{Muro:determinant} the notion of multiplicative functor was extended to general Waldhausen categories, which includes the case of $(C^b(\Ccal), \hbox{quasi-iso})$. \footnote{There seems to be missing a compatibility condition in the definition of their determinant functors, which should be analogous to our Definition \ref{def:multiplicative-functor} \eqref{item:mult-funct-2}--\eqref{item:mult-funct-3}.} Here we refer to the first section of \cite{Waldhausen} or the first section of \cite{ThomasonTrobaugh}  for a comprehensive introduction to Waldhausen categories and their relation to $K$-theory. We summarize the results of \cite{Muro:determinant} in the following:
\begin{theorem}\label{thm:Muro}
Let $\Wcal$ be a Waldhausen category. Then:
\begin{enumerate}
    \item\label{item:Muro-1} There exists a Picard category $V(\Wcal)$ with the analogous properties of Theorem \ref{thm:delignevirtual}, where the universal property is given with respect to multiplicative functors on Waldhausen categories as in \cite[Definition 1.2.3]{Muro:determinant}. 
    \item\label{item:Muro-2} If $F: \Wcal \to \Wcal'$ is an exact functor of Waldhausen categories, then there is a canonically induced functor $\widetilde{F}: V(\Wcal) \to V(\Wcal')$ of Picard categories. 
    \item\label{item:Muro-3} If $F$ as above induces a homotopy equivalence of the associated Waldhausen $K$-theory spectra, then $\widetilde{F}$ is an equivalence of Picard categories.
\end{enumerate}
\end{theorem}
\begin{proof}
 The first point then follows from \cite[Theorem 4.3.4, Theorem 4.4.2, Corollary 4.4.5]{Muro:determinant}. We notice that in \emph{loc. cit.} the existence of the virtual category is part of the existence property of a universal determinant functor. The second point follows by the universal property. The third point follows from  \cite[Theorem 4.5.2]{Muro:determinant}, which gives the connection between the virtual category and the Waldhausen $K$-theory spectra $K(\Wcal).$ 
\end{proof}

In \cite{ThomasonTrobaugh}, for Waldhausen categories of complexes in an abelian category, also called complicial, the cofibrations are implicitly assumed to be degreewise split monomorphisms whose quotient also lies in the same Waldhausen category. By Theorem 1.9.2 and Theorem 1.11.7 in \emph{op.cit.}, the same categories but with cofibrations being the degreewise admissible monomorphisms have homotopy equivalent $K$-theory spectra, and by Theorem \ref{thm:Muro}, they have equivalent virtual categories. In the sequel, we always work with degreewise admissible monomorphisms as cofibrations.

From Knudsen's results in \cite{Knudsen, Knudsen-err} and the universal property of virtual categories, we conclude the following.
\begin{corollary}\label{cor:equivalence-virtual-C-Cb}
Let $\Ccal$ be an exact category, and consider $C^{b}(\Ccal)$ as a Waldhausen category, whose weak equivalences are the quasi-isomorphisms. Then there is a natural equivalence of commutative Picard categories:
\begin{equation}\label{eq:functor-VC-VCb}
    V(\Ccal) \to V(C^b(\Ccal)). 
\end{equation}
Furthermore, sending a bounded complex $E^{\bullet}$ to $\sum (-1)^{k}[E^{k}]$ in $V(\Ccal)$, induces a functor of commutative Picard categories $V(C^{b}(\Ccal))\to V(\Ccal)$, which is a quasi-inverse of \eqref{eq:functor-VC-VCb}.
\end{corollary}
\qed

We recall that given a complicial biWaldhausen category\footnote{biWaldhausen means that the opposite category is also a Waldhausen category. To form the opposite category, one interchanges the role of fibrations and cofibrations.} $\Wcal$, the derived category $D(\Wcal)$ is obtained from that of $\Wcal$ by localizing the weak equivalences, hence formally inverting the weak equivalences. By assumption, the weak equivalences in a complicial biWaldhausen category contain the quasi-isomorphisms. We remark that inverting the quasi-isomorphisms automatically identifies homotopic maps, as explained in \cite[\textsection 1.9.6]{ThomasonTrobaugh}.

\begin{proposition}\label{prop:derivedvirtual}
 Consider an abelian category $\Acal$, and suppose that $\Wcal, \Wcal'$ are complicial biWaldhausen categories considered as subcategories of $(C(\Acal), \hbox{quasi-iso})$. We require that they are both closed under extensions and shifts. If $\Wcal$ is a subcategory of $\Wcal'$, inducing an equivalence of derived categories $D(\Wcal) \to D(\Wcal')$, then:

\begin{enumerate}
    \item There is an equivalence of virtual categories $V(\Wcal)\to V(\Wcal').$
    \item The universal functor $\Wcal \to V(\Wcal)$ factors through  $D(\Wcal)\to V(\Wcal)$.
\end{enumerate}
Furthermore, $D(\Wcal) \to D(\Wcal')$ is an equivalence of categories if for each object $W'$ of $\Wcal'$, there is a quasi-isomorphism $W \to W'$ with $W$ in $\Wcal$.

\end{proposition}
\begin{proof}
 The first point follows from Theorem \ref{thm:Muro} \eqref{item:Muro-3} combined with \cite[Theorem 1.9.8 \& Example 1.3.6]{ThomasonTrobaugh}. The second point is trivially true since weak equivalences are sent to isomorphisms in the virtual category. The last part of the proposition follows from the discussion in \cite[\textsection 1.9.7]{ThomasonTrobaugh}.

\end{proof}

We conclude with an analog of Proposition \ref{prop:product} for virtual categories of exact categories, also see Remark \ref{rem:orderdevelopproduct}. The argument goes along the same lines as in \emph{loc. cit.}, and we don't provide the details.
\begin{proposition}\label{prop:product-virtual}
Let $\Ccal_{1},\ldots,\Ccal_{n},\Ccal$ be exact categories, and $F\colon \Ccal_{1}\times\ldots\times\Ccal_{n}\to\Ccal$ a multiexact functor. Then there is an extension of $F$ into a functor, which is multiplicative in every entry,
\begin{displaymath}
    \widetilde{F}\colon V(\Ccal_{1})\times\ldots \times V(\Ccal_{n})\to V(\Ccal).
\end{displaymath}
The construction is functorial with respect to natural transformations.
\end{proposition}
\qed

\begin{corollary}\label{cor:product-virtual}
Let $\Ccal$ be an exact category, with addition law $\oplus$. Suppose it is further endowed with a biexact functor $\otimes\colon\Ccal\times\Ccal\to\Ccal$, such that $(\Ccal,\oplus,\otimes)$ is a rig category. Then, $V(\Ccal)$ inherits from $\Ccal$ a structure of ring category.
\end{corollary}
\begin{proof}
The proof is formally the same as that of Proposition \ref{prop:rig-to-ring-1}, replacing Proposition \ref{prop:product} with Proposition \ref{prop:product-virtual}.
\end{proof}

\subsection{Virtual categories of schemes}\label{sec:virtual-schemes}
We specialize the previous formalism of virtual categories to schemes. For this, we recall that a perfect complex in the derived category of $\Ocal_X$-modules, $D(\Ocal_X)$, is a complex that is locally quasi-isomorphic to a  bounded complex of vector bundles. Equivalently, a perfect complex is a complex that is pseudo-coherent, \emph{i.e.} locally quasi-isomorphic to a bounded above complex of vector bundles, and locally of finite Tor-amplitude. 

\begin{definition}\label{def:virtual-category-perfect}
Let $X$ be a scheme. 
\begin{enumerate}
    \item The virtual category of $X$ is defined as $V(X) = V(\Vect_X)$, where $\Vect_{X}$ is the exact category of vector bundles on $X$. 

    \item The virtual category of perfect complexes is defined as $V(\Pcal_X),$ where $\Pcal_X$ denotes the complicial biWaldhausen category of perfect complexes of globally bounded Tor-amplitude, in the abelian category of all $\Ocal_X$-modules. Here, we take the quasi-isomorphisms as the weak equivalences in $\Pcal_{X}$.
\end{enumerate}
\end{definition}

The above definition using globally bounded Tor-amplitude is to conform with the definition in \cite[Definition 3.1]{ThomasonTrobaugh}. We notice that if $X$ is quasi-compact, then a perfect complex is automatically of globally bounded Tor-dimension. The complicial biWaldhausen category of perfect complexes on $X$ will simply be denoted $\mathrm{Perf}_{X}$. The latter will only appear in Definition \ref{def:detcoh}, but will otherwise not be needed in the rest of the article. 

In the category $\Pcal_X$, the cofibration sequences, \emph{i.e.} the analogs of exact sequences, can be taken to be sequences of complexes of $\Ocal_X$-modules, which are degreewise exact sequences of $\Ocal_X$-modules.

\begin{definition}\label{def:divisorial}
    A scheme $X$ is called divisorial if $X$ is quasi-compact and quasi-separated and admits an ample family of line bundles \cite[Expos\'e II, D\'efinition 2.2.4]{SGA6}. 
\end{definition}

We record, for future reference, the following comparison of the two virtual categories.
\begin{lemma}\label{lemma:divisorial}
Suppose that $X$ is divisorial. Then the natural inclusion $\Vect_X \to \Pcal_X$ induces an equivalence of categories:
\begin{displaymath}
    V(X) \to V(\Pcal_X).
\end{displaymath}
\end{lemma}
\begin{proof}
This is the conjunction of \cite[Corollary 3.9 \& Proposition 3.10]{ThomasonTrobaugh}, which states a homotopy equivalence for the corresponding $K$-theory spectra, and Theorem \ref{thm:Muro}. 
\end{proof}

The direct sum and tensor product of vector bundles induce a ring category structure on $V(X)$. Indeed, since the tensor product of vector bundles is biexact, by Proposition \ref{prop:product-virtual} one also obtains a bifunctor $V(X) \times V(X)\to V(X)$, which we still denote by $\otimes.$ From the fact that $(\Vect_{X},\oplus,\otimes)$ is a rig category, we infer that $(V(X),\oplus,\otimes)$ has the natural structure of a ring category, by\linebreak Corollary \ref{cor:product-virtual}. 

Virtual categories of schemes enjoy functoriality properties similar to those of the $K$-theory of schemes, by construction. If $f\colon X\to Y$ is a morphism of schemes, the operation of pulling back a vector bundle is exact and induces a pullback functor $f^\ast: V(Y) \to V(X)$, which is actually a functor of ring categories. The assignment $f\mapsto f^{\ast}$ is a contravariant functor, up to natural transformation of the form $(g\circ f)^{\ast}\simeq f^{\ast}\circ g^{\ast}$. It satisfies a cleavage condition, inherited from the case of vector bundles. We recall this condition is expressed by two constraints. The first one is that whenever $g=\id$ or $f=\id$, the isomorphism $(g\circ f)^{\ast}\simeq f^{\ast}\circ g^{\ast}$ is the identity. The second one is a compatibility condition with the associativity law for the composition: there are a priori two natural isomorphisms $(h\circ g\circ f)^{\ast}\simeq f^{\ast}\circ g^{\ast}\circ h^{\ast}$, and they are required to be equal.  
 
The same remarks apply to $V(\Pcal_X)$, replacing the pullback and the tensor product with their derived versions, relying on the constructions in the derived category and Proposition \ref{prop:derivedvirtual}. We leave the precise construction to the reader but refer to \cite[\textsection 3.14, \textsection 3.15]{ThomasonTrobaugh} for the corresponding treatment in the case of the spectra of $K$-theory.  

We next discuss the covariant functoriality of the virtual categories. Taking direct images of complexes of sheaves is not exact, but we can adopt the approach from the situation in algebraic $K$-theory. 

\begin{proposition}\label{prop:properties-direct-image-virtual}
Let $f: X\to Y$ be a proper morphism of schemes, with $Y$ quasi-compact. Suppose that the derived direct image functor $Rf_\ast$ sends perfect complexes to perfect complexes. Then it induces a commutative multiplicative functor $f_!: V(\Pcal_X) \to V(\Pcal_Y)$, referred to as the direct image functor, such that:
\begin{enumerate}
    \item\label{item:properties-direct-image-virtual-1} The direct image functor is multiplicative. It is compatible with the composition of morphisms satisfying the assumptions in the proposition.
    \item\label{item:properties-direct-image-virtual-2} If $E$ is a virtual perfect complex on $X$, and $F$ a virtual perfect complex on $Y$, then there is a natural projection formula isomorphism,
    \begin{displaymath}
        f_! \left(E \otimes f^\ast F\right)\to f_! E \otimes F,
    \end{displaymath}
    of functors of Picard categories $V(\Pcal_X) \times V(\Pcal_Y) \to V(\Pcal_Y).$ 
    \item\label{item:properties-direct-image-virtual-3} Suppose $E$ is a virtual perfect complex on $X$, and suppose $g \colon Y' \to  Y$ is a morphism of quasi-compact schemes, such that $g$ is Tor-independent with  $f$. Suppose in addition that the base change $f^\prime \colon X' \to Y'$ is such that $R{f'}_\ast$ also preserves perfect complexes. Then there is a canonical base change isomorphism 
    \begin{displaymath}
        g^\ast f_! E \to f'_! {g'}^\ast E
    \end{displaymath}
    of functors of Picard categories $V(\Pcal_X) \to V(\Pcal_{Y'}).$
\end{enumerate}
\end{proposition}

\begin{proof}
During the proof, for a quasi-compact scheme $Z$, denote by $\Wcal_{Z}$ the complicial biWaldhausen categories of  bounded below perfect complexes of flasque $\Ocal_{Z}$-modules on $Z$, and $\Pcal^+_Z$ the category of bounded below perfect complexes. There are then natural functors \linebreak  $V(\Wcal_{Z}) \to V(\Pcal_Z)$ and $V(\Pcal^+_Z) \to V(\Pcal_Z)$, which are equivalences of categories, by  \cite[Lemma 3.5]{ThomasonTrobaugh} and Proposition \ref{thm:Muro} \eqref{item:Muro-3}. 

To construct the functor, let $E$ be a bounded below perfect complex on $X$ and fix quasi-isomorphism $E\to I^\bullet$, with $I^\bullet$ bounded below complex of injective $\Ocal_X$-modules. The complex $I^{\bullet}$ is an object of $\Wcal_{X}$. By the assumption that $Rf_{\ast}$ preserves perfect complexes, taking direct images provides a complex $f_\ast I^\bullet$ in $\Wcal_Y$. The image of this object in $V(\Wcal_Y)$ depends, a priori, on the choice of resolution, but only up to unique isomorphism since $\Wcal_{Y}\to V(\Wcal_{Y})$ factors over the derived category $D(\Wcal_Y)$ by Proposition \ref{prop:derivedvirtual}. By the same token, this construction defines a functor 
\begin{displaymath}
    [Rf_{\ast}(\cdot)]\colon (\Pcal^+_{X},\hbox{quasi-iso})\to  V(\Wcal_Y) \to V(\Pcal_Y).
\end{displaymath}
We argue that  this is a multiplicative functor as in \cite[Definition 1.2.3]{Muro:determinant}. Given an exact sequence $\Sigma: 0 \to E \to F \to G \to 0$ of bounded below perfect complexes, by \cite[\href{https://stacks.math.columbia.edu/tag/013T}{013T}]{stacks-project} there is a short exact sequence of injective resolutions $0 \to I^\bullet \to J^\bullet \to K^\bullet \to 0$ and hence a short exact sequence $0 \to f_\ast (I^\bullet) \to f_\ast (J^\bullet) \to f_\ast (K^\bullet) \to 0$. From this, we conclude an isomorphism 
\begin{displaymath}
    [\Sigma]: [Rf_\ast F] \to [Rf_\ast E] + [Rf_\ast G].
\end{displaymath}
It is straightforward to verify the axioms of \cite[Definition 1.2.3]{Muro:determinant}, except possibly the associativity. This however is again a consequence of \cite[\href{https://stacks.math.columbia.edu/tag/013T}{013T}]{stacks-project}, which assures that one can choose the first resolution $I^\bullet$ from which the other ones are constructed. 

Hence by the universal property of virtual categories, $E \mapsto [Rf_\ast E]$ defines a functor of Picard categories 
\begin{displaymath}
    V(\Pcal_{X}^{+}) \to V(\Pcal_{Y})
\end{displaymath}
and by inverting the equivalence $V(\Pcal^+_X) \to V(\Pcal_X)$ we obtain the functor $f_!$. It is formal to verify that for a composition $fg$ of morphisms, there is a natural isomorphism  $(fg)_! \to f_! g_! $.

To prove the other properties, it is convenient to follow the approach of the construction of the direct image functor as in \cite[\textsection 3.16]{ThomasonTrobaugh}. To be able to rely on this, we recall that any bounded below complex $E$ admits a Godement resolution \cite[\href{https://stacks.math.columbia.edu/tag/0FKT}{0FKT}]{stacks-project}, which amounts to a quasi-isomorphism $E\to T(E)$ where $T(E)$ is a complex of flasque sheaves. If $E$ is perfect and bounded below, then $T(E)$ is an object of $\Wcal_{X}$. The Godement resolution is functorial and preserves exact sequences. Taking direct images and considering the image in the virtual category provides another multiplicative functor $(\Pcal_{X}^{+}, \hbox{quasi-iso}) \to V(\Pcal_{Y})$, and hence a functor of Picard categories $f_{?}\colon V(\Pcal_{X})\to V(\Pcal_{Y})$. We have an isomorphism $Rf_\ast E \to f_\ast T(E)$ in $D(\Wcal_{Y})$. This induces a natural isomorphism of functors of Picard categories $f_! \to f_?$, the proof of which we leave to the interested reader. We henceforth identify the two functors.

To prove the projection formula and the base change formula, we can now refer the reader to Proposition 3.17 and Proposition 3.18 of \cite{ThomasonTrobaugh}. While the statements concern homotopies of maps between spectra, the proofs actually construct quasi-isomorphisms of complexes in appropriate Waldhausen categories, where the constructed functors involve the Godement resolution discussed above. This suffices to provide the natural transformations in the virtual category.

\end{proof}

From Lemma \ref{lemma:divisorial} we conclude the following corollary:

\begin{corollary}
    With the notation and assumptions as in Proposition \ref{prop:properties-direct-image-virtual}, assume moreover that $Y$ (and $Y'$ where it applies) are divisorial. Then there is a 
    direct image functor $f_! \colon V(X)\to V(Y), $ satisfying the corresponding properties.
\end{corollary}
\qed

In view of Proposition \ref{prop:properties-direct-image-virtual}, we now provide several criteria which ensure that the derived direct image preserves perfection.
\begin{proposition}\label{prop:perfect-conditions}
Let $f: X\to Y$ be a morphism. In either of the following situations, $Rf_\ast$ sends perfect complexes to perfect complexes:
\begin{enumerate}
    \item\label{item:perfect-conditions-1} $f$ is a proper morphism of finite Tor-dimension, with $X$ and $Y$ are Noetherian . \cite[\href{https://stacks.math.columbia.edu/tag/0B6G}{0B6G}]{stacks-project} \cite[\href{https://stacks.math.columbia.edu/tag/0684}{0684}]{stacks-project}
    \cite[\href{https://stacks.math.columbia.edu/tag/069C}{069C}]{stacks-project}.
    \item\label{item:perfect-conditions-2} $f$ is flat, proper, and of finite presentation. Moreover, in this case, the formation of $Rf_\ast $ commutes with arbitrary base changes  \cite[\href{https://stacks.math.columbia.edu/tag/0B91}{0B91}]{stacks-project}.
    \item\label{item:perfect-conditions-3} More generally, any proper perfect morphism \cite[Example 2.2(a)]{LipNee} .
\end{enumerate}  
\end{proposition}
\qed

Since the virtual category of perfect complexes (cf. Definition \ref{def:virtual-category-perfect}) is defined in terms of perfect complexes with globally finite Tor-amplitude, we also include the following criteria which address the preservation of this.

\begin{proposition}
    Suppose $f: X\to Y$ is a morphism of schemes. In the following cases, $Rf_\ast$ sends perfect complexes to perfect complexes, and moreover sends complexes of globally finite Tor-amplitude to complexes of globally finite Tor-amplitude:
    \begin{enumerate}
        \item Any situation in Proposition \ref{prop:perfect-conditions}, with $Y$ quasi-compact. 
        \item $f$ satisfies the condition $(C_n)$. 
        \item $f$ is a regular closed immersion of codimension $d$, for a positive integer $d$. 
    \end{enumerate}
\end{proposition}
\begin{proof}
    The first item follows from the fact that on a quasi-compact scheme, being of locally bounded Tor-amplitude implies being of globally bounded Tor-amplitude. We next treat the case of a morphism satisfying the condition $(C_n)$.
    
    The fact that $Rf_\ast$ sends perfect complexes to perfect complexes is already stated in\linebreak Proposition \ref{prop:perfect-conditions} \eqref{item:perfect-conditions-2}. We prove the proposition by providing uniform bounds on the Tor-ampli\--tudes for $Rf_\ast E$ in the case $Y$ is affine, and hence $X$ is quasi-compact and separated, and even divisorial since we may assume $X\to Y$ is projective. 
    
    We denote by $E$ our perfect complex of $\Ocal_X$-modules, which we suppose has Tor-amplitude contained in $[a,b]$. We want to prove that $Rf_{\ast}E$ has Tor-amplitude in $[a,b+n]$. By \cite[Theorem 2.4.3]{ThomasonTrobaugh} we can moreover suppose that $E$ is a complex of vector bundles, with non-zero terms contained in $[a,b]$. By successively filtering $E$ by subcomplexes of vector bundles, one readily reduces to the case of $E$ itself being a vector bundle. 
    
    By Noetherian approximation, there is a Noetherian affine scheme $Y_0$, a morphism $g: Y\to Y_0$, a scheme $X_0$ and a vector bundle $E_0$, and a morphism $X_0 \to Y_0$ satisfying the condition $(C_n)$, such that $(X\to Y, E)$ is the pullback of $(X_0 \to Y_0, E_0)$ via $g$. Now, since $f$ is flat and by the vanishing of higher direct images above the fiber dimension for Noetherian schemes  \cite[\href{https://stacks.math.columbia.edu/tag/02V7}{02V7}]{stacks-project}, we deduce that $Rf_{0, \ast} E_0$ is of globally finite Tor-amplitude contained in $[0,n]$ by the criterion in \cite[III, Corollaire 3.7.1]{SGA6}. From this we can conclude, observing that $Lg^\ast$ preserves Tor-amplitude \cite[III, Corollaire 3.5.2]{SGA6}, and $Rf_\ast E \simeq Lg^\ast Rf_{0, \ast} E_0$ by Tor-independent base change in the derived category, cf. \cite[\href{https://stacks.math.columbia.edu/tag/08IB}{08IB}]{stacks-project}.

    The case of a regular closed immersion is similar but more direct. It follows by combining\linebreak \cite[Example 4.1.1, Proposition 4.4, Corollaire 4.5.1]{SGA6}.

\end{proof}

To conclude this section, we take the opportunity to discuss the determinant of the cohomology of the introduction of the article. For this, we recall that by the work of Knudsen--Mumford \cite{KnudsenMumford}, given a perfect complex $E$ on a scheme $Y$, there is a naturally defined graded line bundle $\det E$ on $Y$, with the property that if $E$ is a locally free sheaf of rank $r$, seen as a complex concentrated in degree 0, then
\begin{displaymath}
    \det E = (r, \Lambda^{r} E).
\end{displaymath}
By their construction, it defines a commutative multiplicative functor from the category of perfect complexes on $Y$, which we recall is denoted by $\mathrm{Perf}_{Y}$, to the category of graded line bundles on $Y$, which we recall is denoted by $\Picfr(Y)_{\mathrm{gr}}$. We denote this and the induced functor of Picard categories $V(\mathrm{Perf}_{Y}) \to \Picfr(Y)_{\mathrm{gr}}$ by $\det$. It is sometimes useful to forget the grading, and by abuse of notation, we still denote the same functor by $\det$.

\begin{definition}\label{def:detcoh}
    Suppose $f: X \to Y$ is a flat, proper morphism of finite presentation. The functor that takes a perfect complex $E$ and  produces a graded line bundle $\det Rf_{\ast} (E)$ is called the determinant of the cohomology and is also denoted by $\lambda_{f}(E)$. 
    
    This is the same as the composition of functors $\det f_!$ when $E$ is of finite Tor-amplitude.
\end{definition}

It follows from Proposition \ref{prop:perfect-conditions} \eqref{item:perfect-conditions-2} that the construction of the determinant with the cohomology commutes with arbitrary base change. 

\section{Chern categories and line distributions}\label{subsection:Chern-categories}

To properly formulate statements involving functoriality of intersection bundles, we develop a formalism of Chern categories and categorical Chern classes. This amounts to a naive categorification of the Chow rings and the Chern morphisms appearing in \cite{SGA6} and \cite{FultonLang}. The objects are rational formal power series in the ranks and Chern classes of vector bundles, and the morphisms are abstract incarnations of basic properties of Chern classes, such as the Whitney formula. However, we leave aside finer structures such as projective bundle formulas or triviality of Chern classes of higher degree. While our Chern categories enjoy a pullback functoriality, they do not admit direct image functors. To partially fix this, we introduce the formalism of line distributions, which are functors from Chern categories to Picard categories of line bundles on base schemes. The main example in this paper will be provided by intersection bundles. 

The formalism we develop in this section is based on the categorical material expounded in Section \ref{sec:monoidal-picard}, Section \ref{section:ring-categories}, and Section \ref{sec:multiplicative-functors-virtual-categories}.

\subsection{Universal Chern categories}\label{subsec:universal-chern-categories}
We begin with a purely categorical construction, that we call the universal Chern category. This category is tailored to be universal among graded strict ring categories receiving Chern functors, which are defined on vector bundles and behave like the total Chern class. A drawback of the universal Chern category is that it does not properly account for the topology of our schemes. This issue will be addressed in a later stage \textsection \ref{subsec:Chern-category-quasi-compact}.

\subsubsection{Chern functors}
We introduce the category of Chern functors and the associated notion of universal Chern category. 
\begin{definition}\label{def:axiom-chern-functor}
 Let $X$ be a scheme and $\Rcal$ a graded strict ring category. The category of Chern functors for $X$, with values in $\Rcal$, consists of the following: 
 \begin{enumerate}
     \item The objects, called Chern functors, are symmetric monoidal functors $\cfrak: (V(X),\oplus) \to (\Rcal,\otimes)$, whose components of degree $0$ are endowed with an isomorphism with the constant functor $1_{\Rcal_{0}}$. The functor determined by the component of degree $k$ is denoted by $\cfrak_k : V(X) \to \Rcal_k.$ 
     \item The morphisms are natural transformations of the underlying symmetric monoidal functors, which induce the identity transformations on the components of degree 0.
 \end{enumerate}

 \end{definition}

\begin{remark}\label{rmk:superficial-remarks-Chern-functors}
\begin{enumerate}
    \item\label{rmk:supremChfun-1} A Chern functor $ V(X)\to\Rcal$ defines a functor of commutative Picard categories $V(X)\to 1+\Rcal^{+}$. Therefore, via Theorem \ref{thm:delignevirtual}, the notion of Chern functor is equivalent to the notion of commutative multiplicative functor $(\Vect_{X},\iso)\to 1+\Rcal^{+}$, whose degree 0 component is equipped with an isomorphism with the constant functor $1_{\Rcal_{0}}$. We will pass from one point of view to the other, without further explanation. 
    \item\label{rmk:supremChfun-2} The degree one component of a Chern functor defines a functor of commutative Picard categories $\cfrak_{1}\colon (V(X),+)\to (\Rcal,+)$. 
\end{enumerate}
\end{remark}

\begin{definition}\label{def:universal-chern-category}
A universal Chern category for $X$ is a graded strict ring category $\CHfrak_{u}(X)$, with a Chern functor $\cfrak_{X}\colon V(X)\to\CHfrak_{u}(X)$, such that for any graded strict ring category $\Rcal$, composition with $\cfrak_{X}$ induces an equivalence of categories of functors
\begin{equation}\label{eq:Chern-category-equivalence}
     \Xi: \Hom_{\substack{\mathrm{graded\ ring}\\ \mathrm{category}}}(\CHfrak_{u}(X), \Rcal)  \to \Hom_{\mathrm{Chern}}(V(X), \Rcal).
\end{equation}
Unless there is some danger of confusion, we will write $\cfrak$ instead of $\cfrak_{X}$.
\end{definition}

In practice, the equivalence of categories \eqref{eq:Chern-category-equivalence} is used in the lines of the discussion after Definition \ref{def:Picardification}. In particular, if a universal Chern category exists, it is unique up to equivalence of ring categories. The equivalence itself is determined up to unique natural transformation. We may thus speak about \emph{the}, rather than \emph{a}, universal Chern category. In the sequel, we provide the construction. 

\subsubsection{An intermediate Chern category}\label{subsubsec:intermediate-category} Towards the construction of $\CHfrak_{u}(X)$, we begin by introducing an intermediate category $\CHfrak_{+}(X)$. 

Let $k\geq 0$ be an integer, and define  $\mathfrak{C}^{k}(X)$ as the free category associated with the following graph:\footnote{See \cite[Chapter II, Section 7]{MacLane-cat} for a review of free categories associated with graphs.}
\begin{itemize}
    \item[(V)] \emph{Vertices}. Ordered finite formal sums of ordered finite products of the form
    \begin{equation} \label{eq:Chern-monomial}
        \cfrak_{k_{1}}(E_{1})\cdot\ldots\cdot\cfrak_{k_{m}}(E_{m}),
    \end{equation}
    where the $\cfrak_{k_{i}}(E_{i})$ are symbols associated with integers $k_{i}\geq 0$, with $\sum_{i}k_{i}=k$, and to vector bundles $E_{i}$ on $X$. We follow the usual convention that the empty sum (resp. product) is $0$ (resp. $1$). This can be rephrased in terms of the free additive monoid on the free multiplicative monoid on the symbols $\cfrak_{k}(E)$. Symbols such as \eqref{eq:Chern-monomial} are called Chern monomials of degree $k$. Finite formal sums of Chern monomials of degree $k$ are called Chern polynomials of homogenous degree $k$.
 
    \item[(E)] \emph{Edges}. Given by symmetries and symbolic isomorphisms, defined as follows:
        \begin{enumerate}
            \item \emph{Symmetries}. If $A$ and $B$ are vertices which differ at most by a reordering of the sum and the products defining the monomials, then there is an associated edge $A\rightarrow  B$. In particular, for each object $A$, there is an associated edge $A\to A$, called the identity symmetry.
            \item \emph{Symbolic isomorphisms}. The edges between vertices determined by performing one of the below operations at a time, where the double-headed arrow will indicate that both directions are allowed. 
 \begin{enumerate}
 \item\label{item:formal-iso-1} \emph{Elementary projection formula}: for every vector bundle $E$,
 \begin{displaymath}
    \cfrak_{0}(E) \leftrightarrow  1.
 \end{displaymath}

  \item\label{item:formal-iso-3} \emph{Rank triviality for the zero vector bundle}: for every integer $k > 0$,
 \begin{displaymath}
     \cfrak_k(0) \leftrightarrow 0.
 \end{displaymath}

 \item\label{item:formal-iso-5} \emph{Whitney isomorphism}: for every exact sequence $0\to E^{\prime}\to E\to E^{\bis}\to 0$ of vector bundles and every integer $k\geq 0$,
 \begin{displaymath}
     \cfrak_k(E) \leftrightarrow  \sum_{i=0}^k \cfrak_{i}(E^{\prime}) \cdot \cfrak_{k-i}(E^{\bis}).
 \end{displaymath}
 
 \item \emph{Isomorphisms of vector bundles}: for every isomorphism of vector bundles $\varphi\colon E\to F$ and every integer $k\geq 0$,  
 \begin{displaymath}
     [\varphi] =  [\varphi]_k \colon \cfrak_k(E) \to \cfrak_k(F).
 \end{displaymath}
 \end{enumerate}
\end{enumerate}
We refer to the morphisms in $\Cfrak^{k}(X)$ as formal isomorphisms. 
\end{itemize}

The category $\CHfrak_+^k(X)$ is defined as a quotient category of $\Cfrak^{k}(X)$, by identifying combinations of formal isomorphisms, as follows: 
\begin{itemize}
     \item We identify different orderings of performing the formal isomorphisms, to the effect that all diagrams that might reasonably commute, do.
    \item The identity symmetry is identified with the identity map. Given symmetries $f\colon A\to B$, $g\colon B\to C$ and $h\colon A\to C$, we identify $g\circ f$ with $h$. In particular, between two given objects, there is at most one symmetry. 
    \item The double-headed arrows compose to the identity and are hence mutual inverses. Also, whenever it makes sense, we identify $[\varphi\circ\psi]_{k}$ with $[\varphi]_{k}\circ[\psi]_{k}$, and $[\id]_{k}$ with the identity. 
    \item We perform the minimal necessary identifications so that $\cfrak_{k}\colon (\Vect_{X},\iso)\to\CHfrak_{+}^{k}(X)$ satisfies the analogs of the axioms of commutative multiplicative functors of\linebreak Definition \ref{def:multiplicative-functor}. For instance, this means that in $\CHfrak_{+}^{k}(X)$, the Whitney symbolic isomorphism \eqref{item:formal-iso-5} above satisfies a compatibility condition with admissible filtrations similar to Definition \ref{def:multiplicative-functor} \eqref{item:mult-funct-4}. For an object of the form $\cfrak_{k}(E)$ this goes as follows. If $E^{\prime\prime}\subseteq E^{\prime}\subseteq E$ is an admissible filtration, then there are two possible reiterated Whitney isomorphisms, depending on whether we first filter by $E^{\prime\prime}$ or $E^{\prime}$. These are, respectively,
    \begin{equation}\label{eq:Whitney-ck-filtration-1}
        \cfrak_{k}(E)\to\sum_{i=0}^{k}\sum_{j=0}^{i}\cfrak_{j}(E^{\prime\prime})\cfrak_{j-i}(E^{\prime}/E^{\prime\prime})\cfrak_{k-i}(E/E^{\prime})
    \end{equation}
    and
    \begin{equation}\label{eq:Whitney-ck-filtration-2}
        \cfrak_{k}(E)\to \sum_{j=0}^{k}\sum_{i=0}^{k-j}\cfrak_{j}(E^{\prime\prime})\cfrak_{i}(E^{\prime}/E^{\prime\prime})\cfrak_{k-j-i}(E/E^{\prime}).
    \end{equation}
    We identify these two isomorphisms, after applying the unique symmetry between the right-hand side of \eqref{eq:Whitney-ck-filtration-1} and the right-hand side of \eqref{eq:Whitney-ck-filtration-2}. We extend to Chern polynomials by performing one such identification at a time.   
 \end{itemize}
We record the following properties, which are immediate from the construction.

\begin{lemma}
The category $\CHfrak^k_{+}(X)$ is a strictly symmetric monoidal category with respect to addition, and all of its morphisms are isomorphisms.
\end{lemma}
\qed

Finally, we define $\CHfrak_{+}(X)=\prod_{k}\CHfrak^{k}_{+}(X)$, which is thus a strictly symmetric monoidal category with respect to addition, whose morphisms are isomorphisms. The objects of $\CHfrak_{+}(X)$ will be called positive Chern power series, or simply Chern power series if there is no risk of confusion.

\begin{lemma}\label{lemma:cfrak-CHplus}
The category $\CHfrak_{+}(X)$ is canonically equipped with a structure of graded strict rig category. With respect to the product structure, the functor $\cfrak=\sum_{k}\cfrak_{k}\colon (\Vect_{X},\iso)\to\CHfrak_{+}(X)$ is commutative multiplicative in the sense of Definition \ref{def:multiplicative-functor}, and $\cfrak_{0}$ is isomorphic to the constant functor 1.
\end{lemma}
\begin{proof}
For the first claim, we need to define graded products
\begin{equation}\label{eq:product+k}
    \CHfrak_+^k(X) \times \CHfrak_+^{k'}(X)  \to \CHfrak_+^{k+k'}(X).
\end{equation}
On objects, these are defined by multiplication of Chern monomials, and extended to sums of monomials by developing as in Remark \ref{rem:orderdevelopproduct} \eqref{item:orderdevelopproduct-1}. It is easily checked that this defines a bidistributive product functor $\CHfrak_{+}(X) \times \CHfrak_{+}(X) \to \CHfrak_{+}(X)$ by the formulas in \eqref{eq:graded-product}, and that with this product, $\CHfrak_{+}(X)$ is naturally a graded strict rig category. There are two key observations for the proof: \emph{i)} by construction, between two given objects of $\CHfrak^{k}_{+}(X)$ there is at most one symmetry, and \emph{ii)} in the specific case of $\CHfrak_{+}(X)$, all the morphisms involved in the axioms of a rig category are induced by symmetries.

The second claim is clear since the construction of $\CHfrak_{+}(X)$ has in particular been tailored so that $\cfrak$ is a commutative multplicative functor, with $\cfrak_{0}$ isomorphic to the constant functor 1. 
\end{proof}

\begin{remark}
In $\CHfrak_{+}(X)$ there are a priori two isomorphisms $\cfrak(0) \to \cfrak(0) \cfrak(0) $, one given by the identification $\cfrak(0)\simeq 1$ and one from the Whitney isomorphism of the zero short exact sequence. These are implicitly identified since $\cfrak$ is a commutative multiplicative functor. However, we notice that the one deduced from the Whitney isomorphism does not imply the other, because $\CHfrak_{+}(X)$ is not a Picard category. 
\end{remark}

The following lemma describes the piece of degree 0 of $\CHfrak_{+}(X)$. It points towards the pathology of the construction anticipated at the beginning of this subsection: if $X$ decomposes as a disjoint union of schemes, then $\CHfrak_{+}(X)$ can not detect this decomposition.

\begin{lemma}\label{lemma:CHplus0-N}
Let $\NBbb$ be the set of natural numbers, considered as a rig category. Then, there is a natural equivalence of rig categories $\NBbb\to\CHfrak_{+}^{0}(X)$.
\end{lemma}
\begin{proof}
We define a functor $\NBbb\to\CHfrak_{+}^{0}(X)$ by sending a natural integer $n\geq 0$ to $n$ times the sum of the object 1, and all the morphisms to identities. This is clearly a functor of rig categories. A functor in the other direction, with the property of being compatible with the rig structures, is determined by the relationship $\cfrak_{0}(E)\simeq 1$ for every vector bundle, and sending all the morphisms to identities. The construction of the morphisms in $\CHfrak^{0}_{+}(X)$ is such that, via the isomorphisms $\cfrak_{0}(E)\simeq 1$, all the isomorphisms correspond to identities of objects of the form $n\cdot 1$. From this, it follows that the exhibited functors are mutually quasi-inverse.
\end{proof}
\begin{remark}\label{rmk:pre-Chern-empty}
If $X$ is the empty scheme, then all the vector bundles are isomorphic to $0$. This entails that all the categories $\CHfrak^{k}_{+}(X)$, for $k\geq 1$, are equivalent to the zero group $\lbrace 0\rbrace$ seen as a symmetric monoidal category. Hence, by the lemma, the whole category $\CHfrak_{+}(X)$ is equivalent to $\NBbb$.
\end{remark}

\subsubsection{Construction of the universal Chern category} The category $\CHfrak_{+}(X)$ fulfills the assumptions of Proposition \ref{prop:rig-to-ring-2}. Thus, we can make the following definition.
\begin{definition}
The universal Chern category of $X$, denoted by $\CHfrak_{u}(X)$, is the graded Picardification of $\CHfrak_{+}(X)$. The objects of this category are called Chern power series.
\end{definition}
By construction of the graded Picardification (cf. Definition \ref{def:gr-picardification}), the category $\CHfrak_{u}(X)$ is the product $\prod_{k}\CHfrak_{u}^{k}(X)$, where $\CHfrak_{u}^{k}(X)$ is the Picardification of $\CHfrak_{+}^{k}(X)$. 

The functor $\cfrak\colon (\Vect_{X},\iso)\to \CHfrak_{+}(X)$ followed by $\CHfrak_{+}(X)\to\CHfrak_{u}(X)$ gives rise to a commutative multiplicative functor $ (\Vect_{X},\iso)\to \CHfrak_{u}(X)$, with respect to the product structure on $\CHfrak_{u}(X)$. By Lemma \ref{lemma:cfrak-CHplus}, it actually lands in $1+\CHfrak_{u}(X)^{+}$, which is a stritly commutative Picard category. Hence, there is an induced functor $V(X)\to 1+\CHfrak_{u}(X)^{+}$, that we still denote by $\cfrak$. The component of degree 0 is necessarily isomorphic to the constant funcor 1, and hence $\cfrak\colon V(X)\to\CHfrak_{u}(X)$ is a Chern functor.

\begin{theorem}\label{thm:CH-is-universal}
The universal Chern category $\CHfrak_{u}(X)$, together with the Chern functor \linebreak $\cfrak\colon V(X)\to\CHfrak_{u}(X)$, is universal with respect to graded strict ring categories equipped with Chern functors, that is \eqref{eq:Chern-category-equivalence} is an equivalence of categories.
\end{theorem}
\begin{proof}

We just need to prove that \eqref{eq:Chern-category-equivalence} is an equivalence of categories. First, composition with $\cfrak$ indeed induces a functor $\Xi$ as in \eqref{eq:Chern-category-equivalence}. For this, we have to see that if $F\colon\CHfrak_{u}(X)\to\Rcal$ is a functor of graded strict ring categories, then $F\circ\cfrak$ is a symmetric monoidal functor whose component of degree 0 is isomorphic to the constant functor $1_{\Rcal}$. The first property is clear, while the second one is deduced from $\cfrak_{0}\simeq 1$ together with $F(1)\simeq 1_{\Rcal}=(1_{\Rcal_{0}},0_{\Rcal_{1}},\ldots)$, which is part of the bimonoidal structure of $F$.

We need to prove that the Chern functor $\cfrak: V(X) \to \CHfrak_{u}(X)$ is universal, in the sense that $\Xi$ in \eqref{eq:Chern-category-equivalence} is an equivalence of categories. For this, we construct an inverse functor to $\Xi$. Given a Chern functor $F\colon V(X)\to \Rcal$, with graded components $F_k$, we define a functor $ \Theta (F) \colon  \CHfrak_+(X) \to \Rcal$, given on objects by the formula 
\begin{displaymath}
    \Theta(F)\left(\sum_I \prod_{i\in I} c_{k_i}(E_i)\right) = \sum_I \prod_{i\in I} F_{k_i}(E_i).
\end{displaymath}
In this expression, the bracketing convention \eqref{eq:convention-sum} is chosen both for the products and the sum. We also fix the convention that the object $0$ (resp. $1$) is sent to $0_{\Rcal}$ (resp. $1_{\Rcal}$). Since $\Rcal$ is a graded ring category, to define $\Theta(F)$ on morphisms it is enough to deal with symbolic isomorphisms and then check that the necessary relations are fulfilled. We associate with a symbolic isomorphism in $\Cfrak^{k}(X)$ the analogous isomorphism in $\Rcal_{k}$, involving objects $F_{k}(E)$ and deduced from the Chern functor property of $F$. The latter also ensures that all the necessary relations between symbolic isomorphism are satisfied, so that we indeed obtain a functor on $\CHfrak_{+}(X)$. 

For the sake of clarity, let us elaborate on the particular case of Whitney-type isomorphisms. Let $0\to E^{\prime}\to E\to E^{\bis}\to 0$ be an exact sequence of vector bundles. Consider a symbolic isomorphism $\cfrak_{k}(E)\simeq \sum_{i}\cfrak_{i}(E^{\prime})\cdot\cfrak_{k-i}(E^{\bis})$ in the category $\Cfrak^{k}(X)$. To it, we associate the isomorphism\linebreak  $F_{k}(E)\simeq\sum_{i} F_{i}(E^{\prime})\cdot F_{k-i}(E^{\bis})$ in $\Rcal_{k}$, extracted from the degree $k$ component of the isomorphism $F(E)\simeq F(E^{\prime})\cdot F(E^{\bis})$ in $\Rcal$. In order to descend to the quotient category $\CHfrak_{+}^{k}(X)$, we need to check the compatibility of the construction with isomorphisms of exact sequences and admissible filtrations. This is automatic, because $F$ is a Chern functor, thus defined on the virtual category $V(X)$. 

From the very construction of $\Theta(F)$ and the fact that $\Rcal$ is strongly coherent, it follows that $\Theta(F)$ has a natural structure of functor of graded rig categories. The strong coherence is used here since the condition defining the bimonoidal structure on $F$ concerns arbitrary objects, and not only the regular ones as in Laplaza's coherence theorem. By Proposition \ref{prop:rig-to-ring-2}, we can extend $\Theta(F)$ to a functor of graded ring categories $\CHfrak_{u}(X)\to \Rcal$. The assignment $F\mapsto\Theta(F)$ is compatible with natural transformations of Chern functors $F\to F'$. This is a formal verification before passing to the graded Picardification. After this operation, the claim follows from the second part of Proposition \ref{prop:rig-to-ring-2}.

Given a functor of graded strict ring categories $G\colon\CHfrak_{u}(X)\to \Rcal$, we claim that the functor $(\Theta \circ \Xi)(G)$ is naturally equivalent to $G$. Since both are functors of graded strict ring categories, and the morphisms in $\CHfrak_{+}(X)$ are generated by symbolic isomorphisms, we can restrict to studying the action on objects of the form $\cfrak_{k}(E)$. By construction of $\Theta$, evaluating $(\Theta \circ \Xi)(G)$ on $\cfrak_{k}(E)$ yields $G(\cfrak_{k}(E))$, functorially in the symbolic isomorphisms. The isomorphism $(\Theta\circ\Xi)(G)\to G$ just exhibited clearly preserves natural transformations $G\to G'$ of functors of graded strict ring categories, and hence defines a natural transformation $\Theta \circ \Xi \to \id.$

It is straightforward to construct a natural transformation $ \Xi  \circ \Theta \to \id$ in the same vein.

\end{proof}

Below is a counterpart of Lemma \ref{lemma:CHplus0-N} for the universal Chern category.
\begin{lemma}\label{lemma:CHu0-Z}
Let $\ZBbb$ be the set of integers, considered as a ring category. Then, there is a natural functor of ring categories $\ZBbb\to\CHfrak_{u}^{0}(X)$, which is an equivalence of categories.
\end{lemma}
\begin{proof}
The claim results from Lemma \ref{lemma:CHplus0-N} by passing to the Picardifications.
\end{proof}
\begin{remark}
Continuing with Remark \ref{rmk:pre-Chern-empty}, we derive from the lemma that if $X$ is the empty scheme, then $\CHfrak_{u}(X)$ is equivalent to $\ZBbb$.
\end{remark}

\subsubsection{Functoriality} The universal property of $\CHfrak_{u}(X)$ ensures the functorial behavior with respect to morphisms of schemes. For the formulation, in the next statement, we return to the notation $\cfrak_{X}$ for the universal Chern functor for $X$. We recall that the pullback functoriality of vector bundles induces a pullback functoriality for the virtual categories of schemes. The standard cleavage for the pullback of vector bundles induces a cleavage for the pullback on virtual categories. 

\begin{corollary}\label{cor:CH-pullback}
The pullback functoriality for the virtual categories of schemes induces a pullback functoriality for the universal Chern categories. More precisely:
\begin{enumerate}
    \item Let $g\colon X^{\bis}\to X^{\prime}$ be a morphism of schemes. There exists a functor of graded ring categories $g^{\ast}\colon \CHfrak_{u}(X^{\prime}) \to \CHfrak_{u}(X^{\bis})$, together with an isomorphism of functors of commutative \linebreak Picard categories $g^{\ast}\circ\cfrak_{X^{\prime}} \simeq \cfrak_{X^{\bis}}\circ g^{\ast}$.
    \item If $h\colon X^{\tris}\to X^{\bis}$ is another morphism, there is a canonical isomorphism $(g\circ h)^{\ast}\simeq h^{\ast} \circ g^{\ast}$ of functors of graded ring categories $\CHfrak_{u}(X^{\prime})\to\CHfrak_{u}(X^{\tris})$, characterized by the commutativity of the following diagram:
    \begin{equation}\label{eq:CH-pullback-diagram}
                \xymatrix{ h^\ast \circ g^\ast \circ \cfrak_{X^{\prime}} \ar[r] \ar[d] & h^\ast \circ\cfrak_{X^{\bis}}\circ g^{\ast}  \ar[r] & \cfrak_{X^{\tris}}\circ h^{\ast}\circ g^{\ast}\ar[d] \\
             (g\circ h)^{\ast} \circ\cfrak_{X^{\prime}}\ar[rr] &    & \cfrak_{X^{\tris}}\circ(g\circ h)^{\ast},
             } 
       \end{equation}
    where the vertical arrow on the right is induced by the isomorphism of functors\linebreak  $(g\circ h)^{\ast}\simeq h^{\ast}\circ g^{\ast}$ between virtual categories $V(X^{\prime})\to V(X^{\tris})$.
    \item The collection of pullback functors between universal Chern categories is naturally cloven.
\end{enumerate}
\end{corollary}
\begin{proof}
The claims follow from the fact that \eqref{eq:Chern-category-equivalence} is an equivalence of categories, along the lines of the discussion following Definition \ref{def:Picardification}. 

The composition $V(X^{\prime})\overset{g^{\ast}}{\to} V(X^{\bis})\overset{\cfrak_{X^{\bis}}}{\to}\CHfrak_{u}(X^{\bis})$  is a Chern functor. By the universal property, established in Theorem \ref{thm:CH-is-universal}, it corresponds to a functor of graded ring categories \linebreak  $g^\ast \colon \CHfrak_{u}(X^{\prime}) \to \CHfrak_{u}(X^{\bis})$ as in the first point of the corollary. 

The isomorphism $(g\circ h)^{\ast}\simeq h^{\ast}\circ g^{\ast}$ of the second claim is deduced from the corresponding isomorphism of functors between virtual categories, by imposing the commutativity of the diagram and using that \eqref{eq:Chern-category-equivalence} is fully faithful. With this construction, the cleavage for the pullback functors between virtual categories is automatically inherited by the pullback functors between universal Chern categories. 
\end{proof}
\begin{remark}
\begin{enumerate}\label{rem:pullback-chern}
    \item Because the rationalization of graded ring categories is functorial, we observe that the pullback described in the previous corollary induces a functor \linebreak $f^{\ast}\colon\CHfrak_{u}(X)_{\QBbb}\to\CHfrak_{u}(X^{\prime})_{\QBbb}$, with analogous properties. 
    \item By the very definition of a functor of graded ring categories, the pullback functor $f^{\ast}$ is induced by functors $\CHfrak^{k}_{u}(X)\to \CHfrak_{u}^{k}(X^{\prime})$, which we still denote by $f^{\ast}$.
\end{enumerate}
\end{remark}

\subsubsection{A coherence property}\label{subsubsec:coherence-property}
Because the universal Chern category $\CHfrak_{u}(X)$ is a graded strict ring category, it is strongly coherent. Due to its particularly simple structure, the strong coherence can be upgraded to include the contraction of an object against an inverse. 

\begin{corollary}\label{cor:coherence}
In $\CHfrak_{u}(X)$, all diagrams involving only addition, multiplication, associativity, commutativity, distributivity, neutral and unit elements, inversion, and contractions against inverses, must commute. 
\end{corollary}
\begin{proof}

For the proof, recall the discussion in \textsection\ref{subsubsec:strict-rig-categories} regarding the Picardification of strict rig categories, in terms of the description provided by Lemma \ref{lemma:VA-tilde}. In particular, we will implicitly use that we have natural isomorphisms $(-Z)\cdot W\simeq -(Z\cdot W)\simeq Z\cdot (-W)$, which are componentwise given by symmetries. 

Let $\varphi\colon Z\to W$ be a morphism involving only the operations as in the statement. Suppose that it factors as
\begin{displaymath}
    Z\to T\to T'\to W,
\end{displaymath}
where $T\to T'$ involves a contraction, and $Z\to T$ does not involve any contraction. We may write $T=T_{0}+Y-Y+T_{1}$, and $T'=T_{0}+T_{1}$. Then, the morphism $T\to T'$ can be decomposed as
\begin{displaymath}
    T\to T_{0}+T_{1}+Y-Y\to T',
\end{displaymath}
where the first morphism permutes $Y-Y$ and $T_{1}$. Hence, we may suppose that $T=T'+Y-Y$, while keeping the property that $Z\to T$ does not involve any contraction. Next, the morphism $T'\to W$ no longer modifies the contracted factor $Y-Y$. It is thus possible to factor $T\to W$ as $T'+Y-Y\to W+Y-Y\to W$, where the first arrow is induced by $T'\to W$ and the identity of $Y-Y$, and the second arrow is a contraction. Going on in this fashion, we can finally suppose that $\varphi$ factors as $Z\to T\to W$, where $Z\to T$ does not involve any contraction, $T=W+\sum_{i}(Y_{i}-Y_{i})$, and $T\to W$ is a sequence of contractions. 

To conclude, let $\psi\colon Z\to W$ be another such morphism. We can similarly decompose it as $Z\to T'\to W$, where $Z\to T'$ does not involve any contraction, $T'=W+\sum_{j} (Y_{j}'-Y_{j}')$ and $T'\to W$ is a sequence of contractions. Since $Z\to T$ and $Z\to T'$ do not involve any contractions, these are actually induced by symmetries. Hence, the objects $\sum_{i}(Y_{i}-Y_{i})$ and $\sum_{j}(Y_{j}'-Y_{j}')$ are necessarily isomorphic through symmetries. Let $\sigma$ be such an isomorphism. Therefore, $\varphi,\psi$ fit into the following diagram
\begin{displaymath}
    \xymatrix{
            &T=W+\sum_{i}(Y_{i}-Y_{i})\ar[dd]\ar[dr]       &\\
        Z\ar[ru]\ar[rd]   &                               &W\\
            &T'=W+\sum_{j}(Y_{j}'-Y_{j}'),\ar[ur]      &
    }
\end{displaymath}
where the vertical map is $\id_{W}+\sigma$. The triangle on the right commutes, by the functoriality of addition. The triangle on the left does not involve any contractions, and thus it commutes since $\CHfrak_{u}(X)$ is strongly coherent; or, alternatively, because between two objects there is at most one symmetry. We infer that $\varphi=\psi$, thus finishing the proof. 

\end{proof}

\begin{remark}\label{rmk:product-CH}
\begin{enumerate}
    \item The practical content of the corollary is that in $\CHfrak_{u}(X)$, we may perform polynomial manipulations with objects, without caring about the order of the operations. 
    \item It seems likely, but we didn't check, that the statement of Corollary \ref{cor:coherence} holds for any strict ring category. This guess is based on the fact that the very construction of $\CHfrak_{+}(X)$, and the reasoning of the proof of Corollary \ref{cor:coherence}, are within the spirit of the proof of Laplaza's coherence theorems. 
\end{enumerate}
\end{remark}

\subsection{Chern categories}\label{subsec:Chern-category-quasi-compact}
Building on the previous theory of universal Chern categories, we proceed to develop a theory of Chern categories for schemes, which is sensitive to the fact that the rank of a vector bundle is in general not constant, but locally constant. For a scheme given by a finite disjoint open covering of connected schemes, for example in the  Noetherian case \linebreak (cf. \cite[\href{https://stacks.math.columbia.edu/tag/04MF}{04MF}]{stacks-project}), it would be enough to consider the universal Chern categories of these components separately. In the general case, the connected components are not necessarily open, and hence the space of locally constant functions has a more complicated structure. The construction we propose accounts for this. 

\subsubsection{Construction of the Chern category}
Let $X$ be a scheme. An open cover $\Ucal=\lbrace U_{i}\rbrace_{i\in I}$ of $X$ is said to be an open partition of $X$ if the various sets $U_i$ are pairwise disjoint. It is called a finite open partition if $I$ is finite. An open partition of $X$ always exists, for instance, $X$ itself. If $X$ is quasi-compact, an open partition is necessarily finite.

A refinement $\Vcal=\lbrace V_{j}\rbrace_{j\in J}$ of an open partition $\Ucal=\lbrace U_{i}\rbrace_{i\in I}$ of $X$ is an open partition of $X$ such that for every $j \in J$, there is an $i \in I$ such that $V_j \subseteq U_i$. We denote the refinement relation by $\Vcal\prec\Ucal$. Notice that two open partitions $\Ucal_{1}$ and $\Ucal_{2}$ always have a common refinement obtained by forming the intersections of the open subsets of $\Ucal_{1}$ and $\Ucal_{2}$. 

For an open partition  $\Ucal$ of $X$, we denote by $\CHfrak_{u}^{k}(\Ucal)$ the product category $\prod_{i\in I}\CHfrak_u^{k}(U_{i})$, which is a strictly commutative Picard category. Notice that if we have a refinement $\Vcal$ of $\Ucal$ as above, the inclusion maps $V_j \subseteq U_i$ define, by Corollary \ref{cor:CH-pullback}, a natural restriction functor $\rho_{\Vcal\Ucal}\colon \CHfrak_{u}^{k}(\Ucal) \to \CHfrak_{u}^{k}(\Vcal)$, denoted on objects by $A\mapsto A|_{\Vcal}$. If $\Wcal$ is a refinement of $\Vcal$, then there is an isomorphism of functors $\rho_{\Wcal\Ucal}\simeq\rho_{\Wcal\Vcal}\circ\rho_{\Vcal\Ucal}$. The collection of restriction functors carries a cleavage. We similarly define $\CHfrak_{u}(\Ucal)$, which enjoys of analogous properties. This is a strict graded ring category, with graded pieces $\CHfrak^{k}_{u}(\Ucal)$. 

The first step towards the definition of the Chern category consists in constructing its graded piece of degree $k$.
 
\begin{definition}\label{def:CH-category-final}
\begin{enumerate}
\item The degree $k$ Chern category of $X$ is defined as the following direct limit over all the open partitions of $X$:  
\begin{displaymath}
    \CHfrak^{k}(X):=\underset{\substack{\longrightarrow\\ \Ucal}}{\lim}\CHfrak_{u}^{k}(\Ucal).
\end{displaymath}
More precisely:
    \begin{enumerate}
        \item\label{item:CH-category-final-1} An object consists in giving an open partition $\Ucal$ of $X$ and an object $A_{\Ucal}$ of $\CHfrak_{u}^{k}(\Ucal)$.
        
        \item\label{item:CH-category-final-2} The morphisms between two objects $A_{\Ucal}$ and  $B_{\Vcal}$ are given by 
        \begin{displaymath}
            \underset{\longrightarrow}{\lim} \Hom(A_{\Ucal}|_{\Wcal}, B_{\Vcal}|_{\Wcal}),
        \end{displaymath} 
        where the limit runs over the open partitions $\Wcal$ refining both $\Ucal$ and $\Vcal$. The transition maps are as follows: given  $f_{\Wcal}\colon A_{\Ucal}|_{\Wcal}\to B_{\Vcal} |_{\Wcal}$ and a refinement $\Wcal^{\prime}$ of $\Wcal$, define $f_{\Wcal'}$ by the commutativity of the diagram
        \begin{displaymath}
            \xymatrix{
                (A_{\Ucal}|_{\Wcal})|_{\Wcal^{\prime}}\ar[r]^{(f_{\Wcal})|_{\Wcal^{\prime}}}\ar[d]       &(B_{\Vcal}|_{\Wcal})|_{\Wcal^{\prime}}\ar[d]\\
                A_{\Ucal}|_{\Wcal^{\prime}}\ar[r]^{f_{\Wcal^{\prime}}}              &B_{\Vcal}|_{\Wcal^{\prime}},
            }
        \end{displaymath}
        where the vertical maps are induced by the isomorphism $\rho_{\Wcal^{\prime}\Wcal}\circ\rho_{\Wcal\Ucal}\simeq\rho_{\Wcal^{\prime}\Ucal}$ of restriction functors, and the analogous isomorphism for $\Vcal$. The cleavage condition guarantees the transitivity condition for a directed system.
    \end{enumerate}
    \item We let $\pi_{\Ucal}\colon\CHfrak^{k}_{u}(\Ucal)\to \CHfrak^{k}(X)$ be the functor sending an object (resp. morphism) to itself.
\end{enumerate}
\end{definition}

The next statement asserts that the degree $k$ Chern category satisfies a universal property. It is a direct application of the definitions, and we again omit the proof.
\begin{proposition}\label{prop:CHX-is-a-direct-limit}
The degree $k$ Chern category $\CHfrak^{k}(X)$ satisfies the following properties.
\begin{enumerate}
    \item\label{CHX-is-a-direct-limit-2} The category $\CHfrak^{k}(X)$ has a natural structure of strictly commutative Picard category, such that the $\pi_{\Ucal}\colon\CHfrak_{u}^{k}(\Ucal)\to\CHfrak^{k}(X)$ are functors of commutative Picard categories.

    \item\label{CHX-is-a-direct-limit-4} If $\Vcal$ is a refinement of $\Ucal$, then the restriction from $\Ucal$ to $\Vcal$ fits into a canonically \linebreak 2-commutative diagram of functors of commutative Picard categories
\begin{displaymath}
\begin{tikzcd}
    \CHfrak_{u}^{k}(\Ucal) \ar[rd, "\pi_{\Ucal}"] \ar[rd, phantom, shift right = 0.2ex, ""{name=PIU}]  \ar["\rho_{\Vcal\Ucal}"', dd]  \ar[dd, phantom, shift right = 0.2ex, ""{name=PVU}, very near end]  & \\
    & \CHfrak^{k}(X), \\
    \CHfrak_{u}^{k}(\Vcal) \ar[ru, "\pi_{\Vcal}"' ]   \arrow[Rightarrow, from = PVU, to = PIU, "\mu_{\Ucal \Vcal}" description, shorten=2mm]
\end{tikzcd}
\end{displaymath}
where $\mu_{\Ucal\Vcal}$ is an isomorphism of functors $\pi_{\Ucal}\simeq\pi_{\Vcal}\circ\rho_{\Vcal\Ucal}$. If $\Wcal$ is a refinement of $\Vcal$, then under the natural isomorphism $\rho_{\Wcal\Ucal}\simeq \rho_{\Wcal\Vcal}\circ\rho_{\Vcal\Ucal}$, the transformation $\mu_{\Ucal\Wcal}$ corresponds to $\mu_{\Ucal\Vcal}\circ\mu_{\Vcal\Wcal}$.
    \item\label{CHX-is-a-direct-limit-5} The category $\CHfrak^{k}(X)$ is universal in the 2-category of small commutative Picard categories, with the properties \eqref{CHX-is-a-direct-limit-2}--\eqref{CHX-is-a-direct-limit-4}. Precisely, if $\Pcal$ is a commutative Picard category endowed with functors of commutative Picard categories $F_{\Ucal}\colon\CHfrak_{u}^{k}(\Ucal)\to\Pcal$ satisfying \eqref{CHX-is-a-direct-limit-2}--\eqref{CHX-is-a-direct-limit-4}, then there exists a unique functor of commutative Picard categories $F\colon\CHfrak^{k}(X)\to\Pcal$ such that $F_{\Ucal}=F\circ\pi_{\Ucal}$.\footnote{This is stronger than the existence of a natural transformation, but knowing the latter would suffice for our purposes.}
\end{enumerate}
Analogous facts hold for the finite product $\CHfrak^{\leq N}(X)=\prod_{k\leq N}\CHfrak^{k}(X)$ and the rationalizations $\CHfrak(X)_{\QBbb}^{k}$ and $\CHfrak^{\leq N}(X)_{\QBbb}=\prod_{k\leq N}\CHfrak^{k}(X)_{\QBbb}$.
\end{proposition}
\qed

We are now ready to define the Chern category of $X$.

\begin{definition}
The Chern category of $X$ is defined as $\CHfrak(X)=\prod_{k}\CHfrak^{k}(X)$. The rational Chern category of $X$ is defined as $\CHfrak(X)_{\QBbb}=\prod_{k}\CHfrak^{k}(X)_{\QBbb}$, where $\CHfrak^{k}(X)_{\QBbb}$. Their objects are called Chern power series.
\end{definition}

\begin{remark}\label{rmk:Grothendieck-construction}
\begin{enumerate}
    \item\label{item:Grothendieck-construction-1} It is not clear whether the category $\CHfrak(X)$ admits a universal property as in Proposition \ref{prop:CHX-is-a-direct-limit}, but it seems it does not. There could be objects $(A_{k})_{k}$, where $A_{k}$ is an object of $\CHfrak^{k}_{u}(\Ucal_{k})$, which can not be defined, up to isomorphism, on a common refinement of the $\Ucal_{k}$. A candidate scheme where this could happen is $X=\Spec(\prod_{\NBbb}k)$, for a field $k$. The underlying topological space is known to be homeomorphic to the Stone--\v{C}ech compactification of $\NBbb$, see \cite[Example G, Chapter II]{Mumford-red}. In particular, it is totally disconnected \cite[\href{https://stacks.math.columbia.edu/tag/090C}{090C}]{stacks-project}. Combining this with the quasi-compactness of $X$, one can construct a strictly decreasing sequence of open partitions $\Ucal_{1}\succ\Ucal_{2}\succ\ldots$, which do not share a common refinement. In practice, however, we will only be concerned with objects which can be defined on a given partition. See \textsection \ref{subsubsec:categorical-characteristic} and \textsection \ref{subsec:line-distributions} below.
    \item Let $\Kbold$ be the category of open partitions of $X$. Morphisms are given by the refinement relationship: for each refinement $\Vcal\prec\Ucal$, there is a morphism $\Ucal\to\Vcal$. The assignment $\Ucal\mapsto\CHfrak_{u}^{k}(\Ucal)$, together with the restriction maps and the cleavage condition, defines an op-lax functor $F\colon \Kbold\to\mathrm{(CAT)}$, where $\mathrm{(CAT)}$ is the category of small categories. Following Thomason \cite[Section 3]{ThomasonHomCol}, one can associate with $F$ a category $\Kbold\int F$, fibered over $\Kbold$, called the Grothendieck construction of $F$. It satisfies a universal property, from which one deduces a functor $\Kbold\int F\to\CHfrak^{k}(X)$. Because $\CHfrak^{k}(X)$ is a groupoid, this functor extends to the localization $(\Kbold\int F)[\Sigma^{-1}]\to\CHfrak^{k}(X)$, where $\Sigma$ is the set of all the morphisms in $\Kbold\int F$. It can be checked that this functor is an equivalence of categories.
 
\end{enumerate}
\end{remark}

We record the basic properties of the Chern categories. The verification, which is again formal, is omitted.
\begin{proposition}\label{prop:formal-properties-CH-quasi-compact}
The following structures are induced from the corresponding ones for the universal Chern categories, via the direct limit construction:
\begin{enumerate}
    \item\label{item:formal-properties-CH-quasi-compact-1} The category $\CHfrak(X)$ has a structure of graded strict ring category. It also inherits the strong form of the coherence property of Corollary \ref{cor:coherence}.
    \item\label{item:formal-properties-CH-quasi-compact-2} For every open partition $\Ucal$, there is a natural functor of graded ring categories \linebreak $\CHfrak_{u}(\Ucal)\to\CHfrak(X)$. These functors are compatible in the sense of Proposition \ref{prop:CHX-is-a-direct-limit} \eqref{CHX-is-a-direct-limit-4}. In particular, there is a natural functor of graded ring categories  $\CHfrak_{u}(X)\to\CHfrak(X)$.
    \item\label{item:formal-properties-CH-quasi-compact-5} There is a Chern functor $\cfrak\colon V(X)\to\CHfrak(X)$, induced by the universal Chern functor $\cfrak\colon V(X)\to\CHfrak_{u}(X)$ and the natural functor $\CHfrak_{u}(X)\to\CHfrak(X)$.
    \item\label{item:formal-properties-CH-quasi-compact-6} The categories $\CHfrak^{k}(X)$ and $\CHfrak(X)$ enjoy of pullback functoriality as in Corollary \ref{cor:CH-pullback}. 
\end{enumerate}
The corresponding properties hold for $\CHfrak(X)_{\QBbb}$.
\end{proposition}
\qed

\begin{remark}
\begin{enumerate}
    \item We deduce from the proposition that there is a natural functor of ring categories $\underset{\substack{\longrightarrow\\ \Ucal}}{\lim}\CHfrak_{u}(\Ucal)\to\CHfrak(X)$. We do not expect this functor to be an equivalence of categories. In this regard, see Remark \eqref{rmk:Grothendieck-construction} \eqref{item:Grothendieck-construction-1}. 
    \item The rational Chern category $\CHfrak(X)_{\QBbb}$ is the graded rationalization of $\CHfrak(X)$, as introduced in \textsection \ref{subsec:rationalization-ring-categories}. 
\end{enumerate}
\end{remark}

The next corollary confirms that the zeroth Chern category is sensitive to the topology of the scheme, which was not the case with the universal counterpart. 

\begin{corollary}\label{cor:CH0-H0}
Let $H^{0}(X,\ZBbb)$ be the set of locally constant $\ZBbb$-valued functions on $X$, considered as a ring category. Then, there is a natural equivalence of ring categories $H^{0}(X,\ZBbb)\to\CHfrak^{0}(X)$. In particular, any locally constant $\ZBbb$-valued function on $X$ naturally determines an object of $\CHfrak(X)$. Similarly, there is an equivalence of ring categories $H^{0}(X,\QBbb)\to\CHfrak^{0}(X)_{\QBbb}$.
\end{corollary}
\begin{proof}
The first claim is a direct consequence of Lemma \ref{lemma:CHu0-Z} and Proposition \ref{prop:CHX-is-a-direct-limit}. The second claim is obtained by composing $H^{0}(X,\ZBbb)\to \CHfrak^{0}(X)$ with $\CHfrak^{0}(X)\to\CHfrak(X)$.
\end{proof}

\begin{remark}
As a prototypical example of the corollary, we mention the rank function associated with a vector bundle $E$ on $X$. The corresponding object in $\CHfrak^{0}(X)$ or $\CHfrak(X)$ is denoted simply by $\rk E$.
\end{remark}

\subsubsection{Categorical characteristic classes}\label{subsubsec:categorical-characteristic}
The typical objects of $\CHfrak(X)$ we will be dealing with are given by power series in the ranks and the Chern classes of a finite number of vector bundles. These are categorical avatars of characteristic classes, as formalized in the next definition.
\begin{definition}\label{def:charclass}
A categorical characteristic class $\phi$ for $X$, in $\ell \geq 1$ variables, consists of the following data:
\begin{enumerate}
    \item\label{item:charclass-1} For every morphism $X^{\prime}\to X$, we are given a functor
    \begin{displaymath}
        \phi_{X^{\prime}}\colon(\Vect_{X^{\prime}},\iso)^{\times \ell}\to\CHfrak(X^{\prime})_{\QBbb}.
    \end{displaymath}
    \item\label{item:charclass-2} If $g\colon X^{\bis}\to X^{\prime}$ is another morphism, we are given a natural transformation of functors $g^{\ast}\circ\phi_{X^{\prime}}\simeq\phi_{X^{\bis}}\circ g^{\ast}$.
    \item If $h\colon X^{\tris}\to X^{\bis}$ is yet another morphism, there is a commutative diagram analogous to \eqref{eq:CH-pullback-diagram}.
        \item There is a universal bound on the denominators: for every $k\geq 0$, there exists an integer $m=m(k,\phi)\geq 1$, and a family of commutative diagrams
    \begin{displaymath}
       \xymatrix{
           (\Vect_{X^{\prime}},\iso)^{\times \ell}\ar[r]^-{m\cdot\phi_{X^{\prime}}} \ar[rd]        &\CHfrak(X^{\prime})_{\QBbb}\ar[r]        &\CHfrak^{k}(X^{\prime})_{\QBbb}\\
                &\CHfrak^{k}(X^{\prime}),\ar[ru]
       }
    \end{displaymath}
    for $X^{\prime}$ running over the morphisms $X^{\prime}\to X$. This family of diagrams is supposed to be compatible with the pullback functors. If we can take $m=1$ for every $k$, then we say that the categorical characteristic class is integral.
\end{enumerate}
A categorical characteristic class in one variable will simply be called a categorical characteristic class.
\end{definition}

We notice that the addition and the product of two categorical characteristic classes is again a categorical characteristic class. 

The first example of categorical characteristic class, which is moreover integral, is provided by the Chern functor $\cfrak\colon V(X)\to\CHfrak(X)$ of Proposition \ref{prop:formal-properties-CH-quasi-compact} \eqref{item:formal-properties-CH-quasi-compact-5}. A second example is given by the rank of vector bundles. See Corollary \ref{cor:CH0-H0} and the subsequent remark. More generally, a prominent family of categorical characteristic classes is provided by additive (resp. multiplicative) categorical characteristic classes. These are constructed from the Chern classes along the same lines as in the theory of Chow groups. We expound on the main points of the argument.

Suppose we are first given $\phi\in\QBbb\llbracket T\rrbracket$ with $\phi(0)=0$. Consider the algebra of power series in an infinite number of variables $A=\QBbb\llbracket c_{1},c_{2},\ldots\rrbracket$. A grading on $A$ is defined by declaring that $c_{k}$ has degree $k$. We associate with $\phi$ a power series $\Phi\in A$. The component of degree 0 of $\phi$ is 0. The component of degree $k\geq 1$ of $\Phi$, denoted by $\Phi_{k}$, is related to that of $\phi$ as follows. Let $r\geq k$ be an integer. If $\sigma_{1},\ldots,\sigma_{r}$ are the elementary symmetric polynomials in $T_{1},\ldots, T_{r}$, then we can write
\begin{equation}\label{eq:symmetric-decomposition}
    \sum_{j}\phi_{k}(T_{j})=\Phi_{k}(\sigma_{1},\ldots,\sigma_{k}),
\end{equation}
for a unique $\Phi_{k}\in\QBbb[c_{1},\ldots, c_{k}]$ of degree $k$. This polynomial does not depend on the choice of $r\geq k$. Indeed, to pass from the construction in $r+1$ variables to the construction in $r$ variables, it is enough to set $T_{r+1}=0$ and use the uniqueness of the expression in the elementary symmetric polynomials. From this independence property, it follows that $\Phi_{k}$ satisfies the following additivity property. Introduce new variables $c_{j}^{\prime}$ and $c_{j}^{\prime\prime}$, for every $j\geq 0$, with the convention $c_{0}^{\prime}=c_{0}^{\prime\prime}=1$. Then, upon substituting $c_{i}$ by $\sum_{m+n=i}c_{m}^{\prime}c_{n}^{\prime\prime}$, we obtain
\begin{equation}\label{eq:phi-additive-property}
    \Phi_{k}(c_{1},\ldots,c_{k})=\Phi_{k}(c_{1}^{\prime},\ldots,c_{k}^{\prime})+\Phi_{k}(c_{1}^{\prime\prime},\ldots,c_{k}^{\prime\prime}).
\end{equation}
To see this, we consider the analog of \eqref{eq:symmetric-decomposition} for the variables $T_{1}^{\prime},\ldots,T_{r}^{\prime},T_{1}^{\prime\prime},\ldots, T_{r}^{\prime\prime}$, with $r\geq k$, and then use the relationship between elementary symmetric polynomials
\begin{displaymath}
    \sigma_{i}(T_{1}^{\prime},\ldots,T_{r}^{\prime},T_{1}^{\prime\prime},\ldots, T_{r}^{\prime\prime})
    =\sum_{m+n=i}\sigma_{m}(T_{1}^{\prime},\ldots,T_{r}^{\prime})\sigma_{n}(T_{1}^{\prime\prime},\ldots,T_{r}^{\prime\prime}).
\end{displaymath}

Let now $E$ be a vector bundle on $X$. Because of the coherence result in Corollary \ref{cor:coherence} and Proposition \ref{prop:formal-properties-CH-quasi-compact} \eqref{item:formal-properties-CH-quasi-compact-1}, up to canonical isomorphism, we can associate an object $\phi(E)\in\CHfrak(X)_{\QBbb}$ by 
\begin{displaymath}
    \phi(E)=\sum_{k\geq 1}\Phi_{k}(\cfrak_{1}(E),\ldots, \cfrak_{k}(E)).
\end{displaymath}
More precisely, to make sense of $\Phi_{k}(\cfrak_{1}(E),\ldots, \cfrak_{k}(E))$, we need to fix an order on the degree $k$ monomials in $c_{1},\ldots,c_{k}$, once and for all.

In the general case of a power series $\phi\in\QBbb\llbracket T\rrbracket$, we decompose $\phi=\phi_{0}+\phi_{+}$, where $\phi_{0}\in\QBbb$ is the component of degree 0 of $\phi$. We then set
\begin{displaymath}
    \phi(E)=\phi_{0}\rk E +\phi_{+}(E),
\end{displaymath}
where $\rk E$ is understood as a locally constant function on $X$, and hence canonically defines an object of $\CHfrak(X)$, by Corollary \ref{cor:CH0-H0}.

Analogously, in the multiplicative case, to $\psi\in 1+T\QBbb\llbracket T\rrbracket$, we can associate a natural object $\psi(E) \in 1+\CHfrak(X)_{\QBbb}^{+}$ to any vector bundle $E$. 
 
\begin{proposition}\label{prop:additcharclass}
The following assertions hold:
\begin{enumerate}\label{enum:characteristic-classes}
    \item\label{enum:characteristic-classes-1} For $\phi\in\QBbb[T]$, the assignment $E\mapsto \phi(E)\in\CHfrak(X)_{\QBbb}$ can be naturally upgraded into a categorical characteristic class, such that the functors  $(\Vect_{X^{\prime}},\iso) \to (\CHfrak(X^{\prime})_{\QBbb}, +)$ are commutative multiplicative. In particular, these induce functors of commutative Picard categories $(V(X^{\prime}),\oplus) \to (\CHfrak(X^{\prime})_{\QBbb}, +)$.
    \item\label{enum:characteristic-classes-2} For $\psi\in 1+T\QBbb[T]$, the assignment $E \mapsto \psi(E) \in 1+\CHfrak(X)_{\QBbb}^{+}$ can be naturally upgraded into a categorical characteristic class, such that the functors $(\Vect_{X^{\prime}},\iso)\to (1+\CHfrak(X^{\prime})_{\QBbb}^{+},\ \cdot\ )$ are commutative multiplicative. In particular, these induce functors of commutative Picard categories $(V(X^{\prime}),\oplus) \to (1+\CHfrak(X^{\prime})_{\QBbb}^{+}, \ \cdot\ )$.
\end{enumerate}
Furthermore, the corresponding base change isomorphisms as in Definition \ref{def:charclass} \eqref{item:charclass-2} are isomorphisms of commutative multiplicative functors. 
\end{proposition}
\begin{proof}
For concreteness, we treat the additive case only. If $\phi=\phi_{0}$ is a constant power series, there is nothing to prove, for this case reduces to the rank function of vector bundles. 

Suppose now that $\phi(0)=0$. The properties of $E\mapsto \phi(E)$ then follow from the additivity property of the polynomials $\Phi_{k}$ as in \eqref{eq:symmetric-decomposition}, the fact that $\cfrak\colon V(X)\to 1+\CHfrak(X)^{+}$ is a commutative functor of Picard categories, the strong coherence property of Corollary \ref{cor:coherence} and Proposition \ref{prop:formal-properties-CH-quasi-compact} \eqref{item:formal-properties-CH-quasi-compact-1}, and the pullback functoriality of Corollary \ref{cor:CH-pullback} and Proposition \ref{prop:formal-properties-CH-quasi-compact} \eqref{item:formal-properties-CH-quasi-compact-6}.
\end{proof}

We notice that the proposition contains, as particular cases, the categorical characteristic classes defined by the Chern functors and the rank function of vector bundles.

\begin{definition}\label{def:ChernToddchar}
In the particular cases of the Chern and Todd characteristic classes, associated with the series $\exp(T)$ and $T/(1-e^{-T})$ respectively, we denote the corresponding additive and multiplicative categorical characteristic classes by $\chfrak$ and $\tdfrak$. We also denote by $\tdfrak^{*}$ the multiplicative categorical characteristic class obtained from $\tdfrak$ by multiplying the piece of degree $k$ by $(-1)^{k}$. 
\end{definition}

\begin{remark}\label{rmk:ch-E-otimes-F}
The functor $\chfrak\colon V(X)\to\CHfrak(X)_{\QBbb}$ is a functor of commutative Picard categories, with respect to addition. However, it is not a functor of ring categories. Our Chern categories do not incorporate enough axioms to guarantee that $\cfrak_{k}(E\otimes F)$ can be expressed as a polynomial in the $\cfrak_{i}(E)$ and $\cfrak_{j}(F)$. One of the aims of the forthcoming theory of line distributions is to address this inconvenience. See Remark \ref{rmk:ch-E-otimes-F-2} below.
\end{remark}

 The theory of additive and multiplicative categorical characteristic classes extends to the category of bounded complexes of vector bundles, endowed with the quasi-isomorphisms and exact sequences of complexes, as in Corollary \ref{cor:equivalence-virtual-C-Cb}. By Proposition \ref{prop:derivedvirtual}, it even extends to the derived category. This is summarized in the following corollary.

\begin{corollary}\label{cor:characteristic-derived}
The additive (resp. multiplicative) categorical characteristic classes extend to functors on $(D^b(\Vect_X), \qiso)$, multiplicative on true triangles.  
\end{corollary}
\qed

\subsection{Line distributions}\label{subsec:line-distributions} 
Let $X\to S$ be a proper \fppf morphism of schemes. In practice, this will be a morphism satisfying the condition $(C_{n})$. We introduce a direct image formalism for the virtual Chern category of $X$, with values in line bundles on $S$.

To lighten the notation, in the discussion below, we adopt the following conventions. Whenever we are given a morphism of schemes $f\colon S^{\prime}\to S$, we write $X^{\prime}\to S^{\prime}$ for the morphism deduced from $X\to S$ by base change to $S^{\prime}$. Also, we still denote by $f$ the projection morphism $X^{\prime}\to X$.
\begin{definition}\label{def:distribution}
Let $X\to S$ be a proper fppf morphism of schemes. The category of entire line distributions for $X\to S$, denoted by $\Dcal_{\ZBbb}(X/S)$, has objects and morphisms as follows:
\begin{enumerate}
    \item Objects: entire line distributions $T$ for $X\to S$, consisting in giving  :
\begin{enumerate}
        \item\label{item:distribution-1} An integer $N\geq 0$, and for every morphism $S^{\prime}\to S$, a commutative functor of Picard categories
        \begin{equation}\label{eq:entire-distribution}
            T_{S^{\prime}}\colon (\CHfrak^{\leq N}(X^{\prime}), +) \to \Picfr(S^{\prime}).
        \end{equation}
        Equivalently, $T_{S^{\prime}}$ can be seen as a functor defined on $\CHfrak(X^{\prime})$, which factors through $\CHfrak^{\leq N}(X^{\prime})$.

         \item\label{item:distribution-2} If $g: S^{\bis} \to S^{\prime}$ is a morphism, there is an isomorphism of functors of commutative Picard categories $T_{S^{\bis}}\circ g^\ast  \simeq g^\ast \circ T_{S^\prime}$.  We refer to this as base change identification or simply base change.
        \item\label{item:distribution-3} If $h: S^{\tris} \to S^{\bis}$ is another morphism, the natural diagram akin to \eqref{eq:CH-pullback-diagram} commutes.
    \end{enumerate}
    \item Morphisms: isomorphisms of entire line distributions $T_{1}\simeq T_{2}$, consisting in giving  natural transformations of functors of commutative Picard categories $T_{1, S^{\prime}}\simeq T_{2, S^{\prime}}$, compatibly with the base change identifications.
\end{enumerate}
An entire line distribution is said to have degree $N$ if \eqref{eq:entire-distribution} factors through $\CHfrak^{N}(X^{\prime})$. 
\end{definition}
The tensor product of line bundles induces a structure of strictly commutative Picard category on $\Dcal_{\ZBbb}(X/S)$. The following definition thus makes sense.

\begin{definition}\label{def:Q-distribution}
The category of line distributions for $X\to S$, denoted by $\Dcal(X/S)$, is the rationalization of $\Dcal_{\ZBbb}(X/S)$. Equivalently, it can be described as in Definition \ref{def:distribution}, by replacing the categories with their rationalizations, and imposing universal bounds:
\begin{enumerate}
    \item Objects: for a line distribution $T$, there exists a non-zero integer $m$ such that $T^{\otimes m}$ is induced by an entire line distribution. We say that $T$ has degree $N$ is $T^{\otimes m}$ has degree $N$.
    \item Morphisms: a morphism of line distributions $T\to T^{\prime}$ is induced by a morphism of entire line distributions $T^{\otimes m}\to T^{\prime\otimes m} $, for some non-zero integer $m$.
\end{enumerate}
\end{definition}
Given a line distribution, the smallest integer $m$ as in the definition is called the denominator of the distribution. In particular, we notice that an entire line distribution defines a line distribution with denominator 1.

\begin{remark} 
\begin{enumerate}
    \item Occasionally, we will adopt additive notation for the product of line distributions.
    \item A line distribution which is both of degree $N$ and $N^{\prime}\neq N$, is necessarily isomorphic to the constant line distribution with image $\Ocal_{S}$. 
    \item A tautological consequence of the definition is that a line distribution for $X\to S$ induces a line distribution for $X^{\prime}\to S^{\prime}$, for any morphism $S^{\prime}\to S$.
    \item The line distributions considered in this article are provided by intersection bundles of Elkik \cite{Elkikfib}, discussed in Section \ref{subsubsec:general-intersection-bundles} below. The fact that these indeed define line distributions, proven in Section \ref{subsec:int-bundles-line-distributions}, together with the coherence result for the Chern categories, will allow us to treat intersection bundles formally as polynomials in the Chern classes in an unambiguous way. 
\end{enumerate}
\end{remark}

\begin{proposition}\label{prop:line-distribution-affine}
 Suppose that, in the definition of line distributions (cf. Definition \ref{def:distribution} and Definition \ref{def:Q-distribution}), one only assumes  $S^{\prime}$, $S^{\prime\prime}$, and $S^{\prime\prime\prime}$ are affine. Then this defines a category that is equivalent to the category of line distributions. 
\end{proposition}

\begin{proof}
By the very definition of line distribution, we can reduce to the case of an entire line distribution. Let $P$ be an object of $\CHfrak(X)$. Suppose that we are given two open affines $S_1$ and $S_2$ of $S$, with preimages $X_{1}$ and $X_{2}$ in $X$. To construct a line bundle $T_{S}(P)$ over $X$, we first construct a gluing isomorphism $\varphi_{12}: T_{S_{1}}(P)|_{S_{1}\cap S_{2}} \to T_{S_2}(P)|_{S_{1}\cap S_{2}}$, which depends functorially on $P$. For this, it is enough to construct such an isomorphism for any open affine $U$ of $S_{1} \cap S_{2}$, which further satisfies a gluing condition. We define $\varphi_U$ as the composition of natural restriction isomorphisms (cf. Definition \ref{def:distribution} \eqref{item:distribution-2}):
\begin{displaymath}
    T_{S_1}(P)|_U \simeq T_{U}(P) \simeq T_{S_2}(P)|_U .
\end{displaymath}
This isomorphism depends functorially on $P$ since the restriction isomorphisms are given by isomorphisms of functors of commutative Picard categories. 

Given another open affine $U'$ of $S_{1} \cap S_{2}$, we need to prove that $\varphi_{U}|_{U \cap U'}$ and $\varphi_{U'}|_{U \cap U'}$ agree. For this, it is enough to do so on an open affine $V \subseteq U \cap U'.$ The equality $\varphi_U|_V = \varphi_V = \varphi_{U'}|_V$ is guaranteed by the condition in Definition \ref{def:distribution} \eqref{item:distribution-3}.

Given an affine covering of $S$, the gluing isomorphisms constructed as above satisfy a cocycle condition of the form $\varphi_{13} = \varphi_{12} \varphi_{23}$. Indeed, from the construction, this is immediate locally on open affines in the triple intersections, hence everywhere.  Hence, we can glue the affine locally defined line bundles into a line bundle $T_{S}(P).$ Since the gluing isomorphisms depend functorially on $P$, we see that $P\mapsto T_{S}(P)$ indeed defines a functor of commutative Picard categories.

The above discussion can be applied to any base change $X_{S^{\prime}}\to S^{\prime}$. Once the line bundles $T_{S^{\prime}}(P')$ have been constructed, we still need to verify the axioms of Definition \ref{def:distribution}. Since they are of local nature, we can check them on open affines, where they are satisfied by assumption. 

One similarly globalizes isomorphisms of line distributions. This concludes the proof.
\end{proof}

\begin{remark}
After the previous proposition, the category of line distributions can be seen as a stack over $(\mathrm{Aff}/S)$, endowed with the Zariski topology. Recall that, by descent, line bundles for the Zariski topology are equivalent to line bundles for the \'etale, \fppf, or \emph{fpqc} topologies, see  \linebreak \cite[\href{https://stacks.math.columbia.edu/tag/03P7}{03P7}]{stacks-project}. From this, we deduce that the category of line distributions is also a stack over $(\mathrm{Aff}/S)$, endowed with any of these other topologies. 
\end{remark}

To conclude this subsection, we record some formal constructions on line distributions. 
\begin{definition}\label{def:prod-Chern-series-line-distribution}
Let $T$ be a line distribution for $X\to S$. Let $P$ be a Chern power series in $\CHfrak(X)_{\QBbb}$. We define a line distribution $P \cdot T$ by the formula on objects given by, for $f: S^{\prime} \to S$ and a Chern power series $Q$ on $X'$,  
\begin{displaymath}
    (P \cdot T)_{S^{\prime}}(Q) = T_{S^{\prime}}(Q\cdot f^\ast P).
\end{displaymath}
\end{definition}
We observe that this definition actually provides us with a bisymmetric and bimonoidal functor $(\CHfrak(X)_{\QBbb}, +)\times(\Dcal(X/S),\otimes) \to (\Dcal(X/S),\otimes)$. In applications, we will work with the particular situation when $P=\phi(E)$ is a categorical characteristic class. In this case, sending $E$ to $\phi(E)\cdot T$ defines a functor $(\Vect_{X},\iso)\to\Dcal(X/S)$. 

\begin{definition}\label{def:direct-image-distributions}
Let $X$ and $Y$ be proper and fppf over $S$, and $h\colon Y\to X$ a morphism over $S$. Suppose we are given a line distribution $T$ for $Y\to S$. We define a line distribution $h_\ast T$ for $X\to S$, by the formula on objects given by  
\begin{displaymath}
    (h_\ast^{\prime} T)(Q)= T(h^{\prime\ast} Q),
\end{displaymath}
where $h^{\prime}$ is the base change of $h$ by a morphism $S^{\prime}\to S$, and $Q$ is a Chern power series on $X^{\prime}$.
\end{definition}

Notice that, by construction, we formally have:

\begin{equation}\label{linedist:projformula}
    h_\ast \left( h^\ast P \cdot T \right) = P \cdot  h_\ast T .
\end{equation}
This can be seen as a projection formula for line distributions.

\subsection{Line functors and splitting principles}\label{subsubsec:line-functors}
In this subsection we recall for the convenience of the reader and without proof, statements from \cite[Section 2]{Eriksson-Freixas-Wentworth} on line functors and splitting principles. 

\subsubsection{Line functors}
 We first review the notion of line functor in \cite[Definition 2.1]{Eriksson-Freixas-Wentworth}, adapted to the current article. We refer to \emph{op. cit.} for details.

\begin{definition}\label{def:admissible-class}
An admissible class of morphisms $\Pcal$ is a class of proper fppf morphisms, satisfying the following properties:
\begin{enumerate}
    \item[$(P_1)$]  The class $\Pcal$ is stable under flat base changes and isomorphisms of families.
    \item[$(P_2)$] Let $f\colon X\to S$ be a proper fppf morphism. Suppose that $\Ucal=\lbrace X_{i}\rbrace_{i}$ and $\Vcal=\lbrace S_{j}\rbrace_{j}$ are open partitions of $X$ and $S$, and that $f$ induces surjective morphisms $f_{i}\colon X_{i}\to S_{j(i)}$. Then, $f$ belongs to $\Pcal$ if, and only if, all the $f_{i}$ belong to $\Pcal$. 
\end{enumerate}

\end{definition}
In this article, the prototypical example of this definition will be the smallest admissible class containing a given proper \fppf morphism $f\colon X\to S$. This we call the class generated by $f$.

\begin{definition}
The category of line functors in $k$-variables, for an admissible class $\Pcal$, has objects and morphisms as follows:
    \begin{enumerate}
    \item Objects: line functors in $k$ variables, consisting in associating, to any $f\colon X \to S$ in $\Pcal$, a functor 
    \begin{equation} \label{def:determinantfunctor} 
        \Gcal_f \colon (\Vect_{X}, \iso)^{\times k} \to \Picfr(S).
    \end{equation}
     We require compatibility with $(P_1)$ and $(P_2)$ in Definition \ref{def:admissible-class}, in the following sense:
    \begin{enumerate}
        \item[$(F_1)$] The base changes and isomorphisms of families are supposed to identify the various $\Gcal_{f}$, satisfying a  condition akin to the commutative diagram \eqref{eq:CH-pullback-diagram}.
      \item[$(F_2)$] In the setting of $(P_2)$, the functor $\Gcal_f$ is isomorphic to the product of the functors $\Gcal_{f_i}$. Here the functor $\Gcal_{f_i}$ determines a line functor on $S$ by trivially extending the line bundles from $S_{j(i)}$ to $S$. This isomorphism is supposed to be compatible with $(F_1).$
    \end{enumerate}
    \item Morphisms: isomorphisms of line functors $\Gcal\to\Gcal^{\prime}$, consisting in associating, to every $f$ in $\Pcal$, an isomorphism of functors $\Gcal_{f}\to\Gcal_{f}^{\prime}$, in a way compatible with $(F_{1})$ and $(F_{2})$. 
    \end{enumerate}
A line functor in one variable is simply called a line functor. 
\end{definition}
Notice that, by the property $(F_{1})$, line functors are Zariski locally determined. Since the families are quasi-compact morphisms, the products appearing in $(F_2)$ are, locally on $S$, finite. Hence, property $(F_2)$ is meaningful. 

In these definitions, we notice that, unlike the reference \cite{Eriksson-Freixas-Wentworth}, the schemes and morphisms are more general and the vector bundles are not necessarily of constant rank. The definition of the class $\Pcal$ and the gluing property is more restrictive to be able to reduce to the case of constant rank.

The category of line functors, endowed with the tensor product, has a natural structure of a strictly commutative Picard category. The following definition thus makes sense.

\begin{definition}
The category of $\QBbb$-line functors in $k$-variables, for an admissible family $\Pcal$, is the rationalization of the category of line functors for $\Pcal$ in $k$-variables. Its objects and morphisms can equivalently be described as follows:
\begin{enumerate}
    \item Objects: $\QBbb$-lines functors in $k$-variables, consisting in functors $\Gcal_{f} \colon (\Vect_{X}, \iso)^{k} \to \Picfr(S)_{\QBbb}$, such that the various $\Gcal_{f}^{\otimes m}$ define a line functor, for a non-zero integer $m$ independent of $f$.
    \item Morphisms: an isomorphism of $\QBbb$-line functors in $k$-variables $\Gcal\to\Gcal^{\prime}$ consists in an isomorphism of line functors in $k$-variables $\Gcal^{\otimes m}\to\Gcal^{\prime\otimes m}$, for some non-zero integer $m$. 
\end{enumerate}
\end{definition}
    
We will encounter variants of the above notions, where the functors are endowed with an additional multiplicative datum:

\begin{definition}
A line functor $\Gcal$ in $k$-variables, for $\Pcal$, is commutative multiplicative if the various $\Gcal_{f}$ are commutative multiplicative functors in every entry, in a way compatible with the properties $(F_{1})$ and $(F_{2})$. We similarly define commutative multiplicative $\QBbb$-line functors. 
\end{definition}
Most of the time, for simplicity, we will omit the mention of the class $\Pcal$. It will thus be implicit. 

\subsubsection{Splitting principles}\label{subsubsec:splitting-principle}
We recall the splitting principles for line functors. For concreteness, the below statements are formulated for line functors only, but they hold more generally for line functors in several variables and their rational counterparts. 
 
\begin{proposition}\label{prop:split} $($cf.  \cite[Proposition 2.2]{Eriksson-Freixas-Wentworth}$)$ Let $\Gcal$ be a line functor. Then, for any exact sequence of vector bundles $\Sigma: 0 \to E^{\prime} \to E \to E^{\bis} \to 0$, there is a unique isomorphism
\begin{displaymath}
    \psi_\Sigma\colon \Gcal(E) \to \Gcal(E^{\prime} \oplus E^{\bis})
\end{displaymath}
which:
\begin{enumerate}
    \item Is functorial with respect to pullback and isomorphisms of exact sequences.
    \item Is the identity whenever $\Sigma$ is the standard split exact sequence 
    \begin{equation}\label{eq:standardsplit}
        0 \to E^{\prime} \to E^{\prime} \oplus E^{\bis} \to E^{\bis} \to 0. 
    \end{equation}
\end{enumerate}

\end{proposition}
\qed
\begin{remark} \label{rem:linebundlesonprojective}
We mention that the proof of \cite[Proposition 2.2]{Eriksson-Freixas-Wentworth} relies on the following elementary fact, which was stated under simplifying assumptions on the base scheme. If $S$ is any scheme, every line bundle $L$ on $\PBbb^{1}_{S}$ is isomorphic to a line bundle of the form $\Ocal(k)\otimes p^{\ast}M$, where $k$ is a locally constant function $S\to\ZBbb$, $p\colon\PBbb^{1}_{S}\to S $ is the projection. Necessarily,  $M$ is isomorphic to $\sigma^{\ast}L$, where $\sigma\colon S\to\PBbb^{1}_{S}$ is the $\infty$ section. In \emph{loc. cit.} this is proven by citing a result of \cite{Quillen:K-theory-I}, valid for quasi-compact schemes. For arbitrary $S$, the claim thus holds locally and can be seen to hold in general by gluing. 
\end{remark}

The proof of the below proposition uses the isomorphism in the previous proposition to reduce considerations to the split case and constitutes our second splitting principle.
\begin{proposition} \label{Prop:2ndsplitting}
$($cf. \cite[Theorem 2.10 \& Remark 2.12]{Eriksson-Freixas-Wentworth}$)$ Let $\Gcal, \Gcal'$ be  line functors. Suppose that for direct sums of line bundles, $E = \oplus_{i}^{r} L_i$, we have an isomorphism $\Gcal(E) \to \Gcal'(E)$, compatible with isomorphisms in $E$ preserving the flag $L_1 \subseteq L_1 \oplus L_2 \subseteq \ldots \subseteq E$, and compatible with base change in $\Scal$. Then there is an isomorphism of line functors $\Gcal \to \Gcal'$.
\end{proposition}
\qed

\begin{remark}\label{rmk:indepedenceoffiltration}
    The above is, in fact, a combination of the splitting principle in Proposition \ref{prop:split} and a splitting principle that reduces to constructions depending on a flag. It is part of the result that the outcome is independent of the flag. 
\end{remark}

The following corollary corresponds to \cite[Proposition 2.16 \& Proposition 2.19 \& Corollary 2.20 ]{Eriksson-Freixas-Wentworth}. It relies on the two splitting principles recalled above and allows us to reduce the study of commutative multiplicative line functors to the case of split sequences and line bundles. 
\begin{corollary} \label{Cor:2ndsplitting}
There is an equivalence of categories, between:
\begin{enumerate}
    \item The category of commutative multiplicative line functors. 
    \item The category described as follows:
    \begin{enumerate}
        \item Objects:  line functors $\Gcal$ with an isomorphism $\Gcal_f(E \oplus E') \simeq \Gcal_f(E) \otimes \Gcal_f(E'),$ for any $f, E, E'$. 
        It should be commutative whenever $E$ and $E'$ are line bundles. These properties should be compatible with $(F_1)$ and $(F_2).$
        \item Morphisms: isomorphisms $\Gcal_f(L) \to \Gcal'_f(L)$ for any line bundle $L$, compatible with $(F_1)$ and $(F_2)$.
    \end{enumerate}
    
\end{enumerate}   
\end{corollary}

\qed

As an application, we conclude this section with a splitting principle for line distributions. For simplicity, we state the result for a single vector bundle $E$, but analogous results hold for categorical characteristic classes in more variables. 

\begin{corollary} \label{cor:splittingdistribution}
Suppose we are given two categorical characteristic classes $\phi$ and $\psi$ for $X$. Moreover, suppose that for any $S' \to S$ and any (ordered) direct sum of line bundles $E = \bigoplus_{i=1}^r L_i$ on $X'$, where $r$ is locally constant, there is an isomorphism of line distributions: 
\begin{displaymath}
    \phi(E) \cdot T_{S^{\prime}} \to \psi(E) \cdot T_{S^{\prime}}. 
\end{displaymath}
Then, there is such an isomorphism for any vector bundle $E.$
\end{corollary}

\begin{proof}
By the very definition of line distributions and categorical characteristic classes, we can reduce to the case of entire line distributions and integral characteristic classes. Denote by $S'\to S$ any morphism of schemes. For a fixed Chern power series $P$ on $X'$, the functors $E \mapsto (\phi(E) \cdot T_{S^{\prime}})(P)$ and $E \mapsto (\psi(E) \cdot T_{S'})(P)$ define line functors for the family generated by $X^{\prime}\to S^{\prime}$. By the splitting principle in Proposition \ref{Prop:2ndsplitting}, the datum of the statement defines an isomorphism of line bundles for every vector bundle $E$, say $\rho_P(E): (\phi(E) \cdot T_{S^{\prime}})(P) \to (\psi(E) \cdot T_{S^{\prime}})(P)$. In the case when $E$ is a sum of line bundles, these isomorphisms define isomorphisms of line distributions by assumption. Hence, by the splitting principle, they also define isomorphisms of line distributions for general $E$. 
\end{proof}

\begin{remark}\label{rmk:ch-E-otimes-F-2}
Continuing with Remark \ref{rmk:ch-E-otimes-F}, the splitting principle above will allow us to show that, in the setting of intersection bundles, the categorical characteristic classes $\chfrak(E\otimes F)$ and $\chfrak(E)\cdot\chfrak(F)$ induce isomorphic line distributions. This is accomplished in Proposition \ref{prop:ch-multiplicative-properties} below.
\end{remark}

\section{Deligne pairings}\label{section:intersection-bundles}
In this section we discuss the construction and some properties of Deligne pairings, for morphisms satisfying the condition $(C_n)$, following \cite{Boucksom-Eriksson,  Deligne-determinant, Ducrot, Elkikfib, Eriksson-Freixas-Wentworth, Munoz}. In these references, several levels of generality are allowed. Below, we provide a unified presentation using generators and relations. In particular, the construction does not assume the base schemes to be locally Noetherian or the morphisms to have Cohen--Macaulay fibers, as was required in the original approach by Elkik \cite{Elkikfib}. 

We recall that when we deal with intersection bundles, functorial means compatible with base change, and isomorphisms of the involved $S$-schemes and vector bundles.

\subsection{Construction of Deligne pairings}\label{subsec:Deligne-pairings}  In this subsection, we will consider morphisms of schemes $f\colon X \to S$ which satisfy the condition $(C_n)$, for some integer $n\geq 0$. 

\subsubsection{Construction with the determinant of the cohomology}

The Deligne pairings for $f\colon X\to S$ consist in functorially associating, to any given line bundles $L_{0},\ldots, L_{n}$ on $X$, a line bundle on $S$ denoted by

\begin{equation} \label{def:firstimeintbundle}
    \langle L_0, \ldots, L_n \rangle_{X/S},\quad\text{or simply}\quad\langle L_0, \ldots, L_n \rangle.
\end{equation}
These line bundles are modeled on the intersection classes
\begin{equation}\label{modelintclass}
    f_\ast (c_1(L_0)\cdot \ldots\cdot c_1(L_n) ).
\end{equation}
Over a general base scheme, the bundles \eqref{def:firstimeintbundle} can be directly constructed using the determinant of the cohomology, as developed in \cite{Ducrot} and generalized in \cite[Appendix]{Boucksom-Eriksson}. The definition of the bundle is given by the $(n+1)$-th symmetric difference of determinant of cohomologies \linebreak  (cf. Definition \ref{def:detcoh}):
\begin{equation}\label{eq:definition-ducrot}
    \langle L_0, \ldots, L_n \rangle_{X/S}  = \bigotimes_{I \subseteq \{0, \ldots, n \}} \det Rf_* \left( \bigotimes_{i \in I} L_i \right)^{(-1)^{n+1-|I|}}
\end{equation}
We will elaborate on the precise relationship with \eqref{modelintclass} in \textsection\ref{subsub:Cherndirect}, in particular, see Corollary \ref{cor:deligne-pairing-direct-image}. For the time being, we summarize the main properties proven in these references.

\begin{proposition}\label{prop:generalpropertiesdeligneproducts}
With assumptions as above, we have the following functorial constructions:

\begin{enumerate}
    \item\label{item:Deligne-pairing-multilinear} The Deligne pairing is multilinear for the tensor product of line bundles.
    \item\label{item:symmetry-deligne} Given any permutation $\sigma$ of $\lbrace 0,\ldots,n\rbrace$, there is a canonical symmetry isomorphism 
    \begin{equation}\label{eq:delignepairingsymmetry}
        \langle L_0, \ldots ,L_n\rangle_{X/S} \to \langle L_{\sigma(0)}, \ldots ,L_{\sigma(n)}\rangle_{X/S}. 
    \end{equation}
    If $L_i = L_j$, and $\sigma = (i,j)$, then  $[\sigma] = (-1)^{\kappa},$ where $\kappa = \int_{X/S} \prod_{k\neq i} c_1(L_k)$.
    \item\label{item:deligne-pairing-restriction} Given a relative effective Cartier divisor $D \subseteq X$, there is a canonical restriction isomorphism
\begin{equation}\label{eq:restrictiondivisor}
    \langle L_{0},\ldots, \Ocal(D), \ldots, L_{n-1} \rangle_{X/S} \to \langle L_{0}|_D,\ldots,L_{n-1}|_D\rangle_{D/S},
\end{equation} 
    \item If $X\to S$ is of relative dimension zero, the construction $\langle L_{0} \rangle_{X/S}$ coincides with the norm functor $N_{X/S}(L_{0})$ in \cite[\textsection 6.5.5]{EGAII}. 
    \item\label{item:delmultiplicationbyu} Let $u \in \Ocal_S^\times$. The multiplication map $u: L_i \to L_i$ induces, by functoriality, an automorphism $[u]$ of $\langle L_0, \ldots, L_n \rangle $. We have 
    \begin{displaymath}
       [u] = u^\kappa,\quad\text{where}\quad \kappa=\int_{X/S} \prod_{j\neq i}c_1(L_j).
    \end{displaymath}
    \item\label{item:order-deligne-pairing-isos} Iterated applications of the above isomorphisms can be performed in any order, with the same result. In particular, iterated applications of restrictions along relative Cartier divisors are independent of the order.
\end{enumerate}
\end{proposition}
\qed

\begin{proof}
     The statements are contained in \cite{Ducrot} for projective morphisms of Noetherian schemes, and extended to projective morphisms of schemes in \cite[Appendix]{Boucksom-Eriksson}. For a morphism of type $(C_{n})$, we hence know that the stated properties hold locally over $S$. Since they are all compatible with base change, and the determinant of the cohomology is compatible with base change as well, we can globalize to the whole $S$ by gluing. 
\end{proof}

\begin{remark}\label{rmk:clarification-restriction}
Property \eqref{item:deligne-pairing-restriction} in the statement requires some clarification. The structure morphism $D\to S$ is locally projective, flat, and of finite presentation. In particular, its image is an open and closed subset $T$ of $S$. The morphism $D\to T$ satisfies the condition $(C_{n-1})$. Indeed, the fibers of an effective relative Cartier divisor in $X$ are purely of dimension $n-1$, because $f$ has equidimensional fibers and by \cite[Proposition 16.4.1]{EGAIV1}, or also \cite[Corollaire 5.2.4]{EGAIV2}. The line bundle $\langle L_{0}|_D,\ldots,L_{n-1}|_D\rangle$ is then understood as the line bundle over $S$ determined by
\begin{displaymath} 
    \langle L_{0}|_D,\ldots,L_{n-1}\ |_D\rangle\big|_{T}=\langle L_{0}|_D,\ldots,L_{n-1}|_D\rangle_{D/T}
\end{displaymath}
and
\begin{displaymath}
    \langle L_{0}|_D,\ldots,L_{n-1}\ |_D\rangle\big|_{S\setminus T}=\Ocal_{T}.
\end{displaymath}
In particular, if $D$ is the zero divisor, then \eqref{eq:restrictiondivisor} provides a trivialization.
\end{remark}

\subsubsection{Generators and relations}\label{subsubsec:generators-relations}
Taking Proposition \ref{prop:generalpropertiesdeligneproducts} for granted, one can produce an equivalent presentation of the Deligne pairings in terms of generators, called symbols, and relations. In this and the next subsection, we expound this approach. For the foundations and terminology on rational (or meromorphic) functions, sections, and divisors, we follow \cite[\textsection 20 \& \textsection 21]{EGAIV4}. 

\begin{construction} For $X\to S$ satisfying \eqref{condition-Cn}, the Deligne pairing $\langle L_{0},\ldots,L_{n}\rangle_{X/S}$ can be presented in terms of generators and relations, defined locally Zariski, or \'etale, with respect to $S$:
\begin{enumerate}
    \item[\textbullet] \emph{Generators}. Given rational sections $\ell_{0}, \ldots, \ell_n$ of $L_{0},\ldots,L_{n}$ respectively, in general position, there is a local trivialization 
    \begin{equation}\label{def:symbol}
    \langle \ell_{0}, \ldots, \ell_n \rangle.
    \end{equation} Here general position means:\smallskip
    \begin{enumerate}
        \item\label{item:generators-1} The Cartier divisor corresponding to the rational section $\ell_i$ is of the form \linebreak $D_{i}=D_i^0 - D_i^1,$ where $D_i^0 $ and $D_i^1$ are relative effective relative Cartier divisors. In particular, $\ell_{i}$ is a relative regular meromorphic section \cite[\textsection 21.15]{EGAIV4}. 
        \item\label{item:generators-2} For any permutation $\sigma$ of $I=\lbrace 0,\ldots,n\rbrace$, any function $\epsilon: I \to \{0,1 \}$ and any index $i\in I$, the scheme theoretic intersection
        \begin{displaymath}
            D_{\sigma(i)}^{\epsilon(i)} \cap \bigcap_{j=0}^{i-1} D_{\sigma(j)}^{\epsilon(j)}
        \end{displaymath} 
        defines an effective relative Cartier divisor in $\bigcap_{j=0}^{i-1} D_{\sigma(j)}^{\epsilon(j)}$. Here, for $i=0$ we interpret the intersection over the empty set as all of $X$, and for $i=n$ the intersection is empty.
    \end{enumerate} 
    The existence of sufficiently many sections in general position is addressed in \textsection \ref{subsubsec:sections-general-position} below.
    \smallskip
    \item[\textbullet] \emph{Relations.} If $\ell_i = h \ell_i^{\prime}$ for a rational function $h$ and the symbol $\langle\ell_{0},\ldots,\ell_{i}^{\prime},\ldots,\ell_{n}\rangle$ is defined, then 
    \begin{equation}\label{eq:relationElkikDucrot}
    \langle \ell_{0}, \ldots, \ell_i ,\ldots   \ell_n \rangle = N_{D/S}(h_{\mid D}) \langle \ell_{0}, \ldots, \ell_i', \ldots, \ell_n \rangle
    \end{equation}
    where $D = \cap_{j \neq i} D_j $ and $N_{D/S}(h_{\mid D})$ denotes the norm of the regular function $h_{\mid D}$ on $D$. Here, $D$ is understood as a cycle whose components are finite and flat over $S$. It is non-necessarily effective, and the norm is extended multiplicatively with respect to the addition of cycles.
\end{enumerate}
\end{construction}

In this article, arguments involving generators of Deligne pairings will be carried out locally in the Zariski topology. However,  we refer to Proposition \ref{prop:construction-generators} and the subsequent remark regarding the use of generators in the \'etale topology.

\begin{remark}
\begin{enumerate}
    \item The properties listed in Proposition \ref{prop:generalpropertiesdeligneproducts} can all be written in terms of generators and relations. From this, one can also derive the compatibility between those. For instance, the multilinearity of Deligne pairings sends a section $\langle \ell_0\otimes\ell_{0}^{\prime}, \ldots, \ell_n \rangle $ to $\langle \ell_0, \ldots, \ell_n \rangle \otimes \langle \ell_0^{\prime}, \ldots, \ell_n \rangle$, and analogously in other entries. The symmetry isomorphism sends a symbol of the form $\langle \ell_0, \ldots, \ell_n \rangle $ to $\langle \ell_{\sigma(0)}, \ldots, \ell_{\sigma(n)} \rangle $. We thus see that both operations commute in a natural way. Similarly, the restriction isomorphism \eqref{eq:restrictiondivisor} sends a symbol $\langle \ell_0, \ldots, \mathbf{1}_{D},\ldots, \ell_{n-1} \rangle$ to $\langle \ell_0 |_{D}, \ldots, \ell_{n-1} |_{D} \rangle$, whenever defined. Here, $\mathbf{1}_{D}$ denotes the canonical section of $\Ocal(D)$.
    \item If $X$ admits a finite open partition $\lbrace X_{i}\rbrace_{i}$, whose components $X_{i}$ are faithful over $S$, then there is a canonical isomorphism 
    \begin{equation}\label{eq:decomposition-Deligne-pairing-partition}
      \langle L_{0},\ldots, L_{n}\rangle_{X/S}\to\bigotimes_{i}\langle L_{0} |_{X_{i}},\ldots, L_{n} |_{X_{i}}\rangle_{X_{i}/S},
    \end{equation}
    given by sending a symbol $\langle \ell_{0}, \ldots, \ell_{n} \rangle $ to the product of the symbols $\langle \ell_{0} |_{X_{i}}, \ldots, \ell_{n} |_{X_{i}} \rangle$. Actually, since the construction of the Deligne pairing is local over the base, this can be extended to the case when the $X_{i}$ are not necessarily faithful over the base, as in\linebreak Remark \ref{rmk:clarification-restriction}. 
\end{enumerate}
\end{remark}

\subsubsection{Sections in general position}\label{subsubsec:sections-general-position}
We announced an equivalent construction of the Deligne pairings in \textsection \ref{subsubsec:generators-relations}, provided we know  the existence of sufficiently many rational sections in general position. For the sake of completeness, we now justify this fact. We begin by recalling an instance of the prime avoidance lemma.

\begin{lemma}\label{lemma:prime-avoidance}
Let $Y$ be a projective scheme over a field $k$, $L$ a line bundle, and $A$ a very ample line bundle on $Y$. Let $y_{1},\ldots, y_{m}$ be given distinct points of $Y$. Then:
\begin{enumerate}
    \item\label{item:prime-avoidance-1} If $k$ is infinite, there exists a global section of $A$ which avoids $y_{1},\ldots, y_{m}$.
    \item\label{item:prime-avoidance-2} In general, for some $N_{0}\geq 0$ and for every integer $N\geq N_{0}$, $L\otimes A^{N}$ admits a global section that avoids $y_{1},\ldots, y_{m}$.
\end{enumerate}

\end{lemma}
\begin{proof}
The first statement is well-known, and an argument is provided by Nakai in the proof of \cite[Theorem 4]{Nakai}. For the second claim, consider the Zariski closures $\ov{\lbrace y_{i}\rbrace}$ with the reduced scheme structure. We choose distinct closed points $z_{i}\in \ov{\lbrace y_{i}\rbrace}$, which exist because $Y$ is of finite type over $k$. Let $\Ical$ be the coherent sheaf of ideals of the closed reduced subscheme $Z=\lbrace z_{1},\ldots, z_{m}\rbrace$. Then, for some $N_{0}\geq 0$ and for every $N\geq N_{0}$, $H^{1}(Y,L\otimes A^{N}\otimes\Ical)=0$, and hence restricting to $Z$ defines a surjection $H^{0}(Y,L\otimes A^{N})\to H^{0}(Z,L\otimes A^{N}|_{Z})$. The latter is isomorphic to $\oplus_{i}k(z_{i})$. We can thus find $s\in H^{0}(X,L\otimes A^{N})$ which does not vanish at any of the $z_{i}$. In particular, it does not vanish identically on any of the $\ov{\lbrace y_{i}\rbrace}$. Since the $y_{i}$ are the generic points of the latter, we are done.
\end{proof}

We next establish a variant of the previous lemma in a families situation. To lighten the statement, we first introduce the setting, which is tailored to the ulterior proof of the birational invariance of the Deligne pairings in Proposition \ref{prop:elvira}. Let $S$ be a scheme and $Y\to S$ a flat and proper morphism of finite presentation. We assume that $Y$ is equipped with a line bundle $L$ and a relatively very ample line bundle $A$. We also suppose given a finite collection of commutative diagrams
\begin{equation}\label{eq:geom-setting-avoidance}
    \xymatrix{
        Z_{i}\ar[r]^{h_{i}}\ar[d]   &Y\ar[d]\\
        T_{i}\ar[r]     &S,
    }
\end{equation}
for $i=1,\ldots, m$. Here, the morphisms $T_{i}\to S$ are closed immersions of finite presentation, and the morphisms $Z_{i}\to T_{i}$ are flat and proper of finite presentation. The $h_{i}$ are automatically proper of finite presentation.

\begin{lemma}\label{lemma:prime-avoidance-families}
In the setting \eqref{eq:geom-setting-avoidance}, the following hold locally with respect to $S$ and the $T_{i}$, in the specified topologies:
\begin{enumerate}
    \item\label{item:prime-avoidance-families-1} (\'Etale topology) There exists a global section of $A$ that defines a relative effective Cartier divisor over $S$, and whose pullbacks to the $Z_{i}$ are defined and are relative effective Cartier divisors over the $T_{i}$.
    \item\label{item:prime-avoidance-families-2} (Zariski topology) For some $N_{0}\geq 0$ and for every integer $N\geq N_{0}$, there exists a global section of $L\otimes A^{N}$ which defines a relative effective Cartier divisor over $S$, and whose pullbacks to the $Z_{i}$ are defined and are relative effective Cartier divisors over the $T_{i}$.
\end{enumerate}

\end{lemma}
\begin{proof}
We can suppose that $S$ is affine. By a Noetherian approximation argument, we can reduce to the case when the involved schemes are Noetherian. It is then enough to establish the claims locally around any given closed point $s\in S$. Because the morphisms $T_{i}\to S$ are closed immersions, without loss of generality we can suppose that $s$ is in the image of all these maps. 

For the first claim, after possibly restricting $S$, we can fix a closed embedding $j\colon Y\to\PBbb^{M}_{S}$, such that $A\simeq j^{\ast}\Ocal(1)$. We begin by looking at the fiber of \eqref{eq:geom-setting-avoidance} over $s$. Let $k$ be the separable closure of the residue field $k(s)$, so that $k$ is an infinite field. By Lemma \ref{lemma:prime-avoidance} \eqref{item:prime-avoidance-1}, we can find a global section $\sigma$ of $\Ocal(1)|_{\PBbb^{M}_{k}}$, which avoids the associated primes of $Y_{k}$ and the images of the associated points of the $(Z_{i})_{k}$. Indeed, there are only finitely many such points. The section $\sigma$ is defined over some finite separable extension of $k(s)$. We redefine $k$ to be such an extension. Using \cite[\href{https://stacks.math.columbia.edu/tag/01ZM}{02LF}]{stacks-project} and base change, we can replace $(S,s)$ with an \'etale neighborhood of $s$, and suppose that $k(s)=k$. After possibly shrinking $S$ around $s$, the section $\sigma$ extends to a section of $\Ocal(1)$ in $\PBbb^{M}_{S}$. We denote by $\widetilde{\sigma}$ any choice of extension. This section avoids the associated points of $Y_{s}$. Because $Y\to S$ is proper and flat of finite presentation, by \cite[Proposition 11.3.7 \& Th\'eor\`eme 11.3.8]{EGAIV3} the section $\widetilde{\sigma}$ defines a relative effective Cartier divisor. By construction, the pullbacks $h_{i}^{\ast}(\widetilde{\sigma})$ also avoid the associated points of the $(Z_{i})_{s}$. Hence, they also define relative effective Cartier divisors.

The second claim is proven along the same lines, using Lemma \ref{lemma:prime-avoidance} \eqref{item:prime-avoidance-2}. Nevertheless, a moment's thought indicates that one can no longer use an embedding into a projective space. Instead, we proceed as follows. For some $N_{0}\geq 0$ and for every $N\geq N_{0}$, there exists a section $\sigma$ of $L\otimes A^{N}$ which avoids the associated primes of $Y_{s}$ and the images of the associated primes of the $(Z_{i})_{s}$. Let $\Ical$ be the coherent sheaf of ideals defining the closed subscheme $Y_{s}$ of $Y$. By possibly increasing $N_{0}$, we can moreover suppose that $H^{1}(Y,  L\otimes A^{N}\otimes\Ical)=0$ for $N\geq N_{0}$. This ensures that $\sigma$ extends to $Y$. The rest of the argument is analogous to that of the first point addressed above. 
\end{proof}

With the help of the previous lemmas, we can establish the existence of generators of the Deligne pairings. 
\begin{proposition}\label{prop:construction-generators}
Let $f\colon X\to S$ be a morphism satisfying the condition $(C_{n})$. Let $L_{0},\ldots, L_{n}$ be line bundles on $X$. The following assertions hold, locally with respect to $S$, in the specified topologies: 
\begin{enumerate}
    \item\label{item:construction-generators-1} (\'Etale topology) If the $L_{i}$ are relatively very ample, there exist global sections of the $L_{i}$, in general position in the sense of \textsection \ref{subsubsec:generators-relations}.
    \item\label{item:construction-generators-2} (Zariski topology) In general, there exist rational sections of the $L_{i}$, in general position in the sense of \textsection \ref{subsubsec:generators-relations}.
\end{enumerate}
\end{proposition}
\begin{proof}
We discuss the second statement and leave the first one as an analogous exercise. In the sequel, the constructions are Zariski local on $S$, and we no longer comment on the several necessary and repeated restrictions of the base scheme. 

Since $f$ is locally projective, we can suppose that there is an ample line bundle $A$ on $X$. By Lemma \ref{lemma:prime-avoidance} \eqref{item:prime-avoidance-2}, we can find global sections of $L\otimes A^{N}$ and $A^{N}$ which define relative effective Cartier divisors $D_{0}^{0}$ and $D_{0}^{1}$. The (multiplicative) difference provides a rational section $\ell_{0}$. 

Suppose we have constructed rational sections $\ell_{0},\ldots, \ell_{k}$, for some $k\leq n-1$, defining Cartier divisors $D_{i}^{0}-D_{i}^{1}$ which satisfy the general position conditions in \textsection \ref{subsubsec:generators-relations}. We apply Lemma \ref{lemma:prime-avoidance-families} \eqref{item:prime-avoidance-families-2} to the following situation of the type \eqref{eq:geom-setting-avoidance}: 
\begin{equation}\label{eq:successive-intersection-divisors}
    \xymatrix{
        \bigcap_{j\in J}D_{j}^{\epsilon_{j}}\ \ar@{^{(}->}[r]\ar[d]  &X\ar[d]\\
        S\ar@{=}[r]     &S,
    }
\end{equation}
for $J\subseteq\lbrace 0,\ldots, k\rbrace$ and $\epsilon_{j}= 0$ or $1$. We obtain global sections of $L\otimes A^{M}$ and $A^{M}$, whose divisors are relative effective Cartier divisors, and meet all the intersections in \eqref{eq:successive-intersection-divisors} in relative effective Cartier divisors of the latter. The difference is our choice for $\ell_{k+1}$. For the sections $\ell_{0},\ldots,\ell_{k},\ell_{k+1}$ to satisfy the conditions of \textsection \ref{subsubsec:generators-relations} it only remains to check that they can be permuted in any order. This is a consequence of the criterion for reordering regular sequences in \cite[\href{https://stacks.math.columbia.edu/tag/01ZM}{07DW}]{stacks-project}. The proof is complete.
\end{proof}

\begin{remark}
In Proposition \ref{prop:construction-generators}, suppose that the $L_{i}$ are relatively very ample. If we invoke the existence of generators in the Zariski topology, we can only guarantee the existence of rational sections, as opposed to global ones. In this respect, it can occasionally be better to work in the \'etale topology. In any event, in   constructions that are to be compatible with base change and to preserve the multilinearity of the Deligne pairings, we can reduce to Deligne pairings of relatively ample line bundles, which admit global sections in general position. 
\end{remark}

\subsubsection{Chern classes and direct images}\label{subsub:Cherndirect}

The main goal of this section is to provide the relationship between the isomorphism class of the line bundle \eqref{eq:definition-ducrot} with the direct image of Chern classes in \eqref{def:firstimeintbundle}. While the expression in \eqref{eq:definition-ducrot} is defined for general schemes and morphisms satisfying the condition $(C_n)$, we need to suppose that our base scheme $S$ admits an ample family of line bundles. If there was an extension of the constructions, including $\gamma$-operations and the relation \eqref{iso:GR1} below, to the $K$-theory of perfect complexes, this could be avoided. 

Let $X$ be a scheme, and denote by $K_{0}(X)$ the $K$-theory of vector bundles on $X$. The exterior powers of vector bundles define $\lambda$-operations on $K_0(X)$. Following \cite[Expos\'e V, Section 3]{SGA6}, we define for any positive integer $i$, $\gamma^i(x) = \lambda^i(x+i-1)$. We obtain the $\gamma$-operations on $K_0(X)$, which endows $K_0(X)$ with an alternative $\lambda$-ring structure.  

The $\gamma$-filtration, $F^i K_0(X) \subseteq K_0(X),$ is defined as follows. One sets $F^0 K_0(X) = K_0(X)$ and let $F^1 K_0(X)$ be the kernel of the rank map $K_0(X) \to H^0(X, \mathbb{Z})$. Also, $F^i K_0(X)$ is generated by products of the form $\prod \gamma^{k_j}(x_j)$ where $x_j \in F^1 K_0(X)$ and $\sum k_j \geq i$. This allows one to define Chow-like groups by $\Gr^i K_0(X) = F^i K_0(X)/F^{i+1}K_0(X).$ The $i$-th Chern class of a vector bundle $E$ is defined by the class of $\gamma^{i}([E]-r)$ in $\Gr^i K_0(X)$ , where $r \in H^0(X, \ZBbb)$ is the rank of $E$. If $X$ is quasi-compact and $E$ is a vector bundle of constant rank $r+1$ on $X$, recall that $K_{0}(\PBbb(E))$ is a free module over $K_{0}(X)$, and a basis is given by $1,\xi,\ldots, \xi^{r}$, where $\xi$ is the class of $\Ocal(1)$ in $K_{0}(\PBbb(E))$. From this description, the Chern classes can equivalently be constructed \emph{\`a la} Grothendieck \cite[Exposé VI, Section 5]{SGA6}. 

If $X$ is quasi-compact, by \cite[Exposé X, Théorème 5.3.2]{SGA6}, the determinant induces an isomorphism 
\begin{equation}\label{iso:GR1}
    \Gr^1 K_0(X)\to \Pic(X).
\end{equation}
In particular, $F^{2}K_{0}(X)$ is the kernel of the determinant map on $F^{1}K_{0}(X)$.

If $f: X \to Y$ is a projective local complete intersection morphism of constant relative dimension $d$, with $Y$ admitting an ample line bundle, by \cite[Exposé VIII, Proposition 3.2]{SGA6} the derived direct image of a complex induces a graded morphism of filtrations, after tensoring with $\QBbb$, 
\begin{equation}\label{gradeddirect}
    f_!: F^i K_0(X)_\QBbb \to F^{i-d}K_0(Y)_\QBbb.
\end{equation} 
It hence furnishes a direct image $f_\ast: \Gr^iK_0(X)_{\QBbb}\to \Gr^{i-d}K_0(Y)_\QBbb.$ Here we prefer the notation $f_{\ast}$ for the direct image on the graded pieces since it corresponds to the pushforward functoriality of Chow groups. We notice that the assumptions on the morphism are to ensure that the derived direct image of a vector bundle is perfect, and the assumption of the existence of an ample line bundle is to guarantee that the $K$-theory defined in terms of perfect complexes is equal to that defined by vector bundles. By \cite[Proposition 3.10]{ThomasonTrobaugh}, the latter property also holds replacing an ample line bundle  by $Y$ being divisorial, see Definition \ref{def:divisorial}. The statements on the grading are concrete computations by factoring into a regular embedding and a projective bundle projection. 

In the following proposition, in a restricted setting, we establish an integral version of  \eqref{gradeddirect}. See Proposition 2.25 \cite{Nakayamaintsheaves} for a related argument in the Noetherian case.

\begin{proposition}\label{prop:directimageK}
Suppose that $f: X\to S$ is a morphism satisfying the condition $(C_n)$, with $S$ divisorial. Then, the direct image $Rf_\ast$ naturally furnishes a direct image map 
\begin{displaymath} f_\ast : \Gr^{n+1} K_0(X) \to \Gr^1 K_0(S).
\end{displaymath}
\end{proposition}

\begin{proof}
    As in the discussion leading up to the proposition, to verify that the derived direct image induces a map of $K_0$-groups, it is enough to verify that it sends vector bundles to perfect complexes. Since $f$ is a proper, flat, and finitely presented morphism, by Proposition \ref{prop:perfect-conditions} \eqref{item:perfect-conditions-2} $Rf_{\ast}$ sends perfect complexes, in particular vector bundles, to perfect complexes.

    We need to prove that $f_!$ sends $F^{n+1} K_0(X)$ to $F^1 K_0(S)$ and $F^{n+2} K_0(X)$ to $F^2 K_0(S)$. Given  \linebreak  $x=\prod \gamma^{k_j}(x_j)$ in $F^{n+1} K_0(X)$, to prove that $f_!\ x  \in K_0(S)$  is of virtual rank 0, and hence in $F^1 K_0(S)$, by Proposition \ref{prop:perfect-conditions} \eqref{item:perfect-conditions-2} we can reduce to the case that $S$ is a point. In this case, the virtual rank coincides with the Euler characteristic of $x$. This can be computed by the Riemann--Roch theorem for singular varieties as in \cite[Corollary 18.3.1]{Fulton}. For simplicity, we treat the case $x_j = L_j-1$ and $k_j=1$. One then observes that for $k>0$ the Chern class $\prod_{j=1}^{n+k}\ch(L_{j}-1)$ is concentrated in degrees bigger than $n$, and hence $\chi(x) = \int_{X} \prod_{j=1}^{n+k}\ch(L_{j}-1) \cap \td(X)$ is zero for dimension reasons. We leave the general case to the reader, using for example some version of the splitting principle. 

    If $x$ belongs to $F^{n+2}K_{0}(X)$, to prove that it is sent to $F^{2}K_{0}(S)$, by \eqref{iso:GR1} we need to show that $\det f_!\ \prod \gamma^{k_j}(x_j)= [\Ocal_S] \in \Pic(S)$ whenever $\sum k_j \geq n+2.$ Since any element $x_j$ is a virtual sum  of vector bundles of virtual rank 0, one straightforwardly reduces to the case when $x_j = [E_j]-r_j$, for a vector bundle $E_j.$ For simplicity, we suppose the ranks $r_j$ are constant. We want to reduce to the case when $E_j$ is a line bundle by a version of the splitting principle, and perform the argument for $E = E_0$ and $r=r_0, k = k_0$. If $E$ admits a complete flag this is immediate. The complete flag variety $\pi: D \to X$ of $E$ is a composition of projective bundles, and we denote by $\Ocal_\ell(1), \ell = 1, \ldots, r$ the successive tautological line bundles. It is of relative dimension $m=1+\ldots+(r-1)$. By  \linebreak  \cite[Exposé VI, Corollaire 5.8]{SGA6} it follows that $\pi_!\ F^{n+1+m} K_0(D) \subseteq F^{n+1}K_0(X)$. By the projection formula 
    \begin{displaymath}
        f_! \left(([E]-r) \cdot \prod_{j\geq 1} \gamma^{k_j}(x_j) \right)  = ( f \pi)_!\ \pi^{\ast} \left(\left(\gamma^{k_0}([E]-r) \cdot \prod_{j \geq 1} \gamma^{k_j}(x_j) \right) \prod_\ell \left(1-[\Ocal_\ell(-1)]\right)^{r-\ell}\right).
    \end{displaymath}
    Here, we used that for a projective bundle $g \colon \PBbb(V)\to X$, $Rg_\ast \Ocal(-i) \simeq 0$ for $i = 1, \ldots, (\rk V)-1$ and $\Ocal_X$ for $i = 0$. Hence, possibly replacing $f$ by $ f \circ \pi$ iteratively over the various flag varieties corresponding to the $E_j$, we are reduced to proving that there exists, for $k \geq 1$, an isomorphism
    \begin{displaymath}
        \det Rf_{\ast} \left( (L_0 -1)\cdot \ldots \cdot (L_n -1) \cdot \ldots \cdot (L_{n+k} -1)\right) \simeq \Ocal_S.
    \end{displaymath}
    Notice that whenever $k =0$ the left-hand side is the definition in \eqref{eq:definition-ducrot}. By Noetherian approximation, we can suppose that $S$ is in fact Noetherian. In this case, the statement is a major ingredient in the proof of the main theorem of \cite{Ducrot}, and it can be deduced from \cite[Proposition 4.7.1]{Ducrot}, which produces a canonical such isomorphism for determinants of cohomologies. 
\end{proof}
The proof in fact gives the following more precise form:
\begin{corollary}\label{cor:deligne-pairing-direct-image}
    With the same assumptions and notations as in Proposition \ref{prop:directimageK}, suppose we are given line bundles $L_0, \ldots, L_n$. Then, 
\begin{displaymath}
     c_1(\langle L_0, \ldots, L_n \rangle_{X/S} )=f_\ast \left(c_{1}(L_0)\cdot \ldots\cdot c_{1}(L_n)\right). 
\end{displaymath}
\end{corollary}
\qed

\subsection{Properties of Deligne pairings}
We discuss the projection formula and birational invariance property of Deligne pairings, originally due to Elkik \cite{Elkikfib} and Mu\~noz Garc\'ia \cite{Munoz}, in the Noetherian setting. The case of base schemes is treated by Noetherian approximation. We further address some features which are not covered by these references.  

\begin{proposition}[Projection formula]\label{prop:projformuladeli}
Let $f\colon X\to S$, $g\colon X^{\prime}\to S$, and $h\colon X^{\prime}\to X$ be morphisms satisfying the conditions $(C_{n})$, $(C_{n+n^{\prime}})$, and $(C_{n^{\prime}})$, respectively. Suppose we are given line bundles $L_0, \ldots, L_\ell$ on $X$ and $M_{0}, \ldots, M_{m}$ on $X'$, with $(\ell+1)+(m+1)=n+n'+1$. There are canonical projection formula isomorphisms 
        \begin{displaymath}
            \langle h^\ast L_0, \ldots, h^\ast L_\ell, M_{0}, \ldots, M_{m} \rangle_{X'/S} \simeq 
\begin{cases}
    \langle L_0, \ldots, L_{n-1}, \langle M_{0}, \ldots, M_{n'} \rangle_{X'/X} \rangle_{X/S}, & \text{if }\; m = n' .\\ \\
    \langle  L_0, \ldots, L_{n} \rangle_{X/S}^{\kappa}, & \text{if }\; m= n'-1.\\
    &\\ 
    \Ocal_S, & \text{if }\; m < n'-1. 
\end{cases}
        \end{displaymath}
In the second case, $\kappa = \int_{X'/X} \prod_{i} c_1(M_i)$ is the fiberwise degree and is then assumed to be constant. Moreover:
\begin{enumerate}
    \item The isomorphisms are functorial and multilinear in the line bundles, where, in the third case, we interpret $\Ocal_S$ as a constant functor.
    \item The projection formula isomorphisms are compatible with restrictions along divisors in any of the $L_i$.
    \item The projection formula isomorphisms are compatible with the symmetry isomorphisms in the $L_{i}$ and the $M_{j}$, separately.
\end{enumerate} 
\end{proposition}

\begin{proof}

We first reason locally and suppose $S$ affine and that $f \colon X \to S$ and $g \colon X' \to S$ are projective, in which case $h\colon X^{\prime}\to X$ is projective as well. Since Deligne pairings are functorial, by  Noetherian approximation, we can reduce the construction of the isomorphisms to the case of schemes that are finitely generated over $\ZBbb$. We notice that given Noetherian approximations $f_{\alpha}\colon X_{\alpha}\to S_{\alpha}$ and $g_{\alpha}\colon X^{\prime}_{\alpha}\to S_{\alpha}$ of $f$ and $g$, by  \cite[\href{https://stacks.math.columbia.edu/tag/01ZM}{01ZM}]{stacks-project} we may assume that $h$ descends to a morphism $h_{\alpha}\colon X^{\prime}_{\alpha}\to X_{\alpha}$.  In this situation all schemes are Noetherian, and we can then refer to \cite[\textsection IV.2]{Elkikfib} and \cite[\textsection 5.2]{Munoz}. We notice that in \cite{Munoz}, the base change compatibility of the projection formula for flat morphisms is not explicitly stated. It can however be checked by a careful proofreading. See also the construction below, from which this can also be inferred by a simple inspection.

For a general base scheme $S$, it is enough to compare the constructions for affine $S$ and any affine open $V$ in $S$. Since this open affine embedding can be approximated by Noetherian schemes, the functoriality in the Noetherian situation entails that the local construction above glues together on all of $S$. For similar reasons the resulting isomorphisms are functorial.

For future reference, we recall the construction of the isomorphisms in the Noetherian case. For the first one, working  locally on $S$, we can reduce ourselves to the case when all the line bundles are relatively ample over $S$. Locally on the base, we can suppose there are regular sections $\ell_i$ of $L_{i}, i =0,\ldots, n-1$ which cut out a scheme $D$ in $X$, which satisfies the condition $(C_0)$ over $S$. Denote by $D' = X' \times_X D \to D$, which is a morphism satisfying the condition $(C_{n'})$. After performing a reduction along the sections $\ell_i$ and $h^\ast \ell_i$, as in Proposition \ref{prop:generalpropertiesdeligneproducts} \eqref{item:deligne-pairing-restriction}, the projection formula takes the form of an isomorphism 
\begin{equation}\label{eq:projection-formula-simplified-form}
    \langle M_{0} |_{D'}, \ldots, M_{n'} |_{D'}\rangle_{D'/S}  \simeq N_{D/S} \left(\langle M_{0}|_{D'}, \ldots, M_{n'}|_{D'}\rangle_{D'/D}\right).
\end{equation}
Using that $D \to S$ is finite, one can see that locally on $S$, there exist regular sections $m_{i}$ of the $M_{i}|_{D'}$, for $i=0,\ldots, n'-1$, which cut out a closed subscheme $D''$ of $D'$, such that the natural morphism $D'' \to D$ satisfies the condition $(C_0)$. Then $D^{\prime\prime}\to S$ also satisfies the condition $(C_{0})$. Restricting along the divisors of these sections, the projection formula then takes the form 
\begin{displaymath}
    N_{D''/S} (M_{n'}|_{D''}) \simeq N_{D/S} N_{D''/D} (M_{n'}|_{D''}),
\end{displaymath}
which is the usual transitivity of the determinant construction, see \cite[Lemme 21.5.7.2]{EGAIV4}. It is part of the construction in \cite[\textsection IV.2]{Elkikfib} and \cite[\textsection 5.2]{Munoz} that this is independent of the choice of regular sequences, and hence defines a canonical projection formula isomorphism in general. 

Incidentally, the above description of the first projection formula isomorphism, together with the construction of the Deligne pairings by generators and relations, leads to the following expression for the isomorphism \eqref{eq:projection-formula-simplified-form}. We can actually choose that the sequence $m_{0},\ldots,m_{n^{\prime}-1}$ is part of a larger sequence of sections $m_{0},\ldots,m_{n^{\prime}}$, which are in general position with respect to $D$, and hence with respect to $S$. Then, the isomorphism \eqref{eq:projection-formula-simplified-form} is determined by the assignment
\begin{displaymath}
    \langle m_{0},\ldots, m_{n^{\prime}}\rangle_{D^{\prime}/S}\mapsto N_{D/S}(\langle m_{0},\ldots, m_{n^{\prime}}\rangle_{D^{\prime}/D}).
\end{displaymath}
The transitivity of the norm on functions shows that this is indeed compatible with the relations defining the Deligne pairings \eqref{eq:relationElkikDucrot}, and hence defines an isomorphism of line bundles. 

The second isomorphism is given on symbols by
\begin{equation}\label{eq:explicit-second-projection-formula}
    \langle h^{\ast} \ell_{0},\ldots, h^{\ast}\ell_{n},m_{0},\ldots, m_{n^{\prime}-1}\rangle\mapsto \langle \ell_{0},\ldots,\ell_{n} \rangle^{\kappa}.
\end{equation}
We notice that such symbols exist and provide local bases for the involved line bundles. Also, observe that this assignment is compatible with the relations defining the Deligne pairings. Finally, the resulting isomorphism is compatible with restriction along sections of the line bundles $L_{i}$, which in turn entails it is compatible with the first projection formula. Concretely, it coincides with the following combination of the first projection formula, and the special case of the second projection formula when $n=0$:

\begin{displaymath}
\begin{split}
    \langle h^\ast L_0, \ldots, h^\ast L_{n-1}, h^\ast L_n, M_0, \ldots, M_{n'-1} \rangle & \simeq \langle L_0, \ldots, L_{n-1}, \langle h^\ast L_n, M_0, \ldots, M_{n'-1} \rangle \rangle \\
    & \simeq \langle L_0, \ldots, L_{n-1},  L_n^{\kappa} \rangle \\
    & \simeq \langle L_0, \ldots, L_{n-1},  L_n \rangle^{\kappa}.
\end{split}
\end{displaymath}

The third isomorphism in turn follows from the second one since  $\kappa = 0$ under those assumptions.

The functoriality and the multilinearity statements follow from the above descriptions. The compatibility with restrictions along divisors is contained in the construction, and the compatibility with symmetries is also apparent from the construction.
\end{proof}

\begin{remark}
Since the formation of Deligne pairings is compatible with base change, \emph{a posteriori} we see that the description provided in the proof of Proposition \ref{prop:projformuladeli} remains valid over any base $S$.
\end{remark}

\begin{corollary} \label{cor:compprojformula}
    With the notation as in Proposition \ref{prop:projformuladeli}, suppose $h^{\prime}: X^{\bis} \to X^{\prime}$ is a morphism satisfying the condition $(C_{n^{\bis}})$, and such that the morphisms $X^{\bis} \to X$, $X^{\bis} \to S$ satisfy the conditions $(C_{n^{\prime} + n^{\bis}}), (C_{n+n^{\prime} + n^{\bis}})$. Then, the composition of the projection formula for $h$ and $h^{\prime}$ is that of $ h^{\bis}=h\circ h^{\prime}$ and $h^{\prime}$. More precisely, there are two natural isomorphisms of the form
    \begin{displaymath}
        \begin{split}
        &\langle h^{\bis \ast } L_0, \ldots, h^{\bis \ast}L_{n-1}, h^{\prime\ast} M_0, \ldots, h^{\prime\ast} M_{n^{\prime}-1}, N_{0}, \ldots, N_{n^{\bis}} \rangle_{X^{\bis}/S}\\
           &\hspace{4cm} \to
        \langle L_{0},\ldots, L_{n-1}, \langle M_{0},\ldots, M_{n^{\prime}-1},\langle N_{0},\ldots,N_{n^{\bis}}\rangle_{X^{\bis}/X^{\prime}}\rangle_{X^{\prime}/X}\rangle_{X/S}
        \end{split}
    \end{displaymath}
    which can be derived from the projection formula, and both coincide and likewise in the other possible cases.  
\end{corollary}
\begin{proof}
The statement can be checked locally in the \'etale topology on $S$, and working with symbols. By the compatibility of the projection formula isomorphisms with the multilinearity of the Deligne pairings, we can suppose that all the line bundles are relatively very ample. Recall that the projection formula isomorphisms are compatible with restrictions along relative divisors $D$ of sections of $L_0$. By iteratively choosing sections of $L_0, L_1|_D $, etc., we reduce to the case when $X \to S$ is finite. Then, as in the construction, we can analogously restrict along divisors of the $M_i$, which are relative over both $S$ and $X$, to reduce to the case when also $X^{\prime} \to X$ is finite. Finally, the same restriction argument along the remaining bundles allows us to suppose that $X^{\bis} \to X^{\prime}$ is finite. In this case, the statement reduces to the transitivity of the norm functor. 
\end{proof}

\begin{proposition}[Birational invariance]\label{prop:elvira}
    Let $\pi \colon \widetilde{X} \to X$ be a morphism of schemes, such that $\widetilde{X} \to S$ also satisfies the condition $(C_n)$. Suppose that there exists an open subset $U$ of $X$, fulfilling:
    \begin{itemize} 
    \item  The open immersion $U \to X$ is quasi-compact.
    \item  $U$ is fiberwise dense in $X$. 
    \item $\pi^{-1}(U)\to U$ is an isomorphism. 
    \end{itemize} 
Let $L_{0},\ldots, L_{n}$ denote line bundles on $X$. Then,  there is a canonical isomorphism 

    \begin{equation}\label{eq:iso-elvira}
        \langle L_0, \ldots, L_n \rangle_{X/S} \simeq \langle \pi^\ast L_0, \ldots, \pi^\ast L_n \rangle_{\widetilde{X}/S},
    \end{equation}
which is functorial and multilinear in the line bundles. Moreover:
 \begin{enumerate}   
        \item\label{item:elvira-1} The isomorphism \eqref{eq:iso-elvira} is compatible with the symmetry isomorphisms.
        \item\label{item:elvira-3} If some $L_{i}=\Ocal(D)$ for $D$ a relative effective Cartier divisor, such that $\widetilde{D} = \pi^{\ast} D$ is defined and is also a relative effective divisor, and $D\cap U$ is fiberwise dense in $D$, then the isomorphism \eqref{eq:iso-elvira} is compatible with the restriction isomorphisms along $D$ and $\widetilde{D}$. 
    \end{enumerate}
    
\end{proposition}
\begin{proof}
Since the construction is to be compatible with base change, we can suppose that the base is affine and that $\widetilde{X} \to S$ and $X\to S$ are projective. By a Noetherian approximation argument, we can reduce to the case when the involved schemes are Noetherian. To descend the open immersion $U \to X$ it is necessary to assume it is of finite presentation, which is guaranteed by the assumption of $U \to X$ being quasi-compact. In this situation, the existence of a functorial multilinear isomorphism, modulo the base change compatibility and properties \eqref{item:elvira-1}--\eqref{item:elvira-3}, is the statement of \cite[Th\'eor\`eme 5.3.1]{Munoz}. 

To prove the base change compatibility and properties \eqref{item:elvira-1}--\eqref{item:elvira-3}, we will review the description of the isomorphism \eqref{eq:iso-elvira} in terms of symbols, extracted from \cite[Th\'eor\`eme 5.3.1]{Munoz}. In particular, this will provide a self-contained construction of \eqref{eq:iso-elvira}. All the arguments to follow are local with respect to $S$. 

Fix a closed point $s$ of $S$. Around $s$, we can find a trivialization $\langle\ell_{0},\ldots,\ell_{n}\rangle$ of $\langle L_{0},\ldots, L_{n}\rangle$, such that: 
\begin{enumerate}
    \item[(a)] The $\ell_{i}$ define Cartier divisors $D_{i}=D_{i}^{0}-D_{i}^{1}$ as in \textsection \ref{subsubsec:generators-relations}.
    \item[(b)] The pullbacks $\pi^{\ast}(\ell_{i})$ are defined, and the associated Cartier divisors  \linebreak  $\pi^{\ast}(D_{i})=\pi^{\ast}(D_{i}^{0})-\pi^{\ast}(D_{i}^{1})$ satisfy the conditions in \textsection \ref{subsubsec:generators-relations}.
    \item[(c)] For every $i=0,\ldots, n$, we have $D_{0}^{\epsilon_{0}}\cap\ldots\widehat{D}_{i}^{\epsilon_{i}}\ldots\cap D_{n}^{\epsilon_{n}}\cap (X\setminus U)_{s}=\emptyset$. Here, $\epsilon_{i}=0$ or $1$ and the term $\widehat{D}_{i}^{\epsilon_{i}}$ is omitted.
\end{enumerate}
The existence of such a trivialization is obtained along the same lines as in Proposition \ref{prop:construction-generators}. We will just explain the necessary modifications. 

For the construction of the first section $\ell_{0}$, we apply Lemma \ref{lemma:prime-avoidance-families} \eqref{item:prime-avoidance-families-2}, with the choice of diagrams
\begin{displaymath}
    \xymatrix{
        \widetilde{X}\ar[r]^{\pi}\ar[d]  &X\ar[d]\\
        S\ar@{=}[r]     &S
    }
    \quad\quad\quad
    \xymatrix{
        (X\setminus U)_{s}\ \ar@{^{(}->}[r]\ar[d]  &X\ar[d]\\
        \Spec k(s)\ \ar@{^{(}->}[r]     &S,
    }
\end{displaymath}
where $X\setminus U$ is endowed with the reduced scheme structure. We go on by induction and assume we have constructed $\ell_{0},\ldots,\ell_{k}$, $k\leq n-1$. To construct $\ell_{k+1}$, we apply again Lemma \ref{lemma:prime-avoidance-families} \eqref{item:prime-avoidance-families-2}. Besides the diagrams \eqref{eq:successive-intersection-divisors}, we also take
\begin{displaymath}
    \xymatrix{
        \bigcap_{j\in J}\pi^{\ast}(D_{j}^{\epsilon_{j}})\ \ar[r]^-{\pi}\ar[d]  &X\ar[d]\\
        S\ar@{=}[r]     &S
    }
    \quad\quad\quad
    \xymatrix{
        \bigcap_{j\in J} D_{j}^{\epsilon_{j}}\ \cap (X\setminus U)_{s}\ \ar@{^{(}->}[r]\ar[d]  &X\ar[d]\\
        \Spec k(s)\ \ar@{^{(}->}[r]     &S,
    }
\end{displaymath}
for $J\subseteq\lbrace 0,\ldots, k\rbrace$ and $\epsilon_{j}=0$ or $1$. From this point, it is easy to complete the induction. We notice that, with the sections so constructed, condition (c) is automatically fulfilled for dimension reasons. Moreover, because the morphism $X\to S$ is proper, condition (c) continues to hold even if we replace $(X\setminus U)_{s}$ by $(X\setminus U)_{s'}$, for $s'$ non-necessarily closed and in a well-chosen neighborhood of $s$. We infer that $D_{0}^{\epsilon_{0}}\cap\ldots\widehat{D}_{i}^{\epsilon_{i}}\ldots\cap D_{n}^{\epsilon_{n}}$, which is finite and flat over $S$, is entirely contained in $U$. Consequently, it is necessarily isomorphic to its pullback by $\pi$. 

In terms of the sections above, the isomorphism \eqref{eq:iso-elvira} is given by
\begin{equation}\label{eq:iso-elvira-sections}
        \langle\ell_{0},\ldots,\ell_{n}\rangle\mapsto \langle \pi^{\ast}(\ell_{0}),\ldots,\pi^{\ast}(\ell_{n})\rangle.
\end{equation}
This assignment is compatible with the relations defining the Deligne pairings, whenever we change the sections $\ell_{i}$ by sections $\ell_{i}^{\prime}$ satisfying properties (a)--(c); for the intersection of any of the $n$ associated divisors is contained in $U$. In passing, this entails that the exhibited local construction of \eqref{eq:iso-elvira} globalizes. The base change and symmetry properties are now clear from \eqref{eq:iso-elvira-sections}.

For the compatibility with the restriction isomorphisms, we suppose, without loss of generality, that $L_{0}=\Ocal(D)$ as in the statement. We argue with symbols. The assumptions on $D$ guarantee that, in the construction of the sections $\ell_{i}$, we can take $\ell_{0}=\mathbf{1}_{D}$, the canonical section of $\Ocal(D)$. Then, the collection $\ell_{1}|_{D},\ldots, \ell_{n}|_{D}$ satisfies properties (a)--(c), relative to the morphism $\widetilde{D}\to D$ and the open $D\cap U$. Notice that $\mathbf{1}_{\widetilde{D}}=\pi^{\ast}(\mathbf{1}_{D})$. The restriction isomorphisms send the trivializations $\langle \mathbf{1}_{D},\ell_{1},\ldots,\ell_{n}\rangle$ to $\langle\ell_{1}|_{D},\ldots, \ell_{n}|_{D}\rangle$ and $\langle \mathbf{1}_{\widetilde{D}},\pi^{\ast}(\ell_{1}),\ldots,\pi^{\ast}(\ell_{n})\rangle$ to $\langle\pi^{\ast}(\ell_{1})|_{\widetilde{D}},\ldots, \pi^{\ast}(\ell_{n})|_{\widetilde{D}}\rangle$. Because $\mathbf{1}_{\widetilde{D}}=\pi^{\ast}(\mathbf{1}_{D})$ and $\pi^{\ast}(\ell_{i})|_{\widetilde{D}}=\pi^{\ast}(\ell_{i}|_{D})$, we conclude from the description \eqref{eq:iso-elvira-sections} for $\widetilde{X}\to X$ and $\widetilde{D}\to D$.
    
\end{proof}

The next corollary states that the isomorphisms of Proposition \ref{prop:projformuladeli} and Proposition \ref{prop:elvira} are compatible with each other. In preparation for the statement, we introduce some notation. Suppose that $h \colon X' \to X$ is as in Proposition \ref{prop:projformuladeli} and $\pi \colon \widetilde{X} \to X$ as in Proposition \ref{prop:elvira}. Construct the Cartesian diagram
\begin{equation}\label{eq:diagram-X-tilde-X-prime}
    \xymatrix{
        \widetilde{X}'\ar[r]^{\pi'}\ar[d]_{\widetilde{h}}      &X'\ar[d]^{h}\\
        \widetilde{X}\ar[r]_{\pi}       &X.
    }
\end{equation}
 We notice that $\pi'$ satisfies the assumptions of Proposition \ref{prop:elvira}. For a line bundle $L$ on $X$ (resp. $M$ on $X'$), we write $\widetilde{L}$ for the pullback to $\widetilde{X}$ (resp. $\widetilde{M}$ for the pullback to $\widetilde{X}')$.

\begin{corollary}
With the notation as in Proposition \ref{prop:projformuladeli} and \eqref{eq:diagram-X-tilde-X-prime} above, if $m= n'$, the diagram of isomorphisms
\begin{displaymath}
    \xymatrix{
        \langle h^\ast L_0, \ldots, h^\ast L_{n-1}, M_0, \ldots, M_{n'} \rangle_{X'/S} \ar[d] \ar[r] & \langle \widetilde{h}^\ast \widetilde{L}_0, \ldots, \widetilde{h}^\ast \widetilde{L}_{n-1}, \widetilde{M}_0, \ldots, \widetilde{M}_{n'} \rangle_{\widetilde{X}'/S} \ar[d] \\
        \langle L_0, \ldots, L_{n-1}, \langle M_0, \ldots, M_{n'} \rangle_{X'/X} \rangle_{X/S} \ar[r] & \langle \widetilde{L}_0, \ldots, \widetilde{L}_{n-1},\langle \widetilde{M}_0, \ldots, \widetilde{M}_{n'} \rangle_{\widetilde{X}'/\widetilde{X}}\rangle_{\widetilde{X}/S}
    }
\end{displaymath}
commutes, and similarly for $m = n'-1$ or $m < n'-1.$
\end{corollary}
\begin{proof}
This results by combining the description of the isomorphisms in terms of symbols, provided in the proofs of Proposition \ref{prop:projformuladeli} and Proposition \ref{prop:elvira}. The details are left to the reader. 
\end{proof}

\section{Intersection bundles}\label{subsubsec:general-intersection-bundles} 
In this section, following Elkik \cite[\textsection V]{Elkikfib}, we use the Deligne pairings to define intersection bundles, which are functorial counterparts of direct images of higher Chern classes. We establish basic properties of those, in the form of canonical isomorphisms which lift classical identities between Chern classes. Our approach slightly differs from that of Elkik, in that we systematically use the splitting principles proven in our previous work \cite{Eriksson-Freixas-Wentworth} and recalled in \textsection\ref{subsubsec:line-functors}. However, we show that both yield equivalent results. The comparison is based on various compatibility features between the basic properties, which are only mentioned in passing in \cite{Elkikfib}. These will also be key in the development of a functorial framework for the theory of intersection bundles in Section \ref{subsec:int-bundles-line-distributions}.

\subsection{Construction of intersection bundles}\label{subsubsec:construction-intersection-bundles}
We construct intersection bundles for morphisms of schemes $f\colon X\to S$ satisfying the condition $(C_{n})$. While Elkik's original approach restricts to Cohen--Macaulay morphisms over Noetherian bases, the results in \textsection\ref{subsec:Deligne-pairings} allow us to work in greater generality.

\subsubsection{Intersection bundles of Chern power series} 
Recall the category $\CHfrak_{+}(X)$ constructed in \textsection \ref{subsec:universal-chern-categories}. We proceed to associate, to every object in $\CHfrak_{+}(X)$, a line bundle over $S$. Recall that such objects are called positive Chern power series.

Suppose first that we are given vector bundles $E_1, \ldots, E_m$ on $X$, of non-zero constant ranks $r_1, \ldots, r_m$. Let $k_{1},\ldots,k_{m}\geq 0$  be integers with $\sum_{i}k_{i}=n+1$.  We define a Chern-type intersection bundle on $S$, 
\begin{equation}\label{eq:general-Chern-bundle-1}
    \langle \cfrak_{k_{1}}(E_{1})\cdot\ldots \cdot \cfrak_{k_{m}}(E_{m})\rangle_{X/S}\quad\text{or}\quad \langle \cfrak_{k_{1}}(E_{1})\cdot\ldots \cdot \cfrak_{k_{m}}(E_{m})\rangle,
\end{equation}
as follows. We first define a Segre-type intersection bundle. Let $\Pbold=\PBbb(E_{1})\times_{X}\ldots\times_{X}\PBbb(E_{m})$, and notice that the morphism $\Pbold\to S$ satisfies the condition $(C_{d})$, with $d=n-m+\sum r_{i}$, by \cite[\href{https://stacks.math.columbia.edu/tag/0C4P}{0C4P}]{stacks-project}: since $\Pbold\to X$ is projective and $X\to S$ is locally projective, the composition $\Pbold\to S$ is locally projective too. Denote by $\Ocal_i (1)$ the tautological bundle on $\PBbb (E_{i})$, or its pullback to $X'$. Then, we put
\begin{equation}\label{eq:definition-Segre-classes}
    \langle \sfrak_{k_1}(E_1)\cdot \ldots\cdot \sfrak_{k_{m}}(E_m) \rangle_{X/S} =  \langle \Ocal_1 (1)\{r_1+k_1-1\}, \ldots , \Ocal_k (1)\{r_m+k_{m}-1\}  \rangle_{\Pbold/S},
\end{equation}
where $\Ocal_i (1)\{r_i+k_i-1\}$ means $\Ocal_{i}(1)$ repeated $r_{i}+k_{i}-1$ times. Notice that we follow the convention for Segre classes in \cite{Elkikfib}, and the construction above hence differs from that of Fulton \cite{Fulton} by the sign $(-1)^{n+1}$. This has no effect on the Chern-type constructions below, where both approaches produce the same bundles. 

The intersection bundles \eqref{eq:general-Chern-bundle-1} are then constructed inductively, by  defining 
\begin{displaymath}
\langle \cfrak_{k_1}(E_1) \cdot \ldots \cdot \cfrak_0(E_j) \cdot \ldots \cdot \cfrak_{k_m}(E_m) \rangle_{X/S}  = \langle \cfrak_{k_1}(E_1) \cdot \ldots \cdot \sfrak_0(E_j) \cdot \ldots \cdot \cfrak_{k_m}(E_m) \rangle_{X/S} 
\end{displaymath}
and, for $k \geq 1$, 
\begin{displaymath}
\langle \cfrak_{k_1}(E_1) \cdot \ldots \cdot \cfrak_k(E_j) \cdot \ldots \cdot \cfrak_{k_m}(E_m) \rangle_{X/S}  = \bigotimes_{i=1}^{k} \langle \cfrak_{k_1}(E_1) \cdot \ldots \cdot \sfrak_i(E_j) \cdot \cfrak_{k-i}(E_j) \cdot \ldots \cdot \cfrak_{k_m}(E_m) \rangle_{X/S}^{(-1)^{i+1}} . 
\end{displaymath}
Anticipating the relation with Chern categories later in Section \ref{subsec:int-bundles-line-distributions}, these definitions correspond to the formal identifications
\begin{equation}\label{eq:definition-chern-classes}
    \cfrak_0=\sfrak_0 \quad\text{and}\quad \cfrak_k = \sum_{i=1}^{k} (-1)^{i+1}  \sfrak_i  \cfrak_{k-i}\quad\text{if}\quad k\geq 1.
\end{equation}
 The inductive construction of \eqref{eq:general-Chern-bundle-1} requires an order on the indices $(k_{1},\ldots, k_{m})$, but the result is independent of the order up to canonical isomorphism. Unwinding the definitions, this follows by MacLane's coherence theorem for the tensor product of line bundles since the category of line bundles is strictly commutative. It  allows us to identify double duals of line bundles with the line bundles themselves, without any ordering issues. 
 It is clear from the construction that the intersection bundles thus defined satisfy a decomposition property akin to \eqref{eq:decomposition-Deligne-pairing-partition}.

In addition, we introduce a convention in the following limit cases. For ease of presentation, these trivial cases are not treated in detail below. For the 0 monomial, we declare $\langle 0\rangle_{X/S}=\Ocal_{S}$. Also, in the case some of the vector bundles have rank zero, by convention, we remove the corresponding $\cfrak_0$ terms from the monomial and define the monomial to be zero if there are any corresponding terms with $\cfrak_i$ for $i > 0$.

 Suppose now that $E_{1},\ldots, E_{m}$ have possibly non-constant ranks. We generalize the construction of $\langle \cfrak_{k_{1}}(E_{1})\cdot\ldots \cdot \cfrak_{k_{m}}(E_{m})\rangle_{X/S}$ in a way compatible with the decomposition property of the type \eqref{eq:decomposition-Deligne-pairing-partition}.  For simplicity, we assume first that there is a finite open partition $\Ucal=\lbrace X_{i}\rbrace_{i\in I}$ such that the vector bundles have constant ranks on the $X_{i}$. For example, if $S$, hence $X$ is quasi-compact. Using that $f\colon X\to S$ satisfies the condition $(C_{n})$, after possibly refining $\Ucal$, we can find a finite open partition $\lbrace S_{j}\rbrace_{j\in J}$ of $S$ with the following property: the restriction of $f$ to $X_{i}$ induces a morphism $X_{i}\to S_{j(i)}$, also satisfying condition $(C_{n})$. We define $\langle \cfrak_{k_{1}}(E_{1})\cdot\ldots \cdot \cfrak_{k_{m}}(E_{m})\rangle_{X/S}$ as the line bundle on $S$ whose restriction to $S_{j}$ is
\begin{equation}\label{eq:def-intersection-bundle-open-partitions}
    \bigotimes_{i\in I\ \mid\ j(i)=j}\langle \cfrak_{k_{1}}(E_{1})\cdot\ldots \cdot \cfrak_{k_{m}}(E_{m})\rangle_{X_{i}/S_{j}}.
\end{equation}
Notice that the result does not depend on the open partition $\Ucal$, up to canonical isomorphism, and that the construction is local over $S$. This allows us to glue along a general base scheme $S$ since the affine open schemes are quasi-compact and constitute a basis of Zariski open subsets.

The first Chern class of $\langle \cfrak_{k_{1}}(E_{1})\cdot\ldots \cdot \cfrak_{k_{m}}(E_{m})\rangle_{X/S}$ can be computed via Segre-type intersection bundles. It follows from Corollary \ref{cor:deligne-pairing-direct-image} that we have:
\begin{proposition}\label{prop:chernclassofIntbundle}
Let $f\colon X \to S$ be a morphism satisfying the condition $(C_n)$, with $S$ divisorial, and $E_1, \ldots, E_m$  vector bundles on $X$. Then 
    \begin{equation}\label{eq:c1-intersection-bundles}
        c_1 (\langle \cfrak_{k_{1}}(E_{1})\cdot\ldots \cdot \cfrak_{k_{m}}(E_{m})\rangle_{X/S}) = f_\ast (c_{k_{1}}(E_{1})\cdot\ldots \cdot c_{k_{m}}(E_{m}) ).
    \end{equation}
\end{proposition}
\begin{proof}
Since $S$ is quasi-compact by assumption, the scheme $X$ is quasi-compact too. Then, associated with the vector bundles $E_{1},\ldots, E_{m}$, there are finite open partitions $\lbrace X_{i}\rbrace$ and $\lbrace S_{j}\rbrace$ of $X$ and $S$, as in the construction above \eqref{eq:def-intersection-bundle-open-partitions}. For these finite partitions, we have $K_{0}(X)=\prod_{i}K_{0}(X_{i})$ and $\Pic(S)=\prod_{j}\Pic(S_{j})$, and we can hence reduce to the case of vector bundles of constant rank. It is then enough to prove the counterpart of \eqref{eq:c1-intersection-bundles} for the Segre-type intersection bundles \eqref{eq:definition-Segre-classes}. In this case, the result then follows from an application of Corollary \ref{cor:deligne-pairing-direct-image}.
\end{proof}

We gather in a lemma some elementary facts which follow easily from the construction. 

\begin{lemma}\label{lemma:elementary-properties-segre}
Let $E_{1},\ldots, E_{m}$ be vector bundles on $X$, and $k_{1},\ldots, k_{m}\geq 0$ integers. 
\begin{enumerate}
    \item\label{item:elpropsef-0} The construction $\langle \cfrak_{k_{1}}(E_{1})\cdot\ldots \cdot \cfrak_{k_{m}}(E_{m})\rangle_{X/S}$ is functorial. In particular, if for some $i$, $\alpha: E_i \to E_i'$ is an isomorphism of vector bundles, we denote by $[\alpha]$ the induced isomorphism of the intersection bundles.
\end{enumerate}
Furthermore, there are the following functorial constructions:
\begin{enumerate}
     \setcounter{enumi}{1}
    \item\label{item:elpropseg-1} If $m=n+1$ and the $E_{i}=L_{i}$ are line bundles, then there is a canonical isomorphism
    \begin{displaymath}
        \langle \cfrak_{1}(L_1)\cdot \ldots\cdot \cfrak_{1}(L_{n+1}) \rangle_{X/S}\simeq 
        \langle L_{1},\ldots, L_{n+1}\rangle_{X/S}.
    \end{displaymath}
    \item\label{item:elpropseg-2} Assume that $k_{i}=1$ and $\rk E_{i}=1$ for a subset $I\subseteq\lbrace 1,\ldots, m\rbrace$. Then $\langle \cfrak_{k_1}(E_1)\cdot \ldots\cdot \cfrak_{k_{m}}(E_m) \rangle_{X/S}$ is multilinear with respect to the tensor product in the $E_{i}$, for $i\in I$. In particular, if one of the terms equals $\cfrak_1(\Ocal_X)$, the corresponding intersection bundle is canonically trivial. 
    \item\label{item:elpropseg-3} If $k_{i}=0$, then there is a canonical isomorphism
    \begin{displaymath}
        \langle \cfrak_{k_1}(E_1)\cdot \ldots\cdot \cfrak_{k_m}(E_m) \rangle_{X/S}\simeq
        \langle\cfrak_{k_{1}}(E_{1})\cdot\ldots \cdot\cfrak_{k_{i-1}}(E_{i-1})\cdot\cfrak_{k_{i+1}}(E_{i+1})\cdot\ldots\cdot\cfrak_{k_{m}}(E_m)\rangle_{X/S},
    \end{displaymath}
    which is compatible with the multilinearity property of \eqref{item:elpropseg-2}, when it applies. 
\end{enumerate}
\end{lemma}
\begin{proof}
It is enough to prove the corresponding statements for Segre-type intersection bundles. By construction, the first property reduces to the analogous fact for Deligne pairings of line bundles, which is known. The next two properties are also a direct consequence of the definition of Segre classes as Deligne pairings \eqref{eq:definition-Segre-classes}, and the fact that if $L$ is a line bundle on $X$, then via the canonical isomorphism $\PBbb(L)\simeq X$ the line bundle $\Ocal(1)$ is identified with $L$. The last property follows from the definition and the second projection formula in Proposition \ref{prop:projformuladeli} since $\Ocal_{i}(1)$ has degree one on the fibers of $\PBbb(E_{i})\to X$. 
\end{proof}

\begin{definition}
Let $P$ be an object of $\CHfrak_{+}(X)$, and consider the decomposition of its component of degree $n+1$ in terms of Chern monomials, as $P^{(n+1)}=\sum_{i=1}^{N}P_{i}^{(n+1)}$. Then, we define
\begin{displaymath}
    \langle P\rangle_{X/S}=\langle P^{(n+1)}\rangle_{X/S}=\bigotimes_{i=1}^{N}\langle P_{i}^{(n+1)}\rangle_{X/S}.
\end{displaymath}
\end{definition}
This construction is trivially multilinear with respect to the addition in $\CHfrak_{+}(X)$, and compatible with the strictly commutativity laws in $\CHfrak_{+}(X)$ and $\Picfr(S)$. In particular, we can safely write expressions such as
\begin{displaymath}
    \langle\ (\sum_{i\in I} P_{i})\cdot (\sum_{j\in J}Q_{j})\ \rangle_{X/S}=\bigotimes_{i,j}\langle P_{i}\cdot Q_{j}\rangle_{X/S},
\end{displaymath}
without caring about the order.

\begin{remark}\label{rmk:taut-proj-form}
Suppose that $E$ is a vector bundle on $X$ of constant rank $r$, and $P$ is a positive Chern power series on $X$. Let $h\colon\PBbb(E)\to X$ be the projection map, and $\Ocal(1)$ be the tautological line bundle on $\PBbb(E)$. It follows from the definition that we tautologically have the following projection formula:
\begin{displaymath}
   \langle P \cdot \sfrak_k(E) \rangle_{X/S} = \langle  h^{\ast}P \cdot \cfrak_1(\Ocal(1))^{r+k-1} \rangle_{\PBbb(E)/S}.
\end{displaymath}
This formula is consistent with the definition of the Segre classes in classical Chow theory, as $s_k(E) = h_\ast (c_1(\Ocal(1))^{r+k-1})$.
\end{remark}

\subsection{Properties of intersection bundles}\label{subsec:properties-of-intersection-bundles} We establish several natural isomorphisms between intersection bundles, which lift classical identities in intersection theory to the level of line bundles. For the statement, we recall that for a vector bundle $E$ on $X$, we denote by $\rk E$ the locally constant function defined by the rank of $E$. 

\begin{proposition} \label{prop:properties-chern-bundles}
Let $f\colon X\to S$ be a morphism satisfying the condition $(C_{n})$, and let $P$, $P'$ denote positive Chern power series. The following constructions are functorial:
\begin{enumerate}

\item\label{item:projection-fromula} (Projection formulas) Let $g\colon X^{\prime}\to S$ and $h \colon X' \to X$ be morphisms satisfying the conditions $(C_{n+n'})$ and $(C_{n'})$, respectively, and $P$, $P'$ have pure degrees with $\deg P + \deg P' = n'+n+1$. Then there are canonical projection formula isomorphisms
        \begin{displaymath}
            \langle h^\ast P \cdot P' \rangle_{X'/S} \simeq 
\begin{cases}
\langle P\cdot \langle P' \rangle_{X'/X} \rangle_{X/S}, & \text{if }\; \deg P^{\prime}=n^{\prime}+1.\\ \\
\langle P \rangle_{X/S}^{\kappa}, &  \text{if }\; \deg P^{\prime} = n^{\prime}. \\
        &\\ 
\Ocal_S, & \text{if }\; \deg P^{\prime} < n^{\prime}. \\
\end{cases}
\end{displaymath}
In the first case, we wrote $\langle P' \rangle_{X'/X}$ instead of $\cfrak_{1}(\langle P' \rangle_{X'/X})$ for simplicity. In the second case, $\kappa$ denotes the relative degree $\int_{X'/X} P'$, and is then assumed to be constant. In the third case, we interpret $\Ocal_S$ as a constant functor. These isomorphisms satisfy the analogs of Corollary \ref{cor:compprojformula}.

    \item \label{whitneyiso}  (Whitney isomorphism) Let $0 \to E^{\prime}\to E \to E^{\bis} \to 0$ be a short exact sequence of vector bundles on $X$. Then there is a canonical isomorphism
\begin{displaymath}
    \langle P \cdot \cfrak_{k}(E)\cdot P'\rangle_{X/S} \to \bigotimes_{i=0}^{k} \langle P \cdot \cfrak_{i}(E^{\prime}) \cdot \cfrak_{k-i}(E^{\bis})\cdot P' \rangle_{X/S}, 
\end{displaymath}
 in a way that is compatible with admissible filtrations .

 \item\label{item:first-chern-determinant} (First Chern class isomorphism) Let $E$ be a vector bundle on $X$. Then, there is a canonical isomorphism
 \begin{displaymath}
    \langle P\cdot \cfrak_{1}(E)\cdot P'\rangle_{X/S} \simeq \langle P\cdot \cfrak_{1}(\det E)\cdot P'\rangle_{X/S},
 \end{displaymath}
 that is compatible with the Whitney isomorphism.
    \item \label{item:Chern-rank}(Rank triviality) Let $E$ be a vector bundle on $X$ and $q$ an integer such that $q > \rk E$. Then the intersection bundle
    \begin{displaymath}
        \langle P \cdot \cfrak_q(E) \cdot P' \rangle_{X/S} 
    \end{displaymath}
    is canonically isomorphic to the constant functor $\Ocal_S$. Moreover, the isomorphism is compatible with the Whitney isomorphism.
    \item\label{item:restriction-isomorphism} (Restriction isomorphism) Let $\sigma$ be a regular section of a vector bundle $E$, of constant rank $r$, whose zero locus $Y$, possibly empty, is flat over $S$. Then there is a canonical isomorphism
    \begin{displaymath}
        \langle P \cdot \cfrak_{r}(E) \cdot P' \rangle_{X/S} \to \langle P|_Y \cdot P'|_Y \rangle_{Y/S},
    \end{displaymath}
    where we adopt the same convention as in Proposition \ref{prop:generalpropertiesdeligneproducts} \eqref{item:deligne-pairing-restriction} (cf. Remark \ref{rmk:clarification-restriction}). 
    \item \label{item:birational-invariance} (Birational invariance) Let $g: \widetilde{X} \to S$ be a morphism satisfying condition $(C_n)$. Assume that there exists a morphism $\pi\colon \widetilde{X} \to X$ and a quasi-compact open immersion $U \to X$, such that $\pi^{-1}(U) \to U$ is an isomorphism and $U$ is fiberwise dense in $X.$
    Then there is a canonical isomorphism 
    \begin{displaymath}
        \langle P \rangle_{X/S}\simeq \langle \pi^\ast P \rangle_{\widetilde{X}/S}.
    \end{displaymath}
    \end{enumerate}
\end{proposition}

\begin{proof}

The projection formula formally follows from the construction of Segre and Chern classes and Proposition \ref{prop:projformuladeli}, as in 
\cite[\textsection V.4.2]{Elkikfib}. Notice that in \emph{loc. cit.}, $X$ is only assumed to be connected in the general case, for simplicity. The analog of Corollary \ref{cor:compprojformula} similarly follows.

The Whitney isomorphism is constructed similarly as in \cite[\textsection V.4.8]{Elkikfib}. Before recalling it, we remark it is only defined therein whenever the vector bundles have positive ranks. In the case $\alpha: E^{\prime} \to E$ is an isomorphism, so that $E^{\bis}=0$, we define the Whitney isomorphism as $[\alpha]^{-1}$, where $[\alpha]$ is the isomorphism of intersection bundles induced by $\alpha$ as in Lemma \ref{lemma:elementary-properties-segre} \eqref{item:elpropsef-0}. Likewise, if $\alpha: E\to E^{\bis}$ is an isomorphism and $E^{\prime}=0$, we define the Whitney isomorphism as $[\alpha]$.

The construction in the general case reduces to a Whitney-type isomorphism for Segre classes. By the splitting principle in Proposition \ref{prop:split} and Proposition \ref{Prop:2ndsplitting}, it is enough to describe this in the case when $E^{\prime}=L$ is a line bundle.  Let $r$ be the rank of $E$. We may suppose that $r\geq 2$. Denote by $\pi\colon\PBbb(E)\to X$ the projection map, and $\Ocal(1)$ the tautological bundle on $\PBbb(E)$. We can write, successively replacing a single $\Ocal(1)$ by $\Ocal(1) \otimes \pi^{\ast} L^{\vee} \otimes \pi^\ast L $ and iterating the process $k$ times :
\begin{equation}\label{eq:SegreWhitney}
\begin{split}
    \langle P \cdot \sfrak_k(E) \cdot P' \rangle   =\, & \langle \pi^{\ast}P\cdot \cfrak_1(\Ocal(1))^{r+k-1} \cdot \pi^{\ast}P'  \rangle  \\
     =\, & \langle \pi^{\ast}P\cdot \cfrak_1(\Ocal(1) \otimes \pi^{\ast}L^\vee) \cdot \cfrak_1(\Ocal(1))^{r+k-2}\cdot \pi^{\ast}P'  \rangle \otimes \langle \pi^{\ast}P\cdot \cfrak_1(\pi^{\ast}L)\cdot\cfrak_1(\Ocal(1))^{r+k-2} \cdot \pi^{\ast}P'  \rangle \\
    =\, &\ldots\\
    =\, & \bigotimes_{i=0}^{k-1} \langle \pi^{\ast}P \cdot \cfrak_1(\Ocal(1) \otimes \pi^{\ast} L^\vee) \cdot \cfrak_1(\Ocal(1))^{r+k-2-i} \cdot \cfrak_1(\pi^\ast L)^{i}  \cdot \pi^{\ast}P'\rangle \\
    &\otimes \langle \pi^{\ast}P  \cdot \cfrak_1(\pi^\ast L)^{k} \cdot \cfrak_1(\Ocal(1))^{r-1} \cdot \pi^{\ast}P'\rangle\\
\end{split}
\end{equation}
By Remark \ref{rmk:taut-proj-form}, the last term is identified with $\langle P \cdot \sfrak_k(L) \cdot \sfrak_0(E^{\bis}) \cdot P \rangle.$ Also, the bundle $\Ocal(1) \otimes \pi^\ast L^\vee$ admits a natural section $\sigma$, obtained from combining  $L \to E$ and $\pi^\ast E \to \Ocal(1)$. Its zero scheme is the relative Cartier divisor $\PBbb(E^{\bis})$ of $\PBbb(E).$ Restricting the bundles along this section and applying Proposition \ref{prop:generalpropertiesdeligneproducts} \eqref{item:deligne-pairing-restriction}, provides the sought Whitney-type isomorphism, again by Remark \ref{rmk:taut-proj-form}. The compatibility with admissible filtrations is formal, as in \cite[\textsection V.2.3]{Elkikfib}.

The first Chern class isomorphism can be constructed by induction on the rank and the splitting principle. For rank 0, we impose that the isomorphism is the one respecting the trivializations of both sides, see Lemma \ref{lemma:elementary-properties-segre} \eqref{item:elpropseg-2}. The isomorphism in general is then the one that makes the following diagram commute, associated with an exact sequence $0 \to L \to E \to E^{\bis} \to 0$ with $L$ a line bundle:
\begin{equation}\label{eq:firstChernclassisodef}
    \xymatrix{\langle P\cdot \cfrak_1(E)\cdot P'\rangle \ar[r] \ar[d] & \langle P\cdot\cfrak_1(\det E)\cdot P'\rangle \ar[d] \\ 
    \langle P \cdot \cfrak_1(L) \cdot P'\rangle \otimes \langle P \cdot \cfrak_1(E^{\bis}) \cdot P' \rangle \ar[r] & \langle P \cdot \cfrak_1(L ) \cdot P' \rangle \otimes \langle P \cdot \cfrak_1(\det E^{\bis} ) \cdot P' \rangle .
    }
\end{equation}
Here the left vertical isomorphism is the Whitney isomorphism together with Lemma \ref{lemma:elementary-properties-segre} \eqref{item:elpropseg-3}, the right vertical one follows from $\det E\simeq L\otimes\det E^{\bis}$ and the  multilinearity for $\cfrak_{1}$ (cf. Lemma \ref{lemma:elementary-properties-segre} \eqref{item:elpropseg-2}), and the lower horizontal one is constructed by induction. The compatibility with the Whitney isomorphism is automatic, by the splitting principle and the construction. 

The rank triviality is addressed as follows. By the construction of intersection bundles, we can suppose that $E$ has constant non-zero rank. We begin with the case when $E$ is a line bundle $L$. Observe that in terms of intersection bundles since $\Ocal(1)$ identifies with $L$ under the isomorphism  $\PBbb(L)\simeq X$, we have 
\begin{displaymath}
    \langle P \cdot \sfrak_q(L) \cdot P'  \rangle = \langle P \cdot \sfrak_1(L)^q \cdot P'  \rangle.
\end{displaymath}
From the defining recurrence \eqref{eq:definition-chern-classes}, we deduce that the rank triviality isomorphism is tautological for $q=2$. In general, proceeding by induction on $q$, we find
\begin{equation}\label{eq:Segre-manipulation}
    \langle P \cdot \cfrak_{q}(L) \cdot P' \rangle  \simeq \langle P \cdot \sfrak_{q}(L) \cdot P' \rangle^{(-1)^{q+1}} \otimes \langle P \cdot \sfrak_{q-1}(L) \cdot \sfrak_1(L) \cdot P' \rangle^{(-1)^{q}} \simeq \Ocal_S.
\end{equation}
This identifies the leftmost bundle with the trivial line functor $\Ocal_S$. For the general case, one argues by applying the splitting principle in Proposition \ref{Prop:2ndsplitting} and the Whitney isomorphism above. This reduces us to the case of a product of Chern classes of line bundles with one or more indices strictly bigger than one. We can use the above isomorphism \eqref{eq:Segre-manipulation} for any of those Chern classes. It is formal to verify that the resulting trivialization does not depend on the order, by an argument similar to that after \eqref{eq:definition-chern-classes}. We notice that, in particular, the compatibility with the Whitney isomorphism is part of the construction of the rank triviality isomorphism. 

The restriction isomorphism can be constructed as in  \cite[\textsection V.4.9]{Elkikfib}.  Set $T$ to be the image of $Y$ in $S$. It is both open and closed since $Y\to S$ is proper, flat, and of finite presentation. Since the construction of intersection bundles is local with respect to $S$, and by the convention as in Remark \ref{rmk:clarification-restriction}, we can suppose that either $S = T$ or that $Y$ is empty. 

We first suppose that $S=T$. We observe that, since $\sigma$ is a regular section whose zero locus is faithfully flat over $S$, then $Y$ automatically satisfies the condition $(C_{n-r})$ over $S$. Indeed, to verify that the morphism $Y \to S$ has constant pure relative dimension $n-r$, one can argue locally on $X$, in which case $Y \subseteq X$ is obtained by successively intersecting with relative Cartier divisors which are of pure relative dimension $n-1$, as in the remarks following Proposition \ref{prop:generalpropertiesdeligneproducts}. 

We follow \cite[\textsection V.4.10]{Elkikfib}. The case $r=1$ relies on \eqref{item:deligne-pairing-restriction} in Proposition \ref{prop:generalpropertiesdeligneproducts}. For general $r$, we proceed by induction on the rank. Introduce $\pi\colon\PBbb(E)\to X$, and the universal exact sequence \eqref{eq:tautological-exact-sequence}.

The section $\pi^{\ast}\sigma$ of $\pi^{\ast}E$ induces a regular section of $\Ocal(1)$ on $\PBbb(E)$. We denote by $D\subseteq\PBbb(E)$ its divisor. It is a relative Cartier divisor. The restriction of $\pi^{\ast}\sigma$ on $D$ defines a regular section of $Q|_{D}$. Its zero locus scheme $Z$ equals the restriction of the projective bundle to $Y$, that is $Z=\PBbb(E|_Y)$.  Then there is a string of isomorphisms:
\begin{equation}\label{eq:definitionrestrictionisomorphism}
    \begin{split}
        \langle P \cdot \cfrak_r(E) \cdot P' \rangle_{X/S} & \simeq \langle \pi^\ast P \cdot \cfrak_r(\pi^\ast E) \cdot \pi^\ast P' \cdot \cfrak_{1}(\Ocal(1))^{r-1} \rangle_{\PBbb(E)/S} \\
        & \simeq \langle \pi^\ast P \cdot \cfrak_{r-1}(Q)\cdot \cfrak_{1}(\Ocal(1)) \cdot \pi^\ast P' \cdot \cfrak_{1}(\Ocal(1))^{r-1} \rangle_{\PBbb(E)/S} \\
        & \simeq \langle \pi^{\ast}P|_{D} \cdot \cfrak_{r-1}(Q|_D) \cdot \pi^\ast (P'|_D )\cdot \cfrak_{1}(\Ocal(1))^{r-1} \rangle_{D/S} \\
        & \simeq \langle \pi^{\ast}P|_{Z} \cdot \pi^\ast (P'|_Z )\cdot \cfrak_{1}(\Ocal(1))^{r-1} \rangle_{Z/S} \\
        & \simeq \langle P|_{Y} \cdot P'|_Y \rangle_{Y/S}.
    \end{split}
\end{equation}
The first isomorphism in \eqref{eq:definitionrestrictionisomorphism} is the second case of the projection formula. The second one is the Whitney isomorphism associated with \eqref{eq:tautological-exact-sequence} combined with the rank triviality. The third one uses the case $r=1$, and the fourth isomorphism is constructed by induction. The final one performs the projection formula again.

The argument for the case $Y=\emptyset$ goes along the same lines. The case of rank one again reduces to Proposition \ref{prop:generalpropertiesdeligneproducts}. See Remark \ref{rmk:clarification-restriction}. For the higher rank case, the same manipulation as in \eqref{eq:definitionrestrictionisomorphism} applies, but here instead the scheme $Z$ will be empty, allowing us to apply induction again. 

Finally, we address the birational invariance. We can suppose $P$ is a single monomial in the Segre classes of some vector bundles. If $p: \mathbf{P} \to X$ denotes the projective bundle over $X$ that appears in the definition of the Segre classes in $P$, then $\mathbf{P} \times_X \widetilde{X} \to \widetilde{X}$ is the corresponding projective bundle in the definition of the Segre classes of $\pi^\ast P$. The open subset $p^{-1}(U)$ satisfies the assumptions of Proposition \ref{prop:elvira} and the property easily follows from this. 
\end{proof}
\begin{remark}\label{rem:restriction-iso-chern-class-degree-1}
In the proof of Proposition \ref{prop:properties-chern-bundles} \eqref{item:restriction-isomorphism}, we could introduce a term $\cfrak_{1}(\Ocal(1))^{r-1}$ by means of the projection formula, for instance in the first line of \eqref{eq:definitionrestrictionisomorphism}. More generally, we could have chosen any other Chern class of degree 1 on fibers, with the same result. See \cite[\textsection V.4.10]{Elkikfib}. This can be deduced from the very construction of the second projection formula for Deligne pairings, see in particular \eqref{eq:explicit-second-projection-formula}.
\end{remark}

\subsection{Compatibility}\label{subsec:compatibility}
In this subsection, we address the compatibility between various properties of the intersection bundles.

\subsubsection{The basic compatibilities}\label{subsubsec:basic-compatibilites}
The operations in Proposition \ref{prop:properties-chern-bundles} can be combined, and the question arises whether these can be effected in any order with the same result. We show that this is indeed the case. For this, we begin by stating some simple compatibilities, whose proof is a direct application of the definitions and is omitted.

\begin{proposition}\label{prop:stupid-compatibilities}
Let the notation and terminology be as in Proposition \ref{prop:properties-chern-bundles}. The following statements hold:
\begin{enumerate}
     \item If $P$ or $P^{\prime}$ contain a term of the form $\cfrak_{1}(L)$ for a line bundle $L$, then all the isomorphisms are compatible with the tensor product bilinearity in $L$ (cf. Lemma \ref{lemma:elementary-properties-segre} \eqref{item:elpropseg-2}).
    \item Performed on different variables, the Whitney isomorphism \eqref{whitneyiso}, first Chern class isomorphism \eqref{item:first-chern-determinant}, and rank triviality \eqref{item:Chern-rank}, commute with each other. Here, it is allowed to repeat the same type of operation on several variables. 
    \item\label{item:birational-comp} The birational invariance isomorphism \eqref{item:birational-invariance} is compatible with all the other properties. In the case of the restriction isomorphism \eqref{item:restriction-isomorphism}, it is necessary to assume that $Y \cap U$ is fiberwise dense in $Y$, and $\pi^{\ast}(\sigma)$ is a regular section whose zero-locus is flat over $S$.      
\end{enumerate}
\end{proposition}
\qed

\begin{proposition}\label{prop:further-compatibilites}
Let the notation and terminology be as in Proposition \ref{prop:properties-chern-bundles}. Let $E, F, F'$ be vector bundles on $X$, and let $P$, $P'$ denote positive Chern power series. We assume: 
\begin{itemize}
    \item that $E$ fits into an exact sequence $0\to E^{\prime}\to E\to E^{\bis}\to 0$.
    \item that $m$ is an integer such that $m > \rk F$. 
    \item that $F'$ is of constant rank $r$ and admits a regular section $\sigma$, whose zero-locus $Y$ is flat over $S$. 
\end{itemize}
The following statements hold: 
\begin{enumerate}
    \item\label{item:further-compatibilites-4} The rank triviality is compatible with the projection formula: if $P$ has pure degree $n-m$, then the diagram
        \begin{displaymath}
             \xymatrix{
                    \langle h^\ast (P \cdot \cfrak_m(F) ) \cdot  P' \rangle_{X^{\prime}/S}  \ar[r] \ar[d] & \Ocal_S \ar[d]^{\id} \\
                    \langle P \cdot \cfrak_m(F) \cdot  \langle P' \rangle_{X^{\prime}/X} \rangle_{X/S} \ar[r] & \Ocal_S
            }
        \end{displaymath}
        commutes, and similarly for the cases when $P$ has pure degree $\geq n-m+1$.  The analogous statement holds with respect to $P^{\prime}$. \bigskip

    \item\label{item:further-compatibilites-1} The Whitney isomorphism is compatible with the restriction isomorphism: the diagram 
    \begin{displaymath}
        \xymatrix{
            \langle P \cdot \cfrak_{r}(F^{\prime}) \cdot \cfrak_k(E) \cdot P' \rangle_{X/S} \ar[r] \ar[d] &  \langle P \cdot   \cfrak_k(E|_{Y}) \cdot P' \rangle_{Y/S} \ar[d] \\
            \bigotimes_{i=0}^k \langle P \cdot \cfrak_{r}(F^{\prime}) \cdot \cfrak_i(E^{\prime}) \cdot \cfrak_{k-i}(E^{\bis})  \cdot P' \rangle_{X/S} \ar[r] & \bigotimes_{i=0}^k \langle P  \cdot \cfrak_i(E^{\prime}|_{Y}) \cdot \cfrak_{k-i}(E^{\bis}|_{Y})  \cdot P' \rangle_{Y/S}
        }
    \end{displaymath}
    commutes.\bigskip
    \item\label{item:further-compatibilites-2} The Whitney isomorphism is compatible with the projection formula: if $P$ has pure degree $n-k$,  then the diagram
    \begin{displaymath}
        \xymatrix{
            \langle h^\ast (P \cdot \cfrak_k(E))  \cdot P'  \rangle_{X^{\prime}/S} \ar[r]\ar[d] & \langle  P \cdot \cfrak_k(E)  \cdot \langle P' \rangle \rangle_{X/S} \ar[d] \\
            \bigotimes_{i=0}^k \langle  h^{\ast}(P \cdot \cfrak_i(E^{\prime}) \cdot \cfrak_{{k-i}}(E^{\bis}))  \cdot  P'  \rangle_{X^{\prime}/S} \ar[r] & \bigotimes_{i=0}^k \langle  P \cdot \cfrak_i(E^{\prime}) \cdot \cfrak_{k-i}(E^{\bis})  \cdot \langle P' \rangle \rangle_{X/S}
        }
    \end{displaymath}
    commutes, and similarly for the cases when $P$ has pure degree $\geq n-k+1$. The analogous statement holds with respect $P^{\prime}$.\bigskip
    \item\label{item:further-compatibilites-3} The projection formula is compatible with the restriction isomorphism: if $P$ has pure degree $n-r$, then the diagram
    \begin{displaymath}
        \xymatrix{
            \langle h^\ast(P \cdot  \cfrak_r(F^{\prime})) \cdot P' \rangle_{X'/S} \ar[r] \ar[d] &  \langle h^\ast (P|_{Y}) \cdot P'|_{Y'} \rangle_{Y'/S} \ar[d] \\
            \langle P \cdot \cfrak_{r}(F^{\prime}) \cdot \langle  P'\rangle \rangle_{X/S} \ar[r] &  \langle P|_{Y}  \cdot \langle P'|_{Y'} \rangle \rangle_{Y/S}
        }
    \end{displaymath}
    commutes, and similarly when $P$ has pure degree $\geq n-r+1$. Here, $Y^{\prime}$ denotes $h^{-1}(Y)$ as a scheme. The analogous statement holds with respect to $P^{\prime}$, supposing instead that $F^{\prime}$ is a vector bundle on $X^{\prime}$ and the zero locus of the section is flat, possibly empty, 
    over $X$.\bigskip

    \item\label{item:further-compatibilites-5} The rank triviality is compatible with the restriction isomorphism: the diagram
            \begin{displaymath}
             \xymatrix{
                    \langle P \cdot \cfrak_r(F^{\prime}) \cdot \cfrak_m(F)  \cdot  P' \rangle_{X/S}  \ar[r] \ar[d] & \Ocal_S \ar[d]^{\id} \\
                    \langle P|_Y \cdot \cfrak_m(F|_Y) \cdot  {P'}|_Y\rangle_{Y/S} \ar[r] & \Ocal_S
            }
        \end{displaymath}
        commutes.\bigskip
\end{enumerate}
\end{proposition}
\begin{proof}
We assume that the ranks of all the vector bundles are constant. This is legitimate since the construction of intersection bundles ultimately reduces to this case. Since the principles are the same, we omit below the proofs referring to analogous statements with respect to $P^{\prime}$. 

Assuming  \eqref{item:further-compatibilites-2}, the verification of \eqref{item:further-compatibilites-4} follows formally from the construction of the rank triviality isomorphism and is left to the reader. 

By the splitting principle in Proposition \ref{Prop:2ndsplitting}, the statements involving the Whitney isomorphism reduce to the case of $E^{\prime}$ being a line bundle. To verify  \eqref{item:further-compatibilites-1}, we first treat the case when we restrict along an effective relative Cartier divisor, \emph{i.e.} when $F^{\prime}$ is a line bundle with a regular section. Also, it is enough to treat the case of Segre classes for $E$. The construction of the Whitney-type isomorphism in this case is based  on two points. Firstly, the definition of Segre classes as Deligne pairings of tautological bundles on projective spaces \eqref{eq:definition-Segre-classes}. Secondly, the restriction property \eqref{eq:restrictiondivisor} in Proposition \ref{prop:generalpropertiesdeligneproducts} \eqref{item:deligne-pairing-restriction}. The conclusion follows from the independence of the order of reiterated restrictions to relative Cartier divisors, stated in Proposition \ref{prop:generalpropertiesdeligneproducts} \eqref{item:order-deligne-pairing-isos}.

Before treating the general case of \eqref{item:further-compatibilites-1}, we consider the compatibility \eqref{item:further-compatibilites-2}. It reduces to studying the compatibility of the projection formula and the construction related to \eqref{eq:SegreWhitney}. Since the projection formula is additive in $P$ and  \eqref{eq:SegreWhitney} is a rewriting of terms, the statement can be inferred from the following two facts. The first fact is the compatibility \eqref{item:further-compatibilites-3} in the case of restrictions associated with line bundles, which is already contained in the proof of Proposition \ref{prop:projformuladeli}. The second fact is that the composition of two projection formulas is the projection formula for the composition, in the particular situation that we next discuss. If $V$ is a vector bundle of rank $d$ on $X$, and $h: X' \to X$ is a morphism as in the statement of the projection formula consider the diagram

\begin{displaymath}
    \xymatrix{ 
        \PBbb(h^{\ast}V) \ar[r]^{h'} \ar[d]_{\pi'} & \PBbb(V) \ar[d]^{\pi} \\ 
        X' \ar[r]^{h} \ar[dr] & X \ar[d] \\
        & S.
    }
\end{displaymath}
 Denote by $\xi=c_{1}(\Ocal(1))$ and $\xi^{\prime}=h^{\prime\ast}\xi$. Then the claim is that the following diagram commutes:
\begin{equation}\label{eq:diagram-projection-projection}
    \xymatrix{
         \langle  {h'}^\ast ({\pi}^\ast P \cdot {\xi}^{d-1} ) \cdot {\pi'}^\ast P'\rangle \ar[d] \ar[rr]^{\simeq } & & \langle {\pi'}^\ast h^\ast P \cdot {\pi'}^\ast P'\cdot {\xi'}^{d-1} \rangle  \ar[d] &  \\
        \langle \pi^\ast P \cdot \xi^{d-1} \cdot \langle {\pi'}^\ast P' \rangle  \rangle\ar[d]_{\simeq }   & & \langle h^\ast P \cdot P' \rangle  \ar[d] \\
         \langle \pi^\ast P \cdot \xi^{d-1} \cdot {\pi}^\ast \langle  P' \rangle \rangle \ar[rr]
        & &  \langle P \cdot \langle P' \rangle  \rangle.  & 
        }
\end{equation}
Here the upper horizontal isomorphism is a rewriting of the terms, using the natural transformation ${h'}^\ast {\pi}^\ast \simeq {\pi'}^\ast h^\ast $ and ${h'}^\ast \xi = {\xi}^{\prime}$, and in the lower left vertical isomorphism the base change functoriality of the intersection bundles was used to rewrite $\langle {\pi'}^\ast P'  \rangle \simeq \pi^\ast \langle P' \rangle$. The other isomorphisms are the natural projection formulas. 

By linearity, we can reduce the commutativity of \eqref{eq:diagram-projection-projection} to the case when $P$ and $P'$ are products of Segre classes. By construction of the Segre classes and the projection formulas, it is formal to reduce to the case when $P$ (resp. $P'$) is a product of first Chern classes of line bundles $L_i$ (resp.  $M_j$). Possibly restricting $S$, we can suppose that the $L_{i}$ are differences of relatively very ample line bundles. By the linearity of Deligne products, we can even suppose that the bundles $L_i$ are relatively very ample.  By the construction of the projection formula recalled in the proof of Proposition \ref{prop:projformuladeli}, the commutativity of \eqref{eq:diagram-projection-projection} is equivalent to the statement obtained after restricting along regular sections of the $L_i$, in general position, see \textsection \ref{subsubsec:generators-relations}. This way we are reduced to the case when $P=1$ and  $X \to S$ is of relative dimension 0. In this case, the upper left intersection bundle in \eqref{eq:diagram-projection-projection} admits defining symbols that realize trivializations of all the involved intersection bundles. An inspection of the construction recalled in the proof of Proposition \ref{prop:projformuladeli} shows that the diagram commutes. This proves \eqref{item:further-compatibilites-2} in general, and hence, as mentioned at the beginning of the proof, \eqref{item:further-compatibilites-4} as well. 

To deduce the general case of \eqref{item:further-compatibilites-1}, we first introduce some notation analogous to that of \eqref{eq:definitionrestrictionisomorphism}. Let $\pi\colon\PBbb(F^{\prime})\to X$ be the projection morphism, and consider the universal exact sequence
\begin{displaymath}
    0\to Q\to \pi^{\ast}F^{\prime}\to\Ocal(1)\to 0.
\end{displaymath}
The section $\pi^{\ast}\sigma$ of $\pi^{\ast}F^{\prime}$ induces a section of $\Ocal(1)$ on $\PBbb(F^{\prime})$. We denote by $D\subseteq\PBbb(F^{\prime})$ its divisor. The restriction of $\pi^{\ast}\sigma$ on $D$ defines a regular section of $Q|_{D}$. Its zero locus scheme $Z$ equals the restriction of the projective bundle to $Y$, that is $Z=\PBbb(F^{\prime}|_Y)$. Finally, we let $\xi=\cfrak_{1}(\Ocal(1))$. With this understood, we write down a sequence of diagrams, where we suppose for simplicity of notation that $P$ and $P'$ are 1:
\begin{equation}\label{eq:further-compatibilities-diagram-1}
    \xymatrix{
    \langle \cfrak_r(F^{\prime}) \cdot \cfrak_k(E)  \rangle \ar[r] \ar[d] \ar@{}[dr] |{(\Abold)} & \bigotimes_{i+j=k} \langle \cfrak_r(F^{\prime}) \cdot \cfrak_i(E^{\prime}) \cdot \cfrak_j(E^{\bis}) \rangle \ar[d] \\
    \langle \cfrak_r(\pi^\ast F^{\prime}) \cdot \cfrak_k(\pi^\ast E) \cdot \xi^{r-1} \rangle \ar[r] \ar[d]  \ar@{}[dr] |{(\Bbold)} & \bigotimes_{i+j=k} \langle \cfrak_r(\pi^\ast F^{\prime}) \cdot \cfrak_i(\pi^\ast E^{\prime}) \cdot \cfrak_j(\pi^\ast E^{\bis}) \cdot \xi^{r-1} \rangle \ar[d] \\
    \langle \cfrak_{r-1}(Q)  \cdot \xi \cdot \cfrak_k(\pi^\ast E) \cdot \xi^{r-1} \rangle \ar[r] \ar[d]  \ar@{}[dr] |{(\Cbold)}& \bigotimes_{i+j=k} \langle \cfrak_{r-1}(Q)  \cdot \xi\cdot \cfrak_i(\pi^\ast E^{\prime}) \cdot \cfrak_j(\pi^\ast E^{\bis}) \cdot \xi^{r-1} \rangle \ar[d] \\
    \langle \cfrak_{r-1}(Q|_D)  \cdot \cfrak_k(\pi^\ast E|_D) \cdot \xi^{r-1} \rangle \ar[r] \ar[d] \ar@{}[dr] |{(\Dbold)} & \bigotimes_{i+j=k} \langle \cfrak_{r-1}(Q|_D) \cdot \cfrak_i(\pi^\ast E^{\prime}|_D) \cdot \cfrak_j(\pi^\ast E^{\bis}|_D) \cdot \xi^{r-1} \rangle \ar[d] \\
    \langle \cfrak_k(\pi^\ast E|_Z) \cdot \xi^{r-1} \rangle \ar[r] \ar[d] \ar@{}[dr] |{(\Ebold)}& \bigotimes_{i+j=k} \langle \cfrak_i(\pi^\ast E^{\prime}|_Z) \cdot \cfrak_j(\pi^\ast E^{\bis}|_Z) \cdot \xi^{r-1} \rangle \ar[d]\\
    \langle \cfrak_k(E|_Y)  \rangle \ar[r]  & \bigotimes_{i+j=k} \langle \cfrak_i(E^{\prime}|_Y) \cdot \cfrak_j(E^{\bis}|_Y)  \rangle. 
    }
\end{equation}
The composition of the vertical isomorphisms denotes our restriction isomorphism in \eqref{eq:definitionrestrictionisomorphism}, and the commutativity of the outer contour of the diagram is hence \eqref{item:further-compatibilites-1}. The upper vertical and lower vertical arrows denote projection formula isomorphisms. The diagrams $(\Abold)$ and $(\Ebold)$ hence commute because of the already established \eqref{item:further-compatibilites-2}. The diagram $(\Bbold)$ commutes because rank triviality is compatible with the Whitney isomorphism by Proposition \ref{prop:properties-chern-bundles} \eqref{item:Chern-rank} and the iterated applications of Whitney isomorphisms in different entries clearly commute. The diagram $(\Cbold)$ commutes because of  \eqref{item:further-compatibilites-1} in the case of restrictions to divisors, and the diagram $(\Dbold)$ commutes by induction on the rank, which hence establishes \eqref{item:further-compatibilites-1} in general. 

The general case of  \eqref{item:further-compatibilites-3} is proven along the same lines, by rendering explicit the isomorphisms analogously to \eqref{eq:further-compatibilities-diagram-1} and relying on  the commutativity of \eqref{eq:diagram-projection-projection}, \eqref{item:further-compatibilites-1},  \eqref{item:further-compatibilites-2} and the already proven case of \eqref{item:further-compatibilites-3} when $r=1$. We leave the details to the reader.

\end{proof}

\begin{corollary}
Let the notation and terminology be as in Proposition \ref{prop:properties-chern-bundles}. The following statements hold:
\begin{enumerate}
    \item The first Chern class isomorphism is compatible with the projection formula with respect to $P$ or $P^{\prime}$. 
    \item Performed on different variables, the first Chern class isomorphism is compatible with the restriction isomorphism. 
\end{enumerate}
\end{corollary}
\begin{proof}
The proof of the first Chern class isomorphism proceeds inductively on the rank and involves the splitting principle and the Whitney isomorphism. Therefore, the statement readily follows from Proposition \ref{prop:stupid-compatibilities} and Proposition \ref{prop:further-compatibilites}, which in particular address the compatibility of the Whitney isomorphism with the other operations as in the statement of this corollary. 
\end{proof}

\begin{corollary}\label{Cor:RestrictionRestriction}
Let $E^{\prime}$ and $E^{\bis}$ be two vector bundles of constant ranks $r^{\prime} $ and $r^{\bis}$, with regular sections $\sigma^{\prime}$ and $\sigma^{\bis}$. We suppose that the respective zero loci, $Y^{\prime}$ and $Y^{\bis}$, are flat over $S$, possibly empty, and that their intersection is Tor-independent and flat over $S$, possibly empty. Then, for any objects $P,P^{\prime}$ in $\CHfrak_{+}(X)$, the following diagram of restriction isomorphisms commutes:

\begin{displaymath}
    \xymatrix{ 
    \langle P \cdot \cfrak_{r^{\prime}}(E^{\prime}) \cdot \cfrak_{r^{\bis}}(E^{\bis}) \cdot P^{\prime} \rangle \ar[r] \ar[d] & \langle P|_{Y^{\bis}} \cdot \cfrak_{r^{\prime}}(E^{\prime}|_{Y^{\bis}}) \cdot P^{\prime}|_{Y^{\bis}}\rangle \ar[d] \\
    \langle P|_{Y^{\prime}} \cdot \cfrak_{r^{\bis}}(E^{\bis}|_{Y^{\prime}}) \cdot P^{\prime}|_{Y^{\prime}}\rangle \ar[r] & \langle P|_{Y^{\prime} \cap Y^{\bis}} \cdot P^{\prime}|_{Y^{\prime}\cap Y^{\bis}}\rangle.
    }
\end{displaymath}
\end{corollary}
\begin{proof}
    
    We proceed by induction on the ranks of the vector bundles. If both bundles are line bundles, this reduces to the statement of Proposition \ref{prop:generalpropertiesdeligneproducts} \eqref{item:deligne-pairing-restriction}. 
    
    In general, we suppose the diagram commutes for vector bundles $E^{\prime}$ up to rank $r^{\prime}-1$ and fixed $E^{\bis}$. The construction of the restriction isomorphism for $E^{\prime}$ in \eqref{eq:definitionrestrictionisomorphism} is given in terms of a projection formula, Whitney isomorphisms and rank trivialities, and inductively over iterated restrictions for lower rank vector bundles. By the established compatibility of these operations, and the induction hypothesis, one readily concludes that it also holds for $E^{\prime}$. The analogous argument works for induction over the rank of $E^{\bis}$ of rank $r^{\prime}$.

\end{proof}

We conclude this subsection on basic compatibilities by proving that the isomorphisms constructed in Proposition \ref{prop:properties-chern-bundles} coincide with those of Elkik whenever the latter are defined. 

\begin{corollary}
Suppose that $f\colon X \to S$ has moreover Cohen--Macaulay fibers, with $S$ Noetherian. Then the constructions in Proposition \ref{prop:properties-chern-bundles} coincide with those of Elkik in \cite{Elkikfib}.
\end{corollary}

\begin{proof}
The projection formula in Proposition \ref{prop:properties-chern-bundles} is by construction the same as that in \cite{Elkikfib}. All the other properties coincide with those of \cite{Elkikfib} in the base case of line bundles. For the Whitney isomorphism, we interpret this to mean that $E'$ is a line bundle. 

 In our approach, we then extend the isomorphisms to general vector bundles by the splitting principle, which proceeds by giving a construction based on complete flags of all the involved vector bundles. Part of the splitting principle states that the isomorphism is independent of the choice of a complete flag. 

In \cite[\textsection V.2.3]{Elkikfib}, the following splitting principle is considered. Let $E$ be a vector bundle of constant rank $r$ on $X$, and $\pi\colon D \to X$ the complete flag variety of $E$. It is of relative dimension $1 + \ldots + r-1 = r(r-1)/2.$ Consider the tautological bundles $L_1, \ldots, L_{r}$ on $D$, and let $P = \cfrak_1(L_2)^{2-1}\cdot \ldots\cdot \cfrak_1(L_r)^{r-1}$. Then, the second projection formula in Proposition \ref{prop:properties-chern-bundles} \eqref{item:projection-fromula} yields a canonical isomorphism

\begin{displaymath}
        \langle Q\rangle_{X/S}  \to \langle\pi^{\ast} Q\cdot P\rangle_{D/S}
\end{displaymath}
for every $Q$ of degree $n+1$ on $X$. This way she can assume that $E$ admits a complete flag, and define her isomorphisms using this, or slight modifications of this construction. The compatibility between our Whitney isomorphism and hers hence follows from the line bundle case and the compatibility of our Whitney isomorphism and the projection formula elucidated in\linebreak Proposition \ref{prop:further-compatibilites} \eqref{item:further-compatibilites-2}. 

The same argument applies to the comparison of the rank triviality isomorphisms and the restriction isomorphisms since ours are compatible with the projection formula in\linebreak Proposition \ref{prop:further-compatibilites} \eqref{item:further-compatibilites-4} and \eqref{item:further-compatibilites-3}.

For the first Chern class isomorphism, we first remark that the construction in \cite[\textsection V.1.2 (e)]{Elkikfib} is compatible with the Whitney isomorphism in the same sense as in Proposition \ref{prop:properties-chern-bundles} \eqref{item:first-chern-determinant}. This amounts to the multiplicativity of the isomorphism $\langle \Ocal(1)\{\rk E\} \rangle_{\PBbb(E)/X} \simeq \det E$ in \cite[Proposition IV.3]{Elkikfib} under exact sequences $0 \to L \to E \to E'' \to 0,$ where $L$ is a line bundle. This is proven as in the discussion surrounding \eqref{eq:SegreWhitney}. By the splitting principle for multiplicative functors, the sought comparison reduces to the case when $E$ is of rank one. In this case, both constructions are equal to the identity and hence coincide. 

For the verification of \eqref{item:further-compatibilites-5}, we can reduce to the case when $F$ is a line bundle. For this, we combine the splitting principle, the Whitney isomorphism, and the fact that the latter is compatible with restrictions. In the case of line bundles, the statement is a formality.

\end{proof}

\subsubsection{Further compatibilities}

In the previous subsection, we studied the interaction between the possible combinations of the operations in Proposition \ref{prop:properties-chern-bundles}. As an application, in this subsection, we discuss other options which are less immediate. 
\begin{proposition}\label{prop:newranktriviality}
Let $X\to S$ be a morphism satisfying the condition $(C_{n})$, and $P,P^{\prime}$ objects in $\CHfrak_{+}(X)$. Let $E$ be a vector bundle of constant rank $r$ on $X$, admitting a nowhere vanishing section $\sigma\colon\Ocal_{X}\to E$. Consider the following isomorphisms, obtained by combining the Whitney isomorphism for the associated exact sequence $0\to\Ocal_{X}\to E\to E^{\bis}\to 0$, together with the rank triviality and the restriction isomorphism for $\Ocal_{X}$:
\begin{equation}\label{eq:Whitney-trivialization}
    \langle P\cdot\cfrak_{r}(E)\cdot P^{\prime}\rangle_{X/S}\to\langle P\cdot\cfrak_{1}(\Ocal_{X})\cdot\cfrak_{r-1}(E^{\bis})\cdot P^{\prime}\rangle_{X/S}\to\Ocal_{S}.
\end{equation}Then, the trivialization \eqref{eq:Whitney-trivialization} coincides with the trivialization provided by the restriction isomorphism in Proposition \ref{prop:properties-chern-bundles} \eqref{item:restriction-isomorphism} applied to the section $\sigma$.
\end{proposition}
\begin{proof}
We proceed by induction on the rank $r$. In rank one, $\sigma$ is a trivialization of $E$, and the claim is obvious. In general, since the Whitney isomorphism and the rank triviality are compatible with the projection formula, we can pull back all the objects by $\pi\colon \PBbb(E^{\bis})\to X$, at the expense of introducing a term $\xi^{r-2}=\cfrak_{1}(\Ocal(1))^{r-2}$. Consider the tautological exact sequence on $\PBbb(E^{\bis})$ of the form \eqref{eq:tautological-exact-sequence}, but decorated with a double prime symbol. We also consider the tautological exact sequence \eqref{eq:tautological-exact-sequence} on $\PBbb(E)$, and keep the same notation for the restriction to $\PBbb(E^{\bis}) \subseteq \PBbb(E)$ associated with $\sigma$. All these fit into a commutative diagram of exact sequences on $\PBbb(E^{\bis})$:
\begin{displaymath}
    \xymatrix{
        \Ocal \ar@{=}[r] \ar@{>->}[d] & \Ocal \ar@{>->}[d]  &  \\ 
        Q \ar@{>->}[r] \ar@{->>}[d] & \pi^{\ast } E  \ar@{->>}[d] \ar@{->>}[r] \ar[d] & \Ocal(1) \ar@{=}[d] \\
        Q^{\bis} \ar@{>->}[r] & \pi^{\ast} E^{\bis} \ar@{->>}[r] & \Ocal(1). 
    }
    \end{displaymath}
For simplicity, we omit in the rest of the proof the pullback $\pi^\ast$. Because the Whitney isomorphism is compatible with admissible filtrations, and the rank triviality is compatible with the Whitney isomorphism in separate variables, we have the following commutative diagram:
\begin{displaymath}
    \xymatrix{
        \langle P\cdot\cfrak_{r}(E)\cdot P^{\prime}\cdot\xi^{r-2}\rangle_{\PBbb(E^{\bis})/S}\ar[r]\ar[d]     &\langle P\cdot\cfrak_{r-1}(Q)\cdot \cfrak_{1}(\Ocal(1))\cdot P^{\prime}\cdot\xi^{r-2}\rangle_{\PBbb(E^{\bis})/S}\ar[d]\\
        \langle P\cdot \cfrak_{1}(\Ocal)\cdot\cfrak_{r-1}(E^{\bis})\cdot P^{\prime}\cdot\xi^{r-2}\rangle_{\PBbb(E^{\bis})/S}\ar[r]
        &\langle P\cdot\cfrak_{1}(\Ocal)\cdot\cfrak_{r-2}(Q^{\bis})\cdot\cfrak_{1}(\Ocal(1))\cdot P^{\prime}\cdot\xi^{r-2}\rangle_{\PBbb(E^{\bis})/S}.
    }
\end{displaymath}
In this diagram, the trivializations of the lower terms induced by the presence of $\cfrak_{1}(\Ocal)$ correspond to each other since the Whitney isomorphism is compatible with the restriction isomorphism. The rightmost upper term is also trivial by the restriction isomorphism applied to the section $\Ocal\to Q$. This trivializes, through the upper horizontal morphism, the leftmost upper term. We claim that this coincides with the trivialization provided by the restriction isomorphism for $E$ and $\sigma$. We have then reduced to the case of $\Ocal\to Q$ on $\PBbb(E^{\bis})$, for which the statement holds by induction.

Let us prove the claim.  Since the restriction isomorphism is compatible with the projection formula, the Whitney isomorphism, and the rank triviality, there is a commutative diagram
\begin{displaymath}
    \xymatrix{
        &\langle P\cdot\cfrak_{r}(E)\cdot P^{\prime}\cdot\xi^{r-1}\rangle_{\PBbb(E)/S}\ar[r]\ar[dd]   &\langle P\cdot\cfrak_{r-1}(Q)\cdot\cfrak_{1}(\Ocal(1))\cdot P^{\prime}\cdot\xi^{r-1}\rangle_{\PBbb(E)/S}\ar[dd]\\
      \langle P\cdot\cfrak_{r}(E)\cdot P^{\prime}\rangle_{X/S}\ar[ur]\ar[rd]  &   &\\
        &\langle P\cdot\cfrak_{r}(E)\cdot P^{\prime}\cdot\xi^{r-2}\rangle_{\PBbb(E^{\bis})/S}\ar[r]   &\langle P\cdot\cfrak_{r-1}(Q)\cdot\cfrak_{1}(\Ocal(1))\cdot P^{\prime}\cdot\xi^{r-2}\rangle_{\PBbb(E^{\bis})/S}\simeq\Ocal_{S}.
    }
\end{displaymath}
The diagonal arrows are given by projection formulas. The vertical arrows are associated with the restriction along the divisor of the section $\Ocal\to E\to\Ocal(1)$. The trivialization of the rightmost lower term is provided by the restriction isomorphism for $Q$ on $\PBbb(E^{\bis})$, which is given by induction on the rank. The restriction isomorphism for $E$ is obtained as the composition of the upper diagonal, upper horizontal, and rightmost vertical morphisms, together with the trivialization of the rightmost lower term. Since the diagram commutes, this proves the claim and concludes the proof.
\end{proof}

\begin{lemma}\label{lemma:ilfautpasdeconner}
 Suppose that $h \colon X^{\prime} \to X$ and $X\to S$ are as in Proposition \ref{prop:properties-chern-bundles} \eqref{item:projection-fromula}, and $F$ is a vector bundle on $X'$ of rank $r$, admitting a regular section $\sigma$ whose zero-locus defines a section of $X^{\prime} \to X$. Suppose that $r=n^{\prime}$, the relative dimension of $h$. Then, for any object $P$ in $\CHfrak_{+}(X)$, the projection formula isomorphism and the restriction isomorphism both define isomorphisms:
 \begin{displaymath}
     \langle h^{\ast }P \cdot \cfrak_r(F) \rangle_{X^{\prime}/S} \to \langle P   \rangle_{X/S}
 \end{displaymath}
These coincide.
\end{lemma}

\begin{proof}
This is an immediate application of Proposition \ref{prop:further-compatibilites} \eqref{item:further-compatibilites-3}. With the notation therein, it is indeed enough to use the compatibility of the second projection formula isomorphism and the restriction isomorphism applied to the term $P^{\prime}=\cfrak_{r}(F)$. For the projection formula, one observes that $c_{r}(F)$ has degree one along the fibers since the zero locus of $\sigma$ defines a section $X\to X^{\prime}$ of $X^{\prime}\to X$.
\end{proof}

\begin{proposition}\label{prop:sectionsandrestrictions}
    Let $0 \to E^{\prime} \to E \to E^{\bis} \to 0$ be a short exact sequence of vector bundles, of constant ranks $r^{\prime}, r, r^{\bis}$ respectively. Suppose that $\sigma$ is a regular section of $E$, with zero-locus $Z$, flat over $S$, satisfying: \begin{enumerate}
        \item The projection of $\sigma$ onto $E^{\bis}$ is a regular section $\sigma^{\bis}$ of $E^{\bis}$, with zero-locus $Y$ flat over $S$.
        \item The section $\sigma|_{Y}$ induces a regular section $\sigma^{\prime}$ of $E^{\prime}|_{Y}$. 
    \end{enumerate} 
    Then, for any positive Chern power series $P, P'$, the following diagram of restrictions, Whitney isomorphisms, and rank trivialities is commutative: 
    \begin{displaymath}
        \xymatrix{\langle P \cdot \cfrak_{r}(E) \cdot P^{\prime} \rangle_{X/S} \ar[d] \ar[r] & \langle P|_Z \cdot P^{\prime}|_Z \rangle_{Z/S}  \\ 
        \langle P \cdot \cfrak_{r^{\prime}}(E') \cdot \cfrak_{r^{\bis}}(E^{\bis}) \cdot P^{\prime} \rangle_{X/S} \ar[r] & \langle P|_Y \cdot \cfrak_{r^{\prime}}(E'|_Y) \cdot P^{\prime}|_Y \rangle_{Y/S} \ar[u] }
    \end{displaymath}
\end{proposition}
\begin{proof}
    For simplicity, we omit $P$ and $P^{\prime}$ from the notation, which are thus formally replaced by $1$. Also, in diagrams involving intersection bundles, we omit the pullback notation to projective bundles.

    We first suppose that $E^{\bis}=L$ is a line bundle. We begin with some geometric observations. Consider the projective bundle $\pi\colon\PBbb(E)\to X$ and the tautological exact sequence \eqref{eq:tautological-exact-sequence} on $\PBbb(E)$. By composition, the section $\sigma$ induces a section of $\Ocal(1)$, whose divisor we denote by $D$. The composition $\pi^{\ast} E^{\prime}\to \pi^{\ast} E\to\Ocal(1)$ induces a global section of $F:=\Ocal(1)\otimes (\pi^{\ast} E^{\prime})^{\vee}$, denoted by $\tau$. Its zero locus is isomorphic to a copy of $X$ embedded in $\PBbb(E)$ as a section of $\pi$. Furthermore, the restriction of the universal exact sequence \eqref{eq:tautological-exact-sequence} to this copy of $X$ gives back the sequence $0\to E^{\prime}\to E\to L\to 0$. Consider the following diagram: 
    \begin{equation}\label{eq:alzheimer-1}
        \xymatrix{ 
           \langle \cfrak_{r}(E) \rangle_{X/S} \ar[d] \ar[r] & \langle 1 \rangle_{Z/S}  \\ 
             \langle\cfrak_{r}(E)\cdot\cfrak_{r-1}(F)\rangle_{\PBbb(E)/S}\ar[r]\ar[d] &\langle \cfrak_{r-1}(F)\rangle_{\pi^{-1}(Z)/S}\ar[u]\\
            \langle\cfrak_{r-1}(Q)\cdot\cfrak_{1}(\Ocal(1))\cdot\cfrak_{r-1}(F)\rangle_{\PBbb(E)/S}\ar[r]\ar[d]  &\langle\cfrak_{r-1}(Q|_{D})\cdot\cfrak_{r-1}(F|_{D})\rangle_{D/S}\ar[u] \\
             \langle \cfrak_{r-1}(E^{\prime})\cdot \cfrak_1(L) \rangle_{X/S} \ar[r]    &\langle\cfrak_{r-1}(E^{\prime}|_{Y})\rangle_{Y/S} \ar[u].
        }
    \end{equation}
    The outer contour of this diagram corresponds to the diagram in the proposition, as we shall now explain. The two upper vertical arrows are projection formulas, which hold since the zero locus of $\tau$ induces a section $X\to\PBbb(F)$. By Lemma \ref{lemma:ilfautpasdeconner}, these are actually equal to the restriction isomorphisms for $F$ and $\tau$. The two upper horizontal arrows are given by the restriction isomorphisms for $E$ and $\sigma$. The upper square diagram thus commutes by Corollary \ref{Cor:RestrictionRestriction}. The middle square corresponds to the very construction of the restriction isomorphism for $E$ and $\sigma$, as described in \eqref{eq:definitionrestrictionisomorphism}. Notice that in \emph{loc. cit.} we used $\cfrak_{1}(\Ocal(1))^{r-1}$ instead of $\cfrak_{r-1}(F)$, but both choices yield the same result, after Remark \ref{rem:restriction-iso-chern-class-degree-1}. The lower diagram is a composition of restrictions, which again commute by Corollary \ref{Cor:RestrictionRestriction}. 

    The left horizontal composition of morphisms is equal to the Whitney isomorphism and rank triviality since these commute with restrictions. The composition of the rightmost vertical morphisms corresponds to the definition of the restriction isomorphism for $E^{\prime}|_Y$ and $\sigma^{\prime}$. This proves the proposition in the case of $E^{\bis}$ being a line bundle.

    We proceed by induction on the rank of $E^{\bis}$, assuming the statement is known up to rank $r^{\bis}-1$. We consider the tautological sequence \eqref{eq:tautological-exact-sequence} on $\PBbb(E) \to X$ and denote by the same letters the restriction to $\PBbb(E^{\bis})$. Denote by $Q^{\bis}$ the tautological subbundle on $\PBbb(E^{\bis}) \to X$. These are all related by the diagram of exact sequences and rows on $\PBbb(E^{\bis})$
    \begin{displaymath}
    \xymatrix{
        \pi^{\ast} E^{\prime} \ar@{=}[r] \ar@{>->}[d] & \pi^{\ast} E^{\prime} \ar@{>->}[d]  &  \\ 
        Q \ar@{>->}[r] \ar@{->>}[d] & \pi^{\ast } E  \ar@{->>}[d] \ar@{->>}[r] \ar[d] & \Ocal(1) \ar@{=}[d] \\
        Q^{\bis} \ar@{>->}[r] & \pi^{\ast} E^{\bis} \ar@{->>}[r] & \Ocal(1). 
    }
    \end{displaymath}
   
    There is an induced section of $\Ocal(1)$ on $\PBbb(E^{\bis})$, with divisor $D^{\bis}$. 
    
    From Proposition \ref{prop:further-compatibilites}, the projection formula commutes with the Whitney isomorphism, the rank triviality, and the restriction isomorphism. We can then consider the following composition of diagrams, where we write $\xi=\cfrak_{1}(\Ocal(1))$:
    \begin{displaymath}
        \xymatrix{ 
            \langle \cfrak_{r}(E) \cdot \xi^{r^{\bis}-1}\rangle_{\PBbb(E^{\bis})/S} \ar[r] \ar[d] & \langle \cfrak_{r-1}(Q)  \cdot \cfrak_{1}(\Ocal(1))\cdot\xi^{r^{\bis}-1}  \rangle_{\PBbb(E^{\bis})/S} \ar[d] \ar@{-}[r]  &\cdots &\hspace{2cm} \\ 
            \langle \cfrak_{r^{\prime}}(E^{\prime}) \cdot \cfrak_{r^{\bis}}(E^{\bis}) \cdot\xi^{r^{\bis}-1} \rangle_{\PBbb(E^{\bis})/S} \ar[r] & \langle \cfrak_{r^{\prime}}(E^{\prime}) \cdot \cfrak_{r^{\bis}-1}(Q^{\bis})  \cdot\cfrak_{1}(\Ocal(1))\cdot \xi^{r^{\bis}-1} \rangle_{\PBbb(E^{\bis})/S} \ar@{-}[r] & \cdots &\hspace{2cm}\\
        }
    \end{displaymath}
    \begin{displaymath}
        \xymatrix{
            \hspace{2.3cm} &\cdots\ar[r] & \langle \cfrak_{r-1}(Q|_{D^{\bis}})  \cdot \xi^{r^{\bis}-1}  \rangle_{D^{\bis}/S} \ar[r] \ar[d] & \langle \xi^{r^{\bis}-1} \rangle_{\pi^{-1}(Z)/S} \ar[d] \\ 
            \hspace{2.3cm} &\cdots\ar[r] &\langle \cfrak_{r^{\prime}}(E^{\prime}|_{D^{\bis}}) \cdot \cfrak_{r^{\bis}-1}(Q^{\bis}|_{D^{\bis}}) \cdot \xi^{r^{\bis}-1} \rangle_{D^{\bis}/S} \ar[r] & \langle \cfrak_{r^{\prime}}(E^{\prime}|_{Y}) \cdot \xi^{r^{\bis}-1} 
    \rangle_{\pi^{-1}(Y)/S}.
        }
    \end{displaymath}
    We explain the morphisms. The composition of the lower horizontal morphisms corresponds to the definition of the restriction morphism for $\sigma^{\bis}$. The rightmost vertical morphism identifies with the restriction morphism of $\sigma^{\prime}$ since this commutes with the projection formula. The composition of upper horizontal morphisms identifies, by the induction hypothesis and the projection formula, with the restriction isomorphism corresponding to $\sigma.$ The other isomorphisms are Whitney isomorphisms and rank triviality isomorphisms. Since the projection formula commutes with all these involved isomorphisms, the commutativity of the above diagram is equivalent to the commutativity of the one in the proposition.

    The two leftmost squares commute since the Whitney isomorphism commutes with restriction and rank triviality. The rightmost square commute by the induction hypothesis on the rank of $E^{\bis}$, which concludes the proof. 
\end{proof}

\section{A functorial framework for intersection bundles}\label{subsec:int-bundles-line-distributions}
In this section, we show that Elkik's intersection bundles provide line distributions. This relies on the compatibility properties established in the previous section, together with a study of the behavior of the Whitney isomorphism for split exact sequences, and the interaction with symmetries.

Below, $f\colon X\to S$ denotes a morphism of schemes satisfying the condition $(C_{n})$. 

\subsection{Symmetries and signs}

In the below, we study how some standard symmetry isomorphisms only produce signs. 

\begin{lemma}\label{lemma:signs} 
Let $E_{1},\ldots, E_{m}$ be vector bundles on $X$. 

\begin{enumerate}
    \item\label{item:symmetry-mult} For any permutation $\sigma$ of $\lbrace 1,\ldots, m\rbrace$, there is a canonical functorial symmetry isomorphism 
\begin{displaymath}
    [\sigma]\colon \langle \cfrak_{k_{1}}(E_{1})\cdot\ldots\cdot \cfrak_{k_{m}}(E_{m}) \rangle_{X/S} \simeq \langle \cfrak_{k_{\sigma(1)}}(E_{\sigma(1)})\cdot\ldots\cdot \cfrak_{k_{\sigma(m)}}(E_{\sigma(m)}) \rangle_{X/S},
\end{displaymath}
which is compatible with the basic properties of the intersection bundles (cf. Lemma \ref{lemma:elementary-properties-segre} and Proposition \ref{prop:properties-chern-bundles}). Moreover, given two permutations $\sigma$ and $\tau$, we have $[\sigma\circ\tau]=[\sigma]\circ[\tau]$. If $\sigma$ is a permutation fixing the Chern mononomial, then $[\sigma]$ is a locally constant unit, and only assumes the values $\pm 1$. 
     \item\label{item:symmetry-add} Suppose that $E_i = E_i^{\prime} \oplus E_i^{\bis}$ and consider the diagram:
    \begin{equation}\label{eq:symmetry-add-Whitney}
        \xymatrix{ 
            &\langle \ldots \cdot \cfrak_{k_{i}}(E_i^{\prime} \oplus E_i^{\bis})  \cdot \ldots  \rangle \ar[ldd]\ar[rdd]   \\
            &   &\\
             \bigotimes_{j=0}^{k_i} \langle \ldots \cdot \cfrak_{j}(E_i^{\prime}) \cdot \cfrak_{k_{i}-j}(E_i^{\bis})  \cdot \ldots \rangle \ar[rr] &     &  \bigotimes_{j=0}^{k_i} \langle \ldots \cdot \cfrak_{k_{i}-j}(E_i^{\bis}) \cdot \cfrak_{j}(E_i^{\prime})  \cdot \ldots \rangle.
        }
    \end{equation}
    Here, the diagonal morphisms are, from left to right, the two Whitney isomorphisms obtained by filtering by $E_{i}^{\prime}$ (on the left) and $E^{\bis}_{i}$ (on the right), and the lower horizontal isomorphism is a product of the symmetry isomorphisms interchanging $\cfrak_{k_{i}-j}(E_{i}^{\prime})$ with $\cfrak_{j}(E_{i}^{\bis})$ and a reordering of the tensor products. 

    Then, this diagram commutes up to a unit $\lambda$, which is locally constant and only assumes the values $\pm 1.$
\end{enumerate}

\end{lemma}

\begin{proof}
Throughout the proof, we suppose the ranks of the vector bundles are constant. 

For the first point, the existence of the symmetry isomorphism is formal and follows from the case of Deligne pairings as in  Proposition \ref{prop:generalpropertiesdeligneproducts} \eqref{item:symmetry-deligne}. The compatibility statement is formal too since all the basic properties ultimately reduce to basic properties of Deligne pairings, which are all compatible with the symmetry isomorphisms. For the statement about permutations fixing the Chern monomial, one reduces to the case of transpositions. If one transposes factors of the form $\cfrak_1(L_{i})$ where $L_{i}$ is a line bundle, it readily follows from Proposition \ref{prop:generalpropertiesdeligneproducts} \eqref{eq:delignepairingsymmetry} that $[\sigma]$ is a sign as in the statement. In the general case, since the flag variety of $E_i$ is faithfully flat over $S$, to determine $[\sigma]$ we can assume that the $E_i$ admits a full flag. By the splitting principle, we can suppose that $E_i$ is a sum of line bundles. By developing the expression using the Whitney isomorphism and rank triviality, one reduces to the line bundle situation. 

For the second point, by the construction of intersection bundles, we reduce to the case when the only terms beyond $\cfrak_{k_{i}}(E_i)$ in the upper expression are of the form $\cfrak_1(L_\ell)$, where $L_\ell$ is a line bundle. Arguing locally on $S$, we can suppose we can write, for any $\ell \neq i$, $L_\ell = A_\ell \otimes B_\ell^{\vee}$ where $A_\ell$ and $B_\ell$ are very ample relative to $S$, and admit non-trivial regular sections, defining effective relative Cartier divisors. Using the multilinearity for the first Chern class, it is enough to determine the unit $\lambda$ whenever $L_\ell = \Ocal(D_{\ell})$ for a non-trivial relative effective Cartier divisor $D_\ell$. Furthermore, applying the restriction isomorphism from Proposition \ref{prop:properties-chern-bundles} \eqref{item:restriction-isomorphism}, we can reduce to the case when $X\to S$ is of relative dimension $k_{i}-1$ and the Chern monomial in the upper vertice of the diagram is of the form $\langle \cfrak_{k_{i}}(E_i) \rangle.$ As in the previous point, we can as well suppose that all the involved vector bundles are sums of line bundles. By using that the Whitney isomorphism is compatible with admissible filtrations, one can finally reduce to the case that $E_{i}^{\prime}$ and $E_{i}^{\bis}$ are line bundles, say $L$ and $M$. Compare with \cite[Corollary 2.20]{Eriksson-Freixas-Wentworth}. At this point, we need to distinguish three cases, namely when $k_i = 1, 2$, and when $k_i>2$.

We first consider the case $k_i= 1$, so that in particular $f\colon X\to S$ is finite and flat, of degree $\deg f$. We will prove that the diagram commutes up to the sign $(-1)^{\deg f}$. By the very definition of the Chern classes, we have
\begin{displaymath}
    \langle\cfrak_{1}(L\oplus M)\rangle_{X/S}=\langle\Ocal(1),\Ocal(1)\rangle_{\PBbb(L\oplus M)/S}.
\end{displaymath}
The Whitney isomorphism corresponding to filtering by $L$ amounts, according to \eqref{eq:SegreWhitney}, to the following string of isomorphisms:
\begin{equation}\label{eq:whitney-isom-for-c1}
    \begin{split}
         \langle\Ocal(1),\Ocal(1)\rangle_{\PBbb(L\oplus M)/S}&\to \langle\Ocal(1)\otimes L^{\vee},\Ocal(1)\rangle_{\PBbb(L\oplus M)/S}\otimes  \langle L,\Ocal(1)\rangle_{\PBbb(L\oplus M)/S}\\
         &\to N_{D/S}(\Ocal(1)|_{D})\otimes N_{X/S}(L)\\
         &\to N_{X/S}(M)\otimes N_{X/S}(L).
    \end{split}
\end{equation}
We recall the meaning of the arrows. The first isomorphism is just a formal rewriting of terms. For the second isomorphism, we denoted by $D$ the divisor of the natural section $\Ocal\to\Ocal(1)\otimes L^{\vee}$, and we used the restriction isomorphism along it. We also used the projection formula isomorphism on the second factor. For the third isomorphism, we observe that $D$ provides a section of $\PBbb(L\oplus M)\to X$, and the restriction of $\Ocal(1)$ along this section is naturally isomorphic to $M$. 

We now write \eqref{eq:whitney-isom-for-c1} in terms of symbols. Possible localizing on $S$, we may assume that $L$ and $M$ admit trivializing sections, denoted by $\ell$ and $m$, respectively. We denote by $\widetilde{\ell}$ and $\widetilde{m}$ the sections of $\Ocal(1)$ induced by $\ell$ and $m$. Then, the symbol $\langle\widetilde{\ell},\widetilde{m}\rangle$ trivializes $\langle\Ocal(1),\Ocal(1)\rangle_{\PBbb(L\oplus M)/S}$. We observe that the section $\Ocal\to\Ocal(1)\otimes L^{\vee}$ is nothing but $\widetilde{\ell}\otimes\ell^{\vee}$, and one sees from the definitions that the restriction of $\widetilde{m}$ to $D$ gets identified to $m$. With this understood, the effect of \eqref{eq:whitney-isom-for-c1} on symbols is
\begin{equation}\label{eq:whitney-isom-for-c1-sections}
    \begin{split}
         \langle\widetilde{\ell},\widetilde{m}\rangle &\mapsto \langle\widetilde{\ell}\otimes\ell^{\vee},\widetilde{m}\rangle\otimes  \langle \ell,\widetilde{m}\rangle\\
         &\mapsto N_{D/S}(\widetilde{m}|_{D})\otimes N_{X/S}(\ell)\\
         &\mapsto N_{X/S}(m)\otimes N_{X/S}(\ell).
    \end{split}
\end{equation}
In the second isomorphism, we used the explicit expression of the projection formula in \eqref{eq:explicit-second-projection-formula}. 

Next, we need to repeat the above argument, but filtering by $M$. We now denote by $E$ the divisor of the section $\Ocal\to \Ocal(1)\otimes M^{\vee}$. The manipulation corresponding to \eqref{eq:whitney-isom-for-c1-sections} is, in this case,
\begin{equation}\label{eq:whitney-isom-for-c1-sections-2}
    \begin{split}
         \langle\widetilde{\ell},\widetilde{m}\rangle=(-1)^{\deg f}\langle\widetilde{m},\widetilde{\ell}\rangle &\mapsto (-1)^{\deg f}\langle\widetilde{m}\otimes m^{\vee},\widetilde{\ell}\rangle\otimes  \langle m,\widetilde{\ell}\rangle\\
         &\mapsto (-1)^{\deg f}N_{E/S}(\widetilde{\ell}|_{E})\otimes N_{X/S}(m)\\
         &\mapsto (-1)^{\deg f}N_{X/S}(\ell)\otimes N_{X/S}(m).
    \end{split}
\end{equation}
The equality in the first line is given by the property of permuting two equal factors in a Deligne pairing, in Proposition \ref{prop:generalpropertiesdeligneproducts} \eqref{item:symmetry-deligne}. This is responsible for the sign $(-1)^{\kappa}$, with 
\begin{displaymath}
    \kappa=\int_{\PBbb(L\oplus M)/S}c_{1}(\Ocal(1))=(\deg f)\int_{\PBbb(L\oplus M)/X}c_{1}(\Ocal(1))=\deg f.
\end{displaymath}
We conclude by comparing \eqref{eq:whitney-isom-for-c1-sections} and \eqref{eq:whitney-isom-for-c1-sections-2}.

The case $k_i = 2 $ is analogously considered in \cite[Proposition 3.9]{Eriksson-Freixas-Wentworth}, and shows that the corresponding diagram commutes up to a locally constant sign. Notice also that the definition of the Whitney isomorphism there is actually modified from that of Elkik by a locally constant sign. To this effect, see the end of the proof of Theorem 3.6 and Remark 3.7 (2) in \cite{Eriksson-Freixas-Wentworth}. In any case, the diagram commutes up to a locally constant sign, and we conclude in this case. 

In the case $k_i > 2$, all the functors are canonically trivial by rank triviality. The rank triviality is moreover constructed via the splitting principle, see in particular Remark \ref{rmk:indepedenceoffiltration}, to be compatible with the Whitney isomorphism for either filtrations $L \subseteq L \oplus M$ or $M \subseteq L \oplus M$, and hence coincide. One straightforwardly deduces that $\lambda = 1$ in this case.

\end{proof}

\subsection{Intersection bundles as line distributions}\label{subsec:Intbunaslindib}

The results of Section \ref{subsubsec:general-intersection-bundles}  together with the previous subsection imply the following, which constitutes the main theorem of this section: 

\begin{theorem}\label{thm:Elkik-distribution}
The intersection bundles 
\begin{displaymath}
    P \mapsto \langle P \rangle_{X/S}
\end{displaymath}  define line distributions of denominator 2 and degree $n+1$. 
\end{theorem}

\begin{proof}
First of all, by Proposition \ref{prop:line-distribution-affine}, we may suppose that all our base schemes are quasi-compact. For the proof in this case, recall the constructions in \textsection \ref{subsubsec:intermediate-category}. We first observe that the intersection bundles define functors $\mathfrak{C}^{n+1}(X) \to \Picfr(S)$, as formally follows from the established properties of intersection bundles. By MacLane's coherence theorem in $\Picfr(S)$, this functor is determined by its image on Chern monomials, and morphisms having a Chern monomial as their source. For the sake of clarity, we comment on the image of the symmetry isomorphisms under the functor. Given a Chern monomial $P$, all the symmetries in $\mathfrak{C}^{n+1}(X)$ with source $P$ are of the form $P\to Q$, where $Q$ is obtained by reordering the factors in $P$. We realize it by applying a permutation of the indices, $\sigma$, appearing in $P$, which in general is not unique. Then we send the morphism $P\to Q$ in $\mathfrak{C}^{n+1}(X)$ to the isomorphism of intersection bundles $[\sigma]\colon\langle P\rangle\to\langle Q\rangle$ as in Lemma \ref{lemma:signs} \eqref{item:symmetry-mult}. Notice that if $\sigma^{\prime}$ is any other such permutation, then it differs from $\sigma$ by a permutation fixing the Chern monomial $P$. By the previous Lemma \ref{lemma:signs} \eqref{item:symmetry-mult}, the latter induces at most a sign on the intersection bundle $\langle P\rangle$. 

To show the above functor passes to the quotient $\CHfrak_+^{n+1}(X)$, we need to check a number of compatibilities on morphisms. That is, if two morphisms in $\Cfrak^{n+1}(X)$ are identified in $\CHfrak_{+}^{n+1}(X)$, then the banana is round and they are sent to the same morphism in $\Picfr(S)$. For this, we first need to square the intersection bundles. This will have the effect of removing some sing ambiguities. 

We proceed to discuss the aforementioned compatibilities. We will focus on the non-trivial ones. First, if $\varphi,\psi\colon P\to Q$ are symmetries in $\Cfrak_+^{n+1}(X)$ between two Chern monomials, inducing the same morphism in $\CHfrak_{+}^{n+1}(X)$, then we can represent them by two permutations of the indices of $P$, say $\sigma$ and $\tau$. We can write $\tau=\sigma\circ\eta$, where $\eta$ fixes the Chern monomial $P$. By Lemma \ref{lemma:signs}, $[\tau]=[\sigma]\circ[\eta]$, and $[\eta]^{\otimes 2}=\id$. Hence, $[\tau]^{\otimes 2}=[\sigma]^{\otimes 2}$, so that $\varphi$ and $\psi$ have the same image in $\Picfr(S)$. Second, we check that for fixed Chern monomials $P$ and $Q$, the assignment $E\mapsto \langle P\cdot\cfrak_{k}(E)\cdot Q\rangle^{\otimes 2}$ satisfies the analogs of the axioms of commutative multiplicative functors in Definition \ref{def:multiplicative-functor}. All but the analog of Definition \ref{def:multiplicative-functor} \eqref{item:mult-funct-5} are established in Section \ref{subsubsec:general-intersection-bundles}. The remaining property \eqref{item:mult-funct-5} is derived from Lemma \ref{lemma:signs} \eqref{item:symmetry-add} since we are squaring the intersection bundles. 

By construction and by MacLane's coherence theorem for $\Picfr(S)$, the functor \linebreak  $\CHfrak_{+}^{n+1}(X)\to\Picfr(S)$ given by the square of the intersection bundles is a symmetric monoidal functor, where $\CHfrak_{+}^{n+1}(X)$ is considered with the addition. Hence, it induces a functor of commutative Picard categories $\CHfrak_{u}^{n+1}(X)\to\Picfr(S)$.

The above construction generalizes to provide functors $\CHfrak_{u}^{n+1}(\Ucal)\to\Picfr(S)$, for any open partition $\Ucal$ of $X$. Since $S$, and hence $X$, is supposed to be quasi-compact, $\Ucal$ is necessarily finite. For the construction of the sought functor, write $\Ucal=\lbrace X_{i}\rbrace$, and $S_{i}=f(X_{i})$. The $S_{i}$ constitute a finite open and closed cover of $S$, but not necessarily a partition. Nevertheless, from it we can canonically construct an open partition $\Scal = \lbrace S_{j}^{\prime}\rbrace_{j}$, and then replace $\Ucal$ by the common refinement $\Ucal'=\Ucal \cap f^{-1}(\Scal)$. This allows us to reduce to the case when the $S_{i}$ form an open partition. In this case, for an object $P=(P_{i})_i$ in $\CHfrak_{u}^{n+1}(\Ucal)=\prod_{i}\CHfrak_{u}^{n+1}(X_{i})$ we define a line bundle, locally over each $S_{i}$:
\begin{displaymath}
    \langle P \rangle_{X/S}\ |_{S_{i}} = \bigotimes_{f(X_{j}) = S_{i} } \langle P_{j}\rangle_{X_{j}/S_{i}}.
\end{displaymath}
The functor we seek is then given on objects by squaring. We declare morphisms to behave accordingly. This construction is compatible with base change and with refinements $\Vcal$ of $\Ucal$. By the universal property of $\CHfrak^{n+1}(X)$ (cf. Proposition \ref{prop:CHX-is-a-direct-limit}), this procedure defines an entire line distribution of degree $n+1$. It hence induces a line distribution of degree $n+1$ and denominator 2. 

\end{proof}

We next introduce some definitions and notations based on the line distribution provided by the intersection bundles. Before, the reader may review \textsection \ref{subsec:line-distributions} for the basic notions on line distributions, and in particular the product operation by Chern power series (cf. Definition \ref{def:prod-Chern-series-line-distribution}) and the direct images (cf. Definition \ref{def:direct-image-distributions}).

\begin{definition}\label{definition:intersectiondistribution}
Suppose that $X\to S$ satisfies the condition $(C_n)$. 

\begin{enumerate}
    \item The line distribution defined by the intersection bundles, referred to as the intersection distribution, is denoted by $\langle \bullet \rangle_{X/S}$. 
    \item If $P$ is a Chern power series, we denote by $[ P ]_{X/S}$, or simply $[P]$, the line distribution $P \cdot \langle \bullet \rangle_{X/S}$.
    \item If $i\colon Y \to X$ is a closed subscheme, such that $Y \to S$ satisfies the condition $(C_m)$, we denote by $\delta_{Y/S}$ the line distribution of degree $m+1$ on $X \to S$ given by $i_\ast[ 1 ]_{Y/S} $. 
    \item By convention, if $Y \to T$ satisfies the condition $(C_m)$ and $T \subseteq S$ is a closed and open immersion of schemes, we equally denote by $\delta_{Y/S}$ the line distribution which is $\delta_{Y/T}$ over $T$ and trivial over $S\setminus T.$
    \item When $S$ is implicit and there is no risk for confusion, we simply denote by $\delta_{Y}$ the line distribution $\delta_{Y/S}$. 
\end{enumerate}
\end{definition}

 For simplicity, we will call intersection distribution any line distribution built out from intersection distributions, such as $[P]_{X/S}$ or $\delta_{Y/S}$ above. 
 
 Notice that, with the notation above, if $d$ is a rational number, the line distribution $[d]_{X/S}$ is identified with $\langle\bullet\rangle_{X/S}^{\otimes d}$. In particular, $[0]_{X/S}$ is isomorphic to the trivial line distribution. Abusing notations, we will sometimes denote the distributions $[0]_{X/S}$ and $[1]_{X/S}$ by $0$ and $1$, respectively.

The intersection distribution satisfies some formal properties, which can be derived straightforwardly from the properties of intersection bundles. For instance, given line bundles $L$ and $M$ on $X$, the bilinearity of the Deligne products translates into an isomorphism of line distributions
\begin{equation}\label{eq:c1-tensor-product-distribution}
    [\cfrak_{1}(L\otimes M)]_{X/S}\simeq [\cfrak_{1}(L)]_{X/S}+[\cfrak_{1}(M)]_{X/S},
\end{equation}
where we recall that we use additive instead of tensor product notation for line distributions.

In the rest of this section, we record some relevant features of the intersection distributions, in analogy with the usual properties in intersection theory. The properties are rather direct applications of the formalism and the developed properties for intersection bundles. The first statement recapitulates, within the formalism of line distributions, the contents of \textsection \ref{subsec:compatibility}.

\begin{corollary}\label{cor:properties-intersection-distribution}
With the notation and assumptions of Proposition \ref{prop:properties-chern-bundles}, the intersection distributions satisfy the following properties:
\begin{enumerate}
    \item (Projection formulas) Let $h\colon X^{\prime}\to X$ be as in Proposition \ref{prop:properties-chern-bundles} \eqref{item:projection-fromula}, and take $P$ in $\CHfrak(X)_{\QBbb}$ and $P^{\prime}$ in $\CHfrak(X^{\prime})_{\QBbb}$. There are canonical isomorphisms of line distributions:
    \begin{displaymath}
            h_\ast [P']_{X^{\prime}/S} \simeq 
\begin{cases}
    \left[ \cfrak_1 (\langle P^{\prime} \rangle_{X^{\prime}/X})\right]_{X/S}, & \text{if }\; \deg P^{\prime}=n^{\prime}+1.\\ \\
    [ \kappa ]_{X/S}, &  \text{if }\; \deg P^{\prime} = n^{\prime}.  \\ \\
         0,  & \text{if }\; \deg P^{\prime} < n^{\prime}. \\
\end{cases}
\end{displaymath}
    \item\label{item:prop-int-dist-whitney} (Whitney isomorphism) Let $0 \to E^{\prime}\to E \to E^{\bis} \to 0$ be a short exact sequence of vector bundles on $X$. Then there is a canonical isomorphism
    \begin{displaymath}
        [\cfrak_{k}(E)]_{X/S}\simeq \sum_{i=0}^{k} [\cfrak_{i}(E^{\prime})\cdot\cfrak_{k-i}(E^{\bis})]_{X/S},
    \end{displaymath}
    in a way compatible with admissible filtrations.
    \item\label{item:prop-int-dist-c1-det} (First Chern class isomorphism) Let $E$ be a vector bundle on $X$. Then, there is a canonical isomorphism 
 \begin{displaymath}
        [\cfrak_1(E)]_{X/S} \simeq [\cfrak_1(\det E)]_{X/S}
    \end{displaymath}
    in a way that is compatible with the Whitney isomorphism. 
    \item\label{item:prop-int-dist-rank} (Rank triviality) Let $E$ be a vector bundle on $X$ and $q$ an integer such that $q > \rk E$. Then there is an isomorphism 
    \begin{displaymath}
        [ \cfrak_q(E)]_{X/S} \simeq 0.
        \end{displaymath}
    \item\label{item:prop-int-dist-restriction} (Restriction isomorphism) Let $E$ be a vector bundle of constant rank $r$ on $X$. Suppose that $\sigma$ is a regular section of $E$,  whose zero locus $Y$, possibly empty, is flat over $S$. Then, there is a canonical isomorphism
    \begin{displaymath}
        [ \cfrak_r(E) ]_{X/S} \simeq \delta_{Y/S}, 
    \end{displaymath}
    where $i\colon Y\hookrightarrow X$ is the closed immersion of $Y$ in $X$. 
    \item\label{item:prop-int-dist-birational} (Birational invariance) Suppose $h \colon X' \to X$  is as in Proposition \ref{prop:properties-chern-bundles} \eqref{item:birational-invariance}. Then, there is a canonical isomorphism
    \begin{displaymath}
        h_\ast \delta_{X'/S}\simeq \delta_{X/S}.
    \end{displaymath}
    In particular, $h_{\ast}[h^{\ast}P]_{X'/S}\simeq [P]_{X/S}$. 
\end{enumerate}
By the definition $P \cdot [Q]_{X/S} = [Q \cdot P]_{X/S}$ these operations can be composed with each other in a natural way. As such, all operations commute with each other. 

\end{corollary}

\begin{proof}
The first part is a reformulation of the formal projection formula for line distributions in \eqref{linedist:projformula}, combined with the projection formula in Proposition \ref{prop:properties-chern-bundles}. The last property similarly follows, by instead appealing to Proposition \ref{prop:stupid-compatibilities} \eqref{item:birational-comp}. The other properties are direct translations of the corresponding properties in Proposition \ref{prop:properties-chern-bundles}. The fact that these indeed induce isomorphisms of line distributions, which commute with each other, follows from the results in \textsection \ref{subsec:compatibility} and a reasoning similar to the proof of Theorem \ref{thm:Elkik-distribution}. 
\end{proof}

\begin{remark}
By construction, the isomorphism \eqref{eq:c1-tensor-product-distribution} can be obtained as a combination of the first Chern class isomorphism applied to $E=L\oplus M$, and the Whitney isomorphism for the obvious filtration $L\subseteq L\oplus M$. Consequently, the fact that the isomorphism \eqref{eq:c1-tensor-product-distribution} commutes with all the other stated isomorphisms, is already contained in the corollary. 
\end{remark}

\begin{corollary}
    Let $p\colon M \to \PBbb^1_S$ verify the property $(C_n)$ and consider the line distributions $\delta_{M_0/S}$ and $\delta_{M_{\infty}/S}$ of degree $n+1$ from $M$ to $S$. Then, there is a canonical isomorphism 
    \begin{displaymath}
        \delta_{M_0/S} \simeq \delta_{M_{\infty}/S}
    \end{displaymath}
     of line distributions of degree $n+1$ from $M$ to $S$, which is compatible with the constructions in Corollary \ref{cor:properties-intersection-distribution}.
\end{corollary}

 \begin{proof}
 On $\PBbb^{1}_{S}$, consider the standard isomorphisms of line bundles $\Ocal(\infty)\simeq\Ocal(0)$. Pulling these back to $M$, and using the isomorphism $[\cfrak_{1}(\Ocal(D))]\simeq\delta_{D}$ provided by Corollary \ref{cor:properties-intersection-distribution} \eqref{item:prop-int-dist-restriction}, we obtain
\begin{displaymath} 
    \delta_{M_\infty/S}\simeq [\cfrak_{1}(p^\ast \Ocal(\infty))]_{M/S} \simeq [\cfrak_{1}(p^\ast \Ocal(0))]_{M/S} \simeq \delta_{M_0/S}.
\end{displaymath}
By construction, it is compatible with the isomorphisms in Corollary \ref{cor:properties-intersection-distribution}.
 \end{proof}

\begin{corollary}
    Let $D = \sum D_i$ be a finite sum of relative effective Cartier divisors $D_i$ on $X$. Suppose moreover that whenever $i\neq j$, $D_i \cap D_j$ is a relative effective Cartier divisor in $D_j$. Then, there is a canonical isomorphism
    \begin{displaymath}
        \delta_{D}\simeq \sum \delta_{D_i},
    \end{displaymath}
    which is compatible with the constructions in Corollary \ref{cor:properties-intersection-distribution}. The hypothesis that all $D_i$ are relative effective Cartier divisors is, in particular, satisfied if $D$ is a relative effective Cartier divisor and all but one of the $D_i$ are. 
\end{corollary}
\begin{proof}
Notice that, by \cite[\href{https://stacks.math.columbia.edu/tag/0B8V}{0B8U}]{stacks-project}, the sum is automatically a relative effective Cartier divisor. The sought isomorphism of distributions follows from Corollary \ref{cor:properties-intersection-distribution} \eqref{item:prop-int-dist-birational} since $\bigsqcup D_i \to \sum D_i$ satisfies the assumptions of Proposition \ref{prop:properties-chern-bundles} \eqref{item:birational-invariance}. The compatibility claim is automatic by the construction and the compatibility claim in Corollary \ref{cor:properties-intersection-distribution}. The last point follows from  \cite[\href{https://stacks.math.columbia.edu/tag/0B8V}{0B8V}]{stacks-project}.
\end{proof}

We conclude this subsection with some further properties of the line distributions defined by Chern classes.

\begin{proposition}\label{prop:isom-first-chern-functorial}
The first Chern class satisfies the following properties.

\begin{enumerate}
    \item  Both $E\mapsto [\cfrak_{1}(E)]$ and $E\mapsto [\cfrak_{1}(\det E)]$ are commutative multiplicative functors from $(\Vect_{X},\iso)$ to the commutative Picard category $\Dcal(X/S).$ 
    \item The first Chern class isomorphism is a natural transformation of these multiplicative functors. 
    \item The functors in the first point and the first Chern class isomorphism extend to the virtual category of vector bundles. 
\end{enumerate}
\end{proposition}

\begin{proof}
The third point is a formal consequence of the first two and the definition of the virtual category. We address the first point. The functor $E\mapsto [\cfrak_{1}(E)]$ is commutative multiplicative, because it is the composition of $\cfrak_{1}$, which is a commutative multiplicative functor\linebreak  (cf. Remark \ref{rmk:superficial-remarks-Chern-functors} \eqref{rmk:supremChfun-2}), and the intersection distribution. 

In a similar vein, the functors $E \mapsto \det (E)$ and $L \mapsto [\cfrak_1(L)]$ are both multiplicative, see \eqref{eq:c1-tensor-product-distribution} for the second one. Actually, the (ungraded) determinant functor is only commutative up to sign, which is not seen on the level of line distributions, and hence the composition $[\cfrak_1(\det E)]$ is a commutative multiplicative functor. 

To complete the proof, we need to show that the natural transformation of the Chern class isomorphism preserve the multiplicative structures. Since the isomorphism is constructed by the splitting principle, this is immediate. 

\end{proof}

We next discuss the behavior of Chern classes with respect to the duality of vector bundles.

\begin{proposition} \label{prop:Cherndual}
    Let $E$ be a vector bundle on $X$. For any integer $k\geq 0$, there is a canonical isomorphism of line distributions, 
    \begin{displaymath}
        [\cfrak_k(E^\vee)]_{X/S} \simeq {(-1)^k} [ \cfrak_k(E) ]_{X/S},
    \end{displaymath}
    which is compatible with the constructions in Corollary \ref{cor:properties-intersection-distribution}. 
\end{proposition}

\begin{proof}
    By the same reductions as in the case of the other isomorphisms of line distributions, the statement amounts to constructing an isomorphism, for any Chern power series $P, P'$ in $\CHfrak_+(X)$,
    \begin{displaymath}
        \langle P \cdot \cfrak_k(E^\vee) \cdot P' \rangle_{X/S} \simeq \langle P \cdot \cfrak_k(E) \cdot P' \rangle_{X/S}^{(-1)^k},
    \end{displaymath} 
    compatible with the Whitney isomorphism. 
    
    We suppose for simplicity that $E$ has constant rank $r$. The general case readily reduces to this. If $k = 0$ both Chern classes appearing are trivial and the isomorphism is defined to be the one respecting this. If $k>r$, both sides are canonically trivial by rank triviality in\linebreak Proposition \ref{prop:properties-chern-bundles} \eqref{item:Chern-rank},  and one thus again defines the isomorphism to be the one respecting this. These are both compatible with the Whitney isomorphism as the rank triviality isomorphism also does. 
    
    In general, by the splitting principle, we reduce to the case when $E$ is a sum of line bundles, $E = L_1 \oplus \ldots \oplus L_r$. Then we have canonically $E^\vee = L_1^\vee \oplus \ldots \oplus L_r^\vee.$ Applying the Whitney isomorphism to the filtrations on $E$ (resp. $E^\vee$) starting with $L_1$ (resp. $L_1^\vee$), and again applying rank triviality, one reduces to the case of $E$ being a line bundle and $k = 1.$ In this case, it follows from\linebreak Lemma \ref{lemma:elementary-properties-segre} \eqref{item:elpropseg-2} which gives the multilinearity in $L$ of the term $\cfrak_1(L).$ 

    Since duals change the order in short exact sequences, the dual of the filtration starting with $L_1$ starts with $L_r^\vee$ and not with $L_1^\vee$. By Lemma \ref{lemma:signs} \eqref{item:symmetry-add} the difference between the Whitney isomorphisms applied to the two filtration orders is $\pm 1$, which is not seen on the level of intersection distributions. One concludes from the construction that the isomorphism is compatible with the Whitney isomorphism. 

    The compatibility with the isomorphisms in Corollary \ref{cor:properties-intersection-distribution} is automatic by construction since we relied on the Whitney isomorphism and the rank triviality, which are themselves compatible. 
\end{proof}

We conclude with a result on the behavior of Chern classes under the twist of a vector bundle by a line bundle.

\begin{proposition}\label{prop:Cherntensorproductwithline}
Let $E$ be a vector bundle on $X$ of constant rank $r\geq 1$, and $L$ a line bundle. Then there is a canonical isomorphism of line distributions
\begin{equation}\label{eq:ck-tensor-product}
    [\cfrak_k(E \otimes L)]_{X/S} \simeq \sum_{i=0}^{k}  \binom{r-k+1}{i} [\cfrak_{k-i}(E) \cdot \cfrak_1(L)^{i}]_{X/S},
\end{equation}
which is compatible with the construction in Corollary \ref{cor:properties-intersection-distribution}. For the Whitney isomorphism with respect to exact sequences in $E$, one needs to take into account Vandermonde's binomial identity
\begin{displaymath}
    \binom{n + m}{k}  = \sum_{i=0}^k  \binom{m}{i}\binom{n}{k-i}.
\end{displaymath}
\end{proposition}
\begin{proof}
First of all, if $k > r$, the right-hand side is identically zero, and we define the isomorphism to be the one compatible with the rank triviality in Corollary \ref{cor:properties-intersection-distribution} \eqref{item:prop-int-dist-rank}. For $k=0$ the two sides are identical. For $1\leq k\leq r$, we perform an induction on the rank of $E$. For $E$ of rank 1 and $k=1$, we define the isomorphism according to $[\cfrak_{1}(E\otimes L)]_{X/S}\simeq [\cfrak_{1}(E)]_{X/S} + [\cfrak_{1}(L)]_{X/S}$. For general rank $r$ and $k\leq r$, by the splitting principle in Corollary \ref{cor:splittingdistribution}, it is enough to prove the statement when $E$ is a sum of line bundles. Developing this by the Whitney isomorphism and reordering according to the binomial identity, and applying the induction hypothesis, one finds that the two sides are canonically isomorphic. The compatibility claim holds automatically by the very construction and Corollary \ref{cor:properties-intersection-distribution}. 
\end{proof}

\section{The Riemann--Roch distribution}\label{sec:RR-distribution}
The goal of this section is to construct and study the properties of the right-hand side of the Grothendieck--Riemann--Roch theorem in the context of intersection bundles and incarnate it in the form of a line distribution, which we refer to as the Riemann--Roch distribution. For this, suppose first that we are given a local complete intersection morphism $f \colon X \to Y$ of schemes satisfying the condition $(C_n)$ and $(C_m)$ over a base scheme $S$. In this setting, there is a cotangent complex\footnote{Referred to as the naive cotangent complex in \cite[\href{https://stacks.math.columbia.edu/tag/08P5}{08P5}]{stacks-project}.} $L_{X/Y}$, which locally with respect to $S$ is quasi-isomorphic to a complex of vector bundles of length two, see \cite[\href{https://stacks.math.columbia.edu/tag/08TJ}{08TJ}]{stacks-project} and  \cite[\href{https://stacks.math.columbia.edu/tag/0FV4}{0FV4}]{stacks-project}. The formalism of categorical characteristic classes and line distributions in Section \ref{subsection:Chern-categories}, and the results on intersection distributions in\linebreak Section \ref{subsec:int-bundles-line-distributions} (see in particular Definition \ref{def:ChernToddchar}, Proposition \ref{prop:line-distribution-affine}, and Theorem \ref{thm:Elkik-distribution}), allows us to associate with $X\to S$, and every vector bundle $E$ on $X$, a line distribution 
\begin{equation}\label{def:RRbundle}
    [\chfrak(E)\cdot \tdfrak^{*}(L_{X/Y})]_{X/S},
\end{equation}
where we recall that $\tdfrak^\ast$ is the multiplicative class obtained from $\tdfrak$ by changing the sign in each degree $k$ by $(-1)^{k}$. In fact, since $X$ is divisorial locally with respect to $S$, we can even suppose that $E$ is a perfect complex, by Lemma \ref{lemma:divisorial} and Proposition \ref{prop:line-distribution-affine}.

Taking direct images of the line distribution \eqref{def:RRbundle} one obtains a line distribution
\begin{equation}\label{def:RRdistribution}
    f_\ast [\chfrak(E)\cdot \tdfrak^{*}(L_{X/Y})]_{X/S}
\end{equation}
for $Y\to S.$ The construction is a multiplicative functor in the variable $E$ for short exact sequences $0 \to E' \to E \to E'' \to 0.$ It does not, however, follow directly from the definition that the categorical Chern character takes tensor products to products of categorical Chern characters since $\chfrak$ does not define a functor of ring categories, see Remark \ref{rmk:ch-E-otimes-F}.

Below, we establish these and other natural properties with respect to the Todd class for \eqref{def:RRbundle} and hence also \eqref{def:RRdistribution}.

\subsection{Chern character}
Let $f\colon X\to S$ be a morphism satisfying the property $(C_{n})$. In this subsection, we study in detail the line distributions induced by the Chern categorical classes. That is, we consider the assignment
\begin{displaymath}
    \begin{split}
            V(\Pcal_{X})&\longrightarrow \Dcal(X/S)\\
            E&\longmapsto [\chfrak(E)]_{X/S},
    \end{split}
\end{displaymath}
which is constructed as in the introduction of the current section. Because $\chfrak$ is an additive categorical characteristic class, this functor has a natural structure of functor of commutative Picard categories. Our aim is to describe the behavior of this functor with respect to the tensor product operation on $V(\Pcal_{X})$. 

\begin{proposition}\label{prop:ch-multiplicative-properties}
The functor $E\mapsto[\chfrak(E)]_{X/S}$ satisfies the following properties:
\begin{enumerate}
    \item\label{item:chernsequence-5} For any virtual perfect complexes $E, F$, there is a canonical isomorphism of line distributions
    \begin{equation}\label{eq:ch-otimes}
        [\chfrak(E \otimes F)]_{X/S} \simeq [\chfrak(E) \cdot \chfrak(F)]_{X/S}.
    \end{equation}
    Moreover:
    \begin{enumerate}
        \item\label{eq:ch-otimes-a} It defines an isomorphism of bimonoidal functors $V(\Pcal_{X})\times V(\Pcal_{X})\to\Dcal(X/S)$.
        \item For virtual perfect complexes $E,F,G$, the isomorphism \eqref{eq:ch-otimes} is compatible with the natural associativity isomorphism  $(E\otimes F)\otimes G\to E\otimes (F\otimes G)$, and the associativity of the product in the Chern categories.  
    \end{enumerate}

    \item \label{item:ch-is-the-unit} 
    \begin{enumerate}
        \item\label{item:ch-acts-trivially} There is a canonical isomorphism
        \begin{equation}\label{eq:ch-OX-is-dead}
        [ \chfrak(\Ocal_{X})  ]_{X/S} \simeq 1.
    \end{equation}
        \item The natural isomorphism $E \to E \otimes \Ocal_X$ induces, together with \eqref{eq:ch-otimes} and \eqref{eq:ch-OX-is-dead}, an isomorphism of monoidal functors $V(\Pcal_X) \to \Dcal(X/S),$
    \begin{equation}\label{eq:OX-not-dead-yet}
       [ \chfrak(E) ]_{X/S} \to [ \chfrak(E \otimes \Ocal_X) ]_{X/S} \to [\chfrak(E)\cdot \chfrak(\Ocal_X) ]_{X/S} \to [\chfrak(E)]_{X/S}.
    \end{equation}
    This is the identity natural transformation of functors in $E$.
    \end{enumerate}

\end{enumerate}
\end{proposition}
\begin{proof}
For \eqref{item:chernsequence-5}, the statement is local on $S$ and we can suppose that $X$ and $S$ are divisorial. In this case, by Lemma \ref{lemma:divisorial} and the universal property of virtual categories, we can even assume $E$ and $F$ are vector bundles. By the splitting principles, we may suppose that $E$ and $F$ admit filtrations with line bundle quotients and that all the filtrations are split. Since the categorical characteristic class $\chfrak$ is additive, see Remark \ref{rmk:ch-E-otimes-F}, we are reduced to consider a product of two line bundles $L,L'$. Notice that because we are working with line bundles with rational coefficients, and because of Lemma \ref{lemma:signs}, this is independent of the order we develop the product. 

By rank triviality, a distribution of the form $[\chfrak(L)]$ is canonically isomorphic to $[\exp(\cfrak_1(L))]$. This is independent of the order in which one applies the isomorphism $[\cfrak_k(L)]\simeq 0$ for $k \geq 2$, as can be deduced from the compatibility of the operations in Corollary \ref{cor:properties-intersection-distribution}. Hence, in the case of rank one bundles, the natural isomorphism is constructed by developing $\exp(\cfrak_1(L \otimes L'))$ in a power series and applying the isomorphism of the type \eqref{eq:c1-tensor-product-distribution}. This corresponds to the formal identity  $\exp(x+y) = \exp(x) \exp(y)$ by formal manipulations. The stated properties follow from the construction.

The first part of \eqref{item:ch-is-the-unit} follows from the isomorphism $[\chfrak(\Ocal_X)]\simeq [\exp(\cfrak_{1}(\Ocal_{X}))]$, and the canonical triviality of the intersection bundles where a term of the form $\cfrak_1(\Ocal_X)$ appears,\linebreak see Lemma \ref{lemma:elementary-properties-segre} \eqref{item:elpropseg-2}. This is again independent of the order in which we apply the isomorphisms $[\cfrak_1(\Ocal_X)]\simeq 0$ in the expression for $[\exp(\cfrak_{1}(\Ocal_{X}))]$, as can also be concluded from the corresponding properties of Deligne pairings.

The second part of the second point, by the construction of \eqref{eq:ch-otimes}, reduces to the case when $E$ is simply a line bundle. One then unravels the definitions, and reduces to show that the following analogous composition of morphisms is the identity:
\begin{displaymath}
    \langle L_{0}, L_1, \ldots, L_n \rangle \to \langle L_{0} \otimes \Ocal_X, L_1, \ldots, L_n \rangle \to \langle L_{0} , L_1, \ldots, L_n \rangle \otimes \langle \Ocal_X, L_1, \ldots, L_n \rangle \to \langle L_{0} , L_1, \ldots, L_n\rangle.
\end{displaymath}
This can be verified at the level of symbols. 

\end{proof}

\subsection{Todd class of the cotangent complex }\label{subsubsec:td-class-tangent}
In this subsection, we discuss in more detail the properties of the Todd class of the tangent complex discussed at the beginning of this section. 

Let $f: X\to Y$ be a local complete intersection morphism of schemes. As recalled at the beginning of this section, the cotangent complex for $f$ is a perfect complex on $X$, of Tor-amplitude contained in $[-1,0]$, $L_{X/Y}$. Suppose now we are moreover given a morphism $X \to S$ satisfying the condition $(C_n)$. Then, there is an associated distribution $[\tdfrak^*(L_{X/Y})]_{X/S}.$

Given a factorization $\sigma$ of $f$, written as $X\to Q\to Y$, where $i\colon X\to Q$ is a Koszul-regular closed immersion and $p\colon Q \to Y$ is a smooth morphism, there is an associated  complex
\begin{equation}\label{eq:cotangent}
    L_{X/Y, \sigma} = [N_{X/Q}^\vee \to i^* \Omega_{Q/Y}].
\end{equation}
Here, the conormal bundle is placed in degree $-1$. There is a canonical isomorphism \linebreak   $L_{X/Y, \sigma}\simeq L_{X/Y}$ in the derived category \cite[\href{https://stacks.math.columbia.edu/tag/0FV4}{0FV4}]{stacks-project}. From this, we conclude the following lemma:

\begin{lemma}\label{lemma:td-factorisation-independence}
Let $X \to S$ and $f\colon X\to Y$ be as above. 
\begin{enumerate} 
    \item Given $\sigma$ a factorization $X\to Q\to Y$ of $f$, there is a canonical isomorphism
    \begin{displaymath}
        [\tdfrak^*(L_{X/Y})]_{X/S} \to  [\tdfrak^*(L_{X/Y,\sigma})]_{X/S} \to [\tdfrak^*(i^\ast \Omega_{Q/Y}) \cdot \tdfrak^*(N^\vee_{X/Q})^{-1}]_{X/S}
    \end{displaymath}
    of line distributions for $X\to S$.
    \item Suppose $Y$ and $Z$ are schemes with morphisms to $S$ satisfying the conditions $(C_m)$ and $(C_{\ell})$.  Suppose we are given a local complete intersection morphism $g \colon Y\to Z$. Then there is a canonical isomorphism of line distributions: 
    \begin{displaymath}
        [\tdfrak^*(L_{X/Z})]_{X/S} \to [\tdfrak^*(L_{X/Y}) \cdot \tdfrak^* (f^\ast L_{Y/Z})]_{X/S}.
    \end{displaymath} 
\end{enumerate}
\end{lemma}

\begin{proof}
From the canonical isomorphism $L_{X/Y}\simeq L_{X/Y,\sigma}$ in the derived category \cite[\href{https://stacks.math.columbia.edu/tag/08TJ}{08TJ}]{stacks-project}, together with Proposition \ref{prop:derivedvirtual} and Corollary \ref{cor:characteristic-derived}, we infer a canonical isomorphism \linebreak 
 $\tdfrak^{\ast}(L_{X/Y})\simeq \tdfrak^{\ast}(L_{X/Y,\sigma})$ in $\CHfrak(X)_{\QBbb}$. This entails the first claim, by taking the associated distributions. 

If $\sigma'$ is another factorization, there is hence an isomorphism in the derived category \linebreak  $L_{X/Y, \sigma} \to L_{X/Y, \sigma'}$. Indeed, if $\sigma$ dominates $\sigma'$, then pullback functoriality induces a quasi\--\linebreak isomorphism of complexes $L_{X/Y, \sigma'} \to L_{X/Y, \sigma}$, see \cite[Expos\'e VIII, proof of Proposition 2.2]{SGA6}. As in \emph{loc. cit.}, the general such isomorphism is constructed by dominating $\sigma$ and $\sigma'$ respectively by the diagonal factorization $\sigma \times \sigma'$. This description will be useful later in this proof. 

For the second point, we notice that, while there is a distinguished triangle \linebreak  $Lf^\ast L_{Y/Z} \to L_{X/Z} \to L_{X/Y} \to Lf^\ast L_{Y/Z}[1]$, it does not automatically imply the statement since the virtual categories are a priori only additive with respect to true triangles.

The question is local and we suppose that $S$ is affine and all the schemes are projective over $S$. It follows that all the morphisms admit global factorizations as a regular closed immersion followed by a smooth morphism. Suppose we fix factorizations $\tau\colon X\to Q'\to Z$ and $\tau^{\prime\prime}\colon Y\to Q\to Z$. We consider the natural diagram 
\begin{displaymath}
 \xymatrix{X \ar[r] \ar[dr] &  Y \times_{Z} Q' \ar[d] \ar[r] & Q \times_{Z} Q' \ar[d] \\
 & Y \ar[rd] \ar[r] & Q \ar[d] \\ 
 & & Z,}
\end{displaymath}
where the square is Cartesian and we notice that the top horizontal arrows are regular immersions. The outer triangle provides another factorization of $X\to Z$, which we denote by $\widetilde{\tau}$. The lower right triangle factorization is $\tau^{\prime \prime}$, and we denote the upper left by $\widetilde{\tau}'$. By the proof of \linebreak  \cite[Expos\'e VIII, Proposition 2.6]{SGA6}, there is an associated exact sequence of the associated complexes, defined as in \eqref{eq:cotangent}: 
\begin{equation}\label{eq:exactcotangentsequence}
    0 \to g^{\ast}  L_{Y/Z, \tau^{\prime \prime}} \to L_{X/Z, \widetilde{\tau}}\to L_{\substack{\vspace{-0.08cm}\\ X/Y, \widetilde{\tau}^{\prime}}}\to 0.
\end{equation}
By the multiplicativity of $\tdfrak^{\ast}$ for short exact sequences of complexes of vector bundles, we derive an isomorphism as in the second point, possibly depending on the factorizations. We denote it by $\Phi_{\tau,\tau^{\prime\prime}}$. 
Suppose we are given a second set of factorizations $\mu, \mu''$. If the second set of factorizations dominates the first one, as explained above there are induced maps of the corresponding cotangent complexes, which fit into actual morphisms of corresponding exact sequences of complexes \eqref{eq:exactcotangentsequence}. We deduce that $\Phi_{\tau,\tau^{\prime\prime}}=\Phi_{\mu,\mu^{\prime \prime}}$. In general, two sets of factorizations can be dominated by a third one, via a diagonal argument.

\end{proof}

\subsection{The Borel--Serre isomorphism}\label{subsec:Borel-Serre}

The classical Borel--Serre identity \cite[Lemme 18]{Borel-Serre} expresses the Chern character of a Koszul-type complex in terms of the Todd class. As such, it plays an essential role in the Grothendieck--Riemann--Roch theorem for closed immersions. In this subsection, we establish the corresponding distributional version and provide natural compatibilities with other operations. 

Let now $E$ be a vector bundle $E$ of constant rank $r$ on $X$. Consider the virtual vector bundle
\begin{equation}\label{eq:lambda-1def}
    \lambda_{-1}(E) = \sum_{p=0}^r (-1)^p [\Lambda^p E]
\end{equation}
in $V(X)$. Notice that if $\sigma$ is a section of $E$, the Koszul complex thereof is 
\begin{equation}\label{def:lambdaKosz}
    K(\sigma)  = \left[\Lambda^r E^\vee \to \Lambda^{r-1} E^\vee \to \ldots \to E^{\vee} \to \Ocal_X\right],
\end{equation}
where the term $\Lambda^{k}E$ is placed in degree $-k$. Hence, $\lambda_{-1}(E^\vee)$ is isomorphic to $K(\sigma)$ in $V(X)$. In particular, if $K(\sigma)$ is acyclic, this provides a trivialization of $\lambda_{-1}(E^{\vee})$.

\begin{lemma}\label{lemma:lambda-1mult}
Let $f\colon X\to S$ be a morphism satisfying the condition $(C_{n})$. Suppose that \linebreak  $0 \to E^{\prime} \to E \to E^{\bis}\to 0 $ is a short exact sequence of vector bundles of constant ranks on $X$. Then there is a canonical isomorphism
\begin{displaymath}
     [ \chfrak(\lambda_{-1}(E)) ]_{X/S} \simeq [ \chfrak(\lambda_{-1}(E^{\prime})) \cdot \chfrak(\lambda_{-1}(E^{\bis})) ]_{X/S}
\end{displaymath}
of line distributions. It is: 
\begin{enumerate}
    \item Natural with respect to isomorphisms of short exact sequences.
    \item Compatible with admissible filtrations.
\end{enumerate} 
\end{lemma}
\begin{proof}

First of all, on $\Lambda^k E$ there is a standard natural filtration   

\begin{displaymath} F^i \Lambda^k E =\Imag ( \Lambda^{i} (E^{\prime}) \otimes \Lambda^{k-i}(E) \to \Lambda^k (E))
\end{displaymath} whose successive quotients are given by
\begin{equation}\label{eq:lambdafiltration}
    F^i\Lambda^k E/F^{i+1}\Lambda^k E \simeq  \Lambda^{i} (E^{\prime}) \otimes \Lambda^{k-i} (E^{\bis}).
\end{equation}
We deduce a canonical isomorphism in the virtual category
\begin{equation}\label{eq:isom-lambdas-virtual}
    [\Lambda^k (E)] \simeq \sum_{i=0}^k \left[ \Lambda^{i} (E^{\prime}) \otimes \Lambda^{k-i} (E^{\bis}) \right].
\end{equation}
The construction is clearly functorial with respect to isomorphisms of short exact sequences. 

From the definition \eqref{eq:lambda-1def} and \eqref{eq:isom-lambdas-virtual}, a formal computation provides an isomorphism in the virtual category
\begin{equation}\label{eq:isom-lambda-1-virtual}
    \lambda_{-1}(E) \simeq \lambda_{-1}(E^{\prime}) \otimes \lambda_{-1}(E^{\bis}).
\end{equation}
This depends on the choice of an isomorphism of the type $(-A) \otimes (-B) \simeq A \otimes B$. By \cite[\textsection 4.11 (b)]{Deligne-determinant}, the dependence is only up to sign \eqref{eq:isawthesign}. Therefore, by Remark \ref{rem:Q-Bousfield} \eqref{item:Q-Bousfield-1}, the isomorphism induced by \eqref{eq:isom-lambda-1-virtual} in $V(X)_{\QBbb}$ is well-defined.

The existence of the canonical isomorphism follows by applying the Chern character to this isomorphism, and using the tensor product multiplicativity of the associated line distributions \eqref{eq:ch-otimes} in Proposition \ref{prop:ch-multiplicative-properties} \eqref{item:chernsequence-5}. 

This construction is clearly base change functorial and functorial for isomorphisms of short exact sequences. The rest of the statement follows from general properties of the filtration \eqref{eq:lambdafiltration} with respect to admissible filtrations of $E$, again relying on Proposition \ref{prop:ch-multiplicative-properties} \eqref{item:chernsequence-5}, which ensures the compatibility of \eqref{eq:ch-otimes} with admissible filtrations. 
\end{proof}
\begin{remark}
Notice that, despite the relationship \eqref{eq:isom-lambda-1-virtual}, the construction of $\lambda_{-1}$ does not extend to the virtual category since the objects $\lambda_{-1}(E)$ are not invertible with respect to the tensor product.
\end{remark}

\begin{theorem}[Borel--Serre isomorphism]\label{thm:borelserre}
Let $E$ be a vector bundle of constant rank $r$ on $X$. Then, there is a canonical isomorphism of line distributions

\begin{equation}\label{eq:BS-isomorphism}
    [ \chfrak(\lambda_{-1}(E)) ]_{X/S} \simeq [ \cfrak_{r}(E^\vee) \cdot \tdfrak(E^\vee)^{-1}]_{X/S},
\end{equation}
which satisfies:
\begin{enumerate}

    \item It is compatible with short exact sequences of vector bundles $0 \to E^{\prime} \to E \to E^{\bis} \to 0$ of constant ranks $r',r,r''$, respectively, in the sense that the diagram of isomorphisms
\begin{displaymath}
    \xymatrix{ [ \chfrak(\lambda_{-1}(E)) ]_{X/S}  \ar[r] \ar[d] & [ \cfrak_{r}(E^\vee) \cdot \tdfrak(E^\vee)^{-1} ]_{X/S} \ar[d]  \\
     [ \chfrak(\lambda_{-1}(E^{\prime})) \cdot \chfrak(\lambda_{-1}(E^{\bis})) ]_{X/S} \ar[r] &  [ \cfrak_{r^{\prime}}((E^{\prime})^\vee) \cdot \tdfrak((E^{\prime})^\vee)^{-1} \cdot \cfrak_{r^{\prime\prime}}((E^{\bis})^\vee) \cdot \tdfrak((E^{\bis})^\vee)^{-1} ]_{X/S}}
\end{displaymath}
commutes. Here, the horizontal arrows are those of \eqref{eq:BS-isomorphism}. The left vertical isomorphism is that of  Lemma \ref{lemma:lambda-1mult}. Finally, the right vertical isomorphism is the one obtained by combining the Whitney isomorphism and the rank triviality in Corollary \ref{cor:properties-intersection-distribution}, and the multiplicativity for the Todd categorical class. 
    \item The isomorphism is compatible with admissible filtrations. 
    \item If $\sigma$ is a regular everywhere non-vanishing section of  $E^\vee$, 
both sides of \eqref{eq:BS-isomorphism} are canonically trivial, and the two trivializations correspond to each other. Here, the left-hand side is trivialized by the acyclicity of the associated Koszul complex $K(\sigma)$, see \eqref{def:lambdaKosz}, and the right-hand side is trivialized by Corollary \ref{cor:properties-intersection-distribution} \eqref{item:prop-int-dist-restriction}. 
\end{enumerate}
\end{theorem}

\begin{proof}
 We construct the isomorphism by induction on the rank of $E$, which is constant by assumption. If $E$ has rank one, given the rank triviality property in Proposition \ref{prop:properties-chern-bundles} \eqref{item:Chern-rank}, the isomorphism is a simple rewriting of Chern power series combined with the isomorphism \ref{prop:Cherndual} for duals of vector bundles.

Assume now that the isomorphism is constructed for vector bundles of rank up to $r-1$, and $E$ is of rank $r$. By the splitting principle in Proposition \ref{Prop:2ndsplitting}, we can suppose that $E$ admits a complete flag, of which the first line bundle sits in a sequence $0 \to L \to E \to Q \to 0$, where $Q$ also admits a complete flag. We then define the isomorphism by the commutativity of the diagram
\begin{equation}\label{eq:Borelserresplit}
    \xymatrix{ \langle P \cdot \chfrak(\lambda_{-1}(E)) \cdot P' \rangle  \ar[r] \ar[d] & \langle P \cdot \cfrak_{r}(E^\vee) \cdot \tdfrak(E^\vee)^{-1}\cdot P' \rangle \ar[d]  \\
     \langle P \cdot \chfrak(\lambda_{-1}(L)) \cdot \chfrak(\lambda_{-1}(Q))\cdot P' \rangle \ar[r] &  \langle P \cdot \cfrak_{1}(L^\vee) \cdot \tdfrak(L^\vee)^{-1} \cdot \cfrak_{r-1}(Q^\vee) \cdot \tdfrak(Q^\vee)^{-1}\cdot P' \rangle.}
\end{equation}
Here, the left vertical isomorphism is provided by Lemma \ref{lemma:lambda-1mult}, and the lower horizontal isomorphism is given by the induction hypothesis. The right vertical isomorphism is obtained by combining the Whitney isomorphism with the rank triviality, together with the multiplicativity of the Todd class. This proves the first two statements. 

The third statement is less immediate since the two trivializations are constructed in a priori completely different ways. We prove it by induction on the rank of $E$. If $E$ is a line bundle, it is trivialized by $\sigma^{\vee}$ and we can suppose $E = \Ocal_X$. In this case, we recall from Proposition \ref{prop:ch-multiplicative-properties} \eqref{item:ch-acts-trivially} that there is a canonical isomorphism $[\chfrak(\Ocal_{X})]\simeq 1$. It relies on the triviality $[\cfrak_k(\Ocal_X)]\simeq 0$ for $k\geq 1$, which can be applied in any order in the construction of the isomorphism. Now, consider the diagram of isomorphisms

\begin{displaymath}
    \xymatrix{[ \chfrak(\Ocal_X)  - \chfrak(\Ocal_X)] \ar[r] \ar[d] & 1-[\chfrak(\Ocal_X)] \ar[d] \ar[r] & [\cfrak_1(\Ocal_X)\cdot  \tdfrak(\Ocal_X)^{-1}]  \ar[d] \\
    0 \ar[r] & 1-1 \ar[r] & [0 \cdot  1]=0.
    }
\end{displaymath}
The upper horizontal isomorphism is the Borel--Serre isomorphism for the trivial bundle, where we have identified $\Ocal_{X}^{\vee}$ to $\Ocal_{X}$. The vertical isomorphisms are the canonical trivializations. The statement that the outer contour commutes corresponds to the desired compatibility of the trivializations in the rank one case. We first remark that the leftmost square commutes for formal reasons. The rightmost square commutes too. Indeed, by construction, the upper right horizontal morphism is obtained by the rank triviality $[\cfrak_{k}(\Ocal_{X})]=0$ for $k\geq 2$ and then rewriting $1-[\exp(\cfrak_{1}(\Ocal_{X}))]$, and both vertical isomorphisms are defined in terms of the same rank triviality together with $[\cfrak_1(\Ocal_X)]\simeq 0.$

In the general case, $\sigma$ corresponds to an exact sequence of vector bundles of constant rank: $0 \to \Ocal_X \to E^\vee \to F^\vee \to 0.$ Consider the below diagram:

\begin{equation}\label{eq:diagramfromhell}
\xymatrix@R=20pt@C=30pt{
 [\chfrak(\lambda_{-1}(E))] \ar[r] \ar[d]  \ar@/_9.0pc/@{->}[dd] & [\cfrak_r(E^\vee) \cdot \tdfrak(E^\vee)^{-1}] \ar[d] \ar@/^9.0pc/@{->}[dd] \\
 [\chfrak(\lambda_{-1}(\Ocal_X)) \cdot \chfrak(\lambda_{-1}(F)) ] \ar[r] \ar[d] & [\cfrak_1(\Ocal_X) \cdot \tdfrak(\Ocal_X)^{-1} \cdot \cfrak_{r-1}(F^\vee) \cdot \tdfrak(F^\vee)^{-1}] \ar[d] \\
 0 \ar[r]^{=} & 0.
} 
\end{equation}
Here, the two upper horizontal morphisms are the Borel--Serre isomorphisms, again identifying $\Ocal_X^\vee \simeq \Ocal_X$. The upper vertical morphisms are the morphisms corresponding to the exact sequence $0 \to \Ocal_X \to E^\vee \to F^\vee \to 0$ and the vertical morphisms in \eqref{eq:Borelserresplit}. The lower vertical isomorphisms are the canonical trivializations induced by $[\chfrak(\lambda_{-1}(\Ocal_{X}))]\simeq 0$ and $[\cfrak_{1}(\Ocal_{X})]\simeq 0$. The curved downwards isomorphisms are the trivializations in the third point of the theorem, and the statement reduces to the assertion that the diagram of the outer contour commutes. The upper middle square commutes by construction, and the lower middle square commutes because of the above treated case. We are left to show that the leftmost and rightmost triangles commute. The rightmost triangle commutes by Proposition \ref{prop:newranktriviality}. Below, we address the leftmost triangle.

First, we notice that the construction of the isomorphism \eqref{eq:isom-lambda-1-virtual} for the exact sequence\linebreak $0\to F\to E\to\Ocal_{X}\to 0$ relies on the induced exact sequences 
\begin{equation}\label{eq:filKosz}
    0\to\Lambda^{k}F\to \Lambda^{k}E\to\Lambda^{k-1}F\to 0,
\end{equation}
where it can be verified that the first map is the inclusion and the second is naturally isomorphic to the image of the Koszul differential $d_{-k}\colon\Lambda^{k}E\to\Lambda^{k-1}E$. Now, the trivialization of the virtual class of the Koszul complex is obtained by the isomorphism
\begin{displaymath}
    [K(\sigma)]\simeq \sum_{k}(-1)^{k} \left([\ker d_{k}]+ [\Imag d_{k}]\right)  
\end{displaymath}
and the equality $\Imag d_{k} = \ker d_{k-1}$. The latter are described in \eqref{eq:filKosz}, which also provides the isomorphism  $\lambda_{-1}(E) \simeq \lambda_{-1}(\Ocal_X) \otimes \lambda_{-1}(F)$. The identification of the kernels and the cokernels corresponds to the obvious triviality $\lambda_{-1}(\Ocal_X) \simeq 0.$  It follows that under the isomorphism $\lambda_{-1}(E)\simeq \lambda_{-1}(\Ocal_X) \otimes \lambda_{-1}(F)$ in 
\eqref{eq:isom-lambda-1-virtual}, the trivialization of $\lambda_{-1}(E)$ coming from the Koszul complex corresponds to the trivialization  $\lambda_{-1}(\Ocal_X)\simeq 0.$ 

Since the functor $[\chfrak(E \otimes F)]$ is bimonoidal in $E$ and $F$ by Lemma \ref{lemma:lambda-1mult}, we conclude that the two trivializations of 
\begin{displaymath}
    [\chfrak(\lambda_{-1}(E))] \simeq [\chfrak(\lambda_{-1}(\Ocal_X) \otimes \lambda_{-1}(F))] \simeq [\chfrak(\lambda_{-1}(\Ocal_X)) \cdot \chfrak(\lambda_{-1}(F))] 
\end{displaymath}
appearing in the left triangle of \eqref{eq:diagramfromhell} correspond to each other. This concludes the proof.

\end{proof}

As an application, we find the following primitive version of the Deligne--Riemann--Roch isomorphism, for regular closed immersions defined by the zeros of a given section of a vector bundle. 

\begin{corollary}\label{cor:Borel-Serre-consequence}
    Let $X\to S$ be a morphism satisfying the condition $(C_n)$ and $E$ a vector bundle of constant rank $r$ on $X$. Suppose that $\sigma$ is a regular section of $E$,  whose zero locus $Y$, possibly empty, is flat over $S$. Then there is a canonical isomorphism, a priori depending on $\sigma$,
    \begin{displaymath}
        [\chfrak(i_! \ \Ocal_Y)]_{X/S} \to i_\ast [\tdfrak(N_{X/Y})^{-1}]_{Y/S}
    \end{displaymath}
    of line distributions. If $Y$ is empty, then both sides are canonically trivial, and the isomorphism identifies with the identity isomorphism.

\end{corollary}

\begin{proof}
This is essentially a rewriting of the Borel--Serre isomorphism \eqref{eq:BS-isomorphism} in Theorem \ref{thm:borelserre}, replacing $E$ by $E^\vee$. Indeed, on the one hand, the Koszul complex $K(\sigma)$ provides a resolution of $i_{\ast}\Ocal_{Y}$, so that $i_{!}\ \Ocal_{Y}\simeq K(\sigma)$ in $V(X)$. Recalling the definition of the Koszul resolution in \eqref{def:lambdaKosz}, we see that  $[\chfrak(\lambda_{-1}(E^{\vee}))]_{X/S}\simeq [\chfrak(i_{!}\ \Ocal_{Y})]_{X/S}$. On the other hand, since $\sigma$ is a regular section, by the restriction property in Corollary \ref{cor:properties-intersection-distribution} \eqref{item:prop-int-dist-restriction} there is an isomorphism $[\cfrak_r(E)]_{X/S} \simeq \delta_{Y/S}$, and it is standard that $E|_{Z} \simeq N_{X/Z}$. We deduce $[\cfrak_{r}(E)\cdot\tdfrak(E)^{-1}]_{X/S}\simeq i_{\ast}[\tdfrak(N_{X/Y})^{-1}]_{Y/S}$. 

Notice that the above proof works independently of if $Y$ is empty or not. 
\end{proof}

\begin{corollary}\label{cor:Borel-Serre-consequence2}
Under the assumptions of Corollary \ref{cor:Borel-Serre-consequence}, suppose moreover that there exists a retraction $p\colon X\to Y$. Then, for any virtual perfect complex $F$ on $Y$, there exists an isomorphism of line distributions

\begin{displaymath}
        [\chfrak(i_!\ F)]_{X/S} \to i_\ast [\chfrak(F) \cdot \tdfrak(N_{X/Y})^{-1}]_{Y/S},
    \end{displaymath}
which is functorial in $F$. That is, it defines an isomorphism of functors of commutative Picard categories $V(\Pcal_{Y})\to\Dcal(X/S)$.
\end{corollary}

\begin{proof}
    Suppose we multiply the isomorphism in Corollary \ref{cor:Borel-Serre-consequence} by the class $\chfrak(p^\ast F)$. By Proposition \ref{prop:ch-multiplicative-properties} \eqref{item:chernsequence-5} and  Proposition \ref{prop:properties-direct-image-virtual} \eqref{item:properties-direct-image-virtual-2}, the left-hand side is isomorphic to 
    \begin{displaymath}
        [\chfrak(p^\ast F \otimes i_!\ \Ocal_X)]_{X/S} \simeq [\chfrak(i_!\ i^\ast p^\ast F)]_{X/S} \simeq [\chfrak(i_!\ F)]_{X/S}.
    \end{displaymath}
    On the other hand, by the projection formula for line distributions \eqref{linedist:projformula}, we have similarly that the right-hand side is isomorphic to 
    \begin{displaymath}
        i_\ast [\chfrak(i^\ast p^\ast F) \cdot \tdfrak(N_{X/Y})^{-1}]_{Y/S} \simeq i_\ast [\chfrak(F) \cdot \tdfrak(N_{X/Y})^{-1}]_{Y/S}.
    \end{displaymath}
The functoriality claim is automatic, by the pullback functoriality of the Chern categories and the virtual categories, together with Proposition \ref{prop:ch-multiplicative-properties} \eqref{eq:ch-otimes-a}.

\end{proof}
\subsection{Formulation of Deligne--Riemann--Roch}\label{subsec:FormDRR}
In this subsection, we propose a formulation, in the framework of line distributions, of the existence of a functorial Riemann--Roch isomorphism. It contains, as a special case, the conjectured functorial form of the Riemann--Roch theorem in the program by Deligne in \cite{Deligne-determinant}. For simplicity, we work in the category of divisorial schemes. 

Let $X\to S$ and $Y\to S$ be morphisms satisfying the conditions $(C_{n})$ and $(C_{m})$, and $f\colon X\to Y$ be a local complete intersection morphism of $S$-schemes. Let $E$ be a virtual perfect complex on $X$. Recall, from the introduction of this section, the line distribution $f_{\ast}[\chfrak(E)\cdot\tdfrak^{\ast}(L_{X/Y})]_{X/S}$. As for any line distribution, it is defined over any $S$-scheme, say $u: S^{\prime}\to S$. Hence, it makes sense to compare it with the distribution obtained from the base changed morphism $f^{\prime}: X^{\prime} \to Y^{\prime}$. We will implicitly identify these distributions over $S^{\prime}$. For this comparison to make sense, it is enough to notice that the morphisms $Y^{\prime} \to Y$ and $X \to Y$ are Tor-independent, as is readily verified from the fact that $X \to S$ and $Y \to S$ are flat, so that there is a canonical isomorphism $u^{\ast} L_{X/Y} \simeq L_{X^{\prime}/Y^{\prime} }$. 

In a similar vein, we will consider the line distribution $[\chfrak(f_{!}\ E)]_{Y/S}$ which, to every $u\colon S^{\prime}\to S$ and every object $P$ of $\CHfrak(Y^{\prime})_{\QBbb}$, associates
\begin{displaymath}
    [\chfrak(f_!\ E)]_{Y^{\prime}/S^{\prime}}(P^{\prime})=\langle\chfrak(f^{\prime}_{!} u^{\ast} E)\cdot P^{\prime}\rangle_{Y^{\prime}/S^{\prime}}.
\end{displaymath}
This indeed defines a line distribution since by  Proposition \ref{prop:properties-direct-image-virtual} \eqref{item:properties-direct-image-virtual-3} there is a natural identification $f^{\prime}_! (u^{\ast} E) \simeq u^{\ast} (f_!\ E)$.

Notice that the line distributions so constructed give rise to functors $V(\Pcal_X)\to \Dcal(Y/S)$ of commutative Picard categories. With this in mind, we provide the following definition:
\begin{definition}
  A Deligne--Riemann--Roch isomorphism for $f$ is an isomorphism
\begin{equation}\label{eq:RR-isomorphism}
    R_{f}(E)\colon [\chfrak(f_!\ E)]_{Y/S}\to f_\ast [\chfrak(E) \cdot \tdfrak^{\ast}(L_{X/Y})]_{X/S}
\end{equation}
of line distributions. 

Considering both sides of the isomorphism \eqref{eq:RR-isomorphism} as functors $V(\Pcal_X)\to \Dcal(Y/S)$ of commutative Picard categories, the assignment $E\mapsto R_{f}(E)$ is further required to be an isomorphism of such functors. 
\end{definition}

We notice that for a given Chern power series $P$ and vector bundle $E$, at least when $S$ is divisorial and $f: X \to S$ and $g: Y \to S$ are globally projective, there is an identity of Chern classes 
\begin{displaymath}
    g_\ast \left(\ch(f_!\ [E]) \cdot [P] \right)= g_\ast \left(f_{\ast}(\ch([E]) \cdot \td^{\ast}(L_{X/Y})) \cdot [P]\right).
\end{displaymath}
Here, $[P]$ is the class in the Chow ring of $P $ and $[E]$ is the class in $K_0(X)$ of $E.$ This follows from the classical Grothendieck--Riemann--Roch theorem for $X \to Y$ in \cite[Expos\'e VIII]{SGA6}, together with the Chern class computation of intersection bundles in Proposition \ref{prop:chernclassofIntbundle}. Taking the codimension one part, this shows that there is always an isomorphism of the $\QBbb$-line bundles underlying a Deligne--Riemann--Roch isomorphism as in \eqref{eq:RR-isomorphism}. 

\begin{conjecture}
There is a canonical Deligne--Riemann--Roch isomorphism. 

\end{conjecture} 
Ideally, one expects a natural such isomorphism to satisfy further compatibilities:

\begin{itemize}
    \item Multiplying both sides of \eqref{eq:RR-isomorphism} with an object $\chfrak(F)$ for a virtual perfect complex $F$ on $Y$, one can transform both sides according to the projection formulas in Proposition \ref{prop:properties-direct-image-virtual} \eqref{item:properties-direct-image-virtual-2} and \eqref{linedist:projformula}. One expects the isomorphism $R_{f}$ to interchange these projection formulas. 
    \item Compatibility with compositions $X \to Y \to Z$ of schemes over $S.$
    \item Compatibility with the Corollary \ref{cor:Borel-Serre-consequence} and Corollary \ref{cor:Borel-Serre-consequence2}.
\end{itemize}
The above list is not exhaustive, in that we expect to need more properties in order to characterize $R_{f}$.

\begin{corollary}
Let $f: X \to S$ be a local complete intersection morphism satisfying the condition $(C_n).$ If there exists a Deligne--Riemann--Roch isomorphism for $f$, then there is a functorial isomorphism  
\begin{equation}\label{eq:DRRdet}
    \lambda_{f}(E) \simeq \langle \chfrak(E) \cdot \tdfrak^{\ast}(L_{X/S}) \rangle_{X/S}.
\end{equation}
In particular, it is functorial in $E$ and compatible with base change.
\end{corollary}

\begin{proof}
    First, we evaluate the isomorphism of line distributions for the Chern power series $P = 1.$ This immediately provides the right-hand side of the isomorphism \eqref{eq:DRRdet}. The left-hand side is given by $\langle \chfrak(f_!\ E) \rangle_{S/S}  = \langle \cfrak_1(f_!\ E) \rangle_{S/S} $. By the first Chern class isomorphism, see also\linebreak Proposition \ref{prop:isom-first-chern-functorial}, the latter is isomorphic to $ \langle \cfrak_1(\det f_!\ E) \rangle_{S/S}, $ which in turn is the line bundle $\lambda_{f}(E).$
\end{proof}

\providecommand{\bysame}{\leavevmode\hbox to3em{\hrulefill}\thinspace}
\providecommand{\MR}{\relax\ifhmode\unskip\space\fi MR }
\providecommand{\MRhref}[2]{%
  \href{http://www.ams.org/mathscinet-getitem?mr=#1}{#2}
}
\providecommand{\href}[2]{#2}


\begin{thebibliography}{10}

\bibitem{SGA6}
\emph{Th\'{e}orie des intersections et th\'{e}or\`eme de {R}iemann-{R}och
  ({SGA} 6)}, Lecture Notes in Mathematics, Vol. 225, Springer-Verlag,
  Berlin-New York, 1971, S\'{e}minaire de G\'{e}om\'{e}trie Alg\'{e}brique du
  Bois-Marie 1966--1967 (SGA 6), Dirig\'{e} par P. Berthelot, A. Grothendieck
  et L. Illusie. Avec la collaboration de D. Ferrand, J. P. Jouanolou, O.
  Jussila, S. Kleiman, M. Raynaud et J. P. Serre.

\bibitem{BDRR}
N.~A. Baas, B.~I. Dundas, B.~Richter, and J.~Rognes, \emph{Ring completion of
  rig categories}, J. Reine Angew. Math. \textbf{674} (2013), 43--80.
  \MR{3010546}

\bibitem{freed1}
J.-M. Bismut and D.~Freed, \emph{The analysis of elliptic families. {I}.
  {M}etrics and connections on determinant bundles}, Comm. Math. Phys.
  \textbf{106} (1986), no.~1, 159--176.

\bibitem{freed2}
\bysame, \emph{The analysis of elliptic families. {II}. {D}irac operators, eta
  invariants, and the holonomy theorem}, Comm. Math. Phys. \textbf{107} (1986),
  no.~1, 103--163.

\bibitem{BGS1}
J.-M. Bismut, H.~Gillet, and C.~Soul{\'e}, \emph{Analytic torsion and
  holomorphic determinant bundles. {I}. {B}ott-{C}hern forms and analytic
  torsion}, Comm. Math. Phys. \textbf{115} (1988), no.~1, 49--78.

\bibitem{BGS2}
\bysame, \emph{Analytic torsion and holomorphic determinant bundles. {II}.
  {D}irect images and {B}ott-{C}hern forms}, Comm. Math. Phys. \textbf{115}
  (1988), no.~1, 79--126.

\bibitem{BGS3}
\bysame, \emph{Analytic torsion and holomorphic determinant bundles. {III}.
  {Q}uillen metrics on holomorphic determinants}, Comm. Math. Phys.
  \textbf{115} (1988), no.~2, 301--351.

\bibitem{Borel-Serre}
A.~Borel and J.-P. Serre, \emph{Le th\'{e}or\`eme de {R}iemann-{R}och}, Bull.
  Soc. Math. France \textbf{86} (1958), 97--136.

\bibitem{Boucksom-Eriksson}
S.~Boucksom and D.~Eriksson, \emph{Spaces of norms, determinant of cohomology
  and {F}ekete points in non-{A}rchimedean geometry}, Adv. Math. \textbf{378}
  (2021), Paper No. 107501, 124.

\bibitem{SebWalMarDiff}
S.~Boucksom, W.~Gubler, and F.~Martin, \emph{Differentiability of relative
  volumes over an arbitrary non-{A}rchimedean field}, Int. Math. Res. Not. IMRN
  (2022), no.~8, 6214--6242.

\bibitem{Deligne-determinant}
P.~Deligne, \emph{Le d\'{e}terminant de la cohomologie}, Current trends in
  arithmetical algebraic geometry ({A}rcata, {C}alif., 1985), Contemp. Math.,
  vol.~67, Amer. Math. Soc., Providence, RI, 1987, pp.~93--177.

\bibitem{Ducrot}
F.~Ducrot, \emph{Cube structures and intersection bundles}, J. Pure Appl.
  Algebra \textbf{195} (2005), no.~1, 33--73.

\bibitem{Elkikfib}
R.~Elkik, \emph{Fibr\'{e}s d'intersections et int\'{e}grales de classes de
  {C}hern}, Ann. Sci. \'{E}cole Norm. Sup. (4) \textbf{22} (1989), no.~2,
  195--226.

\bibitem{Elkikmetr}
\bysame, \emph{M\'{e}triques sur les fibr\'{e}s d'intersection}, Duke Math. J.
  \textbf{61} (1990), no.~1, 303--328.

\bibitem{EriComptes}
D.~Eriksson, \emph{Un isomorphisme de type {D}eligne-{R}iemann-{R}och}, C. R.
  Math. Acad. Sci. Paris \textbf{347} (2009), no.~19-20, 1115--1118.

\bibitem{Eriksson-Freixas-Wentworth}
D.~Eriksson, G.~Freixas i~Montplet, and R.~A. Wentworth, \emph{Complex
  {C}hern--{S}imons bundles in the relative setting},
  \url{https://arxiv.org/abs/2109.02033}, 2021.

\bibitem{FrankeChow}
J.~Franke, \emph{Chow categories}, vol.~76, 1990, Algebraic geometry (Berlin,
  1988), pp.~101--162.

\bibitem{FrankeChern}
\bysame, \emph{Chern functors}, Arithmetic algebraic geometry ({T}exel, 1989),
  Progr. Math., vol.~89, Birkh\"{a}user Boston, Boston, MA, 1991, pp.~75--152.

\bibitem{Franke}
J~Franke, \emph{Riemann--{R}och in functorial form}, Unpublished, 1992.

\bibitem{Fulton}
W.~Fulton, \emph{Intersection theory}, second ed., Ergebnisse der Mathematik
  und ihrer Grenzgebiete. 3. Folge. A Series of Modern Surveys in Mathematics,
  vol.~2, Springer-Verlag, Berlin, 1998.

\bibitem{FultonLang}
W.~Fulton and S.~Lang, \emph{Riemann-{R}och algebra}, Grundlehren der
  mathematischen Wissenschaften [Fundamental Principles of Mathematical
  Sciences], vol. 277, Springer-Verlag, New York, 1985.

\bibitem{Gabriel-Zisman}
P.~Gabriel and M.~Zisman, \emph{Calculus of fractions and homotopy theory},
  Ergebnisse der Mathematik und ihrer Grenzgebiete, Band 35, Springer-Verlag
  New York, Inc., New York, 1967.

\bibitem{GS-ARR}
H.~Gillet and S.~Soul\'{e}, \emph{An arithmetic {R}iemann-{R}och theorem},
  Invent. Math. \textbf{110} (1992), no.~3, 473--543.

\bibitem{JG}
P.~Goerss and J.~Jardine, \emph{Simplicial homotopy theory}, Progress in
  Mathematics, vol. 174, Birkh\"{a}user Verlag, Basel, 1999.

\bibitem{Grayson}
D.~Grayson, \emph{Higher algebraic {$K$}-theory. {II} (after {D}aniel
  {Q}uillen)}, Algebraic {$K$}-theory ({P}roc. {C}onf., {N}orthwestern {U}niv.,
  {E}vanston, {I}ll., 1976), 1976, pp.~217--240. Lecture Notes in Math., Vol.
  551.

\bibitem{EGAII}
A.~Grothendieck, \emph{\'{E}l\'{e}ments de g\'{e}om\'{e}trie alg\'{e}brique.
  {II}. \'{E}tude globale \'{e}l\'{e}mentaire de quelques classes de
  morphismes}, Inst. Hautes \'{E}tudes Sci. Publ. Math. (1961), no.~8, 222.

\bibitem{EGAIV1}
\bysame, \emph{\'{E}l\'{e}ments de g\'{e}om\'{e}trie alg\'{e}brique. {IV}.
  \'{E}tude locale des sch\'{e}mas et des morphismes de sch\'{e}mas. {I}},
  Inst. Hautes \'{E}tudes Sci. Publ. Math. (1964), no.~20, 259.

\bibitem{EGAIV2}
\bysame, \emph{\'{E}l\'{e}ments de g\'{e}om\'{e}trie alg\'{e}brique. {IV}.
  \'{E}tude locale des sch\'{e}mas et des morphismes de sch\'{e}mas. {II}},
  Inst. Hautes \'{E}tudes Sci. Publ. Math. (1965), no.~24, 231.

\bibitem{EGAIV3}
\bysame, \emph{\'{E}l\'{e}ments de g\'{e}om\'{e}trie alg\'{e}brique. {IV}.
  \'{E}tude locale des sch\'{e}mas et des morphismes de sch\'{e}mas. {III}},
  Inst. Hautes \'{E}tudes Sci. Publ. Math. (1966), no.~28, 255.

\bibitem{EGAIV4}
\bysame, \emph{\'{E}l\'{e}ments de g\'{e}om\'{e}trie alg\'{e}brique. {IV}.
  \'{E}tude locale des sch\'{e}mas et des morphismes de sch\'{e}mas {IV}},
  Inst. Hautes \'{E}tudes Sci. Publ. Math. (1967), no.~32, 5--361.

\bibitem{Knudsen}
F.~Knudsen, \emph{Determinant functors on exact categories and their extensions
  to categories of bounded complexes}, Michigan Math. J. \textbf{50} (2002),
  no.~2, 407--444.

\bibitem{Knudsen-err}
\bysame, \emph{Erratum: ``{D}eterminant functors on exact categories and their
  extensions to categories of bounded complexes'' [{M}ichigan {M}ath. {J}. {\bf
  50} (2002), no. 2, 407--444; {MR}1914072 (2003k:18015a)] by {F}. {F}.
  {K}nudsen}, Michigan Math. J. \textbf{50} (2002), no.~3, 665.

\bibitem{KnudsenMumford}
F.~Knudsen and D.~Mumford, \emph{The projectivity of the moduli space of stable
  curves. {I}. {P}reliminaries on ``det'' and ``{D}iv''}, Math. Scand.
  \textbf{39} (1976), no.~1, 19--55.

\bibitem{Laplaza}
M.~L. Laplaza, \emph{Coherence for distributivity}, Coherence in categories,
  Springer-Verlag, Berlin-Heidelberg-New York, 1972, pp.~29--65. Lecture Notes
  in Math., Vol. 281.

\bibitem{LipNee}
J.~Lipman and A.~Neeman, \emph{Quasi-perfect scheme-maps and boundedness of the
  twisted inverse image functor}, Illinois J. Math. \textbf{51} (2007), no.~1,
  209--236.

\bibitem{MacLane}
S.~Mac~Lane, \emph{Natural associativity and commutativity}, Rice Univ. Stud.
  \textbf{49} (1963), no.~4, 28--46.

\bibitem{MacLane-cat}
\bysame, \emph{Categories for the working mathematician}, Graduate Texts in
  Mathematics, Vol. 5, Springer-Verlag, New York-Berlin, 1971.

\bibitem{Munoz}
E.~Mu\~{n}{o}z Garc\'{\i}a, \emph{Fibr\'{e}s d'intersection}, Compositio Math.
  \textbf{124} (2000), no.~3, 219--252.

\bibitem{Mumford-red}
D.~Mumford, \emph{The red book of varieties and schemes}, expanded ed., Lecture
  Notes in Mathematics, vol. 1358, Springer-Verlag, Berlin, 1999, Includes the
  Michigan lectures (1974) on curves and their Jacobians, With contributions by
  Enrico Arbarello.

\bibitem{Muro:determinant}
F.~Muro, A.~Tonks, and M.~Witte, \emph{On determinant functors and
  {$K$}-theory}, Publ. Mat. \textbf{59} (2015), no.~1, 137--233.

\bibitem{Nakai}
Y.~Nakai, \emph{Some fundamental lemmas on projective schemes}, Trans. Amer.
  Math. Soc. \textbf{109} (1963), 296--302.

\bibitem{Nakayamaintsheaves}
N.~Nakayama, \emph{Intersection sheaves over normal schemes}, J. Math. Soc.
  Japan \textbf{62} (2010), no.~2, 487--595.

\bibitem{Patel}
D.~Patel, \emph{de {R}ham {${\mathscr E}$}-factors}, Invent. Math. \textbf{190}
  (2012), no.~2, 299--355.

\bibitem{Quillen:K-theory-I}
D.~Quillen, \emph{Higher algebraic {$K$}-theory. {I}}, Algebraic {$K$}-theory,
  {I}: {H}igher {$K$}-theories ({P}roc. {C}onf., {B}attelle {M}emorial {I}nst.,
  {S}eattle, {W}ash., 1972), 1973, pp.~85--147. Lecture Notes in Math., Vol.
  341.

\bibitem{Richter}
B.~Richter, \emph{From categories to homotopy theory}, Cambridge Studies in
  Advanced Mathematics, vol. 188, Cambridge University Press, Cambridge, 2020.

\bibitem{Rossler-ARR}
D.~R\"{o}ssler, \emph{A local refinement of the {A}dams-{R}iemann-{R}och
  theorem in degree one}, Arithmetic {L}-functions and differential geometric
  methods, Progr. Math., vol. 338, pp.~213--246.

\bibitem{Saavedra}
N.~Saavedra~Rivano, \emph{Cat\'{e}gories {T}annakiennes}, Lecture Notes in
  Mathematics, Vol. 265, Springer-Verlag, Berlin-New York, 1972.

\bibitem{Sinh-these}
H.~X. S\'{\i}nh, \emph{{${\rm Gr}$}-cat\'{e}gories},
  \url{https://webusers.imj-prg.fr/~leila.schneps/grothendieckcircle/SinhThesis.pdf},
  1975, PhD. thesis, {I}nstitut p\'{e}dagogique no. 2 de {H}anoi
  ({D}\'{e}partement de math\'{e}matiques).

\bibitem{stacks-project}
The {Stacks Project Authors}, \emph{\textit{Stacks Project}},
  \url{https://stacks.math.columbia.edu}, 2023.

\bibitem{Suslin-Voevodsky}
A.~Suslin and Vl. Voevodsky, \emph{Relative cycles and {C}how sheaves}, Cycles,
  transfers, and motivic homology theories, Ann. of Math. Stud., vol. 143,
  Princeton Univ. Press, Princeton, NJ, 2000, pp.~10--86.

\bibitem{ThomasonPhony}
R.~Thomason, \emph{Beware the phony multiplication on {Q}uillen's {${\mathscr
  A}^{-1}{\mathscr A}$}}, Proc. Amer. Math. Soc. \textbf{80} (1980), no.~4,
  569--573.

\bibitem{ThomasonSpectral}
\bysame, \emph{First quadrant spectral sequences in algebraic {$K$}-theory via
  homotopy colimits}, Comm. Algebra \textbf{10} (1982), no.~15, 1589--1668.

\bibitem{ThomasonHomCol}
R.~W. Thomason, \emph{Homotopy colimits in the category of small categories},
  Math. Proc. Cambridge Philos. Soc. \textbf{85} (1979), no.~1, 91--109.

\bibitem{ThomasonTrobaugh}
R.~W. Thomason and T.~Trobaugh, \emph{Higher algebraic {$K$}-theory of schemes
  and of derived categories}, The {G}rothendieck {F}estschrift, {V}ol. {III},
  Progr. Math., vol.~88, Birkh\"{a}user Boston, Boston, MA, 1990, pp.~247--435.

\bibitem{Waldhausen}
F.~Waldhausen, \emph{Algebraic {$K$}-theory of spaces}, Algebraic and geometric
  topology ({N}ew {B}runswick, {N}.{J}., 1983), Lecture Notes in Math., vol.
  1126, pp.~318--419.

\end{thebibliography}
\end{document}